\documentclass[a4paper,reqno]{amsart}

\usepackage{adjustbox}
\usepackage[margin=2.5cm]{geometry}

\usepackage{stmaryrd}
\usepackage{amsmath}

\usepackage{float}
\usepackage{hhline}
\usepackage{array}
\usepackage[english]{babel}
\usepackage{amscd,amssymb,amsmath,amsthm,amsfonts}
\usepackage{mathtools}
\usepackage{accents}
\usepackage{tensor}
\usepackage{graphicx}
\usepackage{subfig}
\usepackage{tikz-cd}
\usetikzlibrary{calc,matrix,arrows,decorations.pathmorphing}
\usepackage{calc}
\usepackage{enumitem}
\usepackage{marginnote}
\usepackage{booktabs}
\usepackage{hyperref}
\graphicspath{ {./images/} }

\hypersetup{colorlinks=true,pageanchor=false,
	linkcolor=blue,citecolor=red,urlcolor=red}
\usepackage[all]{hypcap}
\usepackage{cleveref}
\usepackage{appendix}

\allowdisplaybreaks

\newcommand{\ie}{i.\,e.\ }

\newtheorem*{theorem*}{Theorem}
\newtheorem{theorem}{Theorem}
\newtheorem{corollary}[theorem]{Corollary}
\newtheorem{proposition}[theorem]{Proposition}
\newtheorem{lemma}[theorem]{Lemma}

\newtheorem{maintheorem}{Theorem}

\newenvironment{refmaintheorem}[1]{%
	 \addtocounter{maintheorem}{-1}
	\begin{maintheorem}%
	}{%
	\end{maintheorem}%
}

\theoremstyle{definition}
\newtheorem{definition}[theorem]{Definition}

\theoremstyle{remark}
\newtheorem{remark}[theorem]{Remark}

\newcommand{\Z}{\mathbb{Z}}

\newcommand{\R}{\mathbb{R}}
\newcommand{\C}{\mathbb{C}}

\newcommand{\Heis}{\mathcal{H}}
\newcommand{\cL}{\mathcal{L}}

\newcommand{\bbF}{\mathbb{F}}

\newcommand{\bbH}{\mathbb{H}}

\newcommand{\bbP}{\mathbb{P}}
\newcommand{\bbS}{\mathbb{S}}

\newcommand{\bfL}{\mathbf{L}}

\newcommand{\calA}{\mathcal{A}}
\newcommand{\calB}{\mathcal{B}}
\newcommand{\calC}{\mathcal{C}}

\newcommand{\calE}{\mathcal{E}}
\newcommand{\calF}{\mathcal{F}}
\newcommand{\calG}{\mathcal{G}}

\newcommand{\calK}{\mathcal{K}}
\newcommand{\calL}{\cL}

\newcommand{\calR}{\mathcal{R}}
\newcommand{\calS}{\mathcal{S}}
\newcommand{\calT}{\mathcal{T}}

\newcommand{\calW}{\mathcal{W}}

\newcommand{\id}{\mathrm{id}}
\newcommand{\Id}{\mathrm{Id}}

\newcommand{\Cob}{\mathrm{Cob}}
\newcommand{\Conf}{\mathrm{Conf}}
\newcommand{\Loc}{\mathrm{Loc}^{(-,-)}}
\newcommand{\Locpair}{\mathrm{Loc}^{\mathrm{pair}}}

\newcommand{\sqbinom}[2]{\left[ \begin{matrix} #1 \\ #2 \end{matrix} \right]}

\newcommand{\Hom}{\mathrm{Hom}}

\newcommand{\bS}{\mathbb{S}}
\newcommand{\Ker}{\mathrm{Ker}}
\newcommand{\Sym}{\mathrm{Sym}}

\makeatletter
\def\@settitle{%
	\begin{center}%
		\baselineskip14\p@\relax
		\bfseries \@title
	\end{center}%
}
\makeatother

\title{Homological Topological Quantum Field Theories}
\author[A. Andreev]{Aleksei Andreev}
\address{Institut f\"ur Mathematik,
	Universit\"at Z\"urich,
	Winterthurerstrasse 190,
	CH-8057 Z\"urich.}
\email{aleksei.andreev@math.uzh.ch} 

\begin{document}
		\begin{abstract}
		We develop a new framework for quantum invariants of $3$-manifolds by extending to cobordisms a homological construction of mapping class group representations. More specifically, we construct a $(2+1)$-dimensional topological quantum field theory (TQFT) that assigns to each surface the twisted homology of its unordered configuration space. The construction requires a choice of local systems on configuration spaces together with additional data. We formulate sufficient conditions on these data that guarantee the TQFT axioms, and we show that there are at least two useful examples satisfying them. One of them yields a homological construction of the projective Kerler--Lyubashenko TQFT, while the other recovers the Frohman--Nicas--Donaldson TQFT. In contrast to the classical algebraic constructions of quantum invariants, our approach is purely topological and relies on multi-trajectory spaces of cobordisms.
	\end{abstract}

	\maketitle

	\tableofcontents

	\section{Introduction}

	The topological nature of the Alexander polynomial led to significant achievements in low-dimensional topology, including genus bounds \cite{Crowell1959, Murasugi1958a, Murasugi1958b}, Freedman's theorem \cite{FreedmanQuinn1990}, knot concordance \cite{FoxMilnor1966}, and many related developments. At the same time, many quantum invariants of knots, links, or $3$- and $4$-manifolds are defined by means of sophisticated algebraic objects, such as quantum groups, braided tensor categories, or Hopf algebras, while their direct geometric or topological meaning often remains only partially understood. A natural problem is therefore to realize those algebraic constructions in terms of classical topological invariants. The program of seeking a homological interpretation of quantum invariants was initiated by Lawrence in the 1990s \cite{lawrence1990homological}, who gave the first homological construction of a quantum representation of the braid group. One of the greatest achievements of this approach was the proof of braid group linearity in the early 2000s by Bigelow \cite{bigelow2001braid} and Krammer \cite{krammer2002braid} independently. Later, homological constructions led to an interpretation of the Jones polynomial via intersection theory \cite{bigelow2002homological} and were generalized to various quantum link invariants \cite{anghel2022topological,anghel2023topological,bigelow2007homological}. More recently, Martel and Bigelow gave homological constructions of quantum $\mathfrak{sl}_2$-modules \cite{martel2022homological} and of quantum groups \cite{bigelow2024quantum}.

	Developing Lawrence's ideas, Blanchet, Palmer, and Shaukat \cite{blanchet2025heisenberg} studied a homological construction of a projective representation of surface mapping class groups. For a compact connected oriented surface $\Sigma$ with one boundary component, the mapping class group $\mathrm{Mod}(\Sigma,\partial\Sigma)$ acts on the twisted relative Borel--Moore homology of unordered configuration spaces $\Conf_*(\Sigma)=\bigsqcup_n\Conf_n(\Sigma)$, the so-called \emph{Heisenberg homology}:
	\[
	\bbH^{BM}(\Sigma;\calL(\Sigma)):=\bigoplus\limits_{n=0}^{\infty} H^{BM}_n(\Conf_*(\Sigma),\Conf^-_*(\Sigma);\calL(\Sigma)).
	\]
	Here $\calL(\Sigma)$ is a local system whose monodromy action factors through a group $\Heis(\Sigma)$ associated with $\Sigma$. The \emph{Heisenberg group} $\Heis(\Sigma)$ can be described either as a quotient of the surface braid group or, equivalently, as the central extension of $H_1(\Sigma)$ determined by the intersection pairing. When $\Sigma$ is oriented, one can choose a canonical generator $\sigma$ of the center of $\Heis(\Sigma)$. In practice we work with the finite Heisenberg group $\Heis_p(\Sigma)$, a finite quotient of $\Heis(\Sigma)$ in which $\sigma^{2p}=1$. Local systems whose monodromy factors through $\Heis_p(\Sigma)$ are called \emph{$p$-Heisenberg local systems} in this paper.

	De Renzi and Martel \cite{de2022homological} showed that for a specific choice of a $p$-Heisenberg local system, called the \emph{Schr\"odinger local system}, the subspace of \emph{small cycles}
	\[
	\bbF_{\calW_p}(\Sigma):=\operatorname{Im}\left[\bbH(\Sigma;\calW_p(\Sigma)) \to \bbH^{BM}(\Sigma;\calW_p(\Sigma))\right]
	\]
	carries a projective representation of $\mathrm{Mod}(\Sigma,\partial\Sigma)$ isomorphic to Lyubashenko's representation \cite{lyubashenko1995invariants} associated with the small quantum $\mathfrak{sl}_2$ at the $p$-th root of unity $\zeta=e^{2\pi i/p}$, where $p$ is odd and $p\geq 3$. Schr\"odinger local systems $\calW_p(\Sigma)$ arise from complex irreducible finite-dimensional representations of the finite Heisenberg group in which the central generator $\sigma$ acts by $-\zeta^{-2}$. Since the Lyubashenko representation is the mapping class group part of the Kerler--Lyubashenko TQFT \cite{kerler2002non}, it is natural to ask whether the homological construction extends from mapping class groups to cobordisms, and hence to an invariant of $3$-manifolds. This question is addressed in the present paper. We show that the answer is positive and provide a general construction of a \emph{homological topological quantum field theory}.

	Topological quantum field theories (TQFTs) provide a general and practical framework for constructing quantum invariants. The mathematical notion of TQFT was first proposed by Atiyah in \cite{atiyah1988topological}, motivated by works of Segal \cite{segal1988definition} and Witten \cite{witten1988topological}. A large class of quantum invariants was defined in the groundbreaking works of Reshetikhin and Turaev \cite{reshetikhin1991invariants,turaev2016quantum}, where they constructed $3$-dimensional semisimple TQFTs from suitable algebraic categories. Later, various non-semisimple TQFTs were developed \cite{blanchet2016non, de20193}. Lyubashenko's invariant of $3$-manifolds can be organized into a non-semisimple TQFT, called the Kerler--Lyubashenko TQFT \cite{kerler2002non}. In this paper, a TQFT is a monoidal functor
	\[
	3\Cob \to \calC
	\]
	from a cobordism category $3\Cob$ to an algebraic category $\calC$. The target category $\calC$ varies with the particular TQFT, while $3\Cob$ is fixed to be the category of connected $3$-cobordisms, studied by Crane and Yetter \cite{crane1999algebraic}, and later by Kerler \cite{kerler1996genealogy,kerler2001towards}, Habiro \cite{habiro2006bottom}, Bobtcheva, Piergallini, Beliakova and De Renzi \cite{bobtcheva20124,beliakova2312algebraic}. Objects of $3\Cob$ are compact connected oriented surfaces with one boundary component. Morphisms are diffeomorphism classes of connected $3$-dimensional cobordisms. The monoidal structure is given by boundary connected sum. More details can be found in Definition~\ref{def:3Cob}.

	Any cobordism in $3\Cob$ can be decomposed into elementary ones. Juh\'asz's presentation of cobordism categories \cite{juhasz2018defining}, adapted to $3\Cob$, allows us to write a full set of relations among elementary cobordisms. With this presentation, a functor
	\[
	F:3\Cob \to \calC
	\]
	can be constructed in three steps:
	\begin{enumerate}
		\item specify an object $F(\Sigma)$ for each $\Sigma \in 3\Cob$;
		\item specify morphisms $F(M):F(\Sigma_-)\to F(\Sigma_+)$ for each elementary cobordism $M:\Sigma_-\to\Sigma_+$ in $3\Cob$;
		\item check Juh\'asz's relations.
	\end{enumerate}
	We follow these steps in our construction of homological TQFTs. The first step is essentially contained in \cite{blanchet2025heisenberg,de2022homological}, but it requires the additional datum of a collection $\calL=\{\calL(\Sigma)\}_{\Sigma\in 3\Cob}$ of local systems. Each $\calL(\Sigma)$ is a $p$-Heisenberg local system of $R$-modules on $\Conf_*(\Sigma)$, where $R$ is a unital commutative ring with involution and $p$ is either $2$ or an odd integer $\geq 3$. We also require each $\calL(\Sigma)$ to be equipped with a perfect sesquilinear pairing
	\[
	(-,-):\calL(\Sigma)\otimes \overline{\calL(\Sigma)}\to R,
	\]
	where the overline denotes conjugation with respect to the involution on $R$. Switching the orientation of $\Sigma$ (denoted by $\overline{\Sigma}$) is required to be compatible with this conjugation, in the sense that $\overline{\calL(\Sigma)}$ is identified with $\calL(\overline{\Sigma})$. We associate to each $\Sigma\in 3\Cob$ the subspace of small cycles
	\[
	\bbF_{\calL}(\Sigma):= \operatorname{Im}\left[\bbH(\Sigma;\calL(\Sigma)) \to \bbH^{BM}(\Sigma;\calL(\Sigma))\right].
	\]
	The pairing on local systems induces a perfect intersection pairing on small cycles.

	\vspace{0.5cm}

	Our construction of the action of elementary cobordisms is based on the notion of the \emph{trajectory space} of an elementary cobordism. The intuitive idea behind trajectory spaces is illustrated in Fig.~\ref{fig:cob2_trajectories} for an elementary $2$-cobordism; a detailed discussion is given in Section~\ref{sec:Multi_trajectory}.

	\begin{figure}[H]
		\centering
		\includegraphics[width=0.9\linewidth]{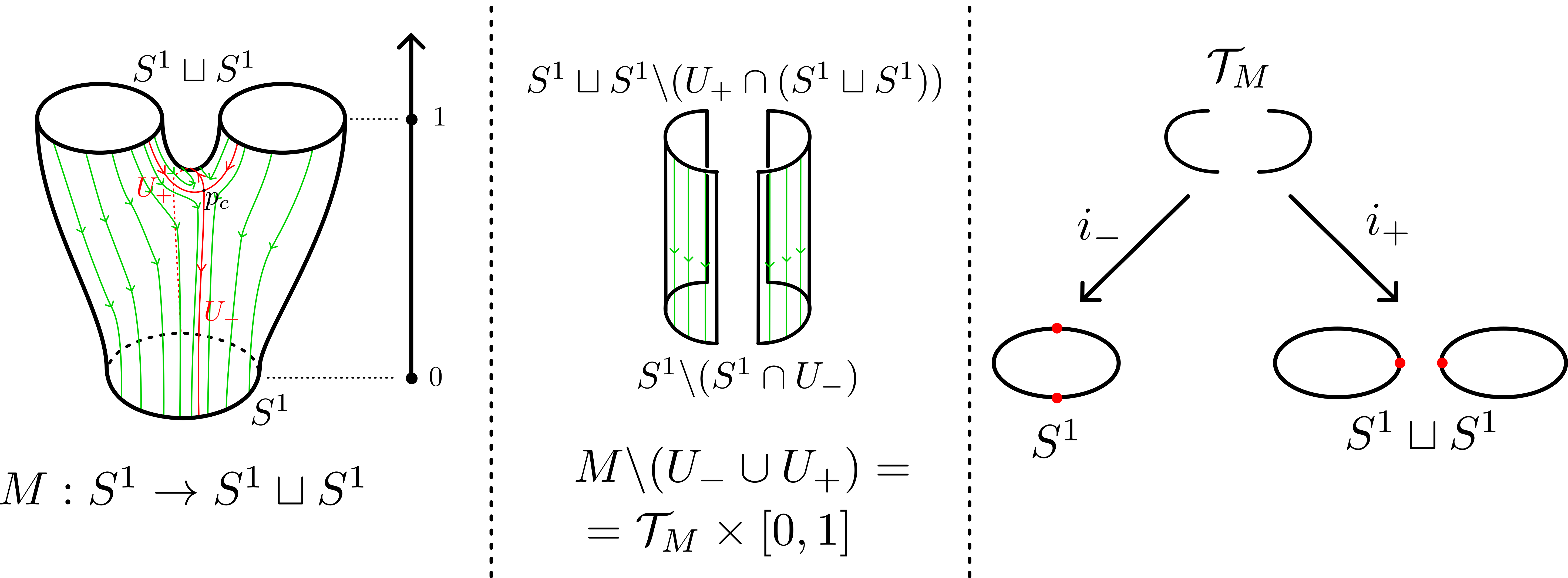}
		\caption{A schematic picture of the trajectory space of an elementary $2$-cobordism $M:S^1\to S^1\sqcup S^1$. The green curves are regular trajectories, while the red curves form the unstable and stable manifolds $U_-$ and $U_+$ of the critical point $p_c$. Removing $U_-\cup U_+$ leaves a cylinder over the trajectory space $\calT_M$, equipped with natural endpoint-evaluation maps $i_-:\calT_M\to S^1$ and $i_+:\calT_M\to S^1\sqcup S^1$.}
		\label{fig:cob2_trajectories}
	\end{figure}

	The main feature of the trajectory space $\calT_M$ of a cobordism $M:\Sigma_-\to \Sigma_+$ is that it comes with maps $i_-$ and $i_+$ evaluating trajectories at their endpoints:
	\[
	\begin{tikzcd}
		&\calT_M\arrow[dl,"i_-"']\arrow[dr, "i_+"]&\\
		\Sigma_- & & \Sigma_+.
	\end{tikzcd}
	\]
	Such a diagram is called a \emph{span}. The maps $i_-$ and $i_+$ are embeddings and therefore induce a similar span on the level of configuration spaces:
	\[
	\begin{tikzcd}
		&\Conf_*(\calT_M)\arrow[dl,"i_-"']\arrow[dr, "i_+"]&\\
		\Conf_*(\Sigma_-) & & \Conf_*(\Sigma_+).
	\end{tikzcd}
	\]
	We would like to apply a twisted homology functor to this span in order to obtain a span of $R$-modules. For this one needs additional data on local systems: namely, for each elementary cobordism, a morphism $\varphi$ between the pull-backs of $\calL(\Sigma_-)$ and $\calL(\Sigma_+)$ to the trajectory space. Given such a morphism, one obtains a span
	\[
	\begin{tikzcd}
		&\bbH^{BM}(\calT_M;i_-^*\calL(\Sigma_-))\arrow[dl,"(i_-)_*"']\arrow[r,"\varphi_*"]&\bbH^{BM}(\calT_M;i_+^*\calL(\Sigma_+))\arrow[dr,"(i_+)_*"]&\\
		\bbH^{BM}(\Sigma_-;\calL(\Sigma_-))&&& \bbH^{BM}(\Sigma_+;\calL(\Sigma_+)).
	\end{tikzcd}
	\]
	We say that $\calL$ is equipped with an action of elementary cobordisms if such a morphism $\varphi$ is specified for each elementary cobordism. If $M$ is an index $1$ cobordism, the span above is modified by inserting a fundamental class of a covering space over $\Conf_{k}(S^1)$
	\[
		\left[\widetilde{\Conf}_{k}(S^1)\right]
	\]
	on the belt sphere, as explained in Section~\ref{subs:Construction_of_homological_TQFTs}. The index $2$ action is defined by duality. For index $1$, the essential step is then to ``invert'' the left arrow after restricting to small cycles: every incoming class has a small-cycle lift to the trajectory space, and the insertion gives the same outgoing class for every such lift. This is possible if the following conditions are satisfied:
	\begin{itemize}
		\item the span is restricted to the subspaces of small cycles;
		\item the local systems $\calL(\Sigma_-)$ and $\calL(\Sigma_+)$ are $p$-Heisenberg for some odd $p\geq 3$ or for $p=2$;
		\item the local systems $\calL(\Sigma_-)$ and $\calL(\Sigma_+)$ are equipped with perfect pairings; mapping-cylinder identifications preserve these pairings, while handle coefficient maps are arbitrary morphisms of paired local systems, as distinguished in Section~\ref{sec:local_systems};
		\item the inserted fundamental class has $k=p-1$ points.
	\end{itemize}
	Under these assumptions, the span with insertion defines the action $\bbF_{\calL}(M)$ directly on small cycles:
	\[
		\bbF_{\calL}(M):\mathring{\bbH}(\Sigma_-;\calL(\Sigma_-))\longrightarrow\mathring{\bbH}(\Sigma_+;\calL(\Sigma_+)).
	\]

	The last step is to check Juh\'asz's relations. Not every action of elementary cobordisms on local systems gives rise to a TQFT. In Section~\ref{sec:nice_ls} we formulate a set of sufficient conditions and call a collection $\calL=\{\calL(\Sigma)\}$, equipped with an action of elementary cobordisms, \emph{Morse compatible} if it satisfies them. We expect that a Morse compatible collection should provide a natural TQFT with values in a category of spans over $\mathrm{Mod}_R$ (or perhaps correspondences), in which one does not need to invert an arrow in each span. We plan to address this in future work.

	\vspace{0.5cm}

	\noindent\textit{Main results.} In cases where the construction is defined only up to sign, we write $\mathrm{Mod}_R^{\pm}$ for the category of $R$-modules and $R$-linear maps up to sign. Our first main result is the following theorem.

	\begin{refmaintheorem}{thm:main_thm}
		Let $R$ be a unital commutative ring with involution. Let $\calL=\{\calL(\Sigma)\}_{\Sigma\in 3\Cob}$ be a collection of $p$-Heisenberg local systems of $R$-modules, such that each $\calL(\Sigma)$ is equipped with a perfect sesquilinear pairing, the action of all $1-(-\sigma)^{k}$, $1\leq k<p$, on $\calL(\Sigma)$ is invertible, $(-\sigma)^p$ acts as the identity, and $\overline{\calL(\Sigma)}$ is identified with $\calL(\overline{\Sigma})$. Suppose $\calL$ is Morse compatible. Then the assignments $\Sigma \mapsto\bbF_{\calL}(\Sigma)$ for all objects $\Sigma\in 3\Cob$ and $M\mapsto\bbF_{\calL}(M)$ for all elementary cobordisms $M\in \mathrm{Mor}(3\Cob)$ give rise to a monoidal functor
		\[
		\bbF_{\calL}:3\Cob\to \mathrm{Mod}_R
		\]
		for odd $p\geq 3$, and to a monoidal functor
		\[
		\bbF_{\calL}:3\Cob\to \mathrm{Mod}_R^{\pm}
		\]
		for $p=2$.
	\end{refmaintheorem}

	Each state space $\bbF_{\calL}(\Sigma)$ carries a perfect sesquilinear intersection pairing, and index $2$ actions are defined by duality. 

	\vspace{0.5cm}

	We provide two natural examples of Morse compatible local systems that recover two well-known non-semisimple TQFTs: the first up to isomorphism, and the second projectively. The first is the Frohman--Nicas--Donaldson TQFT (FND)
	\[
	D:3\Cob\to \mathrm{Vect}_{\Bbbk}^{\pm},
	\]
	where $\mathrm{Vect}_{\Bbbk}^{\pm}$ is the category of vector spaces over a field $\Bbbk$ and $\Bbbk$-linear maps up to sign. Its homological construction via Lagrangian correspondences was given by Frohman and Nicas in \cite{frohman1992alexander} and developed by Kerler in \cite{kerler2003homology}. In particular, Kerler noted that this is a non-semisimple Hennings TQFT. Later, Donaldson outlined a connection between this TQFT and $3$-dimensional Seiberg--Witten invariants in \cite{donaldson1999topological}, and Nguyen further elaborated this approach in \cite{nguyen2014lagrangian}.

	The FND TQFT can be recovered from the homological construction for a particular choice of local systems. In the somewhat degenerate case $p=2$, the collection $\underline{\Bbbk}$ of trivial $2$-Heisenberg local systems of $\Bbbk$-modules is Morse compatible, and the corresponding homological TQFT $\bbF_{\underline{\Bbbk}}$ is isomorphic to the FND TQFT. This is our second main result.

	\begin{refmaintheorem}{thm:Donaldson_iso}
		Let $\Bbbk$ be a field with $\mathrm{char}(\Bbbk)\not= 2$. Then the collection of trivial local systems $\underline{\Bbbk}$ together with the action of elementary cobordisms is Morse compatible. It gives rise to a homological TQFT
		\[
			\bbF_{\underline{\Bbbk}}:3\Cob\to \mathrm{Vect}_{\Bbbk}^{\pm}
		\]
		isomorphic to the Frohman--Nicas--Donaldson TQFT.
	\end{refmaintheorem}

	\vspace{0.5cm}

	The second TQFT recovered by the homological construction is the projective version of the Kerler--Lyubashenko TQFT \cite{lyubashenko1995invariants, kerler2002non} for small quantum $\mathfrak{sl}_2$ at a $p$-th root of unity, where $p$ is odd and $p\geq 3$. It arises from the collection of Schr\"odinger local systems $\calW_p$. We define a projective action of elementary cobordisms on Schr\"odinger local systems and prove in Section~\ref{sec:Schroedinger_ls} that they form a projectively Morse compatible collection of local systems. Our third main result is the following theorem.

	\begin{refmaintheorem}{thm:KL_isomorphism}
		The projective TQFT $\bbF_{\calW_p}$ arising from the collection of Schr\"odinger local systems $\calW_p$ for an odd $p\geq 3$ is projectively isomorphic to the Kerler--Lyubashenko TQFT $J_{\mathfrak{sl}_2}$ associated with the small quantum $\mathfrak{sl}_2$ at the $p$-th root of unity $\zeta=e^{2\pi i/p}$.
	\end{refmaintheorem}

	Thus the paper provides a topological realization of two non-semisimple TQFTs in terms of twisted homology of configuration spaces.

	\vspace{0.5cm}

	In Figure~\ref{fig:Donaldson_action} we give a schematic description of the action of index $1$ and $2$ cobordisms in the homological TQFT.

	\begin{figure}[H]
		\centering
		\includegraphics[width=0.9\linewidth]{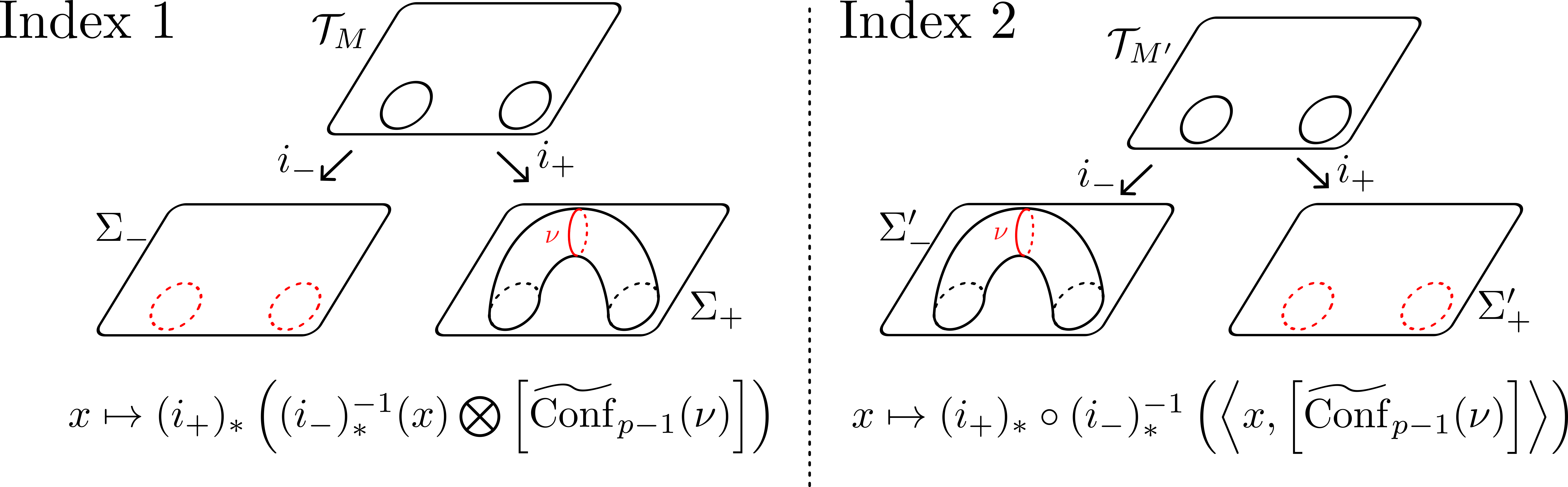}
		\caption{Schematic description of the action of index $1$ and $2$ cobordisms in homological TQFTs for surfaces of genus $0$ or $1$. For an index $1$ cobordism $M$, a small cycle $x$ is lifted to a small cycle on the trajectory space $\calT_M$ and then pushed forward to the homology of $\Sigma_+$ after inserting the (twisted) fundamental class $\left[\widetilde{\Conf}_{p-1}(\nu)\right]$. The result is a small cycle independent of the chosen lift. The action of an index $2$ cobordism $M'$ is defined dually with respect to the intersection pairing: passing to the opposite cobordism of $M'$ gives an index $1$ cobordism $\overline{M}'$ whose action is already defined, and the perfect pairing identifies the desired map with its adjoint. One may think of this dual operation as intersecting with the inserted fundamental class and passing back through the trajectory space. For details, see Section~\ref{sec:general_construction}.}
		\label{fig:Donaldson_action}
	\end{figure}

	\vspace{0.5cm}

	The paper is organized as follows. In Section~\ref{section:preliminaries} we collect preliminaries. In Section~\ref{sec:Heisenberg_homology} we introduce Heisenberg groups and Heisenberg local systems, and define the main homological objects of the paper: ordinary homology and Borel--Moore homology of configuration spaces with twisted coefficients. We also define the subspaces of small cycles and equip them with a perfect pairing. In Section~\ref{sec:general_construction} we construct homological TQFTs from Morse compatible collections of local systems. The action of elementary cobordisms is defined using trajectory spaces, and the relations in the cobordism category are verified using Juh\'asz's presentation. In Section~\ref{sec:trivial_ls} we introduce trivial local systems, show that they are Morse compatible, and prove that the corresponding homological TQFT recovers the FND TQFT. In Section~\ref{sec:Schroedinger_ls} we introduce Schr\"odinger local systems, show that they are projectively Morse compatible, and prove that the corresponding homological TQFT recovers the projective Kerler--Lyubashenko TQFT.

	\subsection*{Acknowledgments}
	The author would like to thank Anna Beliakova and Christian Blanchet for helpful discussions, valuable suggestions and guidance in writing this paper. This work was supported by Simons Collaboration on New Structures in Low Dimensional Topology and Grant 200020\_207374 of Swiss National Science Foundation.

\newpage

	\section{Preliminaries}\label{section:preliminaries}
	In this section we collect the background material and notation needed for the homological construction. 
	

	\subsection{Notation}\label{subs:Notation}
	\begin{itemize}
	\item[-] We write $\Z_k:=\Z/k\Z$. If the coefficients of a homology group are not specified, then integral coefficients are understood: $H_*(X):=H_*(X;\Z)$. 
	
	\item[-] For a topological space $X$ we denote by $[-]:\pi_1(X)\to H_1(X)$ the Hurewicz homomorphism.


	\item[-] Our convention for multiplication in fundamental groups is that the composition $\alpha\beta$ of two loops means concatenation in which \emph{$\alpha$ is traversed first and then $\beta$}.
	
	\item[-] We write $I:=[0,1]\subset \R$.
	
	\item[-] For a commutative ring $R$, let $\mathrm{Mod}_R$ denote the monoidal category of all $R$-modules and $R$-linear maps. For a field $\Bbbk$, we denote by $\mathrm{Vect}_{\Bbbk}$ the monoidal category of vector spaces.

	\item[-] We write $\delta_{ij}$ for the Kronecker delta.

\item[-] \textit{Quantum integers.}
	Consider the ring $\Z[q,q^{-1}]$ of Laurent polynomials in one variable. For every integer $n$ we use the standard notation
	\[
		[n]_q:=\frac{q^n-q^{-n}}{q-q^{-1}}
	\]
	for quantum integers, and
	\[
		[k]_q!:=\prod\limits_{i=1}^{k}[i]_q,\quad\quad\quad \sqbinom{k}{l}_q:=\frac{[k]_q!}{[l]_q!\cdot[k-l]_q!},
	\]
	for quantum factorials and quantum binomial coefficients, where $k$ and $l$ are integers satisfying $k\geq l\geq 0$.
	
	\end{itemize}
	\subsection{Cobordism categories}\label{subs:Cobordism_categories}
	
	We work with the following cobordism category.
	
	\begin{definition}\label{def:3Cob}
		The monoidal braided category $3\Cob$ of connected cobordisms is defined as follows:
		\begin{itemize}
			\item Objects are compact connected oriented surfaces $\Sigma$ with one parametrized boundary component, \ie a homeomorphism $S^1\xrightarrow{\simeq}\partial\Sigma$ is specified. The boundary of each surface $\Sigma$ is split into two arcs $\partial_-\Sigma$ and $\partial_+\Sigma$;
			\item Morphisms between $\Sigma$ and $\Sigma'$ are equivalence classes of compact connected oriented 3-dimensional manifolds $M$ with corners equipped with a parametrization of the boundary, \ie a homeomorphism $(-\Sigma)\cup_{\partial\Sigma}S^1\times I\cup_{\partial\Sigma'}\Sigma'\xrightarrow{\simeq}\partial M$, such that its restriction to $S^1\times\{0\}$ and $S^1\times \{1\}$ coincides with the parametrizations of $-\partial\Sigma$ and $\partial\Sigma'$. The equivalence is given by diffeomorphisms preserving boundary parametrization;
			\item The composition is given by collar gluing of cobordisms;
			\item The monoidal structure is given by the boundary connected sum $\Sigma_1\natural \Sigma_2$ such that $\partial_-\Sigma_1\cup \partial_-\Sigma_2\subset \partial_-(\Sigma_1\natural \Sigma_2)$ (see Fig.~\ref{fig:gluing}).
		\end{itemize}
	\end{definition}

	\begin{figure}[h]
		\centering
		\includegraphics[width=0.5\linewidth]{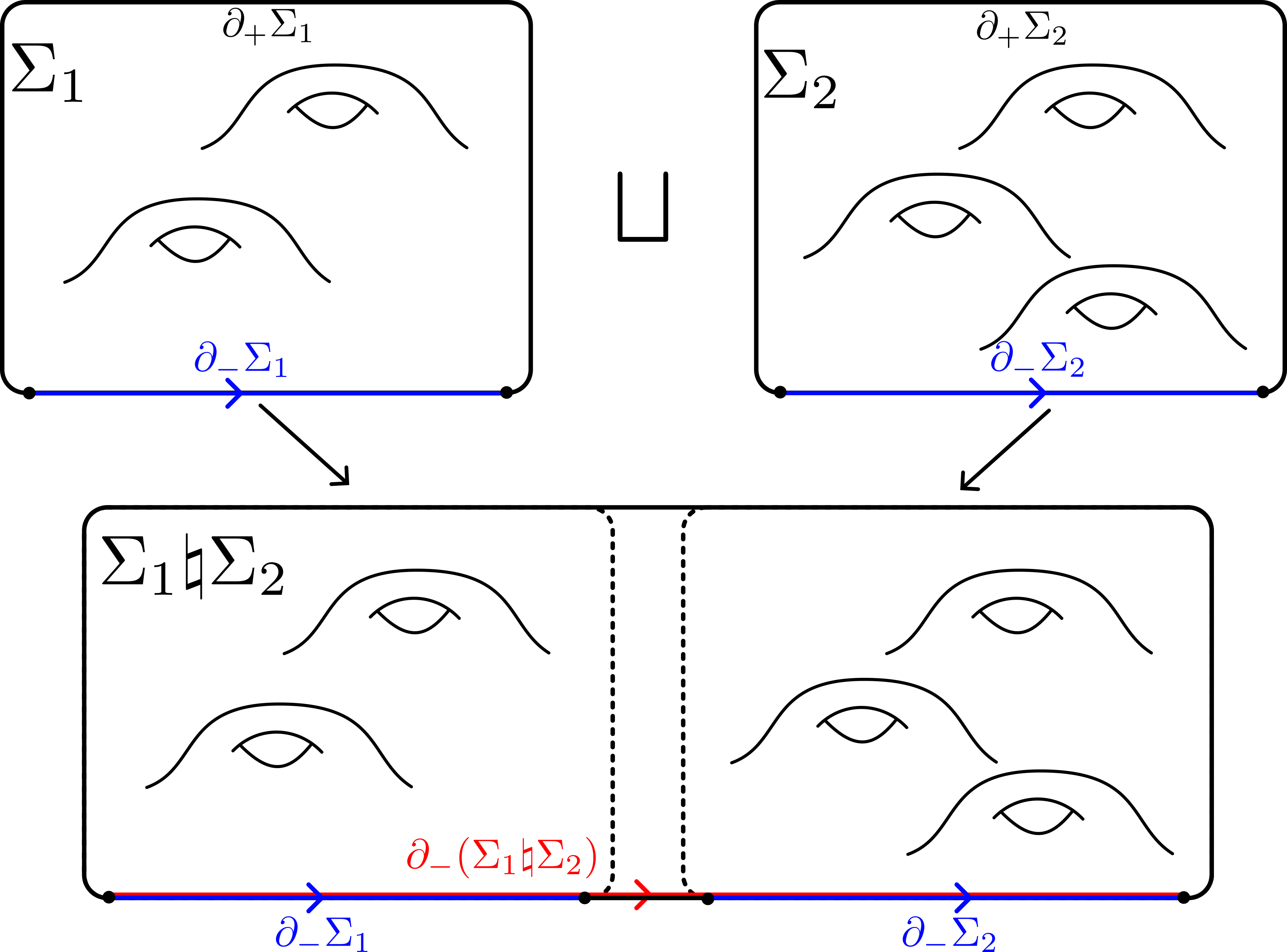}
		\caption{The boundary connected sum $\Sigma_1\natural \Sigma_2$. Surfaces $\Sigma_1$ and $\Sigma_2$ are glued along vertical segments of their boundary. They are viewed as subsurfaces of $\Sigma_1\natural\Sigma_2$ such that $\partial_-\Sigma_1,\partial_-\Sigma_2\subset\partial_-(\Sigma_1\natural\Sigma_2)$ and $\partial_-\Sigma_1$ appears first if one follows the orientation of $\partial_-(\Sigma_1\natural\Sigma_2)$.}
		\label{fig:gluing}
	\end{figure}
	
	Isomorphism classes of objects in this category are parametrized by their genus. We fix a representative for each isomorphism class, namely the standard surface $\Sigma_g$ depicted in Fig.~\ref{fig:standard_picture}.
	
	\begin{figure}[h]
		\centering
		\includegraphics[width=0.8\linewidth]{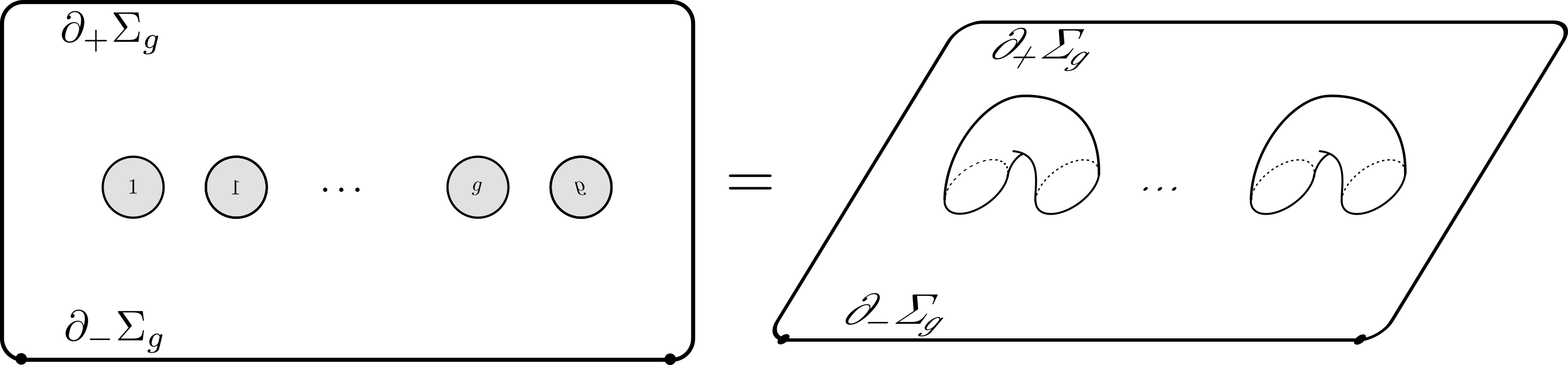}
		\caption{The standard surface $\Sigma_g\in 3\Cob$ of genus $g$. Circles with the same label are identified by reflection across a vertical plane between them. The boundary is split into the arcs $\partial_-\Sigma_g$ and $\partial_+\Sigma_g$.}
		\label{fig:standard_picture}
	\end{figure}

	 \vspace{0.5cm}

	 \begin{definition}\label{def:MCG}
	 	For $\Sigma\in 3\Cob$ the mapping class group $\mathrm{Mod}(\Sigma,\partial\Sigma)$ is the group of orientation-preserving diffeomorphisms of $\Sigma$ that restrict to the identity on the boundary, modulo isotopies relative to the boundary.
	 \end{definition}
	 
	 Let $\Sigma,\Sigma'$ be objects in $3\Cob$ and $d:\Sigma\to \Sigma'$ a diffeomorphism preserving orientation and the boundary parametrization. Recall that the mapping cylinder $M_d$ of $d$ is defined as the equivalence class of the cobordism $\Sigma\times I$ with the parametrization: $\Sigma\xrightarrow{\Id}\Sigma\times\{0\}$, $\Sigma'\xrightarrow{d^{-1}}\Sigma\times \{1\}$. Note that if two diffeomorphisms are equivalent up to boundary-preserving isotopy the corresponding mapping cylinders are equivalent as cobordisms.
	 
	 \vspace{0.5cm}
	 
	The category $3\Cob$ is in fact braided monoidal, \ie for any pair of objects $\Sigma_1,\Sigma_2\in 3\Cob$ there is an isomorphism $M_{\beta}:\Sigma_1\natural\Sigma_2\to \Sigma_2\natural \Sigma_1$ satisfying the hexagon identities (see \cite{street1993braided} for the definition and \cite{kerler2002non}, Lemma 1.3.1, for a discussion of the braided structure on $3\Cob$).

	Specifically, for two objects $\Sigma_1,\Sigma_2\in 3\Cob$, let
	\begin{equation}\label{eq:d_beta_def}
		d_{\beta}=d_{\beta,\Sigma_1,\Sigma_2}:\Sigma_1\natural\Sigma_2\longrightarrow \Sigma_2\natural\Sigma_1
	\end{equation}
	be the boundary-fixing diffeomorphism obtained by moving the two marked surface blocks counterclockwise around each other in an ambient handle model and comparing the endpoint markings with those of $\Sigma_2\natural\Sigma_1$ (see Fig.~\ref{fig:braiding}). The corresponding mapping cylinder gives the braiding
	\[
		M_{\beta}:=M_{d_{\beta}}:\Sigma_1\natural\Sigma_2\longrightarrow \Sigma_2\natural\Sigma_1.
	\]

	\begin{figure}[h]
		\centering
		\includegraphics[width=0.9\linewidth]{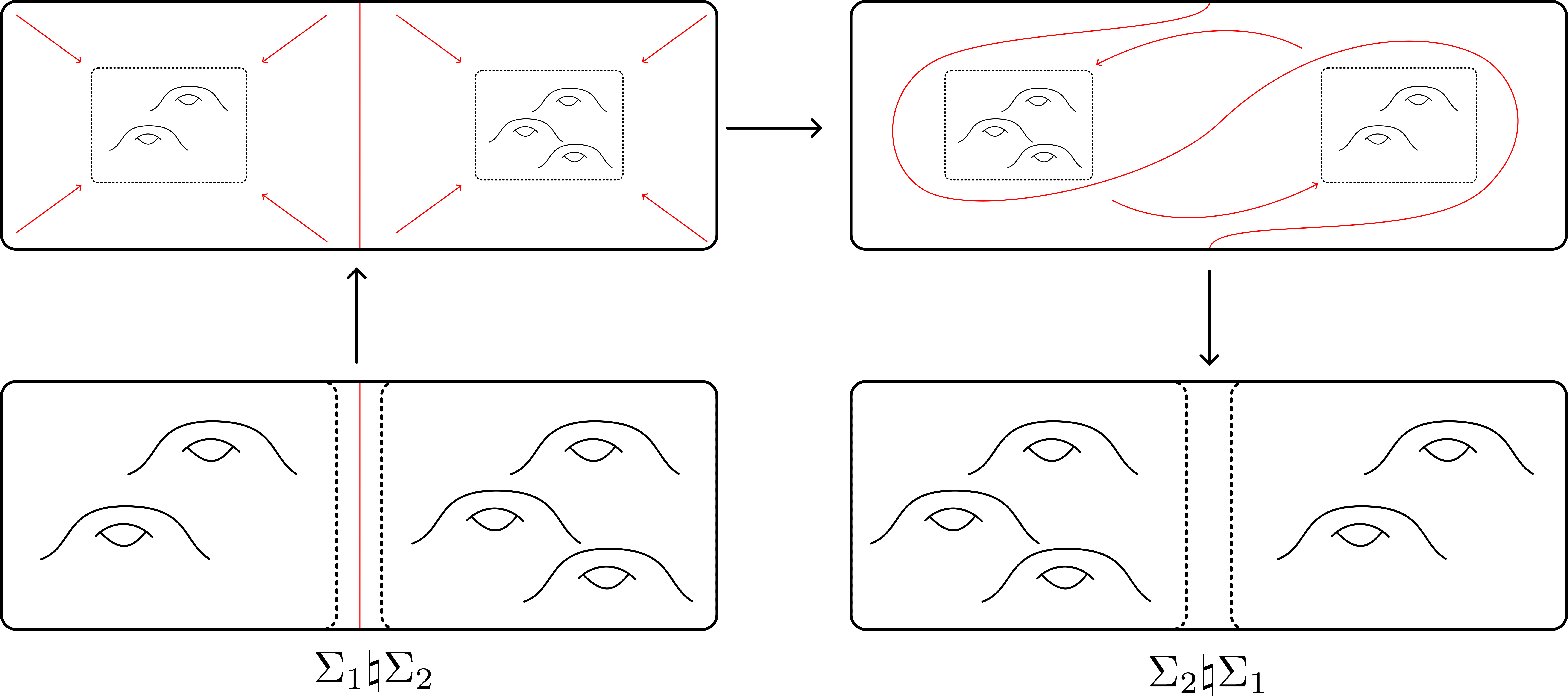}
		\caption{The boundary-preserving diffeomorphism $d_{\beta}$ inducing the braiding on $3\Cob$.}
		\label{fig:braiding}
	\end{figure}

		For an oriented surface $\Sigma$ denote by $\overline{\Sigma}$ the same surface but with inverted orientation. Fix a reflection $r$ of the parametrizing circle preserving each distinguished boundary arc setwise, and precompose the boundary parametrization with $r$. Choose the boundary-sum collars compatibly with this reflection. Reversal then reverses the order of summands:
	\[
		\overline{\Sigma_1\natural\Sigma_2}\cong\overline\Sigma_2\natural\overline\Sigma_1.
	\]
	This identifies the same underlying union of surfaces. 

	For a cobordism $M:\Sigma_1\to\Sigma_2$ in $3\Cob$ we denote by $\overline{M}:\overline{\Sigma}_2\to \overline{\Sigma}_1$ the opposite cobordism obtained by exchanging its incoming and outgoing ends and reversing the orientations of both surfaces. On the side boundary the parametrization is precomposed with $(z,t)\mapsto(r(z),1-t)$. The orientation of the three-manifold $M$ itself is unchanged, so its boundary orientation agrees with this convention. Thus reversal reverses both composition and tensor order: $\overline{M\natural N}\cong\overline N\natural\overline M$.

	\subsection{TQFTs and Juh\'asz's presentation of $3\Cob$}\label{subs:TQFTs_and_Juhasz_presentation}
	
		\begin{definition}\label{def:TQFT}
		Fix a commutative ring $R$. We call a monoidal functor
		\[
		3\Cob\to \mathrm{Mod}_{R}
		\]
		a topological quantum field theory (TQFT).
	\end{definition}

	We briefly describe the presentation of $3\Cob$ given by Juh\'asz in~\cite{juhasz2018defining}. This presentation is the main tool in our construction.
	
	Any cobordism in $3\Cob$ admits a Morse decomposition. For an introduction to Morse theory, see \cite{milnor1963morse}. Recall that a Morse function on a cobordism $M:\Sigma\to \Sigma'$ is a smooth function $f:M\to[0,1]$ such that $f^{-1}(0)=\Sigma$, $f^{-1}(1)=\Sigma'$, its restriction to the vertical boundary $S^1\times I$ is the projection $f|_{S^1\times[0,1]}(s,t)=t$, and all its critical points are non-degenerate. An elementary cobordism $M:\Sigma_-\to \Sigma_+$ is a cobordism admitting a Morse function with at most one critical point. Equivalently, $M$ is determined, up to equivalence, by an attaching tube $\bbS$ in the incoming surface: if $M$ has no critical points, we set $\bbS=\emptyset$, while otherwise $\bbS:S^k\times D^{2-k}\hookrightarrow \Sigma_-$ is an orientation-reversing embedding, with $k\in\{0,1\}$. Then, relative $\Sigma_-$, $M$ is equivalent to the \emph{trace} $M_{\bbS}$ obtained from $\Sigma_-\times I$ by attaching the handle $D^{k+1}\times D^{2-k}$ along $\bbS\subset \Sigma_-\times\{1\}$. The outgoing surface is the surgered surface $\Sigma_+=\Sigma_-(\bbS)$. We denote by $a(\bbS):=\bbS(S^k\times\{0\})\subset \Sigma_-$ the \emph{attaching sphere} and by $b(\bbS)\subset \Sigma_+$ the \emph{belt sphere} defined as the image of $\{0\}\times S^{1-k}$ in the outgoing boundary of the attached handle. Thus $k=0$ gives an index $1$ elementary cobordism, with a framed $S^0$ attaching sphere and a framed $S^1$ belt sphere, while $k=1$ gives an index $2$ elementary cobordism, with a framed attaching sphere $S^1$ and a framed belt sphere $S^0$. For an introduction into Morse theory on cobordisms and more detailed discussion of elementary cobordisms see for example \cite{Milnor1965hCobordism}.
	
	It is in fact possible to give a complete set of relations on elementary cobordisms, and therefore give a presentation of $3\Cob$ as follows. Let $\mathrm{Surf}$ be the category of compact connected oriented surfaces with one parametrized boundary component and orientation preserving diffeomorphisms between them (identical on the boundary). Let $\mathcal{G}$ be the oriented graph obtained from $\mathrm{Surf}$ by adding an edge $e_{\Sigma,\bS}:\Sigma\to\Sigma(\bS)$ for each attaching tube $\bS$ of an index $1$ or $2$ handle in the interior of $\Sigma$ for which the surgered surface $\Sigma(\bS)$ is connected. Define $\calF(\calG)$ as the free category generated by $\calG$. There is a natural functor
	\[
		\calE:\calF(\calG)\to 3\Cob
	\]
	sending each object to itself $\Sigma\mapsto \Sigma$ and each edge $e_{\Sigma,\bbS}$ to the trace $M_{\bbS}:\Sigma \to \Sigma(\bbS)$. The theorem below is the relative-boundary version of \cite[Theorems~1.7 and~2.24]{juhasz2018defining}. One keeps the Morse data fixed on a collar of the vertical boundary; the Cerf moves are then supported in the interior. The connected-level version uses only index $1$ and $2$ handles, and only relations whose intermediate surfaces remain connected. Thus the same generators and local relations apply here.
	\begin{theorem}\label{thm:Juhasz_relations}
		The functor $\calE$ descends to an equivalence functor
		\[
			\calF(\calG)/\calR \xrightarrow{\simeq} 3\Cob
		\]
		from the quotient of the category $\calF(\calG)$ by the set of the following relations $\calR$:
		\begin{enumerate}
			\item $e_d\circ e_{d'}=e_{d\circ d'}$, where $e_d,e_{d'}$ and $e_{d\circ d'}$ are mapping cylinders of diffeomorphisms. If $d$ is isotopic to the identity, then $e_d=\id_\Sigma$;
			\item For an orientation preserving diffeomorphism $d:\Sigma\rightarrow \Sigma'$ and for a framed sphere $\bS\subset \Sigma$, let $\bS' =d\circ \bS$ and $d^{\bS}:\Sigma(\bbS)\to \Sigma'(\bbS')$ be the induced diffeomorphism, then $e_{\Sigma',\bS'}\circ e_{d}= e_{d^{\bS}}\circ e_{\Sigma,\bS}$ where $e_{\Sigma,\bS}$ is the elementary cobordism corresponding to the surgery on $\bS$;
			\item If $\bS$ and $\bS'$ are two disjoint framed spheres in $\Sigma$ and all the following surgered surfaces are connected, then
			\[
				e_{\Sigma(\bS),\bS'}\circ e_{\Sigma,\bS}
				=e_{\Sigma(\bS'),\bS}\circ e_{\Sigma,\bS'}.
			\]
			Here the two outgoing surfaces are identified with $\Sigma(\bS,\bS')$;
			\item If $\bS$ is the attaching sphere of an index $1$ handle, $\bS'\subset \Sigma(\bS)$ is the attaching sphere of an index $2$ handle, and $a(\bS')$ intersects the belt circle $b(\bS)$ transversely at one point, then $e_{\Sigma(\bS),\bS'}\circ e_{\Sigma,\bS}= e_{f}$,
			where $f:\Sigma\to \Sigma(\bS)(\bS')$ is a diffeomorphism (unique up to isotopy) which is identical on $\Sigma\cap\Sigma(\bS)(\bS')$;
			\item $e_{\Sigma,\bS}= e_{\Sigma,\overline{\bS}}$, where $\overline{\bS}$ is obtained by reversing the sphere orientation and the normal framing simultaneously.
		\end{enumerate}
		
	\end{theorem}
	
	Using this presentation one can construct a TQFT $F:3\Cob\to\mathrm{Mod}_R$ by providing a monoidal functor $F:\mathrm{Surf}\to \mathrm{Mod}_R$ and a set of morphisms $F(e_{\bS}):F(\Sigma)\to F(\Sigma(\bS))$ for each $e_{\bS}:\Sigma\to\Sigma(\bS)$ satisfying the relations above and such that for any objects $\Sigma, \Sigma'\in 3\Cob$ and any $\bS$ in $\Sigma$ diagrams
	
	\begin{equation}\label{eq:monoidality_condition}
		\begin{tikzcd}
			F(\Sigma)\otimes F(\Sigma') \arrow[d,"F(e_{\bS})\otimes\mathrm{Id}"] \arrow[r,"\mu"]&F(\Sigma\natural\Sigma')\arrow[d,"F(e_{\bS})"]\\
			F(\Sigma(\bS))\otimes F(\Sigma') \arrow[r,"\mu"]& F(\Sigma(\bS)\natural\Sigma')
		\end{tikzcd}
		\quad
		\begin{tikzcd}
			F(\Sigma')\otimes F(\Sigma) \arrow[d,"\mathrm{Id} \otimes F(e_{\bS})"] \arrow[r,"\mu"]&F(\Sigma'\natural\Sigma)\arrow[d,"F(e_{\bS})"]\\
			F(\Sigma')\otimes F(\Sigma(\bbS)) \arrow[r,"\mu"]& F(\Sigma'\natural \Sigma(\bbS))
		\end{tikzcd}
	\end{equation}
	are commutative, where isomorphisms $\mu$ are a part of the monoidal functor structure.

	
	\subsection{Configuration spaces}\label{sec:conf_spaces}
	
	For $X$ a topological space, the \emph{unordered configuration space} of $n$-points on $X$, $n\geq 1$ is defined as the topological space (we omit the word ``unordered'' in what follows):
	\[
		\Conf_n(X):=(X^{\times n}\setminus \Delta)/S_n, \, n\geq 1
	\]
	where $\Delta:=\big\{(x_1,\dots,x_n)\in X^{\times n}\ \big|\ \ \exists i\not=j, \: x_i=x_j\big\}$ is the big diagonal and $S_n$ is the symmetric group acting by permutation of points. We use the notation $\mathbf{x}$ for a point in $\Conf_n(X)$. Note that a point in $\Conf_n(X)$ is a subset of cardinality $n$ in $X$, hence sometimes the notation $\mathbf{x}=\{x_1,\dots,x_n\}\in \Conf_n(X)$ is used to specify the points $x_i\in X$. We also define $\Conf_0(X)$ as the one-point space.
	
	Note that if $d:X\to Y$ is a diffeomorphism between manifolds, then it induces a homeomorphism of configuration spaces
	\[
		d^{\times n}:\Conf_n(X)\to\Conf_n(Y),
	\]
	\[
		\{x_1,\dots,x_n\}\longmapsto \{d(x_1),\dots,d(x_n)\}.
	\]
	
	\vspace{0.5cm}
	
	One of the central objects in this paper is the configuration space $\Conf_n(\Sigma)$, where $\Sigma\in 3\Cob$, $n\geq 0$. For each $n\geq 1$ define the subspaces
	\[
		\Conf_n^-(\Sigma):=\{\mathbf{x}=\{x_1,\dots,x_n\}\mid \exists i:\ x_i\in\partial_{-}\Sigma\},
	\]
	\[
		\Conf_n^+(\Sigma):=\{\mathbf{x}=\{x_1,\dots,x_n\}\mid \exists i:\ x_i\in\partial_{+}\Sigma\}
	\]
	consisting of configurations with at least one point on $\partial_-\Sigma$ or on $\partial_+\Sigma$, respectively. In degree zero put $\Conf_0^\pm(\Sigma):=\varnothing$ and let $*_0=\varnothing$ be the unique point of $\Conf_0(\Sigma)$. For $n\geq1$ we choose a base point $*_n$ for $\Conf_n(\Sigma)$ by selecting $n$ distinct points on the arc $\partial_-\Sigma$. We always think about $\partial_-\Sigma$ as a horizontal interval as in Fig.~\ref{fig:standard_picture}, hence all configuration points of $*_n$ are ordered from left to right. For $n'>n$ we assume that $*_n$ is given by the $n$ leftmost points of $*_{n'}$. We write $B_n(\Sigma):=\pi_1(\Conf_n(\Sigma),*_n)$ for the corresponding surface braid group. The subgroup of braids supported in a neighborhood of $\partial_-\Sigma$ diffeomorphic to a disc is the classical braid group. With respect to the ordering of the points of $*_n$, its standard generators are denoted by $\sigma_1,\sigma_2,\dots,\sigma_{n-1}$; each $\sigma_i$ exchanges two neighboring points by moving them counter-clockwise around each other (see Fig.~\ref{fig:sigma}). 
	
	\begin{figure}[h]
		\centering
		\includegraphics[width=0.45\linewidth]{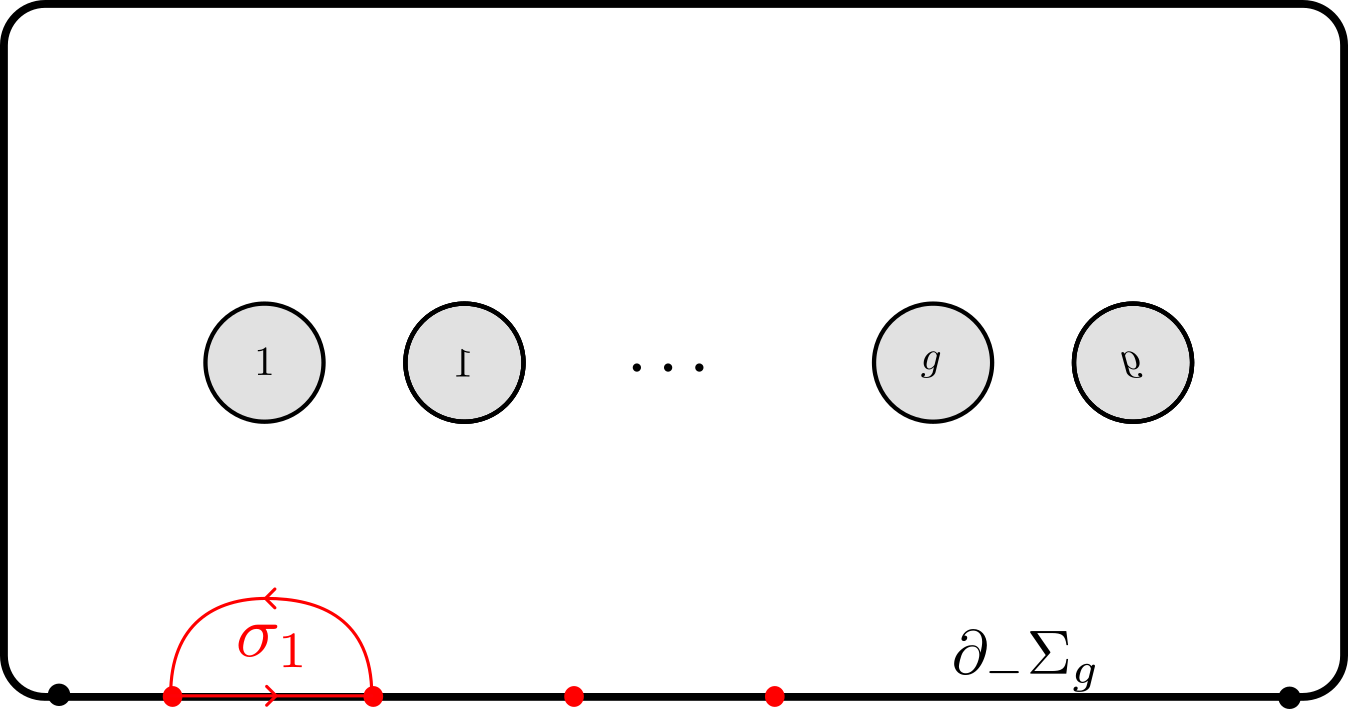}
		\caption{The set of red points on $\partial_-\Sigma_g$ depicts the base-point $*_n$. The classical braid group generator $\sigma_1$ switches two neighboring points by moving them counter-clockwise around each other (red arcs show the trajectories of points under $\sigma_1$).}
		\label{fig:sigma}
	\end{figure}

	\subsection{Multi-trajectory spaces}\label{sec:Multi_trajectory}

	Our construction of TQFTs is motivated by the idea of trajectory spaces induced by a Morse--Smale datum (see the paper \cite{Wehrheim2012} of Wehrheim for definitions and a detailed discussion). Below is an informal motivating description of trajectory spaces in the spirit of Wehrheim. We feel that the intuitive description is important for completeness, but one can skip it and go straight to the Definition~\ref{def:trajectory_spaces}, where trajectory spaces of generating morphisms in Juh\'asz's presentation are defined axiomatically.
	
	Let $M:\Sigma_-\to \Sigma_+$ be an elementary cobordism in $3\Cob$ and let $(f,g)$ be a Morse--Smale pair consisting of a Morse function $f:M\to [0,1]$ and a Riemannian metric $g$ satisfying suitable compatibility conditions. The set of unbroken (regular) trajectories is defined as
	\[
		\calT_M:=\bigl\{\gamma:[0,L]\to M\ \big|\ \dot\gamma=\nabla_g f(\gamma),\ \gamma(0)\in \Sigma_-,\ \gamma(L)\in \Sigma_+,\ L\in \R_{>0}\bigr\}
	\]
	that is, the set of positive gradient flow lines with endpoints on $\Sigma_-$ and $\Sigma_+$. The maps $i_-$ and $i_+$ evaluating the endpoints of trajectories form a diagram:
	\[
	\begin{tikzcd}
		&\calT_M\arrow[dl,"i_-"']\arrow[dr,"i_+"]&\\
		\Sigma_-&&\Sigma_+,
	\end{tikzcd}
	\]
	and therefore $\calT_M$ inherits a topology from $\Sigma_-\times \Sigma_+$.
	
	Starting from a point $p_-\in\Sigma_-$ there is a unique gradient flow line with an endpoint in $p_-$. Then it is either a regular line with the second endpoint $p_+\in \Sigma_+$ or $\gamma:[0,\infty)\to M$, such that $p_-=\gamma(0)$ and $\lim_{t\to \infty}\gamma(t)=p_{c}$ for some critical point $p_c$. Similarly, for any point $p_+\in \Sigma_+$ there is a unique backward trajectory starting at this point which ends up either in a critical point or on $\Sigma_-$. Since the union of all non-regular trajectories forms stable and unstable manifolds $U_-$ and $U_+$, it is natural to identify $
	\calT_M$ with $\Sigma_-\setminus (\Sigma_-\cap U_-)$ or $\Sigma_+\setminus\Sigma_+\cap U_+$. See Fig.~\ref{fig:cob2_trajectories} for a picture of trajectory spaces of a 2-dimensional elementary cobordism.

	We now slightly modify the trajectory space to make it compact, and define it as follows.
	
	\begin{definition}\label{def:trajectory_spaces}
		Let $\Sigma_-\in \mathrm{Surf}$ and let $\bbS_-\subset \Sigma_-$ be an attaching tube. Let $\Sigma_+\simeq\Sigma_-(\bbS_-)$, then the trace $M:=M(\bbS_-):\Sigma_-\to \Sigma_+$ is an elementary cobordism. The trajectory space $\calT_M$ is defined depending on the index of $M$ as follows:
		\begin{itemize}
			\item If $M=M_d$ is the mapping cylinder for some $d:\Sigma_-\to \Sigma_+$ (\ie $\bbS=\varnothing$) then $\calT_M$ is defined as:
			\[
			\calT_{M_d}:=\Sigma_-,\quad\begin{tikzcd}
				&\calT_{M_d}\arrow[dl,"{\Id}"']\arrow[dr,"d"]&\\
				\Sigma_-&&\Sigma_+.
			\end{tikzcd} 
			\]
			\item If $M$ is of index 1 or 2 (\ie $\bbS_-:S^k\times D^{2-k}\to \Sigma_-$ for $k=0,1$) and $\bbS_+:S^{1-k}\times D^{1+k}\to \Sigma_+$ is the belt tube, $\calT_M$ is defined as:
			\begin{equation}\label{eq:trajectory_spaces_span}
			\calT_{M}:=\mathrm{Cl}\left(\Sigma_-\setminus\bS_-\right)\xrightarrow{\simeq} \mathrm{Cl}\left(\Sigma_+\setminus\bS_+\right),\quad \begin{tikzcd}
				&\calT_{M}\arrow[dl,"i_-"']\arrow[dr,"i_+"]&\\
				\Sigma_-&&\Sigma_+,
			\end{tikzcd}
			\end{equation}
			where $\mathrm{Cl}$ means closure inside $\Sigma_-$ or $\Sigma_+$ correspondingly, $i_-$ is the natural inclusion into $\Sigma_-$, and $i_+$ is the natural inclusion into $\Sigma_+$.
		\end{itemize}
		
	\end{definition}
	
	Note that index 1 and 2 cases are dual to each other in the following sense: if $M:\Sigma_-\to \Sigma_+$ is of index 1 (or 2) then $\overline{M}:\overline{\Sigma}_+\to \overline{\Sigma}_-$ is of index 2 (resp. 1). Furthermore $\overline{\bbS}_-$ becomes the belt tube for $\overline{M}$, $\overline{\bbS}_+$ becomes the attaching tube for $\overline{M}$ and $\calT_{\overline{M}}=\overline{\calT}_M$, see Fig.~\ref{fig:index_12_duality}.

	\begin{figure}[h]
		\centering
		\includegraphics[width=0.95\linewidth]{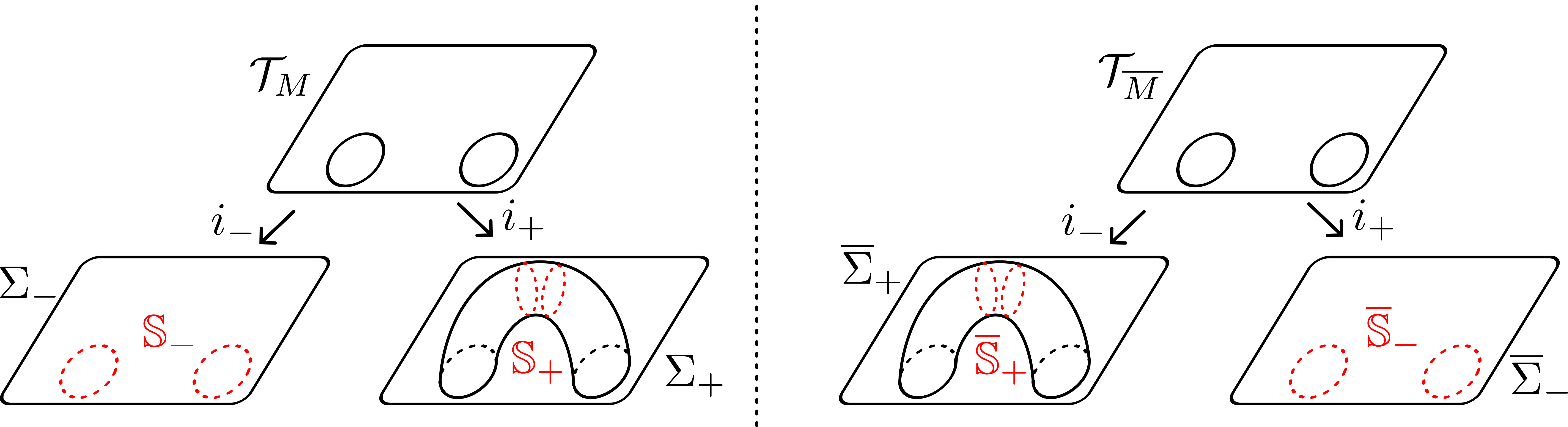}
		\caption{On the left is the trajectory space $\calT_M$ of an index $1$ cobordism $M$: it is a closed subspace of $\Sigma_-$ obtained by removing two open discs; equivalently, it is obtained from $\Sigma_+$ by removing an open cylinder over $S^1$. On the right is the trajectory space of the dual cobordism $\overline{M}$. One may view this duality on the level of trajectory spaces as reflecting the diagram \eqref{eq:trajectory_spaces_span} and simultaneously reversing the orientations of the trajectory space and of the surfaces $\Sigma_-$ and $\Sigma_+$.}
		\label{fig:index_12_duality}
	\end{figure}


	\vspace{0.5cm}
	
	We define the $n$-th multi-trajectory space of an elementary cobordism $M:\Sigma_-\to \Sigma_+$ as the configuration space $\Conf_n(\calT_M)$ together with the induced maps to $\Conf_n(\Sigma_-)$ and $\Conf_n(\Sigma_+)$ which we also denote by $i_-=(i_-)^{\times n}$ and $i_+=(i_+)^{\times n}$:
	\[
		\begin{tikzcd}
			&\Conf_n(\calT_{M})\arrow[dl,"{i_-}"']\arrow[dr,"{i_+}"]&\\
			\Conf_n(\Sigma_-)& &\Conf_n(\Sigma_+).
		\end{tikzcd}
	\]
	
	Since the maps on the level of surfaces are defined only up to isotopy, the induced maps between configuration spaces are defined up to isotopy as well. The trajectory space $\calT_M$ has a boundary component naturally identified with the boundary of $\Sigma_-$ (or, equivalently, of $\Sigma_+$). Denote by $\partial_-\calT_M$ and $\partial_+\calT_M$ the arcs corresponding to $\partial_-\Sigma_-$ and $\partial_+\Sigma_-$, respectively. Therefore the subspaces
	\[
		\Conf_n^-(\calT_M):=\{\{x_1,\dots,x_n\}\:|\:\exists i:\: x_i\in\partial_{-}\calT_M\},\ \ \ \ \Conf_n^+(\calT_M):=\{\{x_1,\dots,x_n\}\:|\:\exists i:\: x_i\in\partial_{+}\calT_M\}
	\]
	are well defined. The maps $i_-$ and $i_+$ send $\Conf_n^-(\calT_M)$ and $\Conf_n^+(\calT_M)$ to the corresponding subspaces of $\Conf_n(\Sigma_-)$ and $\Conf_n(\Sigma_+)$.

	\subsection{Local systems and monodromy representations}\label{sec:local_systems}
	
	Let $X$, $Y$ be path-connected topological spaces. Let $R$ be a commutative ring with unity. A local system of $R$-modules on $X$ is a covariant functor 
	\[
		\calL:\Pi_1(X)\to \mathrm{Mod}_R
	\]
	from the fundamental groupoid $\Pi_1(X)$ of $X$ to the category of $R$-modules. A morphism of local systems is a natural transformation between such functors. If $\calL_1$ and $\calL_2$ are local systems on $X$, then their tensor product $\calL_1\otimes\calL_2$ is defined pointwise by
	\[
		(\calL_1\otimes\calL_2)(x):=\calL_1(x)\otimes \calL_2(x),\,\forall x\in X,\qquad
		(\calL_1\otimes\calL_2)(h):=\calL_1(h)\otimes\calL_2(h),\,\forall h\in \mathrm{Mor}(\Pi_1(X)).
	\]
	If $X$ and $Y$ are equipped with local systems $\calL_X$ and $\calL_Y$, we write
	\[
		\calL_X\boxtimes \calL_Y:\Pi_1(X\times Y)=\Pi_1(X)\times \Pi_1(Y)\to \mathrm{Mod}_R
	\]
	for their exterior product, given by
	\[
		(x,y)\mapsto \calL_X(x)\otimes \calL_Y(y),\, \forall (x,y)\in X\times Y,
	\]
	\[
		(h_X,h_Y)\mapsto \calL_X(h_X)\otimes \calL_Y(h_Y),\,\forall (h_X,h_Y)\in \mathrm{Mor}(\Pi_1(X)\times \Pi_1(Y)).
	\]
	Given a map of topological spaces $f:X\to Y$ and a local system $\calL_Y$ on $Y$, its pull-back $f^*\calL_Y$ is the local system on $X$ obtained by composing $\calL_Y$ with the induced morphism $f_*:\Pi_1(X)\to \Pi_1(Y)$. Likewise, a morphism of local systems $\rho:\calL_Y\to \calL_Y'$ induces a pull-back morphism $f^*\rho:f^*\calL_Y\to f^*\calL_Y'$.
	
	 By definition of a local system $\calL$, for any two points $x,x'\in X$ and any homotopy class of paths $h\in\Pi_1(X)$ between them there is an isomorphism between their images $\mathcal{L}(x)\xrightarrow[h]{\simeq}\mathcal{L}(x')$. Therefore, fixing a base point $x_0\in X$ a local system $\mathcal{L}$ can be described as a (left) $\pi_1(X,x_0)$-module $L$. With our convention that $\alpha\beta$ traverses $\alpha$ first, the left action is $\gamma\cdot v=\calL(\gamma^{-1})(v)$. We call $L$ the monodromy representation of $\calL$ at $x_0$. A morphism of local systems induces a morphism of their monodromy representations at the same point. Monodromy representations are an alternative way to describe local systems.
	
	\begin{proposition}[See for example \cite{whitehead2012elements}, Chapter VI, Theorems 1.11, 1.12]
		The category of local systems on a path-connected space $X$ is equivalent to the category of $\pi_1(X,x_0)$-modules, for any choice of base point $x_0\in X$.
	\end{proposition}
	
	If $f:(X,x_0)\to (Y,y_0)$ is a map of pointed spaces and $L_Y$ is the monodromy representation of a local system $\calL_Y$ on $Y$ at $y_0$, then the pull-back $f^*\calL_{Y}$ has monodromy representation $f^*L_Y$ identified with $L_Y$ as an $R$-module, but with twisted action: any $\gamma\in \pi_1(X,x_0)$ acts by $f_*(\gamma)$.

	We denote by $\underline{R}$ the trivial local system on $X$ that sends every point to $R$ and every homotopy class of paths to the identity. Suppose $R$ is a ring with involution:
	\[
		R\to R, \quad x\to \overline{x}.
	\]
	If $M$ is an $R$-module then $\overline{M}$ is the $R$-module with conjugated action: 
	\[
		r\cdot_{\overline{M}} m :=\overline{r}\cdot_M m, \quad \forall r\in R, \ \forall m\in \overline{M}.
	\]
	Similarly, if $\calL$ is a local system of $R$-modules, then $\overline{\calL}$ is the local system obtained by pointwise replacement of $\calL(x)$ by $\overline{\calL(x)}$.

	 A sesquilinear pairing on a local system $\calL$ of $R$-modules is a morphism of local systems
	\[
		(-,-):\calL\otimes\overline{\calL}\to \underline{R}.
	\]
	Recall that a pairing of local systems $\calA\otimes \calB\to \underline{R}$ is perfect if both dual maps $\calA\to \Hom(\calB,\underline{R})$ and $\calB\to \Hom(\calA,\underline{R})$ are isomorphisms. For local systems equipped with perfect sesquilinear pairings, we distinguish two categories with the same objects.
	\begin{definition}\label{def:paired_local_systems}
	For a topological space $X$, let $\Locpair_R(X)$ be the category whose objects are local systems $\calL:\Pi_1(X)\to\mathrm{Mod}_R$ equipped with a perfect sesquilinear pairing
	\[
		(-,-):\calL\otimes\overline{\calL}\to\underline{R},
	\]
	and whose morphisms are all $R$-linear morphisms of the underlying local systems. 
	Let $\Loc_R(X)$ be the subcategory with the same objects and with morphisms $\psi:\calL_1\to\calL_2$ preserving the pairings, \ie making the diagram
	\[
	\begin{tikzcd}
		\calL_1\otimes\overline{\calL}_1\arrow[r,"\psi\otimes\overline{\psi}"]\arrow[d,"{(-,-)}"]&\calL_2\otimes\overline{\calL}_2\arrow[d,"{(-,-)}"]\\
		\underline{R}\arrow[r,"\Id"]&\underline{R}
	\end{tikzcd}
	\]
	commute. Thus there is an identity-on-objects inclusion $\Loc_R(X)\hookrightarrow\Locpair_R(X)$. A pairing-preserving morphism need not be an isomorphism. Whenever an isomorphism in $\Loc_R(X)$ is required, this will be stated explicitly.
	\end{definition}


	If $L$ is the monodromy representation of $\calL\in \Locpair_R(X)$ at a point $x_0$, then the pairing is equivalently given by a sesquilinear $R$-pairing
	\[
	(-,-):L\otimes \overline{L}\to R,
	\]
	which is equivariant in the following sense:
	\[
	(\gamma\cdot v,\gamma\cdot u)=(v,u),\qquad \forall\gamma\in \pi_1(X,x_0),\ \forall v,u\in L.
	\]
	Morphisms in $\Locpair_R(X)$ correspond to equivariant $R$-linear maps of monodromy representations, while morphisms in $\Loc_R(X)$ additionally preserve the pairings. In both categories the monodromy itself preserves the pairing. Perfectness also gives adjoints of arbitrary coefficient maps: for $f:L\to L'$ there is a unique $R$-linear map $f^{\vee}:L'\to L$ such that $b'(f(v),\overline w)=b(v,\overline{f^{\vee}(w)})$. Equivariance of the pairings implies equivariance of $f^{\vee}$.
	
	The conjugate local system carries the conjugate pairing: on fibers, $\overline b(\overline v,w)=\overline{b(v,\overline w)}$ if $b:L\otimes\overline L\to R$ is the original pairing. Tensor products and exterior products inherit product pairings. These don't need to be perfect in general. We regard such a product as an object of either paired category only when its product pairing is perfect.
	
	Although no projectivity or finite-generation assumption is imposed on the local systems, the monoidality condition in the following sections will require those properties. This will follow from the technical lemma below.

	\begin{lemma}\label{lem:paired_tensor_projective}
		Let $R$ be a commutative ring with involution, and let $L$ be an
		$R$-module equipped with a non-degenerate sesquilinear pairing
		\[
		b:L\otimes_R\overline L\longrightarrow R.
		\]
		Set $M=L\otimes_R\overline L$. Equip $M$ with the sesquilinear
		pairing induced by $b$ and its conjugate:
		\[
		B:M\otimes_R\overline M\longrightarrow R,
		\]
		defined, under the natural identification
		$\overline M\cong\overline L\otimes_R L$, by
		\[
		B(x\otimes\overline y,\overline u\otimes v)
		=b(x,\overline u)\,\overline{b(y,\overline v)},
		\qquad x,y,u,v\in L.
		\]
		If this induced pairing $B$ is also perfect, then $L$ is
		finitely generated projective over $R$.
	\end{lemma}
	\begin{proof}
		Since $B$ is perfect it induces an isomorphism $\overline{L}\otimes_R  L\to \Hom_R(L\otimes_R \overline{L},R)$. Hence the functional $b$ is represented by a finite tensor $\sum_i\overline{u}_i\otimes v_i\in \overline{L}\otimes_R L$:
		
		\[
			b(x,\overline{y})=\sum_i b(x,\overline{u}_i)\overline{b(y,\overline{v}_i)},\quad \forall x,y\in L.
		\]
		Since $b$ is non-degenerate, we have:
		\[
			y=\sum_iu_i b(y,\overline{v}_i), \quad \forall y\in L.
		\]
		Thus $(u_i,b(-,\overline{v}_i))$ is a finite dual basis: the maps 
		\[
			L\to R^m, \quad y\mapsto(b(y,\overline{v}_1),\dots,b(y,\overline{v}_m)),\quad \text{and}\quad R^m\to L,\quad (r_1,\dots,r_m)\mapsto\sum_i r_i u_i
		\] 
		exhibit $L$ as a direct summand of a finite free module.
	\end{proof}

	\subsection{Twisted homology}\label{subs:Twisted_homology}
		We briefly give a definition and some properties of twisted (Borel--Moore) homology in terms of monodromy representations here. For a more detailed introduction see \cite{davis2001lecture}, Chapter~5.
		Let $X$ be a path-connected, locally path-connected, semilocally simply connected topological space. Fix a local system $\mathcal{L}$ on it with monodromy representation $L$ at some base point $x_0$. Let $\tilde{p}:\tilde{X}\to X$ be the universal cover of $X$ together with a choice of a base point $\tilde{x}_0\in\tilde{X}$ such that $\tilde{p}(\tilde{x}_0)=x_0$. Then it is equipped with a left action of $\pi_1(X,x_0)$ as follows: a point $x$ of $\tilde{X}$ represented by a path $\tilde{x}:[0,1]\to X$ from $\tilde{x}_0$ is sent by a loop $\gamma$ to the point $\gamma\cdot x$ represented by the path $\gamma \tilde{x}$. It can also be regarded as a right action by setting $x\cdot \gamma:=\gamma^{-1}\cdot x$. Then the (singular) chain complex $C_*(\tilde{X})$ is equipped with a right $\pi_1(X,x_0)$ action and the chain complex $C_*(\tilde{X})\otimes_{\pi_1(X,x_0)}L$ is defined. Homology of this complex is the twisted homology of $X$ with coefficients in $L$:
	\[
		H_*(X;\calL):=H_*(C_*(\tilde{X})\otimes_{\pi_1(X,x_0)}L, \partial\otimes \Id).
	\]
	The relative version of twisted homology $H_*(X,Y;\calL)$ for a subspace $Y\subset X$ is defined as homology of the complex $C_*(\tilde{X};\tilde{p}^{-1}(Y))\otimes_{\pi_1(X,x_0)}L$.

	Twisted Borel--Moore homology of the pair $(X,Y)$ with coefficients in $\calL$ is defined as the (projective) limit:
	\begin{equation}
		H^{BM}_*(X,Y;\calL):=\lim\limits_{K\in\calK}H_*(C_*(\tilde{X},\tilde{p}^{-1}(Y\cup X\setminus K))\otimes_{\pi_1(X,x_0)}L, \partial\otimes \Id)
	\end{equation}
	over the set $\calK$ of compact subspaces $K\subset X$ partially ordered by inclusion. 
	
	
	Since any cycle in ordinary homology is also a cycle in Borel--Moore homology there is a natural homomorphism 
	\begin{equation}\label{eq:iota}
		\iota: H_*(X,Y;\calL)\to  H_*^{BM}(X,Y;\calL).
	\end{equation}

	Twisted homology is functorial with respect to local systems: if $\phi:\calL_1\to\calL_2$ is a morphism of local systems on $(X,x_0)$ then for both versions of twisted homology there are induced homomorphisms:
	\[
		\phi_*:H_*(X;\calL_1)\to H_*(X;\calL_2),\quad \phi_*:H^{BM}_*(X;\calL_1)\to H^{BM}_*(X;\calL_2).
	\]
	Twisted homology is also functorial with respect to maps between topological spaces in the following sense. Let $f:X\to Y$ be a map between topological spaces and $\calL$ a local system on $Y$. Then there is an induced homomorphism of twisted homology:
	\begin{equation}\label{eq:functoriality_H}
		f_*:H_*(X;f^*\calL)\to H_*(Y;\calL).
	\end{equation}
	If in addition $f$ is a proper map then there is an induced homomorphism of twisted Borel--Moore homology as well:
	\begin{equation}\label{eq:functoriality_H_BM}
		f_*:H_*^{BM}(X;f^*\calL)\to H_*^{BM}(Y;\calL).
	\end{equation}
	Therefore, when we need to construct a homomorphism $H_*^{BM}(X;\calL_1)\to H_*^{BM}(Y;\calL_2)$ induced by a (proper) map $f:X\to Y$ we make it by composing two homomorphisms: one induced by a change of local system $\calL_1\to f^*\calL_2$ and another induced by $f:X\to Y$ itself.
	
	\vspace{0.5cm}
	
	Let $\Sigma$ be either an object in $3\Cob$ or a trajectory space of an elementary cobordism. Suppose a local system $\calL_n$ on $\Conf_n(\Sigma)$ is defined for all $n\geq 0$. It is convenient to organize all configuration spaces into the space
	\[
		\Conf_*(\Sigma):=\bigsqcup\limits_{n\geq 0} \Conf_n(\Sigma),
	\]
	and interpret the collection of local systems $\{\calL_n\}_{n\geq 0}$ as a local system $\calL$ on $\Conf_*(\Sigma)$. We denote
	\begin{equation}
		\bbH_n(\Sigma;\calL):=H_n(\Conf_*(\Sigma),\Conf_*^-(\Sigma);\calL)
	\end{equation}
	and
	\begin{equation}
		\bbH_n^{BM}(\Sigma;\calL):=H_n^{BM}(\Conf_*(\Sigma),\Conf_*^-(\Sigma);\calL).
	\end{equation}

	We denote graded sums over all degrees by
	\begin{equation}
		\bbH(\Sigma;\calL):=\bigoplus\limits_{n\geq 0}\bbH_n(\Sigma;\calL),\quad \bbH^{BM}(\Sigma;\calL):=\bigoplus\limits_{n\geq 0}\bbH_n^{BM}(\Sigma;\calL).
	\end{equation}

	\section{Heisenberg homology}\label{sec:Heisenberg_homology}
	In this section we define the state spaces of a homological TQFT. Twisted Borel--Moore homology of configuration spaces of a surface $\Sigma\in 3\Cob$ can be organized into a representation of $\mathrm{Mod}(\Sigma,\partial\Sigma)$, as shown in \cite{blanchet2025heisenberg} and \cite{de2022homological} for a specific local system. In particular, we are interested in the $p$-Heisenberg local systems introduced in Subsection~\ref{sec:Heisenberg_ls}. Twisted Borel--Moore homology with coefficients in a Heisenberg local system is sometimes called \emph{Heisenberg homology}. Thanks to the compression trick (see \cite{blanchet2025heisenberg}, Theorem 11), Heisenberg homology can be computed explicitly and admits a graphical calculus on twisted cycles, briefly reviewed in Subsection~\ref{subs:Twisted_cycles_and_bases_in_twisted_homology}. The action of elementary index $1$ and $2$ cobordisms defined in the following section requires a non-degenerate self-pairing on the state spaces. In Subsection~\ref{sec:pairing} we discuss an intersection pairing on twisted homology with coefficients in an $\calL\in \Locpair_R(\Conf_*(\Sigma))$. This pairing involves two versions of homology and is perfect in each particle degree. Restricting to a certain submodule gives an ordinary perfect pairing. This submodule is called \emph{the subspace of small cycles}, and Subsection~\ref{sec:small_cycles} is devoted to its definition. The state space associated with $\Sigma\in 3\Cob$ is then defined as the subspace of small cycles in the corresponding homology.

	\subsection{Heisenberg local systems}\label{sec:Heisenberg_ls}
		
	We give two equivalent definitions of the (finite) Heisenberg group $\Heis_p(\Sigma)$, for $p=2$ or odd $\geq 3$, associated with $\Sigma$, where $\Sigma$ is either an object of $3\Cob$ or a trajectory space of an elementary cobordism. One definition is useful for an invariant formulation of the action of cobordisms, while the other is better suited for explicit computations in a preferred basis. We identify $H_1(\Sigma;\partial_-\Sigma)\simeq H_1(\Sigma)$.

	\vspace{0.5cm}

	Let $H_1(\Sigma)\times H_1(\Sigma)\to \Z$, $(x,y)\mapsto x.y$ be the intersection pairing on $H_1(\Sigma)$, then the Heisenberg group $\Heis'(\Sigma)$ is the central extension of $H_1(\Sigma)$ given by this pairing. As a set $ \Heis'(\Sigma)\simeq \Z\times H_1(\Sigma)$ and multiplication is given by the formula:
	\[
		(k,x)\cdot (l,y)=(k+l+x.y,x+y),
	\]
	where $k,l\in\Z$ and $x,y\in H_1(\Sigma)$.
	
	Recall that $\sigma_i$, $i=1,\dots,n-1$, are the classical generators of the braid group. The second definition of the Heisenberg group appears as a quotient of the surface braid group $B_n(\Sigma)$:
	\begin{equation}\label{eq:Heisenberg_def}
		\Heis(\Sigma):=B_n(\Sigma)/[B_n(\Sigma),\sigma_1], \quad n\geq 2.
	\end{equation}
		For a loop $\gamma\in B_n(\Sigma)$ denote by $[\gamma]_{\Heis}$ its image in $\Heis(\Sigma)$ under the quotient homomorphism \eqref{eq:Heisenberg_def}. The classical braid relation $\sigma_i\sigma_{i+1}\sigma_i=\sigma_{i+1}\sigma_i\sigma_{i+1}$ involves that each $\sigma_i$ is sent to the same element by this quotient, and we denote
		\[
			\boxed{\sigma:=[\sigma_1]_{\Heis}.}
		\]
	
	The surface braid group $B_n(\Sigma)$ has an explicit presentation first given in \cite{bellingeri2004presentations}, see also \cite{blanchet2025heisenberg} for a discussion in the context of homological representations of mapping class groups. We only describe a set of generators as follows. Call a set of generators $\{\alpha_i,\beta_i\}_{i=1}^g$ of $\pi_1(\Sigma_g,*_1)$ symplectic if their homology classes satisfy: $[\alpha_i].[\alpha_j]=[\beta_i].[\beta_j]=0$ and $[\alpha_i].[\beta_j]=\delta_{ij}$ for any $i,j\in\{1,\dots,g\}$. An oriented embedded arc $\gamma:[0,1]\hookrightarrow \Sigma$ such that $\gamma(0),\gamma(1)\in \partial_-\Sigma$ determines a unique element in $\pi_1(\Sigma,*_1)$ by sliding both endpoints to $*_1$ along $\partial_-$. We will call a set of disjoint arcs symplectic if they are in one to one correspondence with a symplectic set of generators. We choose a standard symplectic set of arcs on each $\Sigma_g$ depicted in Fig.~\ref{fig:symplectic_arcs}.
	
	\begin{figure}[h]
		\centering
		\includegraphics[width=0.4\linewidth]{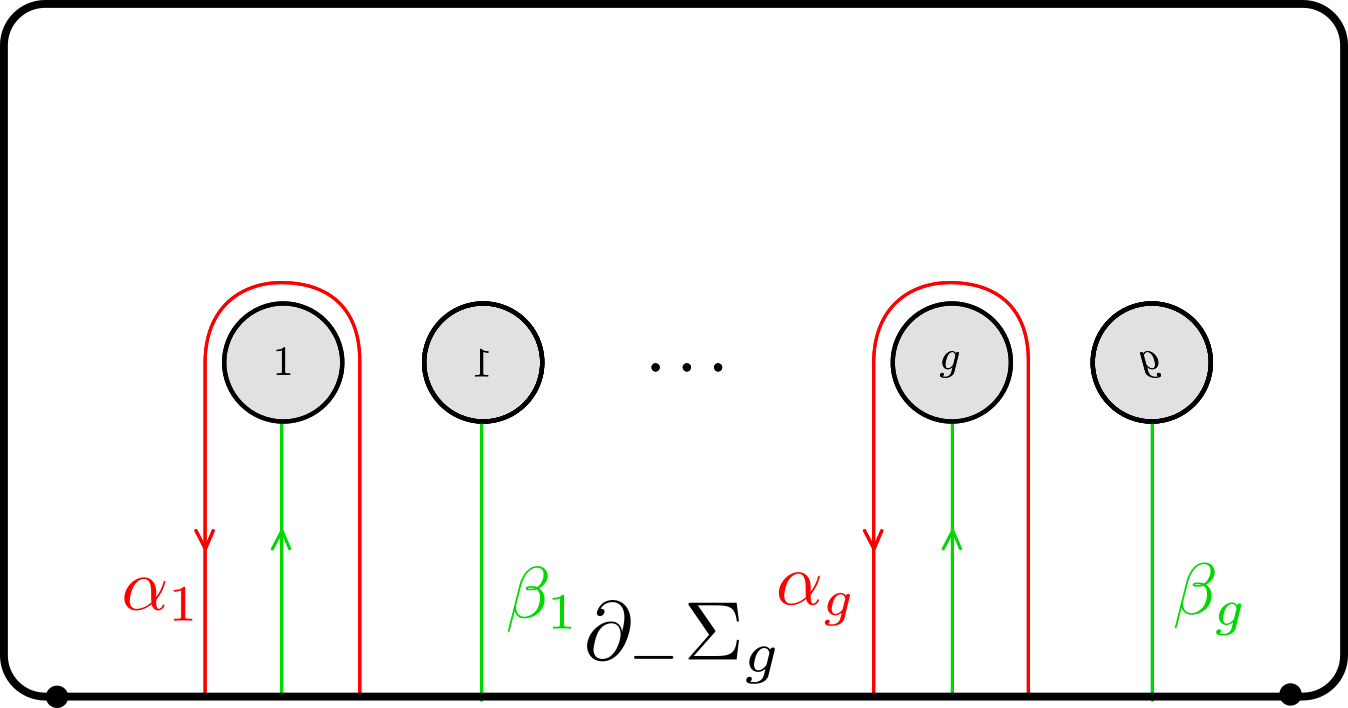}
		\caption{A choice of a symplectic set of arcs in a standard surface $\Sigma_g$.}
		\label{fig:symplectic_arcs}
	\end{figure}

	For a one-boundary surface put $G:=\{\alpha_i,\beta_i\}_{i=1}^g$. For a trajectory surface, adjoin two based loops $\delta_1,\delta_2$ around its additional boundary components to obtain $G$, then their homology classes lie in the radical of the intersection pairing. The braid group $B_n(\Sigma)$ is generated by $\sigma_j$, $1\leq j\leq n-1$, and $\tilde\gamma$ for $\gamma\in G$, where $\tilde\gamma$ moves the leftmost point of $*_n$ along $\gamma$ and is constant on all other points. 
	Then one can define a surjective homomorphism 
	\[
		\rho_{G}:B_n(\Sigma)\to \Heis'(\Sigma),
	\]
	\[
		\sigma_i\mapsto (-1,0),\qquad
		\tilde\gamma\mapsto (0,[\gamma]),\quad \gamma\in G.
	\]
	\begin{proposition}[Compare \cite{blanchet2025heisenberg}]\label{prop:Heisenberg_iso}
		The homomorphism $\rho_{G}$ is well-defined and has kernel $[B_n(\Sigma),\sigma_1]$. Therefore, given a set of generators $G$ as above, there is an isomorphism:
		\begin{equation}\label{eq:def_Heisenberg}
			\rho_{G}:\Heis(\Sigma)\xrightarrow{\simeq}\Heis'(\Sigma).
		\end{equation}
	\end{proposition}
	For completeness, the same presentation argument covers trajectory surfaces. Imposing centrality of $\sigma_1$ in the surface-braid presentation of \cite[Theorem~1.1]{bellingeri2004presentations} identifies all $\sigma_j$ with $\sigma$. The boundary generators commute with every generator, and the only nontrivial handle commutators are $[\tilde\alpha_i,\tilde\beta_j]=\sigma^{-2\delta_{ij}}$. Thus elements have normal forms
	\[
		\sigma^c\prod_{i=1}^g\tilde\alpha_i^{u_i}\tilde\beta_i^{v_i}
		\prod_{j=1}^2\tilde\delta_j^{w_j},
	\]
	with all exponents integral. Their images under $\rho_G$ are distinct. For genus zero, the punctured-disk presentation gives the same argument with the handle generators omitted. For one-boundary surfaces the boundary generators are omitted. This proves the stated isomorphism in all cases used here, including $n=2$.
	Note that in particular, the quotient $B_n(\Sigma)/[B_n(\Sigma),\sigma_1]$ doesn't depend on $n$ for $n\geq2$.
	
	Both versions are intrinsically defined, but their identification $\rho_G$ depends on the chosen generators and need not intertwine their natural mapping-class actions.

	Both versions of the Heisenberg group are naturally a part of the short exact sequence:
	\begin{equation}\label{s_e_s}
		1 \to \Z\xrightarrow{c} \Heis(\Sigma)\xrightarrow{\pi} H_1(\Sigma)\to1
	\end{equation}
	where $c$ is the inclusion of the distinguished central cyclic subgroup (the additional boundary classes of a trajectory surface give further central elements), $\pi$ is the obvious projection in the case of $\Heis'(\Sigma)$. For $\Heis(\Sigma)$ the map $\pi$ can be described as follows. Let $\gamma\in B_n(\Sigma)$ be an element in the preimage of $x\in \Heis(\Sigma)$ under the quotient homomorphism. Let $\gamma_i:[0,1]\to \Sigma$, $i=1,\dots,n$ be the trajectories of each configuration point, then
	\[
		\pi(x):=\sum_i [\gamma_i],
	\]
	where $[\gamma_i]\in H_1(\Sigma;\partial_-\Sigma)\simeq H_1(\Sigma)$. 
	
	Independently of the choice of the isomorphism \eqref{eq:def_Heisenberg} one obtains a formula for commutators in $\Heis(\Sigma)$:
	\begin{equation}\label{eq:commutator_formula}
		\big[[\gamma]_{\Heis},[\gamma']_{\Heis}\big]=\sigma^{-2\pi(\gamma).\pi(\gamma')}.
	\end{equation}

	We now define finite versions of Heisenberg groups in two separate cases. Let $p=2$, then the group $\Heis_p(\Sigma)=\Heis_2(\Sigma)$ is defined as the quotient
	\[
		\Heis_2(\Sigma):=\Heis(\Sigma)/ I_2,\quad I_2=\{x^2\,| \, x\in \Heis(\Sigma) \}.
	\]
	Note that this group is commutative since $\sigma^2\in I_2$. The group $\Heis_2'(\Sigma)$ can be defined as the central extension of $H_1(\Sigma;\Z_2)$ by $\Z_2$. Any isomorphism \eqref{eq:def_Heisenberg} descends to an isomorphism of finite versions $\Heis_2(\Sigma)\xrightarrow{\simeq}\Heis'_2(\Sigma)$.
	
	If $p$ is an odd integer $\geq 3$, we define the finite Heisenberg group $\Heis_p(\Sigma)$ as follows. Recall that the subgroup of pure braids $P_n(\Sigma)\subset B_n(\Sigma)$ is defined as the kernel of the natural homomorphism $B_n(\Sigma)\to S_n$ to the symmetric group $S_n$, \ie each loop in $P_n(\Sigma)$ is a braid with each strand starting and ending at the same point. We define a subset of $\Heis(\Sigma)$
	\begin{equation}\label{eq:finite_Heisenbeg_def}
		I_p:=\{[\gamma]_{\Heis}^p\,| \, \gamma\in P_n(\Sigma)\}, \quad p \text{ is odd }\geq 3,\ n\geq 2.
	\end{equation}
	Using \eqref{eq:commutator_formula} the following computation shows that $I_p$ is in fact a subgroup:
	\[
		[\gamma]_{\Heis}^p[\gamma']_{\Heis}^p=\sigma^{-(p-1)p\pi(\gamma).\pi(\gamma')}[\gamma\gamma']_{\Heis}^p=[\sigma_1^{-(p-1)\pi(\gamma).\pi(\gamma')}\gamma\gamma']_{\Heis}^p, \quad \gamma,\gamma'\in P_n(\Sigma),
	\]
	where $\sigma_1^{-(p-1)\pi(\gamma).\pi(\gamma')}$ is a pure braid since $p-1$ is even.
	It is easy to see that $I_p$ is a normal subgroup since 
	\[
		[\eta]_{\Heis}[\gamma]_{\Heis}^p[\eta^{-1}]_{\Heis}=\sigma^{2p\pi(\gamma).\pi(\eta)}[\gamma]_{\Heis}^p=[\sigma_1^{2\pi(\gamma).\pi(\eta)}\gamma]_{\Heis}^p, \quad \forall \gamma\in P_n(\Sigma), \ \eta\in B_n(\Sigma).
	\]
	The finite quotient is therefore well defined:
	\[
		\Heis_p(\Sigma):=\Heis(\Sigma)/I_p.
	\]
	The normal forms above descend with the exponent of $\sigma$ taken modulo $2p$ for odd $p$ (modulo $2$ for $p=2$), and all other exponents modulo $p$.


	
	
	
	
	
	
	There is a well-defined action of $\mathrm{Mod}(\Sigma,\partial\Sigma)$ on finite Heisenberg groups as follows. Let $d_*:B_n(\Sigma)\to B_n(\Sigma')$ be the homomorphism induced by a mapping class $d:\Sigma\to \Sigma'$. It sends $\sigma_1$ to $\sigma_1$ and therefore descends to a homomorphism $d_*:\Heis(\Sigma)\to \Heis(\Sigma')$. Since $d_*$ preserves the subgroup of pure braids on the level of surface braid groups, the subgroups $I_p$ are invariant under $d_*$ both for $p=2$ and for odd $p\geq 3$. Hence there is a well-defined homomorphism of finite Heisenberg groups
	\[
		d_*:\Heis_p(\Sigma)\to \Heis_p(\Sigma'),
	\]
	induced by $d$.
	
	
	
	
	
	
	\vspace{0.5cm}

	\noindent\textit{Heisenberg local systems.} A \emph{$p$-Heisenberg local system} on $\Conf_*(\Sigma)$ consists of a left $R[\Heis_p(\Sigma)]$-module $L_n$ for every $n\geq0$, together with the associated local system $\calL_n$ on $\Conf_n(\Sigma)$. For $n\geq2$, its monodromy is obtained by restriction along the quotient $B_n(\Sigma)\to\Heis_p(\Sigma)$. For $n=1$, first push the moving point away from a boundary collar and insert one fixed point on $\partial_-\Sigma$. The resulting stabilization induces a homomorphism
	\[
		B_1(\Sigma)\longrightarrow B_2(\Sigma)\longrightarrow\Heis_p(\Sigma),
	\]
	along which the action on $L_1$ is restricted. For $n=0$, restriction is along the unique homomorphism from $B_0(\Sigma)=\{1\}$, so $\calL_0$ is the local system on a point with fiber $L_0$. The chosen Heisenberg action on $L_0$ is additional coefficient data. We write $\calL=\{\calL_n\}_{n\geq0}$ for the resulting collection of local systems.
	
	Any $R$-module representation of $\Heis_p(\Sigma)$ is a representation of the group ring $R[\Heis_p(\Sigma)]$. It is equipped with a canonical basis given by elements of $\Heis_p(\Sigma)$. For convenience we introduce a formal variable $q$ such that
	\[
		\boxed{\sigma=-q^{-2}}
	\]
	and it commutes with all elements of $R[\Heis_p(\Sigma)]$. Although some explicit formulas below contain $q$, it only appears with even powers, therefore all such expressions are well-defined elements of $R[\Heis_p(\Sigma)]$.
	
	Note that if $\calL$ is a $p$-Heisenberg local system in $\Conf_*(\Sigma)$ it can also be viewed as a local system on $\Conf_*(\overline{\Sigma})$ up to the isomorphism $\Heis_p(\Sigma)\to \Heis_p(\overline{\Sigma})$ flipping the central generator: $\sigma\mapsto \sigma^{-1}$.

	\subsection{Twisted cycles and bases in twisted homology}\label{subs:Twisted_cycles_and_bases_in_twisted_homology}

	Let $\Sigma\in 3\Cob$. If $\calL$ is a $p$-Heisenberg local system, the Borel--Moore homology $H_*^{BM}(\Conf_n(\Sigma);\Conf_n^-(\Sigma);\calL)$ is concentrated in the middle dimension of $\Conf_n(\Sigma)$ in the following sense.
	
	\begin{proposition}[\cite{blanchet2025heisenberg}, Theorem A]\label{prop:middle_dimension}
		For any $n\geq 0$ let $\calL$ be a $p$-Heisenberg local system of $R$-modules on $\Conf_n(\Sigma)$ with monodromy representation $L$ at $*_n$. Then there is an isomorphism of $R$-modules
		\begin{equation}\label{eq:homology_decomposition}
			H^{BM}_k(\Conf_n(\Sigma);\Conf_n^-(\Sigma);\calL)\simeq\begin{cases}
				L^{\oplus d}, & k=n\\
				0, & k\neq n,
			\end{cases}
			\quad\quad d=d(g,n)
		\end{equation}
		where $g$ is the genus of $\Sigma$ and $d(g,n)$ is the number of weak compositions of $n$ into $2g$ parts: $d(g,n)=\binom{2g+n-1}{n}$ for $g>0$, while $d(0,0)=1$ and $d(0,n)=0$ for $n>0$.
	\end{proposition}
	\begin{remark}\label{rem:low_degree_disk}
		The cited theorem treats positive genus and $n\geq2$. The remaining cases admit direct proofs. For $n=0$, the configuration space is a point and its relative subspaces are empty, so both ordinary and Borel--Moore homology are $L_0$ in degree zero and vanish otherwise. For $n=1$, compactness of $\Sigma$ gives
		\[
			H^{BM}_k(\Conf_1(\Sigma),\Conf_1^-(\Sigma);\calL_1)=H_k(\Sigma,\partial_-\Sigma;\calL_1).
		\]
		The pair $(\Sigma,\partial_-\Sigma)$ has a relative cellular model with $2g$ one-cells and no relative zero- or two-cells. Its twisted relative cellular complex is therefore $L_1^{\oplus2g}$ concentrated in degree one.

		For a disk $D$ and $n>0$, both ordinary and Borel--Moore relative homology vanish for arbitrary local coefficients. To see this, identify $(D,\partial_-D)$ with the square and its bottom edge. For a configuration $\mathbf{x}=\{(x_i,y_i)\}_{i=1}^n$, put $m(\mathbf{x})=\min_i y_i$. The homotopy
		\[
			H_t(\mathbf{x})=\{(x_i,y_i-tm(\mathbf{x}))\}_{i=1}^n,\qquad 0\leq t\leq1,
		\]
		is a strong deformation retraction onto $\Conf_n^-(D)$. It preserves all pairwise distances, so collisions cannot occur. The homotopy is proper, so both relative homology theories vanish, also with local coefficients.
	\end{remark}

	Let $\calL$ be a $p$-Heisenberg local system on $\Conf_*(\Sigma)$, $\calL_n=\calL|_{\Conf_n(\Sigma)}$. Then by Proposition~\ref{prop:middle_dimension} we have:
	\[
		\bbH^{BM}_n(\Sigma;\calL)=H^{BM}_n(\Conf_n(\Sigma);\Conf_n^-(\Sigma);\calL_n), \quad \bbH^{BM}(\Sigma;\calL)=\bigoplus_{n=0}^{\infty}H^{BM}_n(\Conf_n(\Sigma);\Conf_n^-(\Sigma);\calL_n).
	\]
	
		\vspace{0.5cm}




	We now introduce twisted cycles. Let $\{\gamma_i\}_{i=1}^{r}$ be a collection of disjoint embedded arcs $\gamma_i:(I,\partial I)\to (\Sigma,\partial_-\Sigma)$. 
	Choose a composition $\mathbf{k}\models n$ of $n$ (\ie a sequence of non-negative integers $\mathbf{k}=(k_1,k_2,\dots,k_{r})$ such that $k_1+k_2+\dots+k_{r}=n$).
	Each $\gamma_i$ induces a proper map
	\[
		\Gamma_i:(\Conf_{k_i}(I),\partial\Conf_{k_i}(I))\to (\Conf_{k_i}(\Sigma),\partial_-\Conf_{k_i}(\Sigma))
	\]
	and their product map
	\begin{multline*}
		\Gamma_{\mathbf{k}}:=\Gamma_1\times\dots\times\Gamma_{r}:\left(\Conf_{k_1}(I)\times\dots \times\Conf_{k_{r}}(I),\partial(\Conf_{k_1}(I)\times\dots \times\Conf_{k_{r}}(I))\right)\to\\
		\to(\Conf_n(\Sigma),\Conf_n^-(\Sigma)) 
	\end{multline*}
	is well defined since the arcs don't have common points. Since each configuration space $\Conf_{k_i}(I)$ is contractible the fundamental group of the product $\Conf_{k_1}(I)\times\dots \times\Conf_{k_{r}}(I)$ is trivial. Thus any local system on the product is trivial as well. In particular, this involves an isomorphism of $R$-modules:
	\[
		H^{BM}_n(\Conf_{k_1}(I)\times\dots \times\Conf_{k_{r}}(I),\partial(\Conf_{k_1}(I)\times\dots \times\Conf_{k_{r}}(I));\Gamma_{\mathbf{k}}^*\calL)\xrightarrow{\simeq} \Gamma_{\mathbf{k}}^*L.
	\]
	We choose some base configurations $*^I_{k_j}$ for each $\Conf_{k_j}(I)$, then $*^I:=*^I_{k_1}\times *^I_{k_2}\times \dots \times *^I_{k_{r}}$ is a base configuration in the product. Note that orientation of each arc induces the product orientation on each $\Conf_{k_1}(I)$. Then specifying a path from the base configuration $*_n$ to $*^I$ inside $\Conf_n(\Sigma)$ one obtains an identification of $R$-modules:
	\[
		L\xrightarrow{\simeq} \Gamma_{\mathbf{k}}^*L.
	\]
	Hence there is a homomorphism of $R$-modules:
	\begin{equation}
		(\Gamma_{\mathbf{k}}): L\to H^{BM}_n(\Conf_n(\Sigma),\Conf_n^{-}(\Sigma);\calL)
	\end{equation}
	induced by the isomorphisms above and the homomorphism:
	\begin{multline*}
		(\Gamma_{\mathbf{k}})_*:H^{BM}_n(\Conf_{k_1}(I)\times\dots \times\Conf_{k_r}(I),\partial(\Conf_{k_1}(I)\times\dots \times\Conf_{k_r}(I));\Gamma_{\mathbf{k}}^*\calL)\to\\
		\to H^{BM}_n(\Conf_n(\Sigma),\Conf_n^{-}(\Sigma);\calL).
	\end{multline*}
	For an element $\mathbf{v}\in L$ its image under $	(\Gamma_{\mathbf{k}})_*$ is denoted by
	\[
		\mathbf{\Gamma}(\mathbf{k})\otimes \mathbf{v}:=(\Gamma_{\mathbf{k}})_*\mathbf{v}.
	\]
	It is a homology class that can be represented by an explicit cycle called \emph{twisted cycle}. It is drawn as the collection of arcs $\{\gamma_i\}$ together with a path between the base points and the element $\mathbf{v}$. Different twisted cycles representing the same homology class can be related by a set of moves. For diagrammatic calculus on twisted cycles see Appendix~\ref{sec:Diag_calculus} and \cite{martel2022homological, de2022homological, blanchet2025heisenberg}.

	\vspace{0.5cm}
	
	The decomposition \eqref{eq:homology_decomposition} can be described in terms of twisted cycles. As a set of arcs $\{\gamma_i\}$ above choose a set of symplectic arcs $\{\alpha_i,\beta_i\}_{i=1}^{g}$ in $\Sigma$. We denote by $(\mathbf{a},\mathbf{b})=(a_1,b_1,\dots,a_g,b_g)$ compositions corresponding to these arcs. In particular, the number associated to the arc $\alpha_i$ is $a_i$ and the number associated to the arc $\beta_j$ is $b_j$. Choose a base-point $*^I$ as above and fix a path from $*_n$ to $*^I$. Then for any composition $(\mathbf{a},\mathbf{b})$ the map $(\Gamma_{(\mathbf{a},\mathbf{b})})_*$ is a monomorphism and
	\[
		\bbH_n^{BM}(\Sigma;\calL)=\bigoplus_{(\mathbf{a},\mathbf{b})\models n}(\Gamma_{(\mathbf{a},\mathbf{b})})_* L.
	\]
	In particular, if $L$ admits a basis $\{\mathbf{v}_c\}_{c\in C}$ parametrized by a set $C$, then a basis of $\bbH_n^{BM}(\Sigma;\calL)$ is given by the twisted cycles $\{\mathbf{\Gamma}(\mathbf{a},\mathbf{b})\otimes \mathbf{v}_c\ |\ (\mathbf{a},\mathbf{b})\models n,\ c\in C\}$.

	By considering the standard symplectic arcs in each $\Sigma_g$ twisted cycles in $\bbH(\Sigma_g;\calL)$ can be drawn as in Fig.~\ref{fig:BM_arcs}.
	
	\begin{figure}[h]
		\centering
		\includegraphics[width=0.5\linewidth]{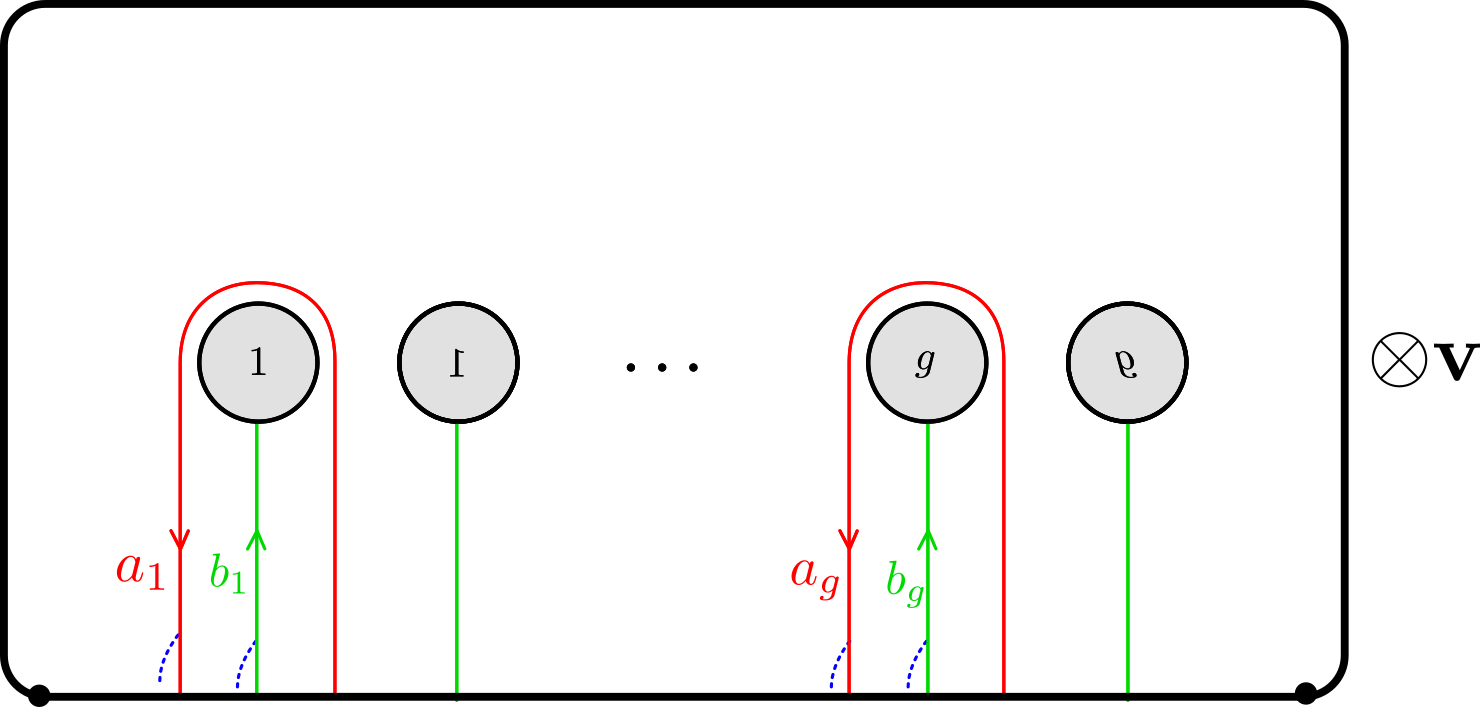}
		\caption{A twisted cycle $\mathbf{\Gamma}(\mathbf{a},\mathbf{b})\otimes \mathbf{v}$ representing a homology class in $\bbH^{BM}_{n}(\Sigma;\calL_n)$, where $\mathbf{a}=(a_1,\dots,a_g)$, $\mathbf{b}=(b_1,\dots,b_g)$ and $(a_1,b_1,\dots,a_g,b_g)$ is a composition of $n$, $\mathbf{v}\in L$. Blue dashed arcs represent a path connecting base configurations on arcs to the base configuration $*_k$: for an arc with label $r$, the corresponding $r$ points move in parallel along the blue line.}
		\label{fig:BM_arcs}
	\end{figure}

	\vspace{0.5cm}
	
	A similar construction of a decomposition of the ordinary homology $H_n(\Conf_n(\Sigma),\Conf^{+}_n(\Sigma);\calL)$ exists. Consider arcs in $\Sigma$ relative to $\partial_+\Sigma$. Similarly to the $\partial_-\Sigma$-relative case, we call a set of disjoint arcs $\{\alpha_i,\beta_i\}_{i=1}^{g}$, relative to $\partial_+\Sigma$, symplectic if their homology classes in $H_1(\Sigma,\partial_+\Sigma)$ form a standard symplectic basis with respect to the intersection pairing. Let $(\mathbf{a},\mathbf{b})=(a_1,b_1,\dots,a_g,b_g)$ be a composition of $n$. Replace each arc $\alpha_j$ (or $\beta_j$) by a collection of $a_j$ (resp $b_j$) parallel non-intersection copies of it. Then a proper map
	\[
		\Gamma^{\vee}_{(\mathbf{a},\mathbf{b})}: (I^{n};\partial I^n)\to (\Conf_n(\Sigma),\Conf^+_n(\Sigma))
	\]
	is defined. Since $I^n$ is contractible any local system on it is trivial, in particular there is an isomorphism:
	\[
		H_n(I^n,\partial I^n;(\Gamma^{\vee}_{(\mathbf{a},\mathbf{b})})^*\calL)\xrightarrow{\simeq}(\Gamma^{\vee}_{(\mathbf{a},\mathbf{b})})^*L.
	\]
	By choosing a base point on $I^n$ and connecting it with $*_n$ by a path one obtains a homomorphism:
	\[
		\left(\Gamma^{\vee}_{(\mathbf{a},\mathbf{b})}\right)_*:L\to H_n(\Conf_n(\Sigma),\Conf_n^+(\Sigma);\calL).
	\]
	It is in fact a monomorphism and there is an induced decomposition:
	\[
		H_n(\Conf_n(\Sigma),\Conf_n^+(\Sigma);\calL)=\bigoplus_{(\mathbf{a},\mathbf{b})\models n}(\Gamma^{\vee}_{(\mathbf{a},\mathbf{b})})_*L.
	\]

	Similarly to the Borel--Moore version we denote the image of $\mathbf{v}\in L$ under $(\Gamma^{\vee}_{(\mathbf{a},\mathbf{b})})_*$ by
	\[
		\mathbf{\Gamma}^{\vee}(\mathbf{a},\mathbf{b})\otimes \mathbf{v}:=(\Gamma^{\vee}_{(\mathbf{a},\mathbf{b})})_*\mathbf{v}.
	\]
	 If $L$ admits a basis $\{\mathbf{v}_c\}_{c\in C}$ then a basis of $H_n(\Conf_n(\Sigma),\Conf_n^+(\Sigma);\calL)$ is given by twisted cycles $\{\mathbf{\Gamma}^{\vee}(\mathbf{a},\mathbf{b})\otimes \mathbf{v}_c\ |\ (\mathbf{a},\mathbf{b})\models n,\ c\in C\}$. For a standard choice of twisted cycles in $\Sigma_g$ see Fig.~\ref{fig:BM_arcs_dual}.
	

	\begin{figure}[h]
		\centering
		\includegraphics[width=1\linewidth]{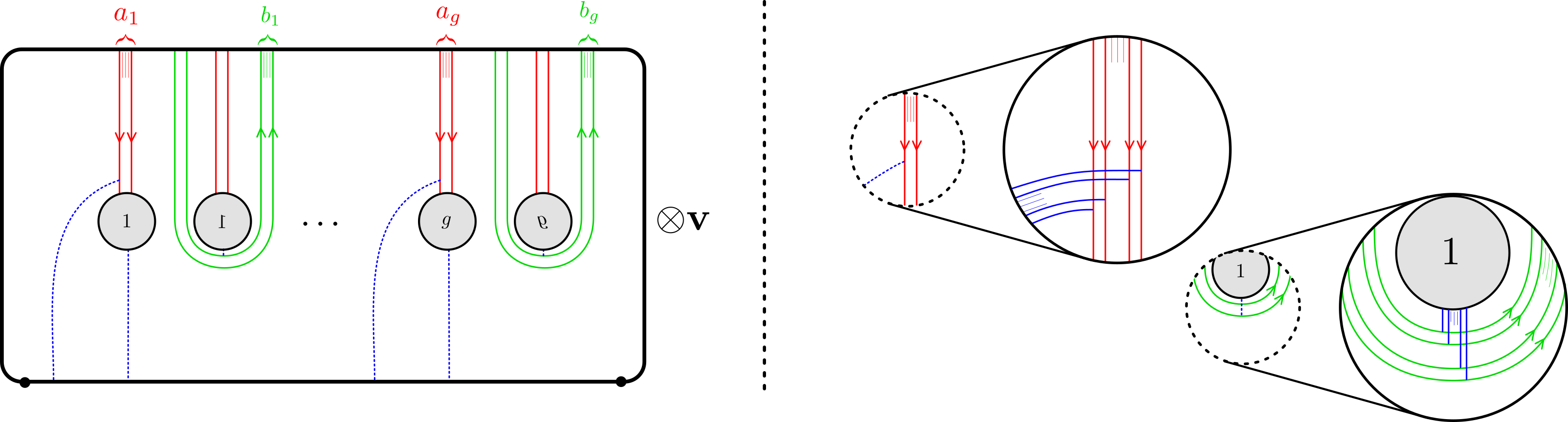}
		\caption{On the left is a twisted cycle $\mathbf{\Gamma}^{\vee}(\mathbf{a},\mathbf{b})\otimes \mathbf{v}$ representing an element of $H_n(\Conf_n(\Sigma),\Conf^+_n(\Sigma);\calL)$. Label $a_i$ (or $b_i$) means $a_i$ (resp. $b_i$) parallel copies of the same arc. Blue dashed lines represent a path between base configuration moving all points in parallel. On the right is the particular way this path connects to the base configuration on $I^n$.}
		\label{fig:BM_arcs_dual}
	\end{figure}

	\vspace{0.5cm}

	Let $\calT_M$ be the trajectory space of an elementary cobordism either of index 1 or 2 (topologically they are the same). The same compression trick as in \cite[Theorem~11]{blanchet2025heisenberg} gives the Borel--Moore arc decomposition for this surface: its genus $g$ and three boundary components give $2g+2$ arcs, and each composition of $n$ into $2g+2$ parts contributes one copy of the coefficient module. The argument is unchanged by the two additional boundary components. The corresponding twisted cycles are described in Fig.~\ref{fig:trajectory_space_basis}.
	
		\begin{figure}[h]
		\centering
		\includegraphics[width=1\linewidth]{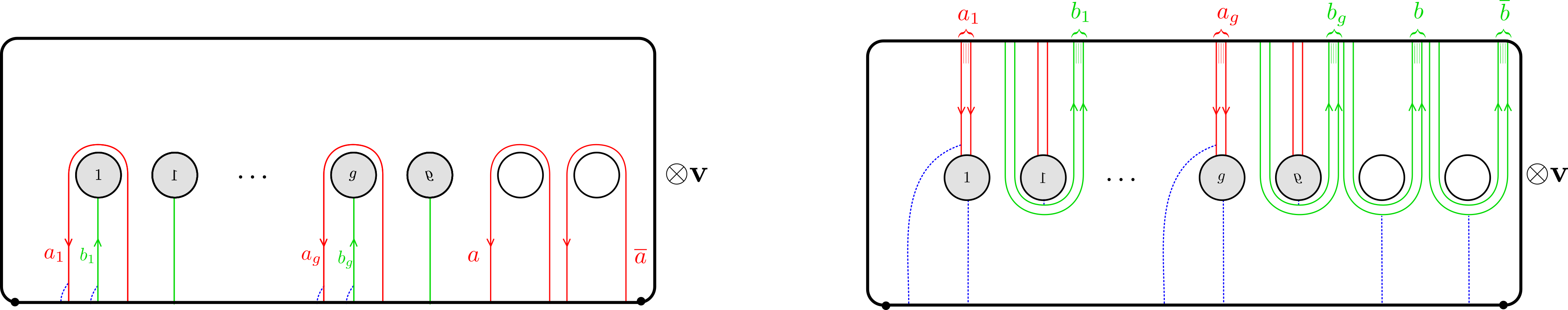}
		\caption{Let $\calL$ be a $p$-Heisenberg local system on $\Conf_n(\calT_M)$ with monodromy representation $L$ at $*_n$. On the left is a twisted cycle in $H^{BM}_n(\Conf_n(\calT_M),\Conf_n^-(\calT_M);\calL)$ for a $\mathbf{v}\in L$. If $L$ admits a basis $\{\mathbf{v}_c\}_{c\in C}$, then a basis of this Borel--Moore homology is given by $\{\mathbf{\Gamma}(\mathbf{a},\mathbf{b},a,\overline{a})\otimes \mathbf{v}_c\mid (\mathbf{a},\mathbf{b},a,\overline{a})\models n,\ c\in C\}$. On the right are compact relative representatives in $H_n(\Conf_n(\calT_M),\Conf_n^+(\calT_M);\calL)$, with each label indicating parallel copies of an arc. We will not need a basis assertion for the ordinary homology of $\calT_M$.}
		\label{fig:trajectory_space_basis}
	\end{figure}

	\subsection{Pairing}\label{sec:pairing}
	Let $\Sigma\in 3\Cob$.
	We use the geometric intersection pairing in each particle degree $n$ multiplied by $(-1)^{n(n-1)/2}$. With product orientations on the arc representatives, this cancels the permutation from the two blocks of $n$ tangent vectors to the $n$ pairs of pointwise tangent vectors. Thus a product of positive arc intersections has sign $+1$. 
	

	\begin{proposition}[Compare {\cite[Section~6.5]{blanchet2025heisenberg}}]\label{prop:intersection_pairing}
		Let $R$ be a commutative ring with involution. Suppose $\calL\in \Locpair_R(\Conf_n(\Sigma))$ is a $p$-Heisenberg local system.
		Then it induces a perfect sesquilinear intersection pairing:
		\[
			\langle-,-\rangle:H^{BM}_n(\Conf_n(\Sigma),\Conf_n^-(\Sigma);\calL)\otimes_R H_n(\Conf_n(\Sigma),\Conf_n^+(\Sigma);\overline{\calL})\to R.
		\]
	\end{proposition}
	\begin{proof}
		The intersection construction of \cite[Section~6.5]{blanchet2025heisenberg} applies to the given coefficient pairing. To check perfectness without a freeness assumption, use the arc decompositions of Section~\ref{subs:Twisted_cycles_and_bases_in_twisted_homology}. Let $L$ be the coefficient fiber at $*_n$. In degree $n$, the two modules are finite direct sums of copies of $L$ and $\overline L$, respectively, indexed by compositions $(\mathbf a,\mathbf b)$ of $n$. With the corresponding arc orientations and coefficient transports, intersections give
		\begin{equation}\label{eq:pairing_blocks}
		\langle\mathbf\Gamma(\mathbf a,\mathbf b)\otimes v,
		\mathbf\Gamma^{\vee}(\mathbf a',\mathbf b')\otimes\overline w\rangle
		=\delta_{\mathbf a,\mathbf a'}\delta_{\mathbf b,\mathbf b'}(v,w).
		\end{equation}
		Indeed, distinct arc indices give disjoint representatives, while equal indices give one intersection with coefficient pairing $(v,w)$. Thus the pairing is the orthogonal sum of finitely many copies of the perfect coefficient pairing. Both induced dual maps are isomorphisms. This extends the free-coefficient argument of \cite[Proposition~40]{blanchet2025heisenberg} to the present coefficient modules.
	\end{proof}

	In particular, consider the standard surface $\Sigma_g$. Let $\calL$ be as in the proposition above and suppose that its monodromy representation $L$ admits a finite basis $\{\mathbf{v}_c\}_{c\in C}$. Since the pairing on $L$ is perfect, there is a dual basis $\{\mathbf{v}_c^{\vee}\}_{c\in C}$ determined by the condition:
	\[
		(\mathbf{v}_c,\mathbf{v}^{\vee}_{c'})=\delta_{c,c'}.
	\]
	Then the standard basis of twisted cycles $\mathbf{\Gamma}^{\vee}(\mathbf{a},\mathbf{b})\otimes \overline{\mathbf{v}_c^{\vee}}$ is dual to standard basis $\mathbf{\Gamma}(\mathbf{a},\mathbf{b})\otimes \mathbf{v}_c$:
	\begin{equation}\label{eq:pairing_formula}
		\langle\mathbf{\Gamma}(\mathbf{a},\mathbf{b})\otimes \mathbf{v}_{c}, \mathbf{\Gamma}^{\vee}(\mathbf{a}',\mathbf{b'})\otimes\overline{\mathbf{v}_{c'}^{\vee}}\rangle=\delta_{\mathbf{a},\mathbf{a}'}\delta_{\mathbf{b},\mathbf{b}'}\delta_{c,c'}
	\end{equation}
	since the only intersection point appears when $a_i=a_i'$, $b_i=b_i'$ for all $i$ (compare configurations on arcs in Fig.~\ref{fig:BM_arcs} and \ref{fig:BM_arcs_dual}). Note that there is no such pairing on index 1 or 2 trajectory spaces since they have two additional boundary components.
	
	\vspace{0.5cm}

	Consider an orientation preserving diffeomorphism $d_{\partial}:\Sigma\to \Sigma$ given by the counter-clockwise half twist along the boundary, \ie $d_{\partial}$ is the identity outside a small neighborhood of $\partial\Sigma$ and sends $\partial_+\Sigma$ to $\partial_-\Sigma$ (see Fig.~\ref{fig:boundary_half_twist}).

	\begin{figure}[h]
		\centering
		\includegraphics[width=0.8\linewidth]{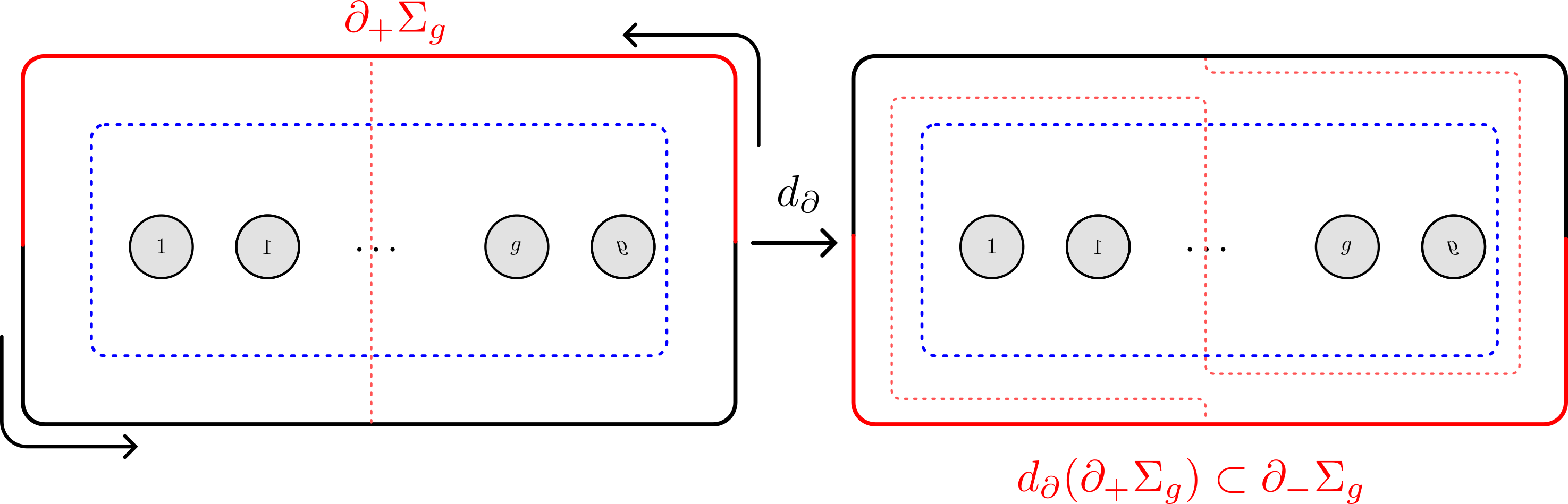}
		\caption{The boundary half twist $d_{\partial}$. It is the identity inside the blue rectangle, and rotates the boundary with its neighborhood counter-clockwise by $\pi$ exchanging $\partial_-\Sigma_g$ and $\partial_+\Sigma_g$.}
		\label{fig:boundary_half_twist}
	\end{figure}
	
	Since $d_{\partial}$ is isotopic to $\Id$, it induces a map $d_{\partial}=d_{\partial}^{\times n}:\Conf_n(\Sigma)\to\Conf_n(\Sigma)$ homotopic to $\Id$ for each $n$. Therefore one can identify $\calL\simeq d_{\partial}^*\calL$. Furthermore $d_{\partial}(\Conf^+_n(\Sigma))\subset \Conf^-_n(\Sigma)$, hence there is an induced isomorphism:
	\[
		H_n(\Conf_n(\Sigma),\Conf_n^+;\calL)\xrightarrow{\simeq}H_n(\Conf_n(\Sigma),\Conf_n^-;\calL)=\bbH_n(\Sigma;\calL).
	\]

	Then $\langle-,-\rangle$ can be interpreted as a pairing:
	\[
		\bbH^{BM}_n(\Sigma;\calL)\otimes_R \bbH_n(\Sigma;\overline{\calL})\to R.
	\]
	
	More generally, summing over all degrees we obtain the following corollary.

	\begin{corollary}
		Let $R$ be a unital commutative ring with involution and let $\calL\in\Locpair_R(\Conf_*(\Sigma))$ be a $p$-Heisenberg local system. There is an induced intersection pairing
		\[
			\langle-,-\rangle:\bbH^{BM}(\Sigma;\calL)\otimes_R\bbH(\Sigma;\overline{\calL})\to R,
		\]
		which pairs distinct particle degrees trivially and is perfect in each degree. For a graded module $A=\bigoplus_{n\geq0}A_n$, write $A^{\circ}:=\bigoplus_{n\geq0}\Hom_R(A_n,R)$ for its restricted dual. The pairing induces isomorphisms
		\[
		\bbH^{BM}(\Sigma;\calL)\cong\bbH(\Sigma;\overline{\calL})^{\circ},\qquad
		\bbH(\Sigma;\overline{\calL})\cong\bbH^{BM}(\Sigma;\calL)^{\circ}.
		\]
	\end{corollary}
	These are restricted-dual statements: the ordinary dual of an infinite direct sum is a product, so the full direct-sum pairing is not asserted to be perfect in the ordinary sense.
	By some abuse of notation we denote by $\mathbf{\Gamma}^{\vee}(\mathbf{a},\mathbf{b})\otimes \mathbf{v}\in \bbH(\Sigma;\calL)$ the twisted cycle obtained by the boundary half twist from $\mathbf{\Gamma}^{\vee}(\mathbf{a},\mathbf{b})\otimes \mathbf{v}\in H(\Conf_*(\Sigma),\Conf_*^+(\Sigma);\calL)$.
	

	\subsection{Subspaces of small cycles}\label{sec:small_cycles}
	
	Let $\Sigma$ be either an object of $3\Cob$ or a trajectory space, and let $\calL$ be a $p$-Heisenberg local system on $\Conf_*(\Sigma)$. Recall that $\iota$ is the natural homomorphism \eqref{eq:iota} from the ordinary homology to the Borel--Moore homology. We define the subspace of small cycles by:
	\begin{equation}
		\mathring{\bbH}(\Sigma;\calL):=\mathrm{Im}[\iota:\bbH(\Sigma;\calL)\to \bbH^{BM}(\Sigma;\calL)].
	\end{equation}
	
	Assume that all $1-q^{2k}$, for $1\leq k<p$, act invertibly on the chosen coefficient modules $L_n$ of $\calL$ at $*_n$, and $q^{2p}$ acts trivially. 
	\begin{remark}\label{rem:quantum_invertibility}
		The fusion factors
		\[
		\mathfrak f_k:=q^{k(k-1)/2}[k]_q!
		=\prod_{j=1}^k(1+q^2+\cdots+q^{2(j-1)}),\qquad \mathfrak f_0=1,
		\]
		involve only even powers of $q$ and therefore act on the coefficient modules. Each factor in this product is $(1-q^{2j})(1-q^2)^{-1}$. Hence $\mathfrak f_k$ acts invertibly for $k<p$ and by zero for $k\geq p$. For odd $p$, the root condition also implies that $\sigma^p$ acts by $-1$.
	\end{remark} 
	
	First suppose $\Sigma\in 3\Cob$. Choose a set of symplectic arcs and consider the corresponding twisted cycles in the ordinary and the Borel--Moore homology. Then for a twisted cycle in the ordinary homology an arc with label $k$ is represented by $k$ parallel arcs. Hence their images in the Borel--Moore homology can be merged by the fusion rule \ref{fig:Fusion_rule} in diagrammatic calculus of twisted cycles (see an example in Fig.~\ref{fig:small_cycles_merging}).
	
	\begin{figure}[h]
		\centering
		\includegraphics[width=0.7\linewidth]{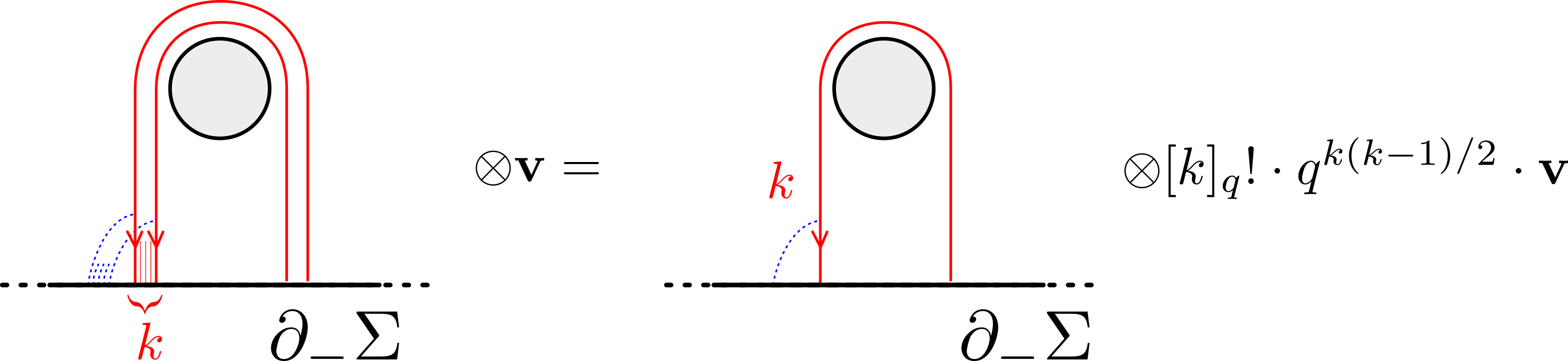}
		\caption{On the left is a collection of $k$ parallel arcs representing the image of a twisted cycle from $\bbH(\Sigma;\calL)$ under $\iota$, $\mathbf{v}\in L_n$. After merging all parallel arcs one obtains the twisted cycle on the right.}
		\label{fig:small_cycles_merging}
	\end{figure}
	
	Therefore one obtains a twisted cycle
	\[
		\mathbf{\Gamma}(\mathbf{a},\mathbf{b})\otimes\left(\prod_{i=1}^g\mathfrak f_{a_i}\mathfrak f_{b_i}\right)\mathbf{v}.
	\]
	Since $\mathfrak f_k$ is invertible for $k<p$ and vanishes for $k\geq p$, the small cycles form the direct sum of the Borel--Moore arc summands $\mathbf\Gamma(\mathbf a,\mathbf b)\otimes L_n$ with all labels below $p$. In the corresponding ordinary parallel-arc decomposition, the kernel consists of the summands with at least one label $\geq p$. The same fusion argument applies with the two distinguished boundary intervals interchanged. 
	For a trajectory space $\calT_M$, the same bound on every arc label describes the intrinsic small cycles:
	\[
	\mathring{\bbH}_n(\calT_M;\calL)
	=\bigoplus_{\substack{(\mathbf a,\mathbf b,a,\overline a)\models n\\0\leq a_i,b_i,a,\overline a<p}}
	\mathbf\Gamma(\mathbf a,\mathbf b,a,\overline a)\otimes L_n.
	\]
	Here each summand denotes a copy of $L_n$. 
	
	\vspace{0.5cm}
	
	Subspaces of small cycles on objects of $3\Cob$ inherit the intersection pairing.
	
	\begin{proposition}\label{prop:pairing_small_cycles}
		Let $\Sigma\in 3\Cob$, let $R$ be a unital commutative ring with involution and let $\calL\in \Locpair_R(\Conf_*(\Sigma))$ be a $p$-Heisenberg local system.
		Suppose that the action of all $1-q^{2k}$, $1\leq k<p$, on the local system is invertible, and $q^{2p}$ acts trivially, then there is a well-defined perfect sesquilinear intersection pairing:
		\[
			\langle-,-\rangle:\mathring{\bbH}(\Sigma;\calL)\otimes \mathring\bbH(\Sigma;\overline{\calL})\to R,
		\]
		induced by the pairing $(-,-)$ on $\calL$.
	\end{proposition}
	\begin{proof}
		Fix a particle degree $n$, and put
		\[
		\begin{aligned}
		H_-&=H_n(\Conf_n(\Sigma),\Conf_n^{-}(\Sigma);\calL_n),& H^{BM}_-&=H_n^{BM}(\Conf_n(\Sigma),\Conf_n^{-}(\Sigma);\calL_n),\\
		H_+&=H_n(\Conf_n(\Sigma),\Conf_n^{+}(\Sigma);\overline{\calL_n}),& H^{BM}_+&=H_n^{BM}(\Conf_n(\Sigma),\Conf_n^{+}(\Sigma);\overline{\calL_n}),
		\end{aligned}
		\]
		and denote the comparison maps by $\iota_\pm:H_\pm\to H^{BM}_\pm$. Proposition~\ref{prop:intersection_pairing}, also with the boundary intervals interchanged, gives perfect mixed pairings $H^{BM}_-\otimes H_+\to R$ and $H_-\otimes H^{BM}_+\to R$. Naturality of intersection, keeping the same order of the two factors and the same degree normalization, gives
		\[
		\langle\iota_-x,y\rangle=\langle x,\iota_+y\rangle.
		\]
		It follows that $\langle\iota_-x,\iota_+y\rangle:=\langle\iota_-x,y\rangle$ is well defined on the two images.

		Each image is a direct summand of its Borel--Moore module by the fusion computation. A functional on $\operatorname{Im}\iota_+$ therefore extends to $H^{BM}_+$ and is represented by some $x\in H_-$ under the second mixed pairing, hence by $\iota_-x$ under the induced pairing. If $\iota_-x$ pairs trivially with $\operatorname{Im}\iota_+$, perfectness of the first mixed pairing gives $\iota_-x=0$. Interchanging the roles of the two mixed pairings proves that the other dual map is also an isomorphism. Thus the induced pairing is perfect.

		Finally, the boundary half twist commutes with the comparison maps and identifies $\operatorname{Im}\iota_+$ with the minus-relative small cycles for $\overline{\calL_n}$. Only degrees $0\leq n\leq2g(p-1)$ can survive. Their finite direct sum therefore has the asserted perfect pairing with respect to the ordinary dual.
	\end{proof}

	Subspaces of small cycles naturally inherit functoriality of Borel--Moore homology in both arguments in the following sense. For either type of surface $\Sigma$, the assignment
	\[
		\mathring{\bbH}(\Sigma, -): \ \Locpair_R(\Conf_*(\Sigma))\to \mathrm{Mod}_R
	\]
	is a covariant functor, since both versions of homology are covariant in coefficient morphisms and the map $\iota$ is natural. Note that the same construction is defined for underlying local systems without pairings. Given a diffeomorphism $d:\Sigma \to \Sigma'$ in $\mathrm{Surf}$, it induces a proper map $d:\Conf_*(\Sigma)\to \Conf_*(\Sigma')$. Then for any local system $\calL\in \Locpair_R(\Conf_*(\Sigma'))$ there is a natural homomorphism of $R$-modules
	\[
		d_*:\mathring{\bbH}(\Sigma; d^*\calL)\to \mathring{\bbH}(\Sigma';\calL)
	\]
	induced by the corresponding homomorphisms \eqref{eq:functoriality_H} and \eqref{eq:functoriality_H_BM} on both versions of twisted homology. The same naturality argument applies to proper embeddings of configuration pairs, in particular to the inclusion maps from trajectory spaces. Cross products of small cycles are small, since the ordinary and Borel--Moore cross products commute with $\iota$.


	\section{Construction of Homological TQFTs}\label{sec:general_construction}
	In this section we provide a general construction of homological TQFTs. As explained in Introduction, the basic idea in the case of an index $1$ or $2$ cobordism is to insert the twisted fundamental class of $\Conf_{p-1}(\nu)$, where $\nu$ is the belt or attaching sphere (see Fig.~\ref{fig:Donaldson_action}). In Subsection~\ref{sec:topological_model} we first describe the topological data naturally associated with each elementary cobordism and with boundary connected sum. In Subsection~\ref{subs:Twisted_homology_of_Conf_S_1} we discuss the twisted homology of $\Conf_{p-1}(S^1)$. In Subsection~\ref{sec:nice_ls} we introduce the notion of Morse compatible local systems, an additional datum sufficient for constructing a homological TQFT. Subsection~\ref{sec:monodromy_reps} reformulates the action of elementary cobordisms on local systems in terms of their monodromy representations which will be needed for some explicit computations. In Subsection~\ref{subs:Construction_of_homological_TQFTs} we define the action of elementary cobordisms on the state spaces of a Morse compatible collection of local systems. In Subsection~\ref{subs:Functorialty} we prove that this action indeed gives rise to a TQFT.

	
	\vspace{0.5cm}
	Fix an integer $p$ which is either $2$ or odd $\geq 3$. In this section all local systems are objects of $\Locpair_R(X)$ and coefficient maps are morphisms in that category, unless explicitly required to be isomorphisms in $\Loc_R(X)$. 

	\subsection{Topological data.}\label{sec:topological_model}
	The construction of homological TQFTs is motivated by a natural topological structure on elementary cobordisms given by multi-trajectory spaces. In this subsection we describe it for each basic piece of the TQFT: elementary cobordisms and the monoidal product in $3\Cob$.

	\vspace{0.5cm}
	
	\textit{Monoidal product.} For a pair of surfaces $\Sigma_1,\Sigma_2\in 3\Cob$ the structure of the boundary connected sum induces a natural (proper) map on configuration spaces. We consider both $\Sigma_1$ and $\Sigma_2$ as subsurfaces of $\Sigma:=\Sigma_1\natural\Sigma_2$ with $\partial_-\Sigma_1,\partial_-\Sigma_2\subset\partial_-\Sigma$, such that $\Sigma_1\cap\Sigma_2=\emptyset$ (see Fig.~\ref{fig:gluing}). Therefore there is an induced embedding 
	\begin{equation}\label{eq:mu_def}
		\mu:\Conf_*(\Sigma_1)\times\Conf_*(\Sigma_2)\hookrightarrow \Conf_{*}(\Sigma),
	\end{equation}
	\[
		\mathbf{x}_1\times \mathbf{x}_2\mapsto \mathbf{x}_1\cup \mathbf{x}_2.
	\]
	
	If $\tau$ interchanges the two configuration-space factors, the underlying union maps satisfy $\mu_{\overline\Sigma_2,\overline\Sigma_1}\circ\tau=\mu_{\Sigma_1,\Sigma_2}$. Reflected base configurations are compared with the chosen standard ones by paths in the contractible configuration space of the distinguished boundary interval. 

	\textit{Mapping cylinder.} For the mapping cylinder $M_d$ given by a mapping class $d:\Sigma\to \Sigma'$ we have
	
	\[
		\begin{tikzcd}
			&\Conf_*(\calT_{M_d})\arrow[dl,"{\Id}"']\arrow[dr,"{d}"]&\\
			\Conf_*(\Sigma)&&\Conf_*(\Sigma'),
		\end{tikzcd}
	\]
	where $d=\Id\sqcup d\sqcup d^{\times 2}\sqcup d^{\times 3}\sqcup \dots$ is the map on configuration spaces induced by $d$.
	
	\textit{Index 1 cobordism.} For an elementary index 1 cobordism $M:\Sigma_-\to \Sigma_+$ with belt sphere $\nu:S^1\to \Sigma_+$ we consider
	\[
	\begin{tikzcd}
		&\Conf_*(\calT_M)\times\Conf_{p-1}(S^1)\arrow[dr,"\varphi"]\arrow[dl,"i_-\circ\pi"']&\\
		\Conf_*(\Sigma_-)&&\Conf_{*+p-1}(\Sigma_+),
	\end{tikzcd}
	\]
	where $i_-=(i_-)^{\times *}:\Conf_*(\calT_M)\to \Conf_*(\Sigma_-)$ is induced by the evaluation at an endpoint $i_-:\calT_M\to \Sigma_-$, $\pi$ is the projection on $\Conf_*(\calT_M)$ and $\varphi$ inserts $p-1$ points on the belt sphere:
	\begin{equation}\label{eq:varphi_def}
		\varphi: \Conf_*(\calT_M)\times \Conf_{p-1}(S^1)\to \Conf_{*+p-1}(\Sigma_+),
	\end{equation}
	\[
		\mathbf{x}\times \{s_1,\dots,s_{p-1}\}\mapsto i_+(\mathbf{x})\cup \{\nu(s_1),\dots,\nu(s_{p-1})\},\ \ \mathbf{x}\in \Conf_n(\calT_M)
	\]
	with $i_+=(i_+)^{\times *}$ induced by $i_+$.
	
	\textit{Index 2 cobordism.} 
	As it was discussed in Section~\ref{sec:Multi_trajectory}, the case of an index 2 cobordism is dual to index 1. Therefore we consider the dual span on the level of configuration spaces. Let $M':\Sigma_-'\to \Sigma_+'$ be an index 2 cobordism with attaching sphere $\nu:S^1\to \Sigma_-'$. Then we have:
	\[
		\begin{tikzcd}
			&\Conf_*(\calT_{M'})\times \Conf_{p-1}(S^1)\arrow[dl,"{\varphi}"']\arrow[rd,"i_+\circ \pi"]&\\
			\Conf_{*+p-1}(\Sigma_-')&& \Conf_*(\Sigma_+'),
		\end{tikzcd}
	\]
	where $i_+=(i_+)^{\times *}:\Conf_*(\calT_{M'})\to \Conf_*(\Sigma_+')$ is the evaluation at an endpoint, $\pi$ is the projection on $\Conf_*(\calT_{M'})$ and $\varphi$ is the same as in \eqref{eq:varphi_def}. 

	\subsection{Twisted homology of $\Conf_{p-1}(S^1)$}\label{subs:Twisted_homology_of_Conf_S_1}

	 In this subsection we give a description of the (top-degree) twisted Borel--Moore homology of $\Conf_{p-1}(S^1)$. 

	\vspace{0.5cm}
	
	Choose a base configuration $s_{p-1}$ on $S^1$. We have
	\[
		\pi_1(\Conf_{p-1}(S^1),s_{p-1})\simeq\Z.
	\]
	We take the generator $t$ that slides each point counter-clockwise to the next point, going once around the circle when $p=2$ (see Fig.~\ref{fig:t_loop}). 
	
	\begin{figure}[h]
		\centering
		\includegraphics[width=0.20\linewidth]{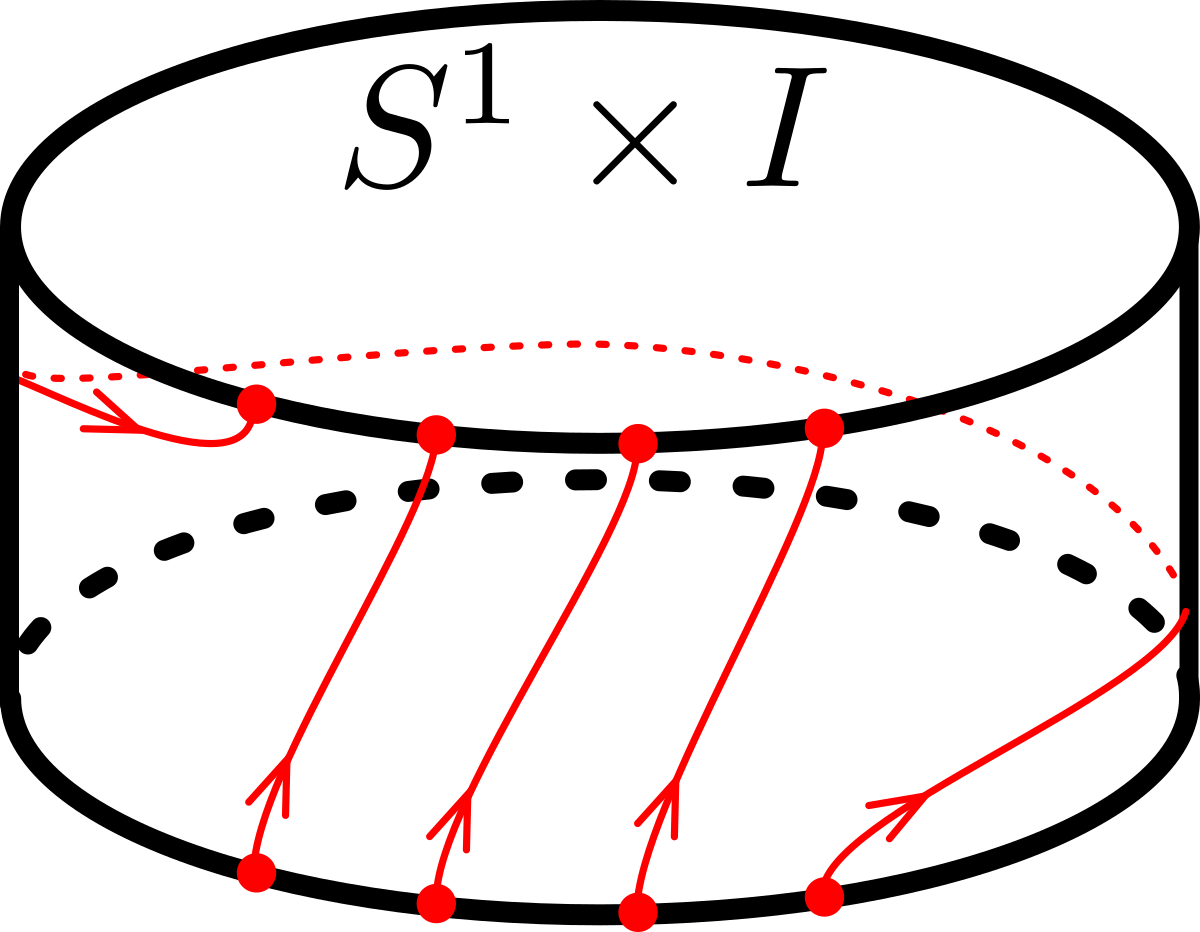}
		\caption{The loop $t:[0,1]\to \Conf_4(S^1)$ represented by trajectories of configuration points in $S^1\times I$, \ie if $t_i:[0,1]\to S^1$ is the trajectory of $i$-th point in $S^1$ the corresponding trajectory in $S^1\times I$ is given by the map $[0,1]\to S^1\times I$, $x\mapsto (t_i(x),x)$.}
		\label{fig:t_loop}
	\end{figure}
	
	Let $p'=2$ if $p=2$ and $p'=2p$ if $p$ is odd. Let $\widetilde{\Conf}_{p-1}(S^1)$ be the covering space associated with $p'\Z\subset \Z=\pi_1(\Conf_{p-1}(S^1),s_{p-1})$, \ie the $p'$-fold covering space with deck transformation group $\Z_{p'}$. Recall that for odd $p$ the manifold $\Conf_{p-1}(S^1)$ is non-orientable and admits a unique orientable 2-fold cover. There is only one double cover of $\Conf_{p-1}(S^1)$ corresponding to $2\Z\subset \Z$, and $p'\Z\subset 2\Z$ since $p'$ is even. Hence the covering space $\widetilde{\Conf}_{p-1}(S^1)$ factors through the orientable double covering and is orientable itself. We consider a fundamental Borel--Moore class $\left[\widetilde{\Conf}_{p-1}(S^1)\right]$ that can be viewed as a twisted fundamental class as follows. We consider the $\pi_1(\Conf_{p-1}(S^1),s_{p-1})$-module and the corresponding local system on $\Conf_{p-1}(S^1)$: 
	\begin{equation}\label{eq:calA_def}
		\boxed{A:=R[t]/(t^{p'}-1)=R[\Z_{p'}], \quad \calA:=\text{local system associated with }A.}
	\end{equation}
	Then
	\[
		[\widetilde{\Conf}_{p-1}(S^1)]\in H^{BM}_{p-1}(\Conf_{p-1}(S^1);\calA).
	\]
	Since $\Conf_{p-1}(S^1)$ can be obtained from the simplex (with some faces removed) $\Conf_{p-1}(I)$ by gluing two faces (which are not removed), a cycle representing this class has the form:
	\[
		\Delta \otimes \sum\limits_{k=0}^{2p-1}(-t)^k
	\]
	for odd $p\geq 3$, where $\Delta$ is the Borel--Moore chain spanning $\Conf_{p-1}(I)$. The sign of this fundamental class is fixed by choosing the product orientation on $\Conf_{p-1}(I)$ induced by the orientation of $S^1$, and then alternating it across the sheets of the covering space. For convenience we denote by $\eta^A$ the homomorphism:
	\begin{equation}\label{eq:eta_A_def}
		\eta^A:R\to A, \quad 1\mapsto\sum\limits_{k=0}^{2p-1}(-t)^k
	\end{equation}
	of $R$-modules. For $p=2$ the cycle has the form
	\[
	\Delta \otimes (1+t)
	\]
	since $\Conf_1(S^1)=S^1$ is oriented, and we denote
	\[
		\eta^A:R\to A,\quad 1\mapsto 1+t.
	\]
	
	We equip $A$ with a perfect pairing:
	\[
	(-,-):A\otimes \overline{A}\to R, \quad (t^a,t^{a'})\mapsto \delta_{a,a'}. 
	\]
	It is easy to see that this pairing is equivariant under the action of $\pi_1(\Conf_{p-1}(S^1),s_{p-1})$, therefore there is an induced perfect pairing on local systems:
	\[
		(-,-):\calA\otimes \overline{\calA}\to \underline{R}.
	\]
	Hence 
	\[
		\calA\in \Locpair_R(\Conf_{p-1}(S^1)).
	\]

	\subsection{Morse compatible local systems on surfaces}\label{sec:nice_ls}
	In this subsection we define \emph{Morse compatible} local systems. We first define an action of elementary cobordisms on local systems and then formulate a set of conditions on such an action that is sufficient for constructing a homological TQFT. 
	
	\vspace{0.5cm}

	Fix a unital commutative ring $R$ with involution. Let $\calL$ be a collection of local systems on configuration spaces of all objects in $3\Cob$:
	\[
		\calL=\left\{\calL(\Sigma)\in \Locpair_R(\Conf_*(\Sigma))\big|\, \Sigma\in 3\Cob\right\}.
	\]
	
	\vspace{0.5cm}
	
	We say that $\calL$ is \textit{monoidal} if for any pair of surfaces $\Sigma_1,\Sigma_2\in 3\Cob$ the product pairing on $\calL(\Sigma_1)\boxtimes\calL(\Sigma_2)$ is perfect and an isomorphism in $\Loc_R(\Conf_*(\Sigma_1)\times\Conf_*(\Sigma_2))$
	\begin{equation}\label{eq:nice_ls_monoidality}
		\mu_{\calL}:\calL(\Sigma_1)\boxtimes \calL(\Sigma_2)\xrightarrow{\simeq}\mu^*\calL(\Sigma_1\natural\Sigma_2)
	\end{equation}
	is chosen, where $\mu$ is the map defined in \eqref{eq:mu_def}. These isomorphisms are required to be associative under the canonical identifications of iterated boundary connected sums. For each disk $D$ we also require a pairing-preserving unit isomorphism
	\begin{equation}\label{eq:disk_unit}
		u_D:R\xrightarrow{\simeq}L_0(D).
	\end{equation}
	The degree $(0,n)$ and $(n,0)$ restrictions of $\mu_{\calL}$, composed with $u_D$, must agree with the canonical left and right unit identifications $D\natural\Sigma\simeq\Sigma$ and $\Sigma\natural D\simeq\Sigma$. 
	\vspace{0.5cm}
	We say that $\calL$ is equipped with an action of elementary cobordisms if:
	\begin{enumerate}
		\item For each mapping class $d:\Sigma\to\Sigma'$ with mapping cylinder $M_d$, the following isomorphism in the category $\Loc_R(\Conf_*(\calT_{M_d}))$:
		\begin{equation}\label{eq:phi_d_def}
			\phi_d=\phi_{M_d}:\calL(\Sigma)\xrightarrow{\simeq} d^*\calL(\Sigma')
		\end{equation}
		is chosen.
		
		\item Let $M:\Sigma_-\to \Sigma_+$ be an index 1 elementary cobordism with the belt sphere $\nu:S^1\to \Sigma_+$ and trajectory space $\calT_M$. Then a morphism in $\Locpair_R(\Conf_*(\calT_M)\times\Conf_{p-1}(S^1))$:
		\begin{equation}\label{eq:psi_M_def}
			\psi_M:i_-^*\calL(\Sigma_-)\boxtimes\calA\to \varphi^*\calL(\Sigma_+)
		\end{equation}
		is chosen, where $\calA$ is the local system defined in \eqref{eq:calA_def} and $\varphi$ is defined in \eqref{eq:varphi_def}. The source has a perfect product pairing since $\calA$ has finite free fibers with the standard pairing.

		These maps must be compatible with change of orientation on the belt sphere. For a reflection $r:S^1\to S^1$, choose the base configuration invariant under $r$ and use the identification
		\[
			\kappa_r:\calA\xrightarrow{\simeq}r^*\calA,\qquad
			t^k\longmapsto (-1)^{(p-1)(p-2)/2}t^{-k}.
		\]
		The sign records the reversal of the order of the $p-1$ points. This is a pairing-preserving identification, intertwines $t$ with $t^{-1}$, and squares to the identity. Writing $\psi_{M,\nu}$ only to display the parametrization in this requirement, we impose
		\[
			\psi_{M,\nu\circ r}
			=(\Id\times r)^*\psi_{M,\nu}\circ(\Id\boxtimes\kappa_r).
		\]
		\item For each $\Sigma\in 3\Cob$ an identification in $\Loc_R(\Conf_*(\Sigma))$, using the underlying identification of the two configuration spaces,
		\[
			\chi_\Sigma:\calL(\overline{\Sigma})\xrightarrow{\simeq}\overline{\calL(\Sigma)}
		\]
		is chosen. Recall that for $\Sigma\in 3\Cob$, $\overline{\Sigma}$ is obtained by switching the orientation. Recall that for a local system $\calL$, by $\overline{\calL}$ we denote the conjugated local system with respect to the involution on $R$. The identification is required to carry the pairing on $\calL(\overline{\Sigma})$ to the conjugate pairing on $\overline{\calL(\Sigma)}$. 
	\end{enumerate}
	
	\begin{remark}
		For a monoidal collection of local systems with the conjugation identifications above, every coefficient fiber is finitely generated projective. Indeed, for a fiber $L$ of $\calL(\Sigma)$ in particle degree $n$, apply \eqref{eq:nice_ls_monoidality} to $\Sigma$ and $\overline\Sigma$ in degrees $(n,n)$. It identifies $L\otimes_R\overline L$, with its product pairing, with a fiber of $\calL(\Sigma\natural\overline\Sigma)$ whose pairing is perfect. Lemma~\ref{lem:paired_tensor_projective} applies. Thus finite projectivity is a consequence of these data.
	\end{remark}

	\vspace{0.5cm}

	We call a monoidal collection of local systems $\calL$ \emph{Morse compatible} if it is equipped with an action of elementary cobordisms and the following conditions are satisfied:
	
	\begin{enumerate}
		\item\label{item:mcg_rep} \textit{Mapping class group representation}. Morphisms $\phi_{d}$ satisfy:
		\begin{itemize}
			\item For the mapping class of the identity $\Id:\Sigma\to \Sigma$ the corresponding morphism of local systems is trivial: $\phi_{\Id}=\Id$, for all $\Sigma\in 3\Cob$;
			\item For two composable mapping classes $\Sigma\xrightarrow{d}\Sigma'\xrightarrow{d'}\Sigma''$ the diagram
			\begin{equation}\label{eq:mcg_rep}
			\left[\begin{tikzcd}
				\calL(\Sigma)\arrow[rr,"\phi_{d'\circ d}"]\arrow[dr,"\phi_{d}"']& &(d'\circ d)^*\calL(\Sigma'')\\
				&d^*\calL(\Sigma')\arrow[ur,"d^*\phi_{d'}"']&
			\end{tikzcd}\right] \quad \in \Loc_R(\Conf_*(\Sigma))
			\end{equation}
			is commutative.
		\end{itemize}
		\item\label{item:invariance}\textit{Invariance of index 1 or 2 action}. Let $M:\Sigma_-\to \Sigma_+$ and $M':\Sigma_-'\to \Sigma_+'$ be two elementary index 1 cobordisms and $d:M\to M'$ is a diffeomorphism that restricts to $d_-:\Sigma_-\to \Sigma_-'$ and $d_+:\Sigma_+\to \Sigma_+'$ on the boundary. It also induces a diffeomorphism $d_{\calT}:\calT_M\to \calT_{M'}$ such that $i_-\circ d_{\calT}=d_-\circ i_-$ and $i_+\circ d_{\calT}=d_+\circ i_+$ up to isotopy. This data can be organized in the commutative diagram below on the left and induces the diagram on the right. 
		\begin{equation}\label{eq:invariance_diag}
			\begin{tikzcd}
				&\calT_{M}\arrow[ld,"i_-"']\arrow[dr,"i_+"]\arrow[d,"d_{\calT}"]&\\
				\Sigma_-\arrow[d, "d_-"] & \calT_{M'}\arrow[dl,"i_-"']\arrow[dr,"i_+"]&\Sigma_+\arrow[d,"d_+"]\\
				\Sigma_-'&&\Sigma_+',
			\end{tikzcd}\quad 
			\begin{tikzcd}[column sep=tiny]
				&\Conf_*(\calT_M)\times \Conf_{p-1}(S^1)\arrow[dr,"\varphi"]\arrow[dl,"i_-\circ \pi"']\arrow[d,"d_{\calT}\times\Id"]&\\
				\Conf_*(\Sigma_-)\arrow[d, "d_-"] &\Conf_*(\calT_{M'})\times \Conf_{p-1}(S^1)\arrow[dl,"i_-\circ \pi"]\arrow[dr,"\varphi"']&\Conf_{*+p-1}(\Sigma_+)\arrow[d,"d_+"]\\
				\Conf_*(\Sigma_-')&&\Conf_{*+p-1}(\Sigma_+'),
			\end{tikzcd}
		\end{equation}
		where $\varphi$ is defined in \eqref{eq:varphi_def}. 
		Then the following diagram is commutative:
		\begin{equation}\label{eq:invariance}
		\begin{gathered}
			\left[\begin{tikzcd}[column sep=large]
				i_-^*\calL(\Sigma_-)\boxtimes\calA\arrow[r,"{i_-^*\phi_{d_-}\boxtimes \Id}"]\arrow[d,"\psi_M"]&i_-^* d_-^* \calL(\Sigma_-')\boxtimes\calA\arrow[r,"\simeq"]&d_{\calT}^* i_-^*\calL(\Sigma_-')\boxtimes \calA\arrow[d,"{d_{\calT}^*\psi_{M'}}"]\\
				\varphi^*\calL(\Sigma_+)\arrow[r,"\varphi^*\phi_{d_+}"]&\varphi^*d_+^*\calL(\Sigma_+')\arrow[r,"\simeq"]&d_{\calT}^*\varphi^*\calL(\Sigma_+')
			\end{tikzcd}\right]\in\\
			\in \Locpair_R(\Conf_*(\calT_M)\times \Conf_{p-1}(S^1))
		\end{gathered}
		\end{equation}

		\item\label{item:locality}\textit{Locality.} We require locality for either placement of the index 1 cobordism.

		\emph{Right locality.} Let $M:\Sigma_-\to \Sigma_+$ be an index 1 cobordism which splits into the boundary connected sum $M=M'\natural M''$ of an index 1 cobordism $M'':\Sigma_-''\to \Sigma_+''$ and a mapping cylinder $M'=M_{\Id}:\Sigma_-'\to \Sigma_+'$ of $\Id$, where $\Sigma_-=\Sigma_-'\natural\Sigma_-''$ and $\Sigma_+=\Sigma_+'\natural \Sigma_+''$. The trajectory space of $M$ naturally splits $\calT_{M}=\calT_{M'}\natural\calT_{M''}$. This data can be organized into the commutative diagram:
			\[
				\begin{tikzcd}
					&\calT_{M'}\sqcup \calT_{M''}\arrow[dl,"i_-\sqcup i_-"']\arrow[dr,"i_+\sqcup i_+"]\arrow[d,hook]&\\
					\Sigma_-'\sqcup\Sigma_-''\arrow[d,hook]&\calT_M\arrow[ld,"i_-"]\arrow[dr,"i_+"]&\Sigma_+'\sqcup \Sigma_+''\arrow[d,hook]\\
					\Sigma_-&&\Sigma_+.
				\end{tikzcd}
			\]
			It induces the commutative diagram:
			\begin{equation}\label{eq:Locality_topologiacl}
				\adjustbox{max width=0.98\textwidth}{$
				\begin{tikzcd}[column sep=tiny]
						&\Conf_*(\calT_{M'})\times \Conf_*(\calT_{M''}) \times\Conf_{p-1}(S^1)\arrow[dr,"i_+\times \varphi"]\arrow[dl,"i_-\times i_-\circ \pi"']\arrow[d,"\mu\times \Id"]&\\
					\Conf_*(\Sigma_-')\times\Conf_*(\Sigma_-'')\arrow[d,"\mu"]&\Conf_*(\calT_M)\times \Conf_{p-1}(S^1)\arrow[ld,"i_-\circ \pi"] \arrow[dr,"\varphi"]&\Conf_{*}(\Sigma_+')\times \Conf_{*+p-1}(\Sigma_+'')\arrow[d,"\mu"]\\
					\Conf_*(\Sigma_-)&&\Conf_{*+p-1}(\Sigma_+),
				\end{tikzcd}
				$}
				\end{equation}
				
		where $\mu$ is defined in \eqref{eq:mu_def}. Then the following diagram is commutative:
		\begin{equation}\label{eq:locality}
			\begin{gathered}
			\left[\begin{tikzcd}[column sep=large]
					i_-^*\calL(\Sigma_-')\boxtimes i_-^*\calL(\Sigma_-'')\boxtimes \calA\arrow[r,"{\Id\boxtimes\psi_{M''}}"]\arrow[d,"{i_-^*\mu_{\calL}\boxtimes\Id} "]& i^*_+\calL(\Sigma'_+)\boxtimes\varphi^*\calL(\Sigma_+'')\arrow[r,equal]&(i_+\times\varphi)^*(\calL(\Sigma_+')\boxtimes\calL(\Sigma_+''))\arrow[d,"(i_+\times\varphi)^*\mu_{\calL}"]\\
					(\mu\times \Id)^*(i_-^*\calL(\Sigma_-)\boxtimes \calA)\arrow[r,"(\mu\times \Id)^*\psi_M"]&(\mu\times \Id)^*\varphi^*\calL(\Sigma_+)\arrow[r,equal]&(i_+\times\varphi)^*\mu^*\calL(\Sigma_+)
			\end{tikzcd}\right]\in\\
			\in \Locpair_R\left(\Conf_*(\calT_{M'})\times \Conf_*(\calT_{M''})\times\Conf_{p-1}(S^1)\right),
			\end{gathered}
		\end{equation}
		where $\mu_{\calL}$ is the morphism \eqref{eq:nice_ls_monoidality}.

		\emph{Left locality.} Keeping $M''$ as the index 1 cobordism and $M'$ as the identity cylinder, suppose instead that $M=M''\natural M'$. Put
		\[
			r: \Conf_*(\calT_{M''})\times\Conf_*(\calT_{M'})\times\Conf_{p-1}(S^1)\to \Conf_*(\calT_{M''})\times \Conf_{p-1}(S^1)\times\Conf_*(\calT_{M'})
		\]
		\[
			r(\mathbf x,\mathbf y,\mathbf z)=(\mathbf x,\mathbf z,\mathbf y),
		\]
		and write $\varphi_L=(\varphi\times i_+)\circ r$. The corresponding topological identity is $\varphi\circ(\mu\times\Id)=\mu\circ\varphi_L$. Let $c_{\calA}$ denote the canonical coefficient permutation
		\[
			i_-^*\calL(\Sigma_-'')\boxtimes i_-^*\calL(\Sigma_-')\boxtimes\calA
			\xrightarrow{c_{\calA}}
			r^*\bigl(i_-^*\calL(\Sigma_-'')\boxtimes\calA\boxtimes i_-^*\calL(\Sigma_-')\bigr),
		\]
		given on fibers by $v\otimes w\otimes a\mapsto v\otimes a\otimes w$. Using the identity-cylinder identification on the $M'$ factor, we require the following square to commute:
		\begin{equation}\label{eq:locality_left}
		\begin{gathered}
		\left[\begin{tikzcd}[column sep=large]
			i_-^*\calL(\Sigma_-'')\boxtimes i_-^*\calL(\Sigma_-')\boxtimes\calA
			\arrow[r,"{r^*(\psi_{M''}\boxtimes\Id)\circ c_{\calA}}"{yshift=3pt}]
			\arrow[d,"{i_-^*\mu_{\calL}\boxtimes\Id}"']
			&\varphi_L^*\bigl(\calL(\Sigma_+'')\boxtimes\calL(\Sigma_+')\bigr)
			\arrow[d,"{\varphi_L^*\mu_{\calL}}"]\\
			(\mu\times\Id)^*\bigl(i_-^*\calL(\Sigma_-)\boxtimes\calA\bigr)
			\arrow[r,"{(\mu\times\Id)^*\psi_M}"]
			&\varphi_L^*\mu^*\calL(\Sigma_+).
		\end{tikzcd}\right]\in\\
		\in \Locpair_R\left(\Conf_*(\calT_{M''})\times \Conf_*(\calT_{M'})\times\Conf_{p-1}(S^1)\right).
		\end{gathered}
		\end{equation}
		Here the left vertical map pulls $\mu_{\calL}$ back along $i_-\times i_-$ on the two trajectory factors, and the bottom target uses the displayed topological identity. 

		\item\label{item:handle_cancellation}\textit{Handle cancellation.} Let $\Sigma_1,\Sigma_2\in 3\Cob$ be of genus 0. For a composition of two cancelling elementary cobordisms $M_{f}:\Sigma_1\xrightarrow{M_1}\Sigma\xrightarrow{M_2}\Sigma_2$ with belt sphere $\nu_1:=b(M_1)$ and attaching sphere $\nu_2=a(M_2)$ we have a pair of morphisms
		\[
			\psi_{M_1}:i_-^*\calL(\Sigma_1)\boxtimes\calA\to \varphi_1^*\calL(\Sigma), \quad \psi_{\overline{M}_2}:i_+^*\calL(\overline{\Sigma}_2)\boxtimes\calA\to \varphi_2^*\calL(\overline{\Sigma}),
		\]
		see \eqref{eq:psi_M_def}. Let 
		\[
			\mathring{\Sigma}:=i_+(\calT_{M_1})\cap i_-(\calT_{M_2}).
		\]
		Since both $\calT_{M_1}$ and $\calT_{M_2}$ are embedded into $\Sigma$, $\mathring{\Sigma}$ can be considered embedded into both of them and we have a commutative diagram (see also Fig.~\ref{fig:handle_cancellation}):
		\[
		\begin{tikzcd}
			&&\mathring{\Sigma}\arrow[dl,hook',"\rho_1"']\arrow[dr,hook,"\rho_2"]&&\\
			&\calT_{M_1}\arrow[dr,"i_+"]\arrow[ld,"i_-"']&&\calT_{M_2}\arrow[dl,"i_-"']\arrow[dr,"i_+"]&\\
			\Sigma_1\arrow[rrrr,bend right=10,"\Id"']&&\Sigma&&\Sigma_2.
		\end{tikzcd}
		\]
	
		\begin{figure}[h]
			\centering
			\includegraphics[width=1\linewidth]{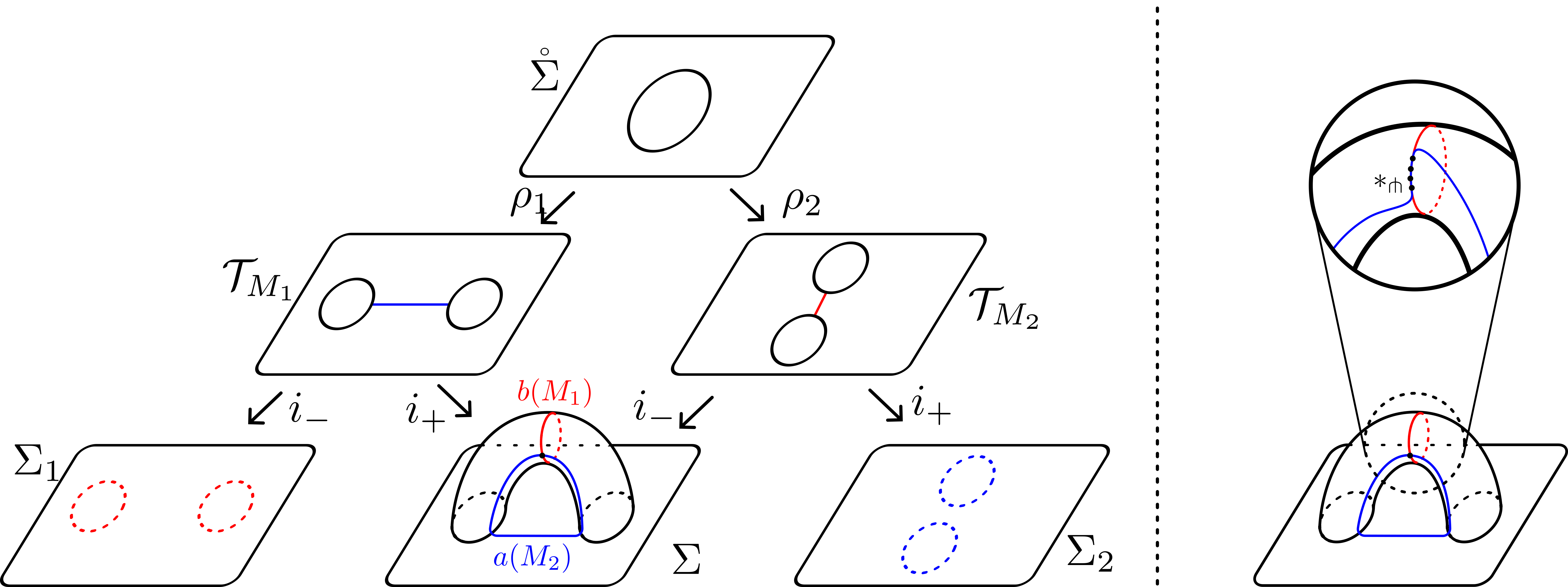}
			\caption{A composition of two cancelling cobordisms. Here $\mathring{\Sigma}$ is obtained from $\calT_{M_1}$ by cutting along the blue line, which is the preimage of $a(M_2)$ under $i_+$. Similarly, $\mathring{\Sigma}$ is obtained from $\calT_{M_2}$ by cutting along the preimage of $b(M_1)$ under $i_-$. We twist by an isotopy of $i_+$ a small segment of the attaching sphere $a(M_2)$ near the intersection point $a(M_2)\cap b(M_1)$ counterclockwise with respect to the orientation of $\Sigma$, in order to have a common segment. We choose a common configuration $*_{\pitchfork}$ and the base points on both copies of $\Conf_{p-1}(S^1)$ are taken to be preimages of $*_{\pitchfork}$ under $i_-$ and $i_+$.}
			\label{fig:handle_cancellation}
		\end{figure}

		In particle degree zero this induces a commutative diagram (up to isotopy described in Fig.~\ref{fig:handle_cancellation}):
		
		\[
		\begin{tikzcd}[column sep=tiny]
			&&\Conf_0(\mathring{\Sigma})\arrow[dl,hook',"\rho_1"]\arrow[dr,hook,"\rho_2"']&&\\
			&\Conf_0(\calT_{M_1})\times \Conf_{p-1}(S^1)\arrow[dr,hook, "\varphi_1"']\arrow[ld,"{i_-\circ \pi}"]&&\Conf_0(\calT_{M_2})\times \Conf_{p-1}(S^1)\arrow[dl,hook',"\varphi_2"]\arrow[dr,"{i_+\circ \pi}"]&\\
			\Conf_0(\Sigma_1)\arrow[rrrr, bend right=10,"\Id"']&&\Conf_{p-1}(\Sigma)&&\Conf_0(\Sigma_2),
		\end{tikzcd}
		\]
		where $\rho_1$ and $\rho_2$ insert the fixed configurations on the two circles whose images are $*_{\pitchfork}$. Write $e_1=u_{\Sigma_1}(1)$ and $\bar e_2=u_{\overline\Sigma_2}(1)$. Evaluating the pulled-back coefficient maps over $\Conf_0(\mathring\Sigma)=\{*_0\}$ gives
		\[
			v_1=\rho_1^*\psi_{M_1}(e_1\otimes\eta^A(1))\in\calL_{p-1}(\Sigma)_{*_{\pitchfork}},\qquad
			v_2=\rho_2^*\psi_{\overline M_2}(\bar e_2\otimes\eta^A(1))\in\calL_{p-1}(\overline\Sigma)_{*_{\pitchfork}}.
		\]
		The required cancellation condition is the scalar identity
		\begin{equation}\label{eq:cancellation}
			(v_1, \mathfrak f_{p-1}^{-1}\cdot v_2)=1,
		\end{equation}
		where the second fiber is identified with the conjugate of the first by $\chi_\Sigma$, and $\mathfrak f_{p-1}^{-1}=([p-1]_q!q^{(p-2)(p-1)/2})^{-1}$ is the inverse fusion factor. Thus this condition is imposed only over $\Conf_0(\mathring\Sigma)$.
		\item\label{item:monoidal} \textit{Monoidal structure}.
		 Suppose $d_1:\Sigma_1\to \Sigma_1'$ and $d_2:\Sigma_2\to\Sigma_2'$ are two mapping classes, $\Sigma:=\Sigma_1\natural \Sigma_2$, $\Sigma'=\Sigma_1'\natural\Sigma_2'$ and $d:\Sigma\to \Sigma'$ is the induced mapping class. In particular, we have the commutative diagram:
		 \begin{equation}\label{eq:monoidal_topological}
		 	\begin{tikzcd}
		 		\Conf_*(\Sigma_1)\times \Conf_*(\Sigma_2)\arrow[r,"d_1\times d_2"]\arrow[d,"\mu"]&\Conf_*(\Sigma_1')\times \Conf_*(\Sigma_2')\arrow[d,"\mu"]\\
		 		\Conf_*(\Sigma)\arrow[r,"d"]&\Conf_*(\Sigma')
		 	\end{tikzcd}
		 \end{equation}
		 Then the diagram
		\begin{equation}\label{eq:Monoidal_structure}
		\begin{gathered}
			\left[\begin{tikzcd}[column sep=large]
				\calL(\Sigma_1)\boxtimes \calL(\Sigma_2)\arrow[r,"\phi_{d_1}\boxtimes \phi_{d_2}"]\arrow[d,"{\mu_{\calL}}"]&d_1^*\calL(\Sigma_1')\boxtimes d_2^*\calL(\Sigma_2')\arrow[r,equal]&(d_1\times d_2)^*(\calL(\Sigma'_1)\boxtimes \calL(\Sigma_2'))\arrow[d,"{(d_1\times d_2)^*\mu_{\calL}}"]\\
				\mu^*\calL(\Sigma)\arrow[r,"{\mu^*\phi_{d}}"]&\mu^*d^*\calL(\Sigma')\arrow[r,equal]&(d_1\times d_2) ^*\mu^*\calL(\Sigma')
			\end{tikzcd}\right]\in\\
			\in \Loc_R(\Conf_*(\Sigma_1)\times \Conf_*(\Sigma_2))
		\end{gathered}
		\end{equation}
		is commutative.
		
		The orientation-reversal identifications are required to be anti-monoidal and natural with respect to mapping cylinders. Let $\tau$ interchange the configuration-space factors and let $c$ be the unsigned interchange of coefficient factors. Then it means
		\begin{equation}\label{eq:conjugation_coherence}
			\begin{aligned}
				\chi_{\Sigma_1\natural\Sigma_2}\,(\tau^*\mu_{\overline\Sigma_2,\overline\Sigma_1})\,c
				&=\overline{\mu_{\Sigma_1,\Sigma_2}}\,(\chi_{\Sigma_1}\boxtimes\chi_{\Sigma_2}),\\
				d^*\chi_{\Sigma'}\,\phi_{\bar d}&=\overline{\phi_d}\,\chi_\Sigma,\qquad
				\chi_D u_{\overline D}=\overline{u_D}.
			\end{aligned}
		\end{equation}
		Here $\mu_{\Sigma_1,\Sigma_2}$ denotes the corresponding $\mu_{\calL}$, and $\bar d:\overline\Sigma\to\overline{\Sigma'}$ is the same underlying diffeomorphism as $d:\Sigma\to\Sigma'$. 
	\end{enumerate} 
	When these conditions are imposed projectively, each diagram must commute up to one invertible scalar common to all particle degrees.

	\subsection{Action of elementary cobordisms on monodromy representations}\label{sec:monodromy_reps}
	In order to work with twisted cycles or specific examples of Morse compatible local systems we need an explicit reformulation of the action of elementary cobordisms on monodromy representations. We derive it here.
	
	\vspace{0.5cm}

	Recall that for each $\Sigma\in3\Cob$ and $n\geq0$ a base point $*_n\in\Conf_n(\Sigma)$ is fixed. It lies in $\Conf_n^-(\Sigma)$ for $n>0$, while $*_0$ is the empty configuration. Let $\{\calL(\Sigma)\}_{\Sigma\in3\Cob}$ be a Morse compatible collection of $p$-Heisenberg local systems, with coefficient modules $L_n(\Sigma)$ at these base points. For each mapping class $d:\Sigma\to\Sigma'$, the morphism $\phi_d$ induces a pairing-preserving map of $B_n(\Sigma)$-modules
	\[
		\phi_d:L_n(\Sigma)\longrightarrow d^*L_n(\Sigma').
	\]
	Equivariance under the chosen full Heisenberg actions is additional information when the braid-group image is not the whole Heisenberg group. In particular, it does not follow from the local-system condition in degree zero. In the examples below the mapping-cylinder maps are chosen to be Heisenberg intertwiners in every degree.	
	
	Consider now an index 1 cobordism $M:\Sigma_-\to \Sigma_+$. Choose a base point $s_{p-1}\in \Conf_{p-1}(S^1)$. The local system $i_-^*\calL(\Sigma_-)\boxtimes \calA$ has a monodromy representation $i_-^*L_n(\Sigma_-)\boxtimes A$ at the point $*_n\times s_{p-1}$. 
	Then $\psi_M$ (see \ref{eq:psi_M_def}) induces a morphism to the pull-back of the monodromy representation of $\calL_{n+p-1}(\Sigma_+)$ at the point $*_n\cup \nu(s_{p-1})$. Since it doesn't coincide with $*_{n+p-1}$ we need to choose a path $h:[0,1]\to \Conf_{n+p-1}(\Sigma_+)$ from $*_{n+p-1}$ to $*_n\cup \nu(s_{p-1})$ in order to identify the monodromy representations at these points. Therefore, we have a morphism of $B_{n}(\calT_M)\times\pi_1(\Conf_{p-1}(S^1))$ representations:
	\[
	\psi_M^h: i_-^*L_n(\Sigma_-)\boxtimes A\to\varphi^*_h L_{n+p-1}(\Sigma_+),
	\]
	where $\varphi_h^*$ denotes the pull-back induced by the homomorphism
	\[
	\varphi_*^h:B_n(\calT_M)\times\pi_1(\Conf_{p-1}(S^1))\to B_{n+p-1}(\Sigma_+)
	\] 
	\[
	\gamma\mapsto h\varphi_*(\gamma)h^{-1}.
	\]
	
	Then for two paths $h$, $h'$ the corresponding morphisms are related by:
	\begin{equation}\label{monodromy_rep_relation}
		\psi_M^h(x)=[h (h')^{-1}]_{\Heis}\cdot\psi_M^{h'}(x),\quad \forall x\in i_-^*L_n(\Sigma_-)\otimes A
	\end{equation}

	For the tensor maps \eqref{eq:nice_ls_monoidality}, one can choose splitting paths from $*_{n+m}$ to $\mu(*_n,*_m)$ through configurations on $\partial_-\Sigma$, and use their transports to express $\mu_{\calL}$ at the reference fibers. Since any two paths of this form are homotopic it gives a canonical identification.
	


	\subsection{State spaces and the action of elementary cobordisms}\label{subs:Construction_of_homological_TQFTs}
	Suppose $\calL$ is a Morse compatible collection of Heisenberg local systems defined in the previous subsection. 
	
	\subsubsection{State spaces}
	The state space associated to $\Sigma\in 3\Cob$ is the $R$-module (subspace of small cycles):
	\[
		\boxed{\bbF_{\calL}(\Sigma):=\mathring{\bbH}(\Sigma;\calL(\Sigma)).}
	\]
	In particle degree zero, $\bbF_{\calL,0}(\Sigma)=L_0(\Sigma)$. For a disk $D$, Remark~\ref{rem:low_degree_disk} gives
	\[
		\bbF_{\calL}(D)=L_0(D)\xleftarrow[\simeq]{u_D}R,
	\]
	with no positive-degree summands. Thus the disk is the tensor unit regardless of its positive-particle monodromy.

	\subsubsection{Monoidality}\label{sec:monoidality}
	
	For a connected sum $\Sigma=\Sigma_1\natural \Sigma_2$ consider the composition:
	\[
		\mathring\bbH(\Sigma_1;\calL(\Sigma_1))\otimes \mathring\bbH(\Sigma_2;\calL(\Sigma_2))\to \mathring\bbH(\Sigma_1\sqcup\Sigma_2;\calL(\Sigma_1)\boxtimes\calL(\Sigma_2))\xrightarrow{(\mu_{\calL})_*} \mathring{\bbH}(\Sigma_1\sqcup\Sigma_2;\mu^*\calL(\Sigma))\xrightarrow{\mu_*} \mathring{\bbH}(\Sigma;\calL(\Sigma))
	\]
	The first arrow is the K\"unneth homomorphism with sign $(-1)^{(p-1)g_1n_2}$ in particle degrees $(n_1,n_2)$, where $g_1$ is the genus of $\Sigma_1$. For odd $p$ this is the ordinary cross product. For $p=2$ the sign compensates for moving an inserted point past the second factor. We denote the composition
	\begin{equation}\label{eq:mu_bbF_def}
		\boxed{\mu^{\bbF}_{\calL}:\bbF_{\calL}(\Sigma_1)\otimes \bbF_{\calL}(\Sigma_2)\to \bbF_{\calL}(\Sigma_1\natural \Sigma_2).}
	\end{equation}

	It is easy to check that $\mu^{\bbF}_{\calL}$ is an isomorphism considering twisted cycles for an appropriate choice of a set of symplectic arcs. In particular, we choose a set of arcs $\{\alpha_i,\beta_i\}_{i=1}^{g_1+g_2}$ on $\Sigma$ such that none of them intersect the common boundary of $\Sigma_1$ and $\Sigma_2$ and all arcs for $i=1,\dots,g_1$ belong to $\Sigma_1$, all arcs for $i=g_1+1,\dots,g_1+g_2$ belong to $\Sigma_2$. Such a picture is always diffeomorphic to Fig.~\ref{fig:homology_splitting}. Let $L_1$ and $L_2$ be monodromy representations of $\calL(\Sigma_1)$ and $\calL(\Sigma_2)$ at $*_{k_1}$ and $*_{k_2}$, and $\mathbf{v}_1\in L_1$, $\mathbf{v}_2\in L_2$, then 
	\begin{multline}
		\mu_{\calL}^{\bbF}: \left(\mathbf{\Gamma}(a_1,b_1,\dots,a_{g_1},b_{g_1})\otimes \mathbf{v}_1\right)\bigotimes \left(\mathbf{\Gamma}(a_{g_1+1},b_{g_1+1},\dots,a_{g_1+g_2},b_{g_1+g_2})\otimes \mathbf{v}_2\right)\mapsto\\
		\mapsto (-1)^{(p-1)g_1n_2}\mathbf{\Gamma}(a_{1},b_{1},\dots,a_{g_1+g_2},b_{g_1+g_2})\otimes \mu_{\calL}(\mathbf{v}_1\otimes \mathbf{v}_2),
	\end{multline}
	where $n_2=\sum_{i=g_1+1}^{g_1+g_2}(a_i+b_i)$.
	It is an isomorphism since $\mu_{\calL}$ is an isomorphism ($\calL$ is monoidal) and all possible compositions $(a_1,b_1,\dots, a_{g_1+g_2},b_{g_1+g_2})$ are obtained uniquely from compositions $(a_1,b_1,\dots,a_{g_1},b_{g_1})$ and $(a_{g_1+1},b_{g_1+1},\dots,a_{g_1+g_2},b_{g_1+g_2})$.

	\begin{figure}[h]
		\centering
		\includegraphics[width=0.7\linewidth]{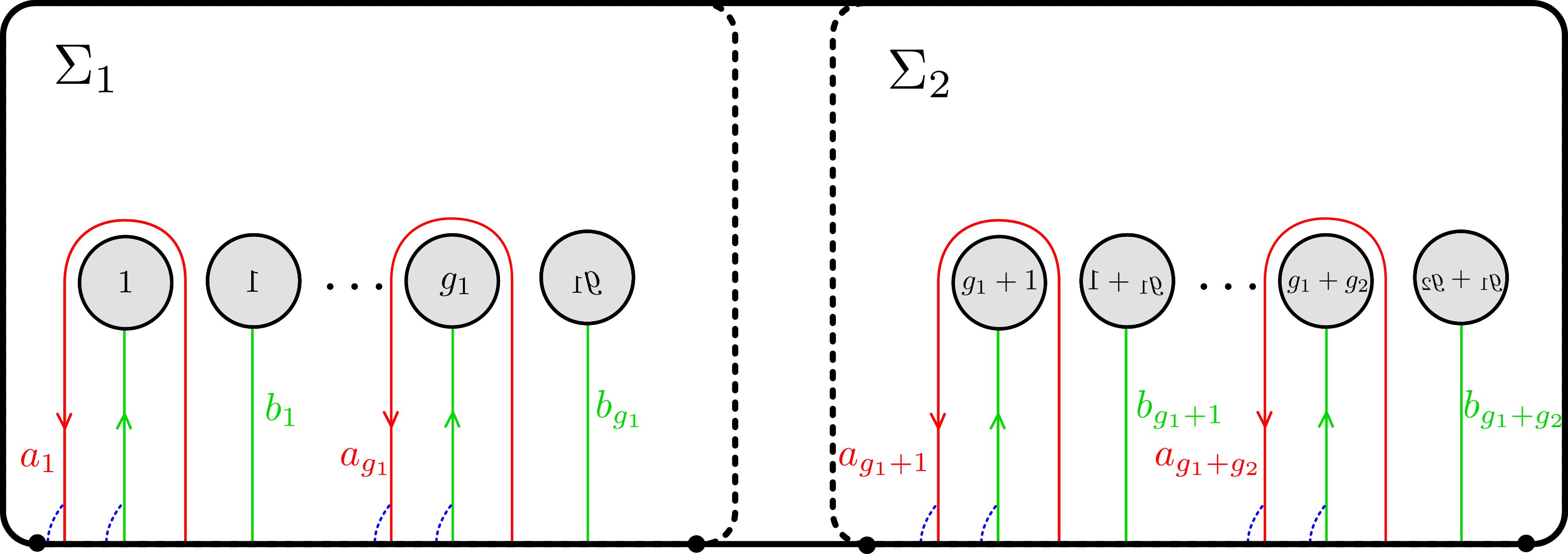}
		\caption{Arc summands in $\bbH^{BM}(\Sigma;\calL(\Sigma))$ suitable for checking the isomorphism $\mu^{\bbF}_{\calL}$. Each summand is a copy of the corresponding coefficient module.}
		\label{fig:homology_splitting}
	\end{figure}
	
	Associativity of $\mu_{\calL}$ induces associativity of $\mu_{\calL}^{\bbF}$: the signs agree because
	\[
		g_1n_2+(g_1+g_2)n_3=g_2n_3+g_1(n_2+n_3).
	\]
	The unit map $R\xrightarrow{u_D}\bbF_{\calL}(D)$ and the degree $(0,n)$ and $(n,0)$ unit conditions give the left and right unit constraints in every particle degree: the additional sign is $1$ for either disk factor.

	\subsubsection{Action of mapping classes}\label{sec:mapping_groupoid_rep}
	
	Let $d:\Sigma\to \Sigma'$ be a mapping class and $M_d$ the corresponding mapping cylinder. We define an action of $M_d$ as the composition
	\[
		\boxed{\bbF_{\calL}(d):\bbF_{\calL}(\Sigma)=\mathring{\mathbb{H}}(\Sigma;\calL(\Sigma))\xrightarrow{\phi_{d}} \mathring{\mathbb{H}}(\Sigma;d^*\calL(\Sigma'))\xrightarrow{d_*}\mathring{\mathbb{H}}(\Sigma';\calL(\Sigma'))=\bbF_{\calL}(\Sigma').}
	\]
	
	The map $\bbF_{\calL}(d)$ preserves the pairing $\langle-,-\rangle$. Indeed, $\phi_d$ is an isomorphism in $\Loc_R(\Conf_*(\Sigma))$, and $d_*$ preserves intersections.
	
	\subsubsection{Index 1 cobordism action}\label{sec:index_1_action}
	Let $M:\Sigma_-\to\Sigma_+$ be an elementary index 1 cobordism with belt sphere $\nu:S^1\to \Sigma_+$. Recall definitions of $\psi_M$ (see \eqref{eq:psi_M_def}) and $\varphi$ (see \eqref{eq:varphi_def}).
	We first consider the following composition which we denote by $\Psi_M$:
	
	\begin{equation}\label{eq:Psi_M_def}
		\begin{tikzcd}
			\bbH^{BM}(\calT_M;i_-^*\calL(\Sigma_-))\arrow[d,"{\mathrm{Id}\otimes \left[\widetilde{\Conf}_{p-1}(S^1)\right]}"]\arrow[dddd,rounded corners=3pt , to path={-- ++(-4cm,0)  |-  node [left, near start] {$\Psi_M$} (\tikztotarget)}]
			\\ 
			\bbH^{BM}(\calT_M;i_-^*\calL(\Sigma_-))\otimes H^{BM}_{p-1}(\Conf_{p-1}(S^1);\calA)\arrow[d]\\
			\bbH^{BM}(\calT_M\sqcup S^1;i_-^*\calL(\Sigma_-)\boxtimes\calA)\arrow[d,"(\psi_M)_*"]\\
			\bbH^{BM}(\calT_M\sqcup S^1;\varphi^*\calL(\Sigma_+))\arrow[d,"\varphi_*"]\\
			\bbH^{BM}(\Sigma_+;\calL(\Sigma_+))
		\end{tikzcd}
	\end{equation}
	where the second vertical map is the K\"unneth homomorphism, and we use the notation
	\[
		\mathbb{H}^{BM}(\calT_M\sqcup S^1;\tilde\calL):=\bigoplus\limits_{n=0}^{\infty} H_{n+p-1}^{BM}(\Conf_n(\calT_M)\times \Conf_{p-1}(S^1),\Conf_n^-(\calT_M)\times \Conf_{p-1}(S^1);\tilde\calL)
	\]
	for any local system $\tilde{\calL}$.

	\begin{lemma}\label{lem:i_inversion}
	Suppose that the action of each $1-q^{2k}$, $1\leq k <p$, on $\calL$ is invertible and the action of $q^{2p}$ is trivial, for all $\Sigma\in 3\Cob$. Then $\Psi_M$ sends small cycles to small cycles, and in the diagram
		\begin{equation}\label{eq:homological_index_1_span}
			\begin{tikzcd}
					&\mathring{\bbH}(\calT_M;i_-^*\calL(\Sigma_-))\arrow[dl,"{(i_-)_*}"']\arrow[rd,"{\Psi_M}"]&\\
					\mathring{\bbH}(\Sigma_-; \calL(\Sigma_-))\arrow[rr,"\exists !","\bbF_{\calL}(M)"',blue]&& \mathring{\bbH}(\Sigma_+;\calL(\Sigma_+))
			\end{tikzcd}
		\end{equation}
		the map $(i_-)_*$ is surjective and $\Ker(i_-)_*\subset \Ker(\Psi_M)$, where both maps are restricted to the displayed small-cycle spaces. Therefore, there is a unique homomorphism $\bbF_{\calL}(M)$ making this diagram commutative.
	\end{lemma}

	\begin{proof}

		The proof is an explicit computation via twisted cycles. We consider a standard index 1 cobordism depicted in Fig.~\ref{fig:cointegral}. This computation is valid for any other index 1 cobordism up to diffeomorphisms of $\calT_M$, $\Sigma_-$ and $\Sigma_+$ commuting with $i_-$ and $i_+$.
		
		\begin{figure}[h]
			\centering
			\includegraphics[width=0.8\linewidth]{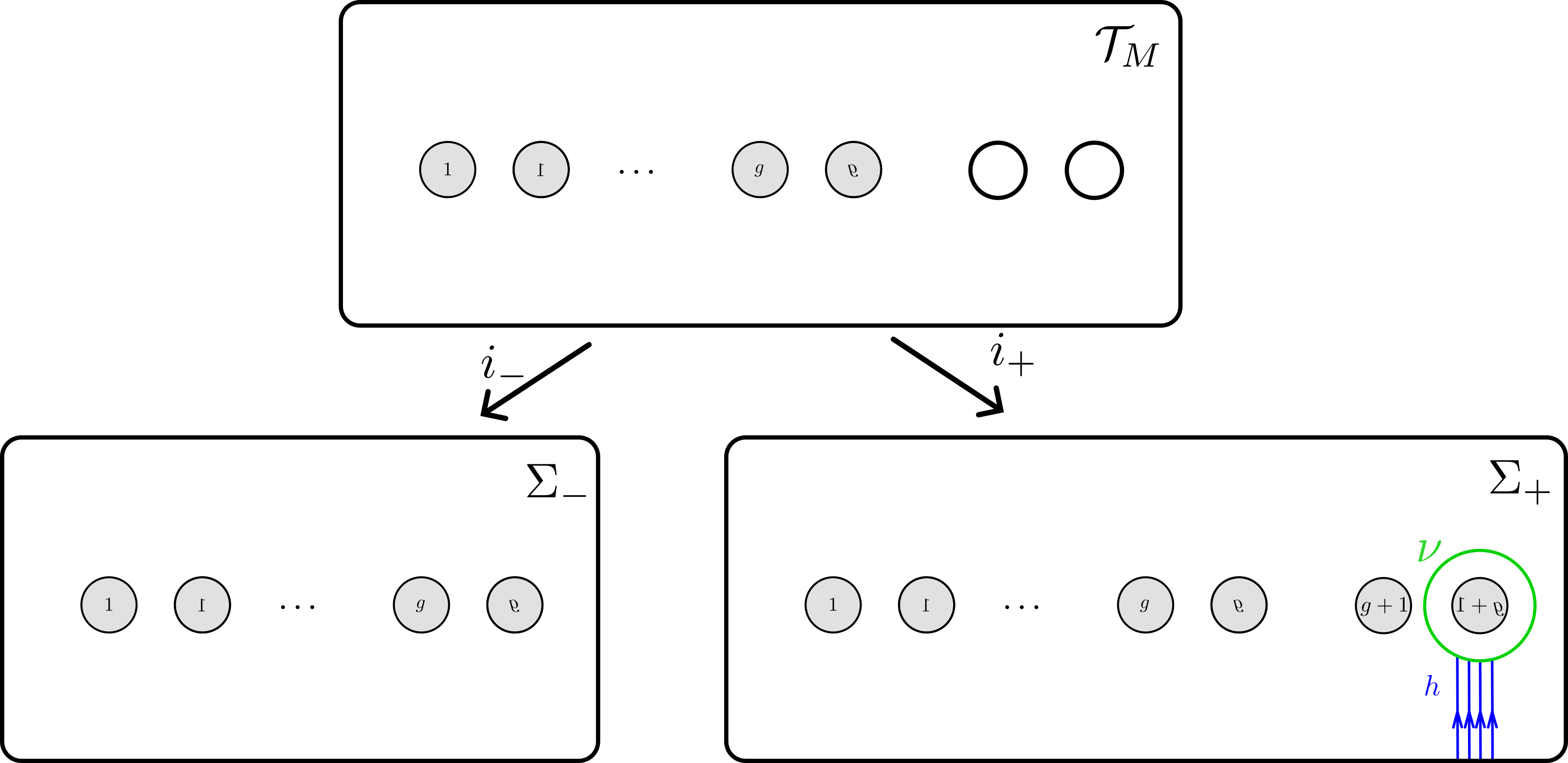}
			\caption{A standard index 1 cobordism with belt sphere $\nu$. Here $h$ is the path from $*_{n+p-1}$ to $*_{n}\cup\nu(s_{p-1})$ moving $p-1$ rightmost points in parallel vertically and constant on other points.}
			\label{fig:cointegral}
		\end{figure}
		
		We choose a set of symplectic arcs on both $\Sigma_-$ and $\Sigma_+$ and a set of arcs in $\calT_M$ consistent with those, see Fig.~\ref{fig:Sigma_W_arcs}. We identify $\Heis(\Sigma)$ and $\Heis'(\Sigma)$ by Proposition~\ref{prop:Heisenberg_iso} using these arcs. We consider twisted cycles given by these arcs and assign labels: $a_i$ to arcs $\alpha_i^{x}$, and $b_i$ to arcs $\beta_i^{x}$, $x\in\{+,-,M\}$, $a$ to $\alpha^M$ and $\overline{a}$ to $\overline{\alpha}^M$ (compare Fig.~\ref{fig:cointegral} and Fig.~\ref{fig:Sigma_W_arcs}).
		\begin{figure}[h]
			\centering
			\includegraphics[width=0.8 \linewidth]{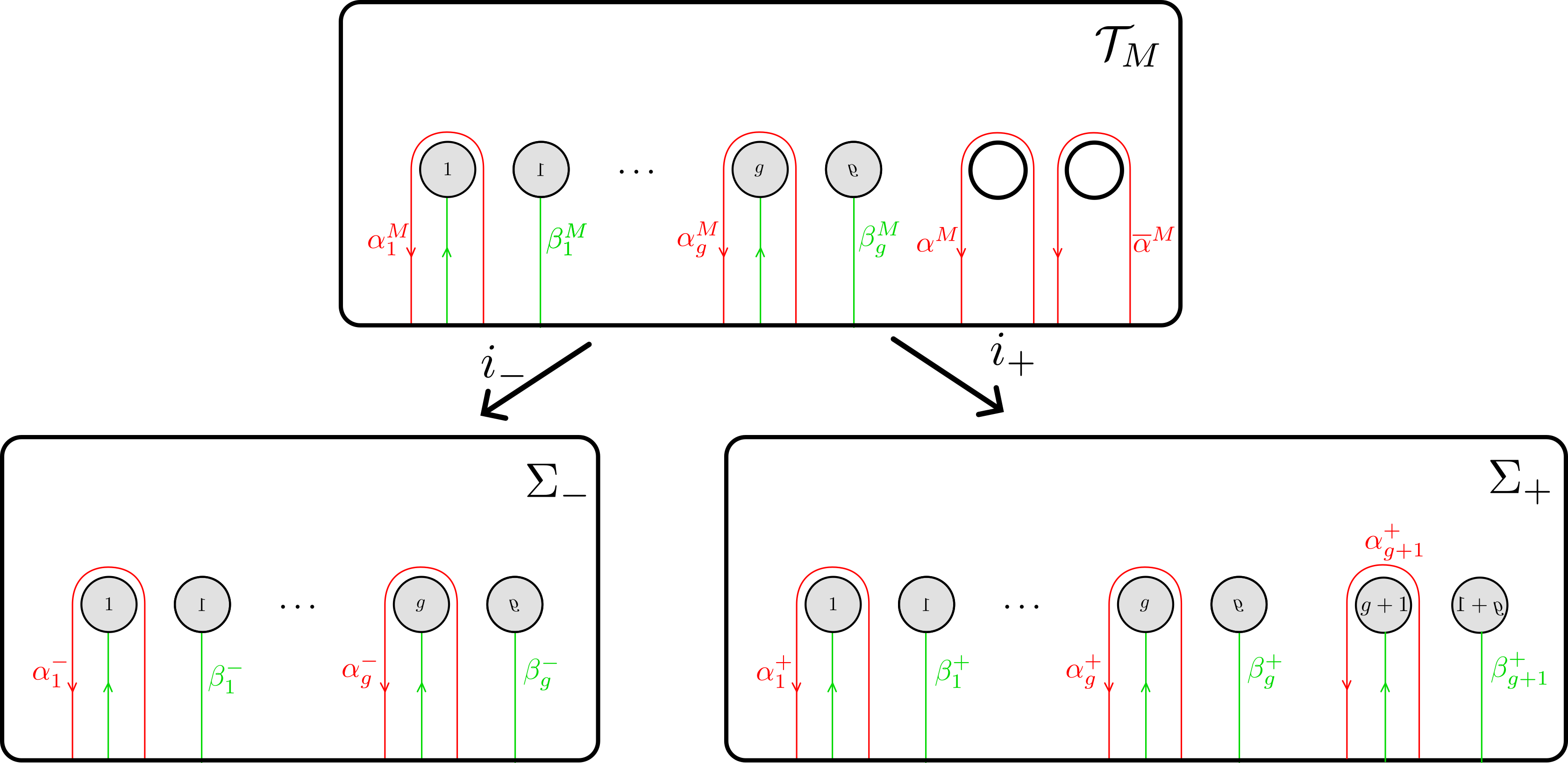}
			\caption{A choice of symplectic arcs for $\Sigma_+$, $\Sigma_-$ and $\calT_M$. The subset of arcs $\alpha_j^M,\beta_j^M$ for $j=1,\dots,g$ sent by $i_-$ to $\Sigma_-$ form a symplectic set of arcs $G^-:=\{\alpha_j^-,\beta_j^-\}$ in $\Sigma_-$ for $\alpha_j^-:=i_-(\alpha_j^M)$, $\beta_j^-:=i_-(\beta_j^M)$. Denote arcs in $\Sigma_+$ by $\alpha_j^+:=i_+(\alpha_j^M)$ and $\beta_j^+:=i_+(\beta_j^M)$ for $j=1,\dots,g$, and $\alpha_{g+1}^+:=i_+(\alpha^M)$. An arc $\beta_{g+1}$ in $\Sigma_+$ can be chosen such that $G^+:=\{\alpha_i^+,\beta_i^+\}_{i=1}^{g+1}$ is a symplectic set of arcs. For $\Sigma_-$ and $\Sigma_+$ we fix isomorphisms \eqref{eq:def_Heisenberg} $\rho_{G^-}$ and $\rho_{G^+}$ correspondingly.}
			\label{fig:Sigma_W_arcs}
		\end{figure}

		Fix $n\geq 0$. Let $L_-$ and $L_+$ be the monodromy representations for $\calL_n(\Sigma_-)$ and $\calL_{n+p-1}(\Sigma_+)$ at $*_n$ and $*_{n+p-1}$ correspondingly. For $\mathbf{v}\in L_-$ consider the element of $\mathring{\bbH}(\calT_M;i_-^*\calL(\Sigma_-))$ given by the twisted cycle $\mathbf{\Gamma}(a_1,b_1,\dots,a_g,b_g,a,\overline{a})\otimes \mathbf{v}$. By Section~\ref{sec:small_cycles}, these cycles span the small-cycle space when all labels satisfy $0\leq a_i,b_i,a,\overline a<p$ and their sum is $n$. Its image in $\mathring{\bbH}(\Sigma_-;\calL(\Sigma_-))$ is:
		\[
			(i_-)_*:\mathbf{\Gamma}(a_1,b_1,\dots,a_g,b_g,a,\overline{a})\otimes \mathbf{v}\mapsto \delta_{a,0}\delta_{\overline{a},0}\cdot\mathbf{\Gamma}(a_1,b_1,\dots,a_g,b_g)\otimes \mathbf{v}
		\]
		since the images of arcs $\alpha^M$ and $\overline{\alpha}^M$ in $\Sigma_-$ can be contracted to $\partial_-\Sigma_-$, the corresponding cycles for non-zero $a$ or $\overline{a}$ vanish.
		Clearly $(i_-)_*$ is surjective on small cycles, by taking $a=\overline a=0$, and its kernel there is spanned by twisted cycles with $a+\overline{a}>0$:
		\[
			\mathrm{Ker}(i_-)_*=R\langle \mathbf{\Gamma}(a_1,b_1,\dots,a_g,b_g,a,\overline{a})\otimes \mathbf{v}\mid a+\overline{a}>0,\ 0\leq a_i,b_i,a,\overline a<p,\ \mathbf{v}\in L_- \rangle.
		\]
		
		We now compute the action of $\Psi_M$. Fix $\mathbf{v}\in L_-$, then $\psi_M(\mathbf{v}\boxtimes 1)$ is an element of the monodromy representation of $\calL(\Sigma_+)$ at $*_n\cup s_{p-1}$. We connect $*_n\cup s_{p-1}$ with $*_{n+p-1}$ by the path $h$ as in Fig.~\ref{fig:cointegral}. This allows us to identify the monodromy representation at $*_n\cup s_{p-1}$ with the monodromy representation at $*_{n+p-1}$, namely $L_+$. Recall that $\psi^h_M(\mathbf{v}\boxtimes 1)$ is the element of $L_+$ obtained by this identification from $\psi_M(\mathbf{v}\boxtimes 1)$ (see Subsection \ref{sec:monodromy_reps}).
		
		Recall that $\varphi$ (defined in \eqref{eq:varphi_def}) inserts $\Conf_{p-1}(S^1)$ by sending its configuration points to the belt sphere $\nu:S^1\to \Sigma_+$. Recall that $t$ is the generator of $ \pi_1(\Conf_{p-1}(S^1),s_{p-1})$ (see Fig.~\ref{fig:t_loop}). Then $\varphi_* (1,t)\in \pi_1(\Conf_{n+p-1}(\Sigma_+), *_n\cup \nu(s_{p-1}))$ is the loop moving $p-1$ configuration points along the belt sphere and constant on $*_n$. We conjugate it by the path $h$ in order to obtain the loop $h\varphi_* (1,t)h^{-1}\in \pi_1(\Conf_{n+p-1}(\Sigma_+), *_{n+p-1})$, see Fig.~\ref{fig:loop_insertion}.
		
		\begin{figure}[h]
			\centering
			\includegraphics[width=0.7\linewidth]{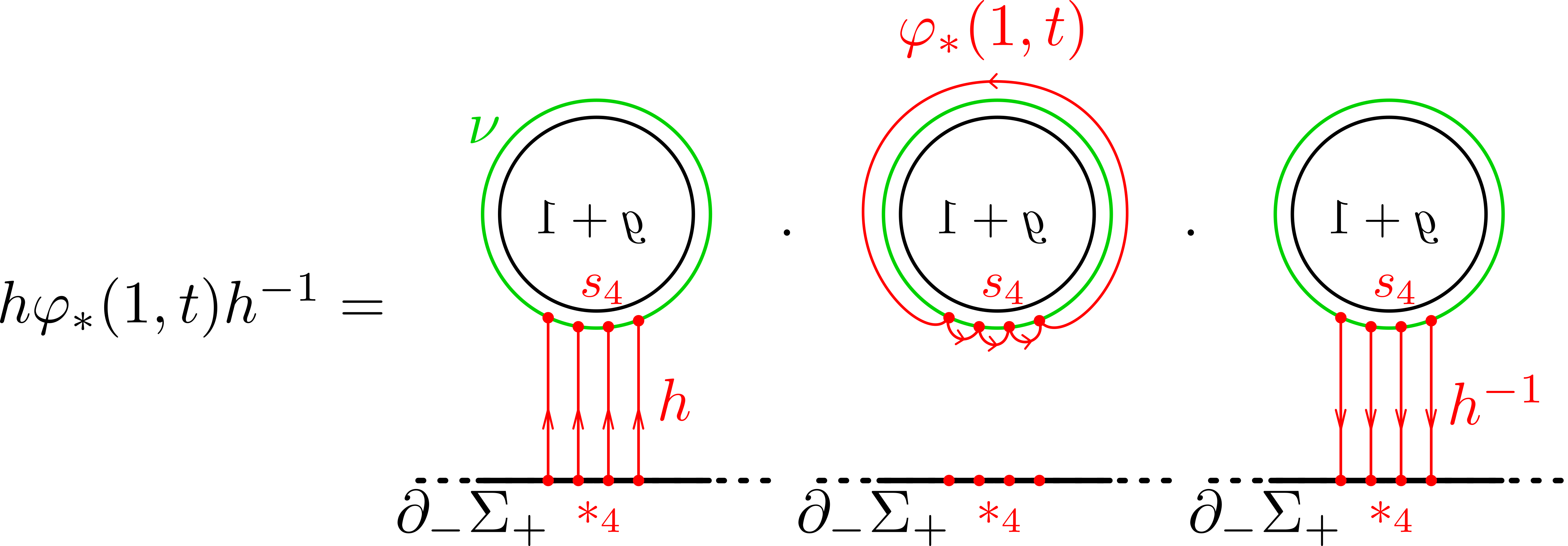}
			\caption{Red lines show the trajectories of each point under $h\varphi_*(1,t)h^{-1}$. For odd $p$, there are two contributions in the Heisenberg evaluation of this loop: the first is $[i_+(\overline{\alpha})]_{\Heis}=(-2,-[\alpha_{g+1}^+])$ and given by a point moving around the meridian of the handle; the second is $[\sigma^{p-2}]_{\Heis}=-[\sigma^{-2}]_{\Heis}$ as operators on the coefficients, and given by the rightmost point moving to the left of the other $p-2$ points counter-clockwise.}
			\label{fig:loop_insertion}
		\end{figure}
		
		For odd $p$, as an operator on the monodromy representation, we have
		\[
			[h\varphi_* (1,t)h^{-1}]_{\Heis}=[\sigma^{p-2} i_+(\overline{\alpha})]_{\Heis}=-(2,0)\cdot[(\beta_{g+1}^+)^{-1}(\alpha_{g+1}^+)^{-1}\beta_{g+1}^+]_{\Heis}=-(0,-[\alpha_{g+1}^+])
		\]
		where we used the assumption that $\sigma^p$ acts by $-1$. Recall that $\eta^A_*$ appearing in the definition of $\Psi_M$ sends $1$ to $\mathbf{t}=\sum_k (-t)^k$ (see \eqref{eq:eta_A_def}). Therefore, one can compute:
		\[
			\psi_M^h\circ (\Id\boxtimes \eta^A) (\mathbf{v}\boxtimes 1)=\psi_M^{h}\left(\sum\limits_{k=0}^{2p-1}(\mathbf{v}\boxtimes(-t)^k)\right)=\sum\limits_{k=0}^{2p-1}(0,-k[\alpha_{g+1}^+])\cdot \psi_M^h(\mathbf{v}\boxtimes 1).
		\]
		For $p=2$, there is one inserted point and no braid contribution. The loop acts by $(0,-[\alpha_{g+1}^+])$ and $\eta^A(1)=1+t$. Thus the last sum has two terms. In both cases the number of terms is $p'$, as defined in Section~\ref{subs:Twisted_homology_of_Conf_S_1}.

		It remains to check where the chain $\Delta$ (spanning $\Conf_{p-1}(I)$) is sent under the map $\varphi$. Moving the configuration points of $s_{p-1}$ along $h^{-1}$ together with a segment of $\nu$ brings that segment to the relative boundary. The resulting relative chain class is represented by the arc cycle parallel to $i_+(\overline{\alpha})$. Hence one obtains the twisted cycle in Fig.~\ref{fig:d_r_action_BM}. For $1\leq\overline a<p$, the fusion coefficient is
		\[
		q^{\overline a(p-1)}\sqbinom{\overline a+p-1}{\overline a}_q
		=\mathfrak f_{\overline a+p-1}(\mathfrak f_{\overline a}\mathfrak f_{p-1})^{-1}=0,
		\]
		since the numerator vanishes and the denominator is invertible. Thus the resulting twisted cycle vanishes.
		
		\begin{figure}[h]
			\centering
			\includegraphics[width=1\linewidth]{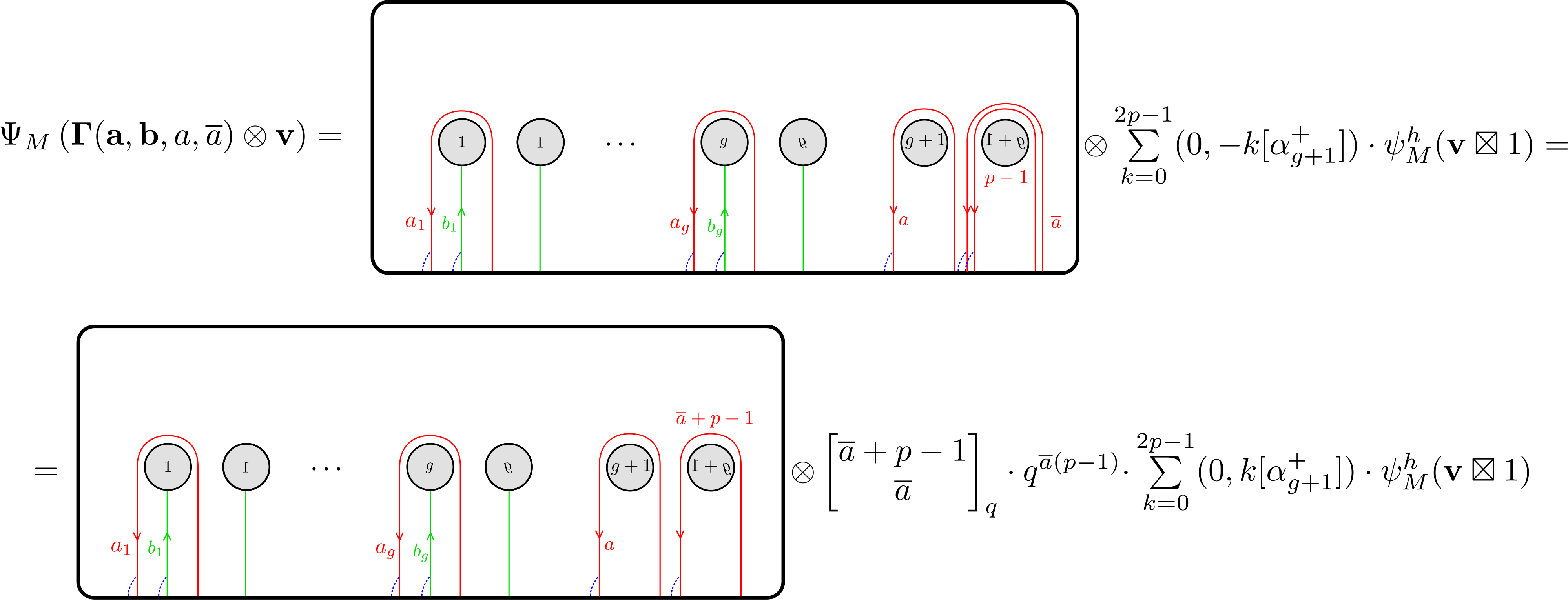}
			\caption{The twisted cycle on the top represents the inserted chain $\Delta$, relative to $\Conf_*^-(\Sigma_+)$, by an arc cycle parallel to $i_+(\overline{\alpha}^M)$. At the bottom is the twisted cycle obtained after applying the fusion rule \ref{fig:Fusion_rule} to the parallel arcs.}
			\label{fig:d_r_action_BM}
		\end{figure}

		Let now $\overline{a}$ be zero. Similarly, the relative chain class is represented by the arc cycle parallel to $i_+(\alpha)$; see Fig.~\ref{fig:second_insertionM}. After switching the orientation of the arc labelled by $(p-1)$, the fusion coefficient contains $q^{a(p-1)}\sqbinom{a+p-1}{a}_q=\mathfrak f_{a+p-1}(\mathfrak f_a\mathfrak f_{p-1})^{-1}$, which vanishes for $1\leq a<p$ by the same argument. Therefore any small twisted cycle with $a+\overline{a}>0$ is sent to zero by $\Psi_M$.

		For $a=\overline a=0$, the latter relative deformation gives an outgoing arc cycle with labels $(\mathbf a,\mathbf b,p-1,0)$. All of these labels are below $p$, so this image is small as well. Hence $\Psi_M$ preserves small cycles and annihilates the kernel of the restricted filling map. Surjectivity now gives the unique homomorphism in \eqref{eq:homological_index_1_span}.
		
		\begin{figure}[h]
			\centering
			\includegraphics[width=0.75\linewidth]{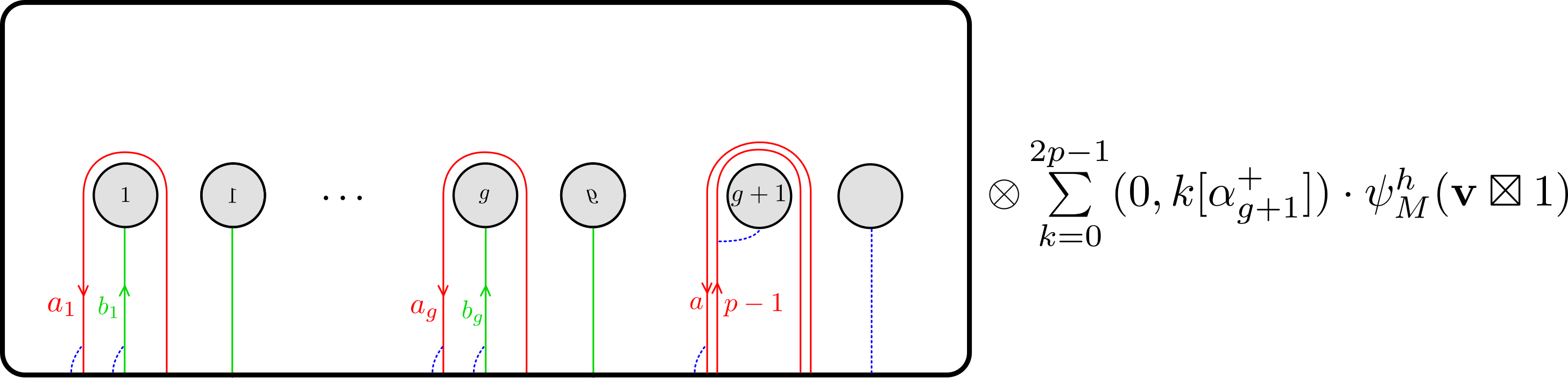}
			\caption{Twisted cycle obtained after sliding $\Delta$ along the attached handle and representing its relative class by an arc cycle parallel to $\alpha_{g+1}^+$. The coefficient is the full $p'$-term sum coming from $\eta^A(1)$.}
			\label{fig:second_insertionM}
		\end{figure}

	\end{proof}

	Putting $a=0$ in Fig.~\ref{fig:second_insertionM} and using the braid rule (see Appendix~\ref{sec:Diag_calculus}) we obtain an explicit formula for the action of $\bbF_{\calL}(M)$ on small cycles for the choice of symplectic arcs as above, with $0\leq a_i,b_i<p$:
	
	\begin{equation}\label{eq:index_1_preaction}
		\boxed{\bbF_{\calL}(M):\mathbf\Gamma(\mathbf{a},\mathbf{b})\otimes \mathbf{v}\mapsto \mathbf{\Gamma}(\mathbf{a},\mathbf{b},p-1,0)\otimes \bigl(\tfrac{(p-2)(p-1)}{2},(p-1)[\beta_{g+1}^+]\bigr)\sum\limits_{k=0}^{p'-1}(0,k[\alpha_{g+1}^+])\psi_M^h(\mathbf{v}\boxtimes 1).}
	\end{equation}

	Here $p'=2p$ for odd $p$ and $p'=2$ for $p=2$. For odd $p$, replacing the full $2p$-term sum by a $p$-term sum would omit a factor of $2$, since $(0,p[\alpha_{g+1}^+])=1$. The map is already defined on small cycles by Lemma~\ref{lem:i_inversion}, and has degree $p-1$.
	
	\subsubsection{Index 2 cobordism action}\label{sec:index_2_action_def}
	
	Let $M':\Sigma_-'\to \Sigma_+'$ be an elementary cobordism of index 2 with attaching sphere $\nu:S^1\to \Sigma_-'$. Consider the dual cobordism $\overline{M}':\overline{\Sigma}'_+\to \overline{\Sigma}'_-$, which is of index 1. Assuming that the action of each $1-q^{2k}$, $1\leq k<p$ on $\calL$ is invertible, and $q^{2p}$ acts trivially, the homomorphism
	\[
		\bbF_{\calL}(\overline{M}'):\bbF_{\calL}(\overline{\Sigma}'_+)\to \bbF_{\calL}(\overline{\Sigma}'_-)
	\]
	is defined on small cycles by Lemma~\ref{lem:i_inversion}.
	
	Since $\calL(\overline{\Sigma})$ is identified with $\overline{\calL(\Sigma)}$ for each $\Sigma$, the intersection pairing on small cycles (see Proposition~\ref{prop:pairing_small_cycles}) can be viewed as a perfect pairing on state spaces:
	\[
		\langle-,-\rangle:\bbF_{\calL}(\Sigma)\otimes \bbF_{\calL}(\overline{\Sigma})\to R.
	\]
	Hence the dual to $\bbF_{\calL}(\overline{M}')$ with respect to this pairing exists uniquely: dualize the linear map and use the isomorphisms with the ordinary duals supplied by the perfect state-space pairings. We define the action of $M'$ as the dual. Specifically, let $\bbF_{\calL}(M')$ be the homomorphism
	\[
		\bbF_{\calL}(M'):\bbF_{\calL}(\Sigma_-')\to \bbF_{\calL}(\Sigma_+')
	\]
	determined by
	\[
	\langle\bbF_{\calL}(M') \Gamma_-,\Gamma_+\rangle=\langle\Gamma_-,\bbF_{\calL}(\overline{M}')\Gamma_+\rangle, \quad \forall \Gamma_-\in \bbF_{\calL}(\Sigma_-'), \ \forall \Gamma_+\in \bbF_{\calL}(\overline{\Sigma}_+').
	\]
	
	For explicit computations it is useful to derive $\bbF_{\calL}(\overline{M}')$ using the dual arc diagrams for small cycles. For this, we take $M'=\overline{M}$, where $M$ is the standard index 1 cobordism as in Fig.~\ref{fig:Sigma_W_arcs}. In the following computation the arc diagrams are read as Borel--Moore representatives, with all labels below $p$ and the boundary identification of Section~\ref{sec:pairing}. Compact lifts for intersection computations are obtained by the inverse fusion factors of Section~\ref{sec:small_cycles}. The insertion is computed in Fig.~\ref{fig:Psi_M_dual}, resulting in the formula:
	\begin{equation}\label{eq:index_2_preaction}
		\boxed{\bbF_{\calL}(M):\mathbf\Gamma^{\vee}(\mathbf{a},\mathbf{b})\otimes \mathbf{v}\mapsto \mathbf{\Gamma}^{\vee}(\mathbf{a},\mathbf{b},0,p-1)\otimes \bigl(\tfrac{(p-2)(p-1)}{2},(p-1)[\beta_{g+1}^+]\bigr)\sum\limits_{k=0}^{p'-1}(0,k[\alpha_{g+1}^+])\psi_M^h(\mathbf{v}\boxtimes 1).}
	\end{equation}
	
	\begin{figure}[h]
		\centering
		\includegraphics[width=1\linewidth]{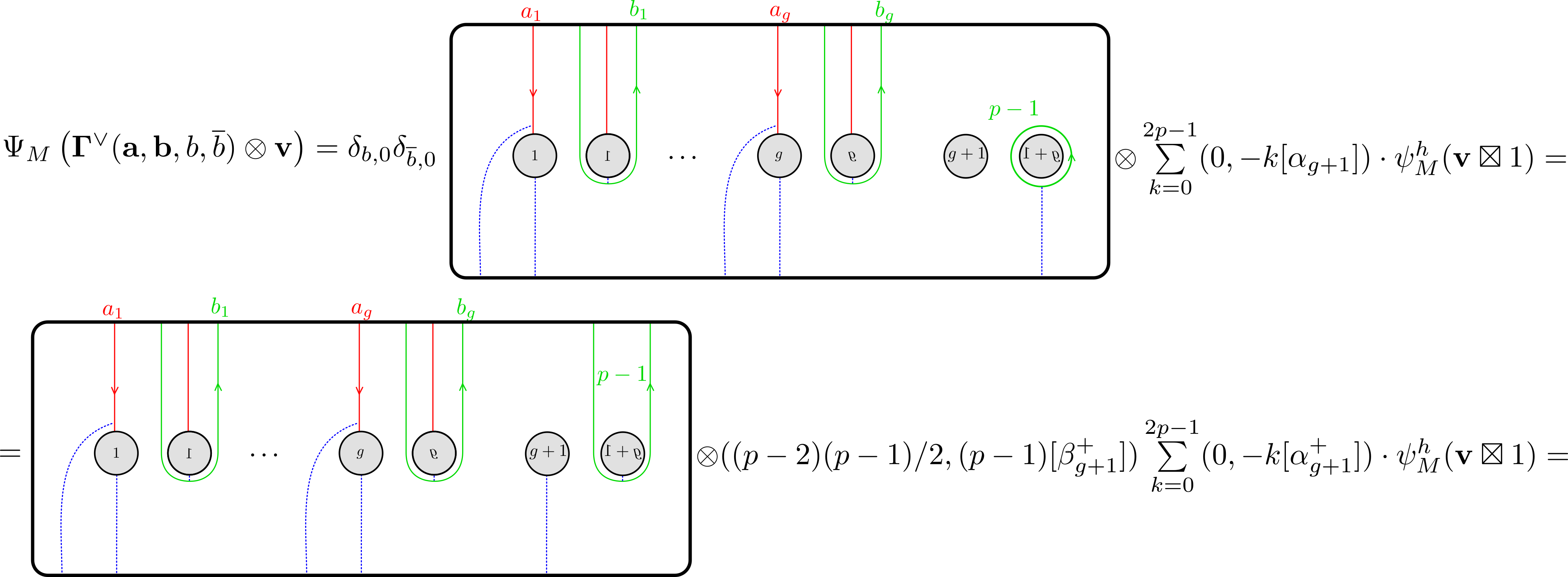}
		\caption{The insertion computation with dual arc representatives of small cycles, shown for odd $p$. All labels are below $p$. Similarly to the computations above, the images of the small cycles with $b+\overline{b}>0$ vanish. In the case $b=\overline{b}=0$, we represent the inserted class by an arc cycle relative to $\partial_+\Sigma_+$ and apply the braid rule from Fig.~\ref{fig:Braid_rule}. For $p=2$, the coefficient sums have two terms.}
		\label{fig:Psi_M_dual}
	\end{figure}

	\subsection{Functoriality}\label{subs:Functorialty}
	
	Denote by $\mathrm{Mod}_R^{\pm}$ the category of $R$-modules and $R$-linear maps up to sign.
	
	\begin{maintheorem}\label{thm:main_thm}
	Let $R$ be a unital commutative ring with involution. Let $\calL=\{\calL(\Sigma)\}_{\Sigma\in 3\Cob}$ be a collection of $p$-Heisenberg local systems of $R$-modules, such that each $\calL(\Sigma)$ is equipped with a perfect sesquilinear pairing, the action of all $1-(-\sigma)^{k}$, $1\leq k<p$, on $\calL(\Sigma)$ is invertible, $(-\sigma)^p$ acts as the identity, and $\overline{\calL(\Sigma)}$ is identified with $\calL(\overline{\Sigma})$. Suppose $\calL$ is Morse compatible. Then the assignments $\Sigma \mapsto\bbF_{\calL}(\Sigma)$ for all objects $\Sigma\in 3\Cob$ and $M\mapsto\bbF_{\calL}(M)$ for all elementary cobordisms $M\in \mathrm{Mor}(3\Cob)$ give rise to a monoidal functor
	\[
	\bbF_{\calL}:3\Cob\to \mathrm{Mod}_R
	\]
	for odd $p\geq 3$, and to a monoidal functor
	\[
	\bbF_{\calL}:3\Cob\to \mathrm{Mod}_R^{\pm}
	\]
	for $p=2$.
	\end{maintheorem}
	
	\begin{proof}

	We first check that $\bbF_{\calL}$ is monoidal on the level of $\mathrm{Surf}$. Equalities of state-space maps below are taken in $\mathrm{Mod}_R$ for odd $p$, and in $\mathrm{Mod}_R^\pm$ for $p=2$. The geometric homology diagrams use ordinary cross products; we account explicitly for the sign in $\mu^{\bbF}_{\calL}$ when passing to the state-space locality identities.
	Recall that if a surface $\Sigma$ has a decomposition $\Sigma=\Sigma_1\natural \Sigma_2$ then by \ref{sec:monoidality} there is a natural isomorphism 
	 \[
		\mu^{\bbF}_{\calL}:\mathring{\bbH}(\Sigma_1;\calL(\Sigma_1))\otimes\mathring{\bbH}(\Sigma_2;\calL(\Sigma_2))\to \mathring{\bbH}(\Sigma;\calL(\Sigma)).
	\]
	 Suppose $d_1:\Sigma_1\to \Sigma_1'$ and $d_2:\Sigma_2\to\Sigma_2'$ are two mapping classes, $\Sigma:=\Sigma_1\natural \Sigma_2$, $\Sigma'=\Sigma_1'\natural\Sigma_2'$ and $d:\Sigma\to \Sigma'$ is the induced mapping class. It is exactly the situation described in diagram \eqref{eq:monoidal_topological} for the \emph{Monoidal structure} condition on Morse compatible local systems. Denote $\calL_1=\calL(\Sigma_1)$, $\calL_2=\calL(\Sigma_2)$, $\calL_1'=\calL(\Sigma'_1)$, $\calL_2'=\calL(\Sigma_2')$, $\calL=\calL(\Sigma)$ and $\calL'=\calL(\Sigma')$. Consider the diagram
		
	\[
		\begin{tikzcd}[column sep=large]
			\mathring{\bbH}(\Sigma_1;\calL_1)\otimes \mathring{\bbH}(\Sigma_2;\calL_2)\arrow[r,"\phi_{d_1*}\otimes \phi_{d_2*}"]\arrow[d]&\mathring{\bbH}(\Sigma_1,d_1^*\calL_1')\otimes \mathring{\bbH}(\Sigma_2,d_2^*\calL_2')\arrow[r,"(d_1)_*\otimes(d_2)_*"]\arrow[d]&\mathring{\bbH}(\Sigma_1',\calL_1')\otimes\mathring{\bbH}(\Sigma_2';\calL_2')\arrow[d]\\
			\mathring{\bbH}(\Sigma_1\sqcup \Sigma_2;\calL_1\boxtimes\calL_2)\arrow[r,"(\phi_{d_1}\boxtimes\phi_{d_2})_*"]\arrow[d,"(\mu_{\calL})_*"]&\mathring{\bbH}(\Sigma_1\sqcup\Sigma_2;d_1^*\calL_1'\boxtimes d_2^*\calL_2')\arrow[r,"(d_1\sqcup d_2)_*"]\arrow[d,"(d_1\sqcup d_2)^*(\mu_{\calL})_*"]&\mathring{\bbH}(\Sigma_1'\sqcup\Sigma_2';\calL_1'\boxtimes \calL_2')\arrow[d,"(\mu_{\calL})_*"]\\
			\mathring{\bbH}(\Sigma_1\sqcup \Sigma_2;\mu^*\calL)\arrow[r,"\mu^*(\phi_d)_*"]\arrow[d,"\mu_*"]&\mathring{\bbH}(\Sigma_1\sqcup \Sigma_2,\mu^*d^*\calL')\arrow[r,"(d_1\sqcup d_2)_*"]\arrow[d,"d^*\mu_*"]&\mathring{\bbH}(\Sigma_1'\sqcup\Sigma'_2;\mu^*\calL')\arrow[d,"\mu_*"]\\
			\mathring{\bbH}(\Sigma;\calL)\arrow[r,"(\phi_d)_*"]&\mathring{\bbH}(\Sigma;d^*\calL')\arrow[r,"d_*"]&\mathring{\bbH}(\Sigma';\calL')
		\end{tikzcd}
	\]		
	Commutativity of two squares in the first row follows from naturality of twisted K\"unneth homomorphism. The left square in the second row is commutative since it is obtained from \eqref{eq:Monoidal_structure} by applying functor $\mathring{\bbH}(\Sigma_1\sqcup \Sigma_2;-)$. All other squares are commutative by functoriality of $\mathring{\bbH}$ in both arguments. Diffeomorphisms preserve the genera and particle degrees, so the sign in $\mu^{\bbF}_{\calL}$ agrees on both sides. Together with the associativity and unit constraints checked in Section~\ref{sec:monoidality}, this gives the required monoidal structure on $\mathrm{Surf}$.
	
	The monoidality condition \eqref{eq:monoidality_condition} can be checked in a similar way. Suppose $\Sigma_-=\Sigma_-'\natural\Sigma_-''$ and $\bbS\subset \Sigma_-''$ is an attaching tube for an index 1 cobordism, then we are exactly in the situation of the \emph{Locality} condition \eqref{eq:Locality_topologiacl}. Since the K\"unneth homomorphism is natural with respect to both arguments of homology, the square
	\[
		\begin{tikzcd}
			\bbH^{BM}(\calT_{M'};i_-^*\calL(\Sigma_-'))\otimes \bbH^{BM}(\calT_{M''};i_-^*\calL(\Sigma_-''))\arrow[r]\arrow[d,"\Id\otimes\Psi_{M''}"]&\bbH^{BM}(\calT_{M'}\sqcup \calT_{M''};i_-^*\calL(\Sigma_-')\boxtimes i_-^*\calL(\Sigma_-''))\arrow[d,"\Psi_{M''}"]\\
			\bbH^{BM}(\Sigma_+';\calL(\Sigma_+'))\otimes \bbH^{BM}(\Sigma_+'';\calL(\Sigma_+''))\arrow[r]&\bbH^{BM}(\Sigma_+'\sqcup\Sigma_+'';\calL(\Sigma_+')\boxtimes\calL(\Sigma_+''))
		\end{tikzcd}
	\]
	is commutative, where $\Psi_{M''}$ on the right is the same as in $\eqref{eq:Psi_M_def}$, but with $\calT_{M'}\sqcup\calT_{M''}$ instead of $\calT_M$. On the other hand, the diagram
	\[
		\begin{tikzcd}
			\bbH^{BM}(\calT_{M'}\sqcup \calT_{M''};\mu^*i_-^*\calL(\Sigma_-))\arrow[r,"\mu_*"]\arrow[d,"\mu^*\Psi_M"]&\bbH^{BM}(\calT_M;i_-^*\calL(\Sigma_-))\arrow[d,"\Psi_M"]\\
			\bbH^{BM}(\Sigma_+'\sqcup \Sigma_+'';\mu^*\calL(\Sigma_+))\arrow[r,"\mu_*"]&\bbH^{BM}(\Sigma_+;\calL(\Sigma_+))
		\end{tikzcd}
	\]
	is commutative, where $\mu^*\Psi_M$ is obtained from $\Psi_M$ by pulling back all local system morphisms in \eqref{eq:Psi_M_def} along $\mu$ and replacing $\varphi$ by $\Id \times \varphi$. Therefore it remains to check that $\mu_{\calL}$ commutes with $\Psi_{M''}$. We verify this step by step, following the composition of homomorphisms in \eqref{eq:Psi_M_def}. The first square:
	\[
		\begin{tikzcd}
			\bbH^{BM}(\calT_{M'}\sqcup\calT_{M''};i_-^*(\calL(\Sigma_-')\boxtimes \calL(\Sigma_-'')))\arrow[d,"{(i_-^*\mu_{\calL})_*}"]\arrow[r,"{\Id\otimes[\widetilde{\Conf}_{p-1}(S^1)]} "{yshift=5pt}]&\bbH^{BM}(\calT_{M'}\sqcup\calT_{M''};i_-^*(\calL(\Sigma_-')\boxtimes \calL(\Sigma_-'')))\otimes \bbH^{BM}_{p-1}(S^1;\calA)\arrow[d,"{(i_-^*\mu_{\calL})_*\otimes \Id}"]\\
			\bbH^{BM}(\calT_{M'}\sqcup\calT_{M''};\mu^{*}i_-^*\calL(\Sigma_-))\arrow[r,"{\Id\otimes [\widetilde{\Conf}_{p-1}(S^1)]}"{yshift=5pt}] &\bbH^{BM}(\calT_{M'}\sqcup\calT_{M''};\mu^{*}i_-^*\calL(\Sigma_-))\otimes \bbH^{BM}_{p-1}(S^1;\calA)\\
		\end{tikzcd}
	\]
	clearly commutes, where we use notation 
	\[
		\bbH^{BM}_{p-1}(S^1;\calA):=H^{BM}_{p-1}(\Conf_{p-1}(S^1);\calA).
	\]
	The second square:
	
	\[
	\adjustbox{max width=\textwidth}{$
		\begin{tikzcd}[column sep=tiny]
			\bbH^{BM}(\calT_{M'}\sqcup\calT_{M''};i_-^*(\calL(\Sigma_-')\boxtimes \calL(\Sigma_-'')))\otimes \bbH^{BM}_{p-1}(S^1;\calA) \arrow[r]\arrow[d,"{(i_-^*\mu_{\calL})\otimes \Id}"]&\bbH^{BM}(\calT_{M'}\sqcup\calT_{M''}\sqcup S^1; i_-^*\calL(\Sigma_-')\boxtimes i_-^*\calL(\Sigma_-'')\boxtimes\calA)\arrow[d, "{(\mu_{\calL}\boxtimes \Id)}"]\\
			\bbH^{BM}(\calT_{M'}\sqcup\calT_{M''};\mu^{*}i_-^*\calL(\Sigma_-))\otimes\bbH^{BM}_{p-1}(S^1;\calA)\arrow[r]& \bbH^{BM}(\calT_{M'}\sqcup\calT_{M''}\sqcup S^1;\mu^*i_-^*\calL(\Sigma_-)\boxtimes \calA)
		\end{tikzcd}$
	}
	\]
	commutes by naturality of K\"unneth homomorphism. Commutativity of the next square
	\[
		\begin{tikzcd}[column sep=large]
			\bbH^{BM}(\calT_{M'}\sqcup\calT_{M''}\sqcup S^1; i_-^*\calL(\Sigma_-')\boxtimes i_-^*\calL(\Sigma_-'')\boxtimes\calA)\arrow[d,"{(i_-^*\mu_{\calL}\boxtimes\Id)_*}"]\arrow[r,"{(\Id\boxtimes \psi_{M''})_*}"]&\bbH^{BM}(\calT_{M'}\sqcup\calT_{M''}\sqcup S^1;\calL(\Sigma_+')\boxtimes \varphi^*\calL(\Sigma_+''))\arrow[d,"{\varphi^*\mu_{\calL}}"]\\
			\bbH^{BM}(\calT_{M'}\sqcup\calT_{M''}\sqcup S^1;\mu^*i_-^*\calL(\Sigma_-)\boxtimes \calA)\arrow[r,"{(\mu^*\psi_M)_*}"]&\bbH^{BM}(\calT_{M'}\sqcup\calT_{M''}\sqcup S^1;\mu^*\varphi^*\calL(\Sigma_+))
		\end{tikzcd}
	\]
	follows from \emph{Locality condition} \eqref{eq:locality} by applying the functor $\bbH^{BM}(\calT_{M'}\sqcup \calT_{M''}\sqcup S^1;-)$. Finally, the last square is simply commutative by functoriality of Borel--Moore homology in the first argument:
	\[
		\begin{tikzcd}
			\bbH^{BM}(\calT_{M'}\sqcup\calT_{M''}\sqcup S^1;\calL(\Sigma_+')\boxtimes \varphi^*\calL(\Sigma_+''))\arrow[r,"{(\Id\times \varphi)_*}"]\arrow[d,"{\varphi^*\mu_{\calL}}"]&\bbH^{BM}(\Sigma_+'\sqcup\Sigma_+'';\calL(\Sigma_+')\boxtimes \calL(\Sigma_+''))\arrow[d,"\mu_{\calL}"]\\
			\bbH^{BM}(\calT_{M'}\sqcup\calT_{M''}\sqcup S^1;\mu^*\varphi^*\calL(\Sigma_+))\arrow[r,"{(\Id\times \varphi)_*}"]&\bbH^{BM}(\Sigma_+'\sqcup \Sigma_+'';\mu^*\calL(\Sigma_+))
		\end{tikzcd}
	\]
	
		Combining everything gives right locality for the unsigned cross products. The insertion raises the second particle degree by $p-1$, so including the tensorator signs changes the two sides by the fixed ratio $(-1)^{(p-1)^2g'}$, where $g'$ is the genus of the first, unchanged surface. This is $1$ for odd $p$ and an overall sign for $p=2$. Thus, in the respective target category, we obtain
	\[
		\Psi_M\circ\mu^{\bbF}_{\calL}=\mu_{\calL}^{\bbF}\circ (\Id\otimes \Psi_{M''}).
	\]
	By functoriality of $\bbH^{BM}$ in the first argument we also have:
	\[
		(i_-)_*\circ \mu_{\calL}^{\bbF}=\mu^{\bbF}_{\calL}\circ (\Id\otimes (i_-)_*).
	\]
	We now restrict the trajectory inputs to small cycles. The cross products, coefficient maps and gluing maps preserve them by Section~\ref{sec:small_cycles}. In these identities, $\mu^{\bbF}_{\calL}$ on a trajectory space denotes the same composition with the same tensorator sign, using the genus of the first incoming surface. By Lemma~\ref{lem:i_inversion}, the two composites in the following square agree after precomposition with $\Id\otimes(i_-)_*$. This map is surjective, since $(i_-)_*$ is surjective on small cycles and tensor products preserve surjections. For $p=2$ the sign is uniform in all particle degrees, so the same cancellation of a surjection applies to the chosen representatives. Thus the square commutes:
	\begin{equation}\label{eq:proof_mon_1}
		\begin{tikzcd}
			\bbF_{\calL}(\Sigma_-')\otimes \bbF_{\calL}(\Sigma_-'')\arrow[r,"\mu^{\bbF}_{\calL}"]\arrow[d,"\Id\otimes \bbF_{\calL}(M'')"]&\bbF_{\calL}(\Sigma_-)\arrow[d,"\bbF_{\calL}(M)"]\\
			\bbF_{\calL}(\Sigma_+')\otimes \bbF_{\calL}(\Sigma_+'')\arrow[r,"\mu^{\bbF}_{\calL}"]&\bbF_{\calL}(\Sigma_+)
		\end{tikzcd}
	\end{equation}
	Thus we proved commutativity of the right diagram in \eqref{eq:monoidality_condition}. For the left diagram, use the left locality condition \eqref{eq:locality_left} and repeat the same argument. The permutation $r$ moves the inserted class of degree $p-1$ past the unchanged second factor of degree $m$, contributing $(-1)^{(p-1)m}$. The genus of the first factor increases by one, so the incoming and outgoing tensorator signs differ by the same factor. These signs cancel. Restricting to small cycles and using surjectivity of $(i_-)_*\otimes\Id$ therefore gives the commutative diagram:
	\begin{equation}\label{eq:proof_mon_2}
	\begin{tikzcd}
		\bbF_{\calL}(\Sigma_-'')\otimes \bbF_{\calL}(\Sigma_-')\arrow[r,"\mu^{\bbF}_{\calL}"]\arrow[d," \bbF_{\calL}(M'')\otimes \Id"]&\bbF_{\calL}(\Sigma_-''\natural \Sigma_-')\arrow[d,"\bbF_{\calL}(M)"]\\
		\bbF_{\calL}(\Sigma_+'')\otimes \bbF_{\calL}(\Sigma_+')\arrow[r,"\mu^{\bbF}_{\calL}"]&\bbF_{\calL}(\Sigma_+''\natural\Sigma_+')
	\end{tikzcd}
	\end{equation}
	This is the left diagram in \eqref{eq:monoidality_condition} for index $1$ cobordisms.

	For index $2$ cobordisms, we use directly the intersection pairing between small cycles relative to $\partial_-$ and to $\partial_+$. Denote the latter spaces by $\mathring{\bbH}^{+}$, on both surfaces and trajectory spaces.  The coefficient locality conditions \eqref{eq:locality} and \eqref{eq:locality_left} are independent of the relative boundary condition. Together with the plus-relative version of Lemma~\ref{lem:i_inversion}, the same argument as above therefore proves both locality identities for the index $1$ action on $\mathring{\bbH}^{+}$. 
	On surfaces, the natural intersection pairing
	\[
	\langle-,-\rangle:
	\mathring{\bbH}(\Sigma;\calL)
	\otimes_R
	\mathring{\bbH}^{+}(\Sigma;\overline{\calL})
	\longrightarrow R
	\]
	is perfect. We isotope the embeddings of the two surfaces into their boundary connected sum so that the image of each $\partial_+$-arc lies in the $\partial_+$-arc of the sum. For these embeddings, intersections are computed separately in the two summands, with the usual cross-product signs and the prescribed tensorator signs. Taking adjoints of the $\partial_+$-relative locality identities therefore gives the required locality identities for index $2$. The sign comparison is trivial for odd $p$. For $p=2$ it leaves at most an overall sign depending on the genus of the unchanged summand, independent of particle degree. Thus \eqref{eq:monoidality_condition} holds in the stated target category.

	\vspace{0.5cm}
	
	We now check Juhász’s relations from Theorem~\ref{thm:Juhasz_relations}. For each relation we use the notation from it.
		
	\begin{enumerate}
		\item By the property \ref{item:mcg_rep} the first relation is satisfied. Let $\Sigma\xrightarrow{d}\Sigma'\xrightarrow{d'}\Sigma''$ be a composition of two diffeomorphisms. Consider the diagram
		\[
			\begin{tikzcd}[column sep=tiny]
				\mathring{\bbH}(\Sigma;\calL(\Sigma))\arrow[rr,"{\phi_{d'\circ d}}"]\arrow[dr,"{\phi_{d}}"]&&\mathring{\bbH}(\Sigma,(d'\circ d)^*\calL(\Sigma''))\arrow[rr,"{(d'\circ d)_*}"]\arrow[dr,"{d_*}"]&&\mathring{\bbH}(\Sigma'',\calL(\Sigma''))\\
				&\mathring{\bbH}(\Sigma,d^*\calL(\Sigma'))\arrow[ur, "{d^*\phi_{d'}}"]\arrow[dr,"d_*"]&&\mathring{\bbH}(\Sigma',(d')^*\calL(\Sigma''))\arrow[ur,"{d'_*}"]&\\
				&&\mathring{\bbH}(\Sigma',\calL(\Sigma'))\arrow[ur,"{\phi_{d'}}"]&&
			\end{tikzcd}
		\]
		Here the middle square and the right triangle are commutative by functoriality of $\mathring{\bbH}$ in both arguments. The left triangle is commutative by property \ref{item:mcg_rep}. Therefore, $\bbF_{\calL}(M_{d'\circ d})=\bbF_{\calL}(M_{d'})\circ\bbF_{\calL}(M_d)$.
		Clearly $\bbF_{\calL}(M_{\mathrm{Id}})=\mathrm{Id}$. If $d$ is isotopic to the identity relative to the boundary, it identifies the pull-back local systems, and homotopy invariance gives the same map on small cycles. Thus $\bbF_{\calL}(M_{\Id})=\Id$.
			
		\item The second relation follows from property \ref{item:invariance}. Consider an elementary index 1 cobordism $M$. Let $d_-:\Sigma_-\to \Sigma_-'$ be a diffeomorphism sending $\bbS_-$ to $\bbS_-'$, then it extends to a diffeomorphism $d:M\to M'$ where $M=M_{\bbS_-}$, $M'=M_{\bbS_-'}$ and induces diffeomorphisms $d_{\calT}:\calT_M\to \calT_{M'}$ and $d_+:\Sigma_+\to \Sigma_+'$, $\Sigma_+=\Sigma_-(\bbS_-)$ and $\Sigma_+'=\Sigma_-'(\bbS_-')$. They satisfy $i_-\circ d_{\calT}=d_-\circ i_-$ and $i_+\circ d_{\calT}=d_+\circ i_+$ up to homotopy (see diagram \eqref{eq:invariance_diag}). Consider the diagram
		\begin{equation}\label{eq:monoidality}
			\adjustbox{max width=0.96\textwidth}{$
			\begin{tikzcd}
				\bbH^{BM}(\calT_M;i_-^*\calL(\Sigma_-))\arrow[r,"(i_-^*\phi_{d_-})_*"]\arrow[d,"\mathrm{Id}\otimes \eta^A"]&\bbH^{BM}(\calT_M;i_-^*d_-^*\calL(\Sigma_-'))\arrow[r,"(d_{\calT})_*"]\arrow[d,"\mathrm{Id}\otimes \eta^A"]&\bbH^{BM}(\calT_{M'};i_-^*\calL(\Sigma_-'))\arrow[d,"\mathrm{Id}\otimes \eta^A"]\\
				\bbH^{BM}(\calT_M\sqcup S^1;i_-^*\calL(\Sigma_-)\boxtimes \calA)\arrow[r,"{(i_-^*\phi_{d_-}\boxtimes \mathrm{Id})_*}"{yshift=4pt}]\arrow[d,"\psi_M"]&\bbH^{BM}(\calT_M\sqcup S^1;i_-^*d_-^*\calL(\Sigma_-')\boxtimes \calA)\arrow[r,"(d_{\calT}\times\mathrm{Id})_*"{yshift=5pt}]\arrow[d,"d_{\calT}^*\psi_{M'}"]&\bbH^{BM}(\calT_{M'}\sqcup S^1;i_-^*\calL(\Sigma_-')\boxtimes \calA)\arrow[d,"\psi_{M'}"]\\
				\bbH^{BM}(\calT_M\sqcup S^1;\varphi^*\calL(\Sigma_+))\arrow[r,"{(\varphi^*\phi_{d_+})_*}"]\arrow[d,"\varphi"]&\bbH^{BM}(\calT_M\sqcup S^1;\varphi^*d_+^*\calL(\Sigma_+'))\arrow[r,"(d_{\calT}\times \mathrm{Id})_*"]\arrow[d,"\varphi"]&\bbH^{BM}(\calT_{M'}\sqcup S^1;\varphi^*\calL(\Sigma_+'))\arrow[d,"\varphi"]\\
				\bbH^{BM}(\Sigma_+;\calL(\Sigma_+))\arrow[r,"(\phi_{d_+})_*"]&\bbH^{BM}(\Sigma_+;d_+^{*}\calL(\Sigma_+'))\arrow[r,"(d_{+})_*"]&\bbH^{BM}(\Sigma_+';\calL(\Sigma_+'))
			\end{tikzcd}
			$}
		\end{equation}
		Commutativity of two squares in the first row follows from naturality of twisted K\"unneth homomorphism. The left square in the second row is commutative since it is obtained from \eqref{eq:invariance} by applying functor $\mathrm{\bbH}^{BM}(\calT_{M}\sqcup S^1;-)$. All other squares are commutative by functoriality and proper homotopy invariance of $\bbH^{BM}$. 
		
		 On the other hand, the diagram
		\[
			\begin{tikzcd}
				\bbH^{BM}(\calT_M;i_-^*\calL(\Sigma_-))\arrow[r,"(i_-)_*"]\arrow[d,"i_-^*\phi_{d_-}"]&\bbH^{BM}(\Sigma_-;\calL(\Sigma_-))\arrow[d,"\phi_{d_-}"]\\
				\bbH^{BM}(\calT_M;i_-^*d_-^* \calL(\Sigma_-'))\arrow[r,"(i_-)_*"]&\bbH^{BM}(\Sigma_-;d_-^*\calL(\Sigma_-'))
			\end{tikzcd}
		\]
		is clearly commutative. Diffeomorphisms and coefficient maps preserve small cycles. Restricting the trajectory inputs to these subspaces and using Lemma~\ref{lem:i_inversion}, the two composites below agree after precomposition with $(i_-)_*$. Since this filling map is surjective on small cycles, we obtain:
		\[
			\bbF_{\calL}(M')\circ \bbF_{\calL}(M_{d_-})=\bbF_{\calL}(M_{d_+})\circ\bbF_{\calL}(M).
		\]
		Naturality of $\chi$ in \eqref{eq:conjugation_coherence}, preservation of coefficient pairings by $\phi_d$, and invariance of intersections under $d$ give the mixed naturality identity
		\[
			\langle\bbF_{\calL}(d)x,\bbF_{\calL}(\bar d)y\rangle=\langle x,y\rangle.
		\]
		Dualizing the index $1$ covariance identity for the reversed cobordisms with this identity and the perfect pairings proves the same covariance for index $2$.
		
	
		\item	
		Disjoint attaching tubes can be put in separate boundary-sum factors by joining their neighborhoods to the distinguished boundary with disjoint arcs. For index $2$ tubes, extend them to disjoint punctured-torus neighborhoods. This covers all three index pairs $(1,1)$, $(1,2)$ and $(2,2)$. Using covariance to make this decomposition, both versions of locality \eqref{eq:proof_mon_1}, \eqref{eq:proof_mon_2} reduce the relation to
		\[
			(\bbF_{\calL}(M_{\bbS_1}')\otimes \mathrm{Id})\circ(\mathrm{Id}\otimes\bbF_{\calL}(M_{\bbS_2}''))= (\mathrm{Id}\otimes\bbF_{\calL}(M_{\bbS_2}''))\circ(\bbF_{\calL}(M_{\bbS_1}')\otimes \mathrm{Id}).
		\]
		Here identity maps have the appropriate incoming or outgoing factors as their domains.

		\item Let $M_-:\Sigma_-\to\Sigma$ have index $1$ and $M_+:\Sigma\to\Sigma_+$ have index $2$, where $\Sigma_-$ and $\Sigma_+$ are of genus $g$. Suppose that the belt circle $\alpha$ of $M_-$ and the attaching circle $\beta$ of $M_+$ intersect transversely once. Let $f:\Sigma_-\to\Sigma_+$ be the cancellation diffeomorphism from Theorem~\ref{thm:Juhasz_relations}. We first isolate both handles in the same genus-one summand.

		A regular neighborhood of $\alpha\cup\beta$ is a once-punctured torus. Join it to the distinguished boundary by an embedded arc and thicken the union. This gives a boundary connected sum $\Sigma\simeq \Sigma_{g}\natural \Sigma_1$, with $\Sigma_1$ of genus one and both circles contained in $\Sigma_1$. Surgery on $\Sigma_1$ along either circle gives a disk, denoted $D_-$ or $D_+$. With compatible boundary parametrizations, the two elementary cobordisms therefore become simultaneously
		\[
			M_-=\Id_{\Sigma_g}\natural M_1:\Sigma_g\natural D_-\to \Sigma_g\natural \Sigma_1,
			\qquad
			M_+=\Id_{\Sigma_g}\natural M_2:\Sigma_g\natural \Sigma_1\to \Sigma_g\natural D_+.
		\]
		The cancellation diffeomorphism is the identity on $\Sigma_g$ and a diffeomorphism $f_D:D_-\to D_+$ on the disk. Applying monoidality properties \eqref{eq:monoidality_condition} already proved above we obtain that it is enough to restrict to cobordisms $M_1$ and $M_2$.
		Hence we compute the case $g=0$, $\Sigma=\Sigma_1$, and $\Sigma_\pm=D_\pm$. Via the units, $\bbF_{\calL}(D_\pm)=R$, and $\bbF_{\calL}(f_D)=\Id_R$. In the following computation $1$ denotes the corresponding unit vector, including on the oppositely oriented disk. By the definition of the index $2$ action, it is enough to check
		\[
			\bigl\langle\bbF_{\calL}(M_1)(1),\bbF_{\calL}(\overline M_2)(1)\bigr\rangle=1.
		\]
		A standard picture for the composition $M_2\circ M_1$ is depicted in Fig.~\ref{fig:relation_4}, where a choice of symplectic arcs in $\Sigma$ is made. Note that $M_1$ is exactly the standard index 1 cobordism considered in Fig.~\ref{fig:Sigma_W_arcs}, for $g=0$.  In Fig.~\ref{fig:overline_M_2_action} we compute the action of $\bbF_{\calL}(\overline{M}_2)$ in the dual basis using \eqref{eq:index_2_preaction}. 
		\[
			\bbF_{\calL}(\overline{M}_2)(1)=\mathbf{\Gamma}^{\vee}(p-1,0)\otimes ((p-2)(p-1)/2, (1-p)[\alpha_{g+1}^+])\psi_{\overline{M}_2}^h( 1\boxtimes \eta^A(1)).
		\]
		
		\begin{figure}[h]
			\centering
			\includegraphics[width=0.5\linewidth]{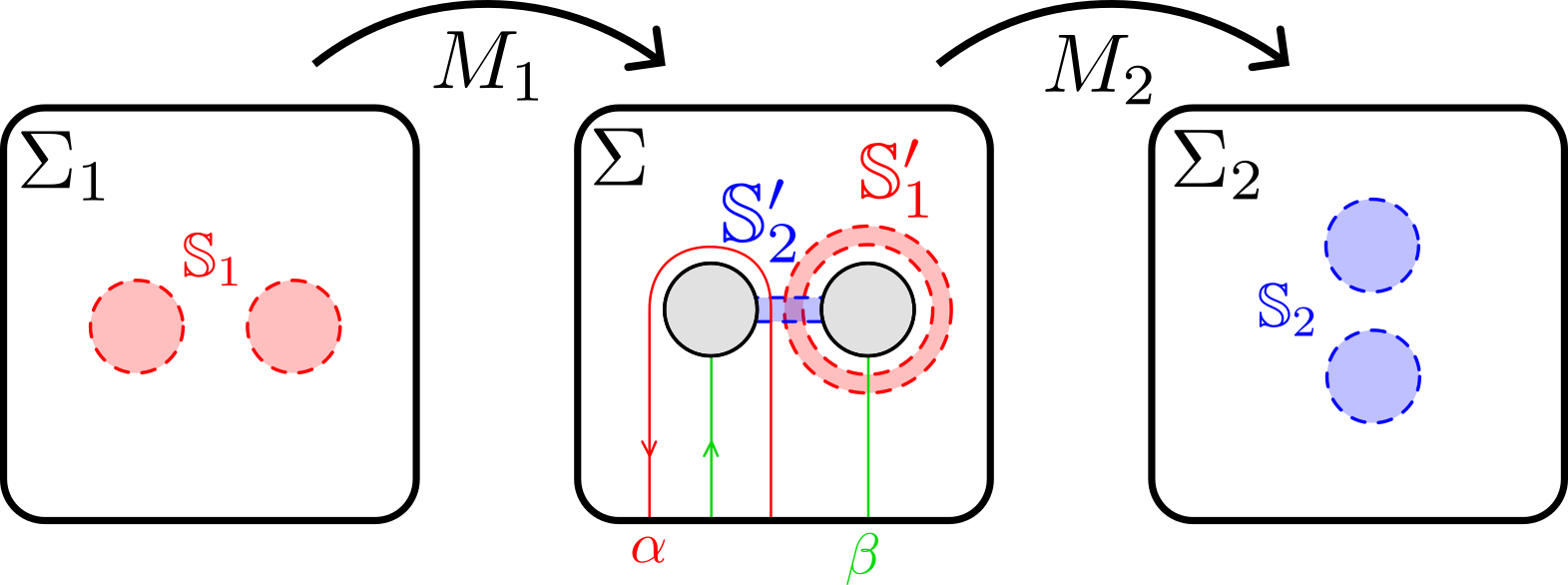}
			\caption{Composition of cancelling cobordisms, $\bbS_1$ and $\bbS_1'$ are attaching and belt tubes of $M_1$, $\bbS_2'$ and $\bbS_2$ are attaching and belt tubes of $M_2$.}
			\label{fig:relation_4}
		\end{figure}
		
			\begin{figure}[h]
			\centering
			\includegraphics[width=0.9\linewidth]{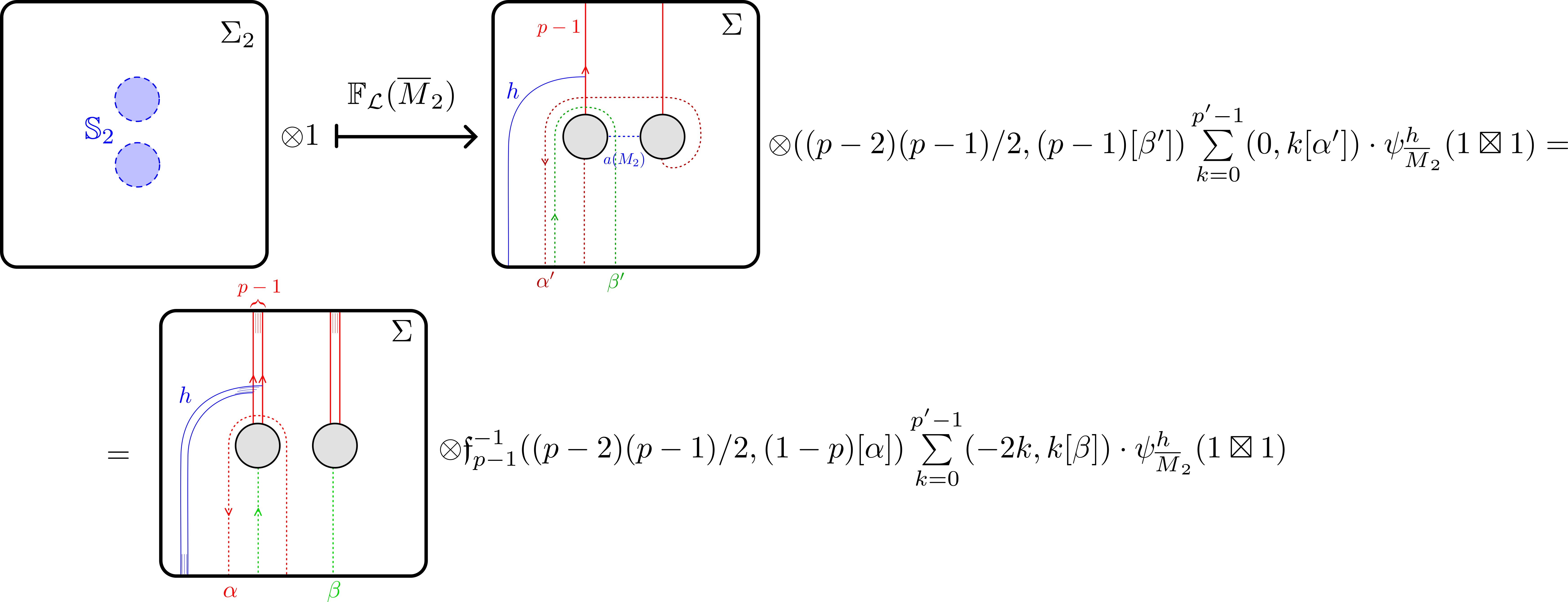}
			\caption{Action of $\bbF_{\calL}(\overline{M}_2)$. The right hand side of the first row is obtained from formula \eqref{eq:index_2_preaction} by a diffeomorphism sending arcs $\alpha_{1}^+$, $\beta_{1}^+$ ($g=0$) into arcs $\alpha'$, $\beta'$ in the picture. As elements of $\pi_1(\Sigma)$ they are related to the generators of Fig.~\ref{fig:relation_3} by $\alpha'=\alpha^{-1}\beta\alpha$ and $\beta'=\alpha^{-1}$. Hence on the level of the Heisenberg group $(0,-[\beta'])=(0,[\alpha])$ and $(0,[\alpha'])=(-2,[\beta])$. The resulting small cycle is then lifted to a cycle in the ordinary twisted homology to compute the intersection pairing. This results in an additional inverse action of $[p-1]_q!q^{(p-2)(p-1)/2}=\mathfrak{f}_{p-1}$.}
			\label{fig:overline_M_2_action}
		\end{figure}

		The representatives of $\bbF_{\calL}(M_1)(1)$ and of the ordinary-homology lift of $\bbF_{\calL}(\overline M_2)(1)$ meet at the single configuration $*_{\pitchfork}$. The permutation from the two tangent-vector blocks to the pointwise pairs contributes $(-1)^{(p-1)(p-2)/2}$, which is canceled by our intersection normalization. The remaining pointwise intersection signs are all the same in this local picture. Their product is $+1$ for even $p-1$, while for $p=2$ it is an overall sign. At this configuration the two coefficient vectors are $v_1$ and $\mathfrak{f} _{p-1}^{-1}\,v_2$ from \eqref{eq:cancellation}: $v_1,v_2$ are obtained by the two coefficient insertion maps, and $\mathfrak{f}_{p-1}^{-1}$ comes from lifting the second small cycle to ordinary homology, as in Fig.~\ref{fig:overline_M_2_action}. Consequently, in the respective target category,
		\[
			\left\langle\bbF_{\calL}(M_1)(1),\bbF_{\calL}(\overline{M}_2)(1)\right\rangle
			=(v_1,\mathfrak{f}_{p-1}^{-1}\,v_2)=1.
		\]
		Therefore by the definition of an index 2 cobordism action we have
		\[
			\bbF_{\calL}(M_2)\circ\bbF_{\calL}(M_1)=\Id.
		\]


		\item Reversing the framed attaching sphere reverses the belt parametrization. In the interval chart $0<x_1<\cdots<x_n<1$, reflection is
		\[
			(x_1,\ldots,x_n)\longmapsto(1-x_n,\ldots,1-x_1),
		\]
		with orientation sign $(-1)^{p(p-1)/2}$. Inversion of $t$ fixes $\eta^A(1)$, and the coefficient identification $\kappa_r$ specified after \eqref{eq:psi_M_def} contributes $(-1)^{(p-1)(p-2)/2}$. The combined sign of the inserted twisted class is therefore $(-1)^{(p-1)^2}=(-1)^{p-1}$. By uniqueness of descent through the surjective filling map in Lemma \ref{lem:i_inversion}, the index $1$ action is unchanged for odd $p$ and changes by an overall sign for $p=2$. The index $2$ statement follows by dualizing the index $1$ statement for the reversed cobordism.
		\end{enumerate}
		
		Theorem~\ref{thm:Juhasz_relations} now shows that composing the elementary maps is independent of the chosen Morse decomposition and defines a functor on $3\Cob$. The two elementary monoidality identities extend to composites, and the already verified associativity and disk units make this functor monoidal. All signs discarded when $p=2$ are overall signs of maps, as required for $\mathrm{Mod}_R^\pm$.
	\end{proof}



\section{Trivial local systems and Frohman--Nicas--Donaldson TQFT}\label{sec:trivial_ls}

In this section we consider the somewhat degenerate case of trivial local systems with $p=2$. We show that the corresponding homological TQFT is isomorphic to the Frohman--Nicas--Donaldson TQFT.

\vspace{0.5cm}
	Let $R=\Bbbk$ be a field with trivial involution and $\mathrm{char}(\Bbbk)\not = 2$, then $\mathrm{Mod}_R=\mathrm{Vect}_{\Bbbk}$. Denote by $\mathrm{Vect}_{\Bbbk}^{\pm}$ the category of $\Bbbk$ vector spaces as objects and $\Bbbk$-linear homomorphisms \emph{up to sign} as morphisms. We put $p=2$, hence by Theorem~\ref{thm:main_thm} any Morse compatible $2$-Heisenberg local system induces a homological TQFT with target category $\mathrm{Vect}_{\Bbbk}^{\pm}$. 
	
	
	Let $\underline{\Bbbk}$ be the trivial collection of local systems
	\[
		\underline{\Bbbk}:=\{\underline{\Bbbk}(\Sigma)=\underline{\Bbbk}\ |\ \Sigma\in 3\Cob\}.
	\]
	We equip its fiber with the trivial Heisenberg action in every particle degree, including degree zero. Each $\underline{\Bbbk}(\Sigma)$ is equipped with the obvious pairing, hence $\underline{\Bbbk}(\Sigma)\in \Locpair_{\Bbbk}(\Conf_*(\Sigma))$. The collection $\underline{\Bbbk}$ is monoidal (see \eqref{eq:nice_ls_monoidality}) with $\mu_{\underline{\Bbbk}}$ given by
	\[
		\mu_{\underline{\Bbbk}}:\underline{\Bbbk}(\Sigma_1)\boxtimes \underline{\Bbbk}(\Sigma_2)\to \mu^*\underline{\Bbbk}(\Sigma_1\natural\Sigma_2), \quad 1\boxtimes 1\mapsto 1.
	\]

	 We define an action of elementary cobordisms as follows.
	 For each mapping class $d$ let $\phi_{d}=\mathrm{Id}$ (see \eqref{eq:phi_d_def} for definition). For an index 1 cobordism $M:\Sigma_-\to \Sigma_+$ let (see \eqref{eq:psi_M_def} for definition) 
	\[
		\psi_M:\underline{\Bbbk}\boxtimes A \to \varphi^*\underline{\Bbbk}
	\]
	\[
		1\boxtimes(a+bt)\mapsto \frac{a+b}{2},\qquad a,b\in\Bbbk,
	\]
	where $A=\Bbbk[t]/(t^2-1)$. This is equivariant since the target monodromy is trivial and multiplication by $t$ preserves the augmentation $a+bt\mapsto a+b$.

	Since the involution is trivial, we tautologically identify $\overline{\underline{\Bbbk}(\Sigma)}=\underline{\Bbbk}=\underline{\Bbbk}(\overline{\Sigma})$ for each $\Sigma$. The disk unit is $u_D(1)=1$. These identifications and tensor maps satisfy associativity, the unit conditions, and the anti-monoidal identity \eqref{eq:conjugation_coherence}, since multiplication in $\Bbbk$ commutes with the unsigned coefficient flip.
	
	\vspace{0.5cm}

	Recall that there is a version of 3-dimensional Frohman--Nicas--Donaldson TQFT adapted to the category $3\Cob$:
		\[
			D:3\Cob\to\mathrm{Vect}_{\Bbbk}^{\pm}
		\]
	A detailed description is given in Appendix~\ref{sec:Donaldson}.
	\begin{maintheorem}\label{thm:Donaldson_iso}
		Let $\Bbbk$ be a field with $\mathrm{char}(\Bbbk)\not= 2$. Then the collection of trivial local systems $\underline{\Bbbk}$ together with the action of elementary cobordisms is Morse compatible. It gives rise to a homological TQFT
		\[
			\bbF_{\underline{\Bbbk}}:3\Cob\to \mathrm{Vect}_{\Bbbk}^{\pm}
		\]
		isomorphic to the Frohman--Nicas--Donaldson TQFT.
	\end{maintheorem}
	\begin{proof}
		Conditions \eqref{eq:mcg_rep}, \eqref{eq:invariance}, \eqref{eq:locality}, \eqref{eq:locality_left}, and \eqref{eq:Monoidal_structure} follow from the identity cylinder and tensor maps and the same augmentation map $\psi_M$ in every degree, including degree zero. More explicitly, both locality squares send $u\boxtimes v\boxtimes(a+bt)$ to $uv(a+b)/2$, for $u,v,a,b\in\Bbbk$. The coefficient permutation $c_{\calA}$ is unsigned. Belt reparametrizations preserve this map: for a reflection the isomorphism after \eqref{eq:psi_M_def} is $\kappa_r(t^k)=t^{-k}$, which preserves the augmentation. To check \eqref{eq:cancellation}, note that $\mathfrak{f}_1=1$ and $\eta^A$ has coefficient $1+t$. Thus $\eta^A$ composed with $\psi_{M_1}$ (resp. $\overline{\psi}_{M_2}$) sends the disk unit $1$ to $2\cdot 1/2=1$. Therefore
		\[
			\left(\psi_{M_1}\circ(\Id\boxtimes \eta^A)(1), \overline{\psi}_{M_2}\circ (\Id\boxtimes \eta^A)(1)\right)=(1,1)=1.
		\]
		 Hence $\underline{\Bbbk}$ is a Morse compatible collection of local systems and by Theorem~\ref{thm:main_thm} the functor $\bbF_{\underline{\Bbbk}}$ is defined.
		\vspace{0.5cm}
		
		We now compare $\bbF_{\underline{\Bbbk}}$ and the FND TQFT.
		Fix $\Sigma\in 3\Cob$. For a choice of symplectic arcs in $\Sigma$ the space of small cycles is spanned by twisted cycles $\mathbf{\Gamma}(a_1,b_1,\dots,a_g,b_g)\otimes 1$ with labels $\leq 1$ (see Fig.~\ref{fig:BM_arcs} and compare with Fig.~\ref{fig:Donaldson_basis}). The comparison isomorphism has the following expression:
		\begin{equation}\label{eq:homological_Donaldson_iso}
			\mathring{\bbH}(\Sigma;\underline{\Bbbk})\to \wedge^*H_1(\Sigma,\partial_-\Sigma;\Bbbk),
		\end{equation}
		\[
			\mathbf{\Gamma}(a_1,b_1,\dots,a_g,b_g)\otimes 1\mapsto [\alpha_1]^{a_1}\wedge[\beta_1]^{b_1}\wedge\dots\wedge [\alpha_g]^{a_g}\wedge [\beta_g]^{b_g}.
		\]

		To see that this isomorphism is natural, consider the map on ordinary relative homology induced by
		\[
		(\Conf_n(\Sigma),\Conf_n^-(\Sigma))\longrightarrow(\Sym^n(\Sigma),\Sym^n_-(\Sigma)).
		\]
		By Section~\ref{sec:small_cycles}, the kernel of $\iota$ is spanned by parallel-arc classes with a repeated arc. Their images in the symmetric product vanish, since the corresponding exterior products have a repeated factor. The map therefore factors through $\mathring\bbH_n(\Sigma;\underline{\Bbbk})$. For labels $0$ or $1$ the fusion factors are $1$, and products of distinct arcs map to the exterior basis in \eqref{eq:Donaldson_decomposition}. Thus the induced map is \eqref{eq:homological_Donaldson_iso}, independently of the chosen arcs.

		We now check that the actions of elementary cobordisms coincide.
		\begin{itemize}
			\item Mapping classes act on a twisted cycle $\mathbf{\Gamma}(\mathbf{a},\mathbf{b})\otimes 1$ by sending it to the twisted cycle given by the arcs $d(\alpha_i)$ and $d(\beta_i)$ with the same labels. Similarly, the classes $[\alpha_i]$, $[\beta_i]$ are sent by $d_*$ to $[d(\alpha_i)]$ and $[d(\beta_i)]$. Naturality of the comparison therefore gives equality of the mapping-cylinder actions: $\bbF_{\underline{\Bbbk}}(M_d)=D(M_d)$.
			\item For an index 1 cobordism $M:\Sigma_-\to \Sigma_+$ we compare explicit formulas. By the formula \eqref{eq:index_1_preaction} for the standard index 1 cobordism we have 
			\begin{equation}\label{eq:index_1_trivial}
				\bbF_{\underline{\Bbbk}}(M):\mathbf{\Gamma}(a_1,b_1,\dots,a_g,b_g)\otimes 1\mapsto \pm\mathbf{\Gamma}(a_1,b_1,\dots,a_g,b_g,1,0)\otimes 1.
			\end{equation}
			On the other hand:
			\[
				D(M):[\alpha_1]^{a_1}\wedge[\beta_1]^{b_1}\wedge\dots\wedge [\alpha_g]^{a_g}\wedge [\beta_g]^{b_g}\mapsto \pm [\alpha_1]^{a_1}\wedge[\beta_1]^{b_1}\wedge\dots\wedge [\alpha_g]^{a_g}\wedge[\beta_g]^{b_g}\wedge [\alpha_{g+1}]
			\]
			by \eqref{eq:Donaldson_index_1}. Thus $\bbF_{\underline{\Bbbk}}(M)$ coincides with $D(M)$ up to the isomorphism \eqref{eq:homological_Donaldson_iso}.
			\item By \eqref{eq:Donaldson_index_2} an index 2 cobordism $M'=\overline{M}$ in standard basis acts by: 
			\begin{multline}
				D(M'):[\alpha_1]^{a_1}\wedge[\beta_1]^{b_1}\wedge\dots\wedge [\alpha_{g+1}]^{a_{g+1}}\wedge [\beta_{g+1}]^{b_{g+1}}\mapsto\\
				\mapsto \pm[\alpha_1]^{a_1}\wedge[\beta_1]^{b_1}\wedge\dots\wedge [\alpha_{g}]^{a_g}\wedge [\beta_{g}]^{b_g} \cdot \delta_{a_{g+1},0}\delta_{b_{g+1},1}.
			\end{multline}
			On the other hand by formula \eqref{eq:index_2_preaction} one obtains
			\[
				\bbF_{\underline{\Bbbk}}(\overline{M}'):\mathbf{\Gamma}^{\vee}(a_1,b_1,\dots,a_g,b_g)\otimes 1\mapsto \pm \mathbf{\Gamma}^{\vee}(a_1,b_1,\dots,a_g,b_g,0,1)\otimes 1.
			\]
			All labels are $0$ or $1$, so these dual arc representatives have compact lifts with fusion factor $1$. Taking the adjoint using \eqref{eq:pairing_formula} gives:
			\begin{equation}\label{eq:index_2_trivial}
				\bbF_{\underline{\Bbbk}}(M'):\mathbf{\Gamma}(a_1,b_1,\dots,a_g,b_g,a_{g+1},b_{g+1})\otimes 1\mapsto \pm\mathbf{\Gamma}(a_1,b_1,\dots,a_g,b_g)\otimes 1\, \delta_{a_{g+1},0}\delta_{b_{g+1},1}.
			\end{equation}
			These actions coincide under \eqref{eq:homological_Donaldson_iso}, with a single overall sign. We compare the operations directly: the configuration-space pairing pairs equal particle degrees, whereas \eqref{Donaldson_pairing} pairs complementary exterior degrees.
			
		\end{itemize}
		Finally, in particle degrees $(n_1,n_2)$ the boundary-sum map of Section~\ref{sec:monoidality} is $(-1)^{g_1n_2}$ times the cross product. Under \eqref{eq:homological_Donaldson_iso} this is precisely the exterior boundary-sum map \eqref{eq:Donaldson_tensorator}. The zero-particle class of the disk maps to $1\in\Bbbk$, so the unit maps also agree. Thus the comparison is an isomorphism of monoidal functors.
	\end{proof}

	\section{Schr\"odinger local systems and Kerler--Lyubashenko TQFT}\label{sec:Schroedinger_ls}
	In this section we study the main example of a (projective) homological TQFT, namely the one arising from Schr\"odinger local systems. In Section~\ref{subs:Schroedinger_local_systems} we define these local systems and provide a projectively Morse compatible action of elementary cobordisms on them. We then show in Section~\ref{sec:KL_TQFT} that the corresponding homological TQFT $\bbF_{\calW_p}$ is projectively isomorphic to the Kerler--Lyubashenko TQFT. In Section~\ref{sec:sl_2_action} we show that $\bbF_{\calW_p}$ intertwines the action of quantum $\mathfrak{sl}_2$ on state spaces, constructed in \cite{martel2022homological}, and hence is projectively isomorphic to the Kerler--Lyubashenko TQFT as a braided monoidal functor.
	
	\vspace{0.5cm}

	For this section we fix:
	\begin{itemize}
		\item an odd integer $p\geq 3$;
		\item $R=\C$ with involution given by the complex conjugation;
		\item $\zeta=e^{2\pi i/p}$.
	\end{itemize}
	Therefore $\mathrm{Mod}_R=\mathrm{Vect}_{\C}$. We denote by $\bbP\mathrm{Vect}_{\C}$ the projective category: objects are vector spaces, morphisms are linear maps up to scalar $\in \C\setminus\{0\}$.
	
	\subsection{Schr\"odinger local systems}\label{subs:Schroedinger_local_systems}
	
	\noindent\textit{Schr\"odinger representations.} Fix $\Sigma\in 3\Cob$. A canonical example of a Heisenberg group representation is the Schr\"odinger representation. There is an analog of the famous Stone--von~Neumann theorem in our context, see for example \cite{gelca2010classical}, Theorem 2.4.
	
	\begin{theorem}[Stone--von~Neumann]\label{thm:Stone_von_Neumann}
		There exists a unique irreducible unitary representation $W_{\zeta}(\Sigma)$ of $\Heis_p(\Sigma)$, up to unitary isomorphism, where the central generator $\sigma$ acts by $-\zeta^{-2}$. This representation is called the Schr\"odinger representation.
	\end{theorem}
\begin{proof}
	Existence is supplied by the induced representation \eqref{eq:Schroedinger_model} constructed below. We first prove uniqueness.
	Let $\xi=-\zeta^{-2}$. Then $\xi$ is a primitive root of unity of order $2p$. Choose a collection of symplectic arcs $\{\alpha_i,\beta_i\}_{i=1}^g$ in $\Sigma$. The commutator formula \eqref{eq:commutator_formula} gives the relations
	\[
	\big[[\alpha_i]_{\Heis},[\beta_j]_{\Heis}\big]
	=\sigma^{-2\delta_{i,j}},\qquad
	\big[[\alpha_i]_{\Heis},[\alpha_j]_{\Heis}\big]=1,\qquad
	\big[[\beta_i]_{\Heis},[\beta_j]_{\Heis}\big]=1,
	\]
	for all $1\leq i,j\leq g$. Moreover,
	\[
	\sigma^{2p}=1,\qquad
	[\alpha_i]_{\Heis}^{p}=1,\qquad
	[\beta_i]_{\Heis}^{p}=1,
	\qquad 1\leq i\leq g.
	\]
	Here we omit the tilde over $\alpha_i$ and $\beta_i$, viewing them as loops in $B_n(\Sigma)$ for some $n\geq 2$. The second set of relations follows from the fact that $\alpha_i$ and $\beta_i$ are pure braids, while $\sigma$ is not represented by a pure braid but $\sigma^2$ is.
	
	Let $W$ be a representation satisfying the assumptions of the theorem. Since the operators $[\alpha_i]_{\Heis}$ commute pairwise, they admit a common eigenvector $\mathbf{v}\in W$. Since each $[\alpha_i]_{\Heis}$ has order $p$, all eigenvalues are $p$-th roots of unity.
	
	Using the commutation relations, for every $j$ the vector $[\beta_j]_{\Heis}\mathbf{v}$ is again a common eigenvector of the operators $[\alpha_i]_{\Heis}$. Its eigenvalues agree with those of $\mathbf{v}$ except for the $j$-th eigenvalue, which is multiplied by $\xi^{-2}$. Consequently, the vectors
	\[
	\prod_{j=1}^{g}[\beta_j]_{\Heis}^{k_j}\mathbf{v},
	\qquad
	0\leq k_j\leq p-1,
	\quad 1\leq j\leq g,
	\]
	are common eigenvectors of all $[\alpha_i]_{\Heis}$ with pairwise distinct sets of eigenvalues. Hence they are linearly independent. Since $W$ is irreducible, they form a basis of $W$. Furthermore, we may choose a basis vector $\mathbf{v}_0$ whose eigenvalues under all $[\alpha_i]_{\Heis}$ are equal to $1$. This gives a canonical model of the representation: any other irreducible representation $W'$ satisfying the assumptions has an analogous basis obtained from a vector $\mathbf{v}_0'\in W'$ with the same eigenvalues. Sending $\mathbf{v}_0$ to $\mathbf{v}_0'$ and extending equivariantly gives an isomorphism
	\[
	W\cong W'.
	\]
	Thus the representation is unique up to isomorphism.
	
	Finally, the invariant Hermitian pairing is unique up to an overall scalar. Indeed, two distinct basis vectors above have different eigenvalues for at least one operator $[\alpha_i]_{\Heis}$, and hence are orthogonal with respect to any invariant Hermitian form. The norm of every basis vector is determined by the norm of $\mathbf{v}_0$ using the unitary action of the operators $[\beta_j]_{\Heis}$. Therefore, after imposing $(\mathbf{v}_0,\mathbf{v}_0)=1$, the invariant Hermitian pairing is uniquely determined.
\end{proof}

\begin{remark}
	The proof essentially follows the argument of Gelca and Uribe \cite{gelca2010classical}, where the finite Heisenberg group is realized using a specific coordinate model. We avoid this realization here because it is not invariant under the mapping class group action. Instead, the quotient description of $\Heis_p(\Sigma)$ via pure braids in \eqref{eq:finite_Heisenbeg_def} is intrinsic to the surface $\Sigma$ and makes the mapping class group invariance transparent.
\end{remark}

	Note that by Stone--von~Neumann theorem $W_{\zeta}(\Sigma)$ is equipped with a unique sesquilinear pairing $(-,-)$, up to scalar, equivariant under the action of the Heisenberg group.
	
	\vspace{0.5cm}
	
	We use an explicit model for $W_{\zeta}(\Sigma)$ as follows. Let $L\subset H_1(\Sigma)$ be a Lagrangian subspace with respect to the intersection form, \ie a maximal subgroup with vanishing intersection pairing. Choose a \emph{section} of $L$ in the Heisenberg group, \ie a homomorphism $s:L\to \Heis(\Sigma)$ such that its composition with the natural projection $\pi$ in \eqref{s_e_s} is identical: $\pi\circ s=\mathrm{Id}_{L}$. We require its image in $\Heis_p(\Sigma)$ to factor through $L/pL$. Such sections exist: lift a basis of $L$ by pure braids, then their Heisenberg images commute, and their $p$-th powers vanish in the finite quotient. All sections below satisfy this requirement. When their target is written as $\Heis_p(\Sigma)$, we use the same notation for their finite images. Denote by $\mathbf{L}\subset \Heis_p(\Sigma)$ the image of $s(L)$. It projects isomorphically onto $L/pL$ and meets $\langle\sigma\rangle$ trivially. The central element $\sigma$ has order $2p$ in $\Heis_p(\Sigma)$, as follows from the symplectic-generator presentation used above. Denote
	\[
		\tilde{\bfL}:=\Z_{2p}\times \bfL\subset \Heis_p(\Sigma)
	\]
	the (maximal) abelian subgroup determined by $L$ and $s$, where $\Z_{2p}$ is the cyclic subgroup generated by $\sigma$. With this datum define a $\Heis_p(\Sigma)$-representation
	\begin{equation}\label{eq:Schroedinger_model}
		W_{\zeta}(L):=\C[\Heis_p(\Sigma)]\otimes_{\tilde{\bfL}}\C_{\zeta}
	\end{equation}
	where $\C_{\zeta}$ is the one-dimensional representation of $\tilde{\bfL}$ with an element of the form $(k,x)\in\Z_{2p}\times \bfL$, acting by $(-\zeta^{-2})^k$. Note that $W_{\zeta}(L)$ is as well a representation of $\Heis(\Sigma)$ via the quotient homomorphism $\Heis(\Sigma)\to \Heis_p(\Sigma)$.
	
		Let $L^{\vee}$ be a complementary Lagrangian to $L$, \ie a Lagrangian subspace of $H_1(\Sigma)$ such that $H_1(\Sigma)=L\oplus L^{\vee}$. Choose a section $s^{\vee}:L^{\vee}\to \Heis(\Sigma)$ satisfying the same finite-quotient requirement. Then one can choose a basis of $W_{\zeta}(L)$ parametrized by elements of $L^{\vee}/pL^{\vee}$:
	\[
	\mathbf{v}_{l^{\vee}}:=s^{\vee}(l^{\vee})\otimes 1\in \C[\Heis_p(\Sigma)]\otimes_{\tilde{\bfL}}\C_{\zeta}, \quad l^{\vee}\in L^{\vee}/pL^{\vee},
	\]
	where $s^{\vee}(l^{\vee})$ is well-defined in the finite Heisenberg group, independently of the representative of $l^{\vee}$.
	Note that in particular, each $\mathbf{v}_{l^{\vee}}$ is an eigenvector of any element in $L$ and uniquely determined (up to scalar) by all eigenvalues:
	\[
		s(l)\cdot\mathbf{v}_{l^{\vee}}=\zeta^{4l.l^{\vee}}\mathbf{v}_{l^{\vee}}, \quad l\in L.
	\]
	 In particular, this implies that $W_{\zeta}(L)$ is irreducible. One can define an equivariant sesquilinear pairing:
	 \[
	 (\mathbf{v}_{l_1^{\vee}},\mathbf{v}_{l_2^{\vee}})=\delta_{l_1^{\vee},l_2^{\vee}},\quad l^{\vee}_1, l_2^{\vee}\in L^{\vee}/pL^{\vee}.
	 \]
  Since $\sigma$ acts on $W_{\zeta}(L)$ by $-\zeta^{-2}$, Stone--von~Neumann gives a unitary intertwining isomorphism, unique up to a scalar of modulus one,
	 \[
	 	W_{\zeta}(L)\simeq W_{\zeta}(\Sigma).
	 \]

	For a Lagrangian $L\subset H_1(\Sigma)$ and a section $s:L\to \Heis(\Sigma)$ we denote
	\[
		\mathbf{1}_{L}\in W_{\zeta}(\Sigma), \quad s(l)\cdot \mathbf{1}_{L}=\mathbf{1}_L, \ \forall l\in L
	\]
	the nonzero vector, unique up to scalar, fixed by $s(L)$. 
	
		For the standard choice of symplectic arcs in $\Sigma_g$ as in Fig.~\ref{fig:BM_arcs} we put $L_g=\langle[\alpha_1],\dots, [\alpha_g]\rangle$ with the section $s([\alpha_i])=[\alpha_i]_{\Heis}$. Let $L_g^{\vee}=\langle[\beta_1],\dots, [\beta_g]\rangle$ and $s^{\vee}([\beta_i])=[\beta_i]_{\Heis}$. The standard basis is given by:
	\begin{equation}\label{standard_Schroedinger_basis}
		\mathbf{v}_{c_1,\dots,c_g}:=s^{\vee}(c_1[\beta_1]+\dots +c_g[\beta_g])\otimes  1,\quad c_1,\dots,c_g \in \{0,1,\dots,p-1\}.
	\end{equation}
	In particular, $\mathbf{1}_{L_g}=\mathbf{v}_{0,\dots,0}$.

	\vspace{0.5cm}
	
  	\noindent\textit{Schr\"odinger local systems.}	For each $\Sigma\in 3\Cob$ set $\calW_p(\Sigma)$ to be the $p$-Heisenberg local system on $\Conf_*(\Sigma)$ with chosen coefficient module $W_{\zeta}(\Sigma)$ at $*_n$ for every $n\geq0$. In particular, its degree-zero fiber is $W_{\zeta}(\Sigma)$, of dimension $p^g$, although its geometric monodromy in that degree is trivial. Denote by
	\[
		\calW_p:=\{\calW_p(\Sigma)|\Sigma\in 3\Cob\}
	\]
	the corresponding collection of Schr\"odinger local systems. Recall that, on the level of Heisenberg groups, reversing the orientation of $\Sigma\in 3\Cob$ is equivalent to inverting the central generator $\sigma$. We have a fiberwise tautological identification
	\[
		\chi^0_\Sigma:\calW_p(\overline{\Sigma})\xrightarrow{\simeq}\overline{\calW_p(\Sigma)}, \quad \forall\Sigma\in 3\Cob,
	\]
	since inverting $\sigma$ is equivalent to the complex conjugation of the Schr\"odinger representation ($\zeta\mapsto \zeta^{-1}$). Hence the equivariant pairing on $W_{\zeta}(\Sigma)$ defined above, gives rise to a perfect sesquilinear pairing:
	\[
		(-,-):\calW_p(\Sigma)\otimes\overline{\calW_p(\Sigma)}\to \underline{\C}.
	\]
	
	For a disk $D$, the module $W_\zeta(D)$ is one-dimensional. Its central generator acts by $\xi=-\zeta^{-2}$, so for $n\geq2$ each disk braid generator acts by $\xi$. For $n=0,1$ the geometric monodromy is trivial. The unit $u_D:\C\to W_\zeta(D)$ sends $1$ to the standard basis vector. Remark~\ref{rem:low_degree_disk} still gives $\bbF_{\calW_p}(D)=\C$. More generally, there are $p^{2g}$ small-cycle arc labels and each coefficient summand has dimension $p^g$, hence
	\[
		\dim_\C\bbF_{\calW_p}(\Sigma)=p^{3g},\qquad
		\bbF_{\calW_p,0}(\Sigma)=W_\zeta(\Sigma).
	\]
	This degree-zero convention agrees with \cite[Definition~2.28]{de2022homological}.
\vspace{0.5cm}
	\noindent\textit{Monoidal structure.} We define a monoidal structure on Schr\"odinger local systems as follows. Let $\Sigma_1,\Sigma_2\in 3\Cob$, and let $L_1$, $L_2$ be Lagrangians together with sections $s_1$, $s_2$. Define $L=L_1\times L_2$ and $s=\mu_*(s_1\times s_2)$, where $\mu_*$ is the homomorphism $\Heis(\Sigma_1)\times\Heis(\Sigma_2)\to \Heis(\Sigma_1\natural \Sigma_2)$ induced by the embedding $\Sigma_1\sqcup \Sigma_2\to \Sigma_1\natural \Sigma_2$. The tensor identification is chosen as a morphism of $\Heis_p(\Sigma_1)\times\Heis_p(\Sigma_2)$-modules via $\mu_*$. Restriction gives, for all $k,l\geq0$, a morphism of $B_k(\Sigma_1)\times B_l(\Sigma_2)$-modules:
	\begin{equation}\label{eq:mu_calW}
		\mu^0_{\Sigma_1,\Sigma_2}:W_{\zeta}(\Sigma_1)\otimes W_{\zeta}(\Sigma_2)\to W_{\zeta}(\Sigma_1\natural \Sigma_2)
	\end{equation}
	\[
		\mathbf{1}_{L_1}\otimes \mathbf{1}_{L_2}\mapsto \mathbf{1}_L.
	\]
	One can choose complementary Lagrangians $L_1^{\vee}$ and $L_2^{\vee}$ to $L_1$ and $L_2$, then $L_1^{\vee}\oplus L_2^{\vee}$ is a complementary Lagrangian to $L_1\oplus L_2$. Since the morphism above sends 
	\[
	\mathbf{v}_{l_1}\otimes \mathbf{v}_{l_2}\to \mathbf{v}_{l_1+l_2},\quad\forall l_1\in L_1^{\vee}/pL_1^{\vee}, \, \forall l_2\in L_2^{\vee}/pL_2^{\vee},
	\]
	it is an isomorphism preserving the pairings, since these bases are orthonormal. Note that this auxiliary isomorphism is unique up to scalar and does not depend projectively on the choice of Lagrangians $L_1$, $L_2$ or sections $s_1$, $s_2$, since both $W_{\zeta}(\Sigma_1)\boxtimes W_{\zeta}(\Sigma_2)$ and $W_{\zeta}(\Sigma_1\natural \Sigma_2)$ are isomorphic irreducible representations of $\Heis_p(\Sigma_1)\times \Heis_p(\Sigma_2)$. Choose this same auxiliary map in all particle degrees. For genera $g,h$ and particle degrees $n,m$, define the coefficient tensor map by
	\[
		\mu_{\calW_p}\big|_{(n,m)}=\zeta^{-2gm}\mu^0_{\Sigma_1,\Sigma_2}.
	\]
	This gives a projective isomorphism in $\Loc_{\C}(\Conf_*(\Sigma_1)\times\Conf_*(\Sigma_2))$: the ambiguity in the auxiliary map is one scalar in all degrees. Associativity follows from
	\[
		gm+(g+h)\ell=h\ell+g(m+\ell),
	\]
	and either degree-zero disk unit has phase $1$.

	For orientation reversal set
	\[
		\chi_\Sigma\big|_{\Conf_n(\Sigma)}=\zeta^{2gn}\chi^0_\Sigma,
		\qquad g=\mathrm{genus}(\Sigma).
	\]
	These maps preserve the coefficient pairings and disk units. Their anti-monoidal coherence follows from the scalar identity
	\[
		\zeta^{2(g+h)(n+m)}\zeta^{-2hn}
		=\overline{\zeta^{-2gm}}\,\zeta^{2gn}\zeta^{2hm}.
	\]
	
	\vspace{0.5cm}
	
	\noindent\textit{Action of mapping cylinders.} If $d:\Sigma\to\Sigma'$ is a mapping class, choose an intertwining isomorphism $\phi_d:W_{\zeta}(\Sigma)\to d^*W_{\zeta}(\Sigma')$. The pull-back of the positive definite invariant Hermitian pairing is a positive scalar multiple of the source pairing. Rescaling $\phi_d$ by the inverse square root of this scalar makes it pairing-preserving. Thus $\phi_d$ defines an isomorphism in $\Loc_{\C}(\Conf_*(\Sigma))$, unique up to a scalar of modulus one. We use the same chosen Heisenberg intertwiner in every particle degree, including degree zero, for the projective mapping-cylinder action. Naturality of $\chi^0$ holds projectively by uniqueness of auxiliary Heisenberg intertwiners. The factor $\zeta^{2gn}$ is the same on both sides, since a diffeomorphism preserves genus and particle degree. Hence $\chi$ satisfies the naturality identity in \eqref{eq:conjugation_coherence} with one scalar in all degrees.

	\vspace{0.5cm}
	
	\noindent\textit{Action of index 1 cobordisms}. Consider an index 1 cobordism $M:\Sigma_-\to \Sigma_+$ with belt sphere $\nu:S^1\to \Sigma_+$. In order to define its action on Schr\"odinger local systems we choose a path $h$ in $\Conf_{n+p-1}(\Sigma_+)$ from $*_{n+p-1}$ to $*_n\cup \nu(s_{p-1})$ and choose compatible Lagrangians on $\Sigma_-$ and $\Sigma_+$ as follows.

	Denote $t_h:=h\varphi_*(1,t)h^{-1}$. Its image in $H_1(\Sigma_+)$ is $[\nu]$. Insertion preserves pure braids in the trajectory factor, so $\varphi_*^h(-,0)$ descends to a homomorphism of finite Heisenberg groups. The circle calculation in Fig.~\ref{fig:loop_insertion} gives $[t_h]_{\Heis}^p=\sigma^p$. Thus the corrected lift $\sigma^p[t_h]_{\Heis}$ has order $p$, and we require
	\[
		s_+([\nu])=\sigma^p[t_h]_{\Heis}\in\Heis_p(\Sigma_+).
	\]
	In particular, $[t_h]_{\Heis}$ acts by $-1$ on a vector fixed by $s_+([\nu])$.

	Choose a Lagrangian $L_-\subset H_1(\Sigma_-)$ and a section $s_-$. Choose $L_{\calT}\subset H_1(\calT_M)$ mapping isomorphically onto $L_-$ under $(i_-)_*$, and a section $s_{\calT}$ lifting $s_-$. Such choices exist by lifting a basis and its commuting pure Heisenberg lifts. We call $L_{\calT}$ a Lagrangian for convenience. Set
	\[
		L_+:=(i_+)_*L_{\calT}\oplus\Z[\nu].
	\]
	This is a Lagrangian in $H_1(\Sigma_+)$. Define $s_+$ on $(i_+)_*L_{\calT}$ by the commutative diagram
	\begin{equation}\label{eq:Heisenberg_span}
		\begin{tikzcd}
			&\Heis_p(\calT_M)\arrow[dl,"(i_-)_*"']\arrow[dr,"{\varphi_*^h(-,0)}"]&\\
			\Heis_p(\Sigma_-)&L_{\calT}\arrow[dl,"(i_-)_*"]\arrow[dr,"(i_+)_*"']\arrow[u,"s_{\calT}"]&\Heis_p(\Sigma_+)\\
			L_-\arrow[u,"s_-"]&&L_+\arrow[u,"s_+"]
		\end{tikzcd}
	\end{equation}
	and impose the displayed value on $[\nu]$ separately. 

	We will also need the relations imposed by the two punctures of $\calT_M$. Their loops generate the kernel of $(i_-)_*$ on finite Heisenberg groups. Under $\varphi_*^h(-,0)$ they map to the two inverse generators of $\langle s_+([\nu])\rangle$. To see this, use the standard chart of Fig.~\ref{fig:Sigma_W_arcs}. The loop $\alpha^M$ avoids the inserted points and maps to $s_+([\nu])^{-1}$ by the circle calculation in Fig.~\ref{fig:loop_insertion}. The product with the other puncture loop is the boundary loop of a punctured torus. Its Heisenberg image is $\sigma^{\pm2p}$: the handle contributes two crossings and the $p-1$ inserted points contribute $2(p-1)$ crossings with the same sign. This is the genus-one surface boundary relation, and is trivial in $\Heis_p(\Sigma_+)$. Hence the other puncture maps to $s_+([\nu])$. Changing $h$ conjugates both identities.

	For the standard paths $h$ of Fig.~\ref{fig:cointegral}, these homomorphisms are independent of $n$. Indeed, old handle generators are supported away from the insertion, and adding stationary boundary points changes neither their Heisenberg lifts nor the local puncture and circle calculations. We use these same auxiliary Heisenberg actions also for $n=0,1$, and choose one common module map in all degrees. Other paths are compared by \eqref{monodromy_rep_relation}. 

	\begin{lemma}\label{lem:Schroedinger_handle}
		There is a unique nonzero projective morphism of local systems, intertwining the auxiliary Heisenberg actions and chosen in all degrees as above:
		\[
			\psi_M:i_-^*\calW_p(\Sigma_-)\boxtimes\calA\to\varphi^*\calW_p(\Sigma_+),
		\]
		such that the corresponding morphism of the monodromy representations
		\begin{equation}\label{W_monodromy1}
			\psi_M^h:W_\zeta(\Sigma_-)\otimes A\to W_\zeta(\Sigma_+)
		\end{equation}
		for any choice of Lagrangians and sections as above sends $\mathbf1_{L_-}\boxtimes1$ to $\mathbf1_{L_+}$ projectively.
	\end{lemma}

	\begin{proof}
		Fix some choice of Lagrangians and sections as above first. The auxiliary $\Heis_p(\calT_M)\times\Z_{2p}$ action generates $W_\zeta(\Sigma_-)\otimes A$ from $\mathbf1_{L_-}\otimes1$. Therefore the equivariant map, if it exists, is unique up to scalar.

		For existence, use the induced description of $W_\zeta(\Sigma_-)$. After pullback along $(i_-)_*$, relations on $\mathbf{1}_{L_-}\otimes 1$ prescribe the central character, invariance under $s_{\calT}(L_{\calT})$, and trivial action of the two puncture loops. The vector $\mathbf1_{L_+}$ satisfies all these relations: the central character is the same, the Lagrangian relations follow from \eqref{eq:Heisenberg_span}, and both puncture images lie in $\langle s_+([\nu])\rangle$. The universal property of induction therefore gives a nonzero $\Heis_p(\calT_M)$-equivariant map sending $\mathbf1_{L_-}$ to $\mathbf1_{L_+}$. Since $[t_h]_{\Heis}$ commutes with this action and acts by $-1$ on $\mathbf{1}_{L_+}$, it acts by $-1$ on the whole image. The extension over $A$ is consequently well-defined by
		\[
			\psi_M^h(v\otimes t^k)=(-1)^k\psi_M^h(v\otimes1).
		\]
		It is equivariant for the circle action, and hence gives the required local-system morphism in every degree.

		Suppose $h,h'$ are two paths from $*_{n+p-1}$ to $*_n\cup\nu(s_{p-1})$. Keep $L_-,L_{\calT}$ and their sections fixed, and choose the outgoing sections by the corresponding diagrams \eqref{eq:Heisenberg_span}. For $l_{\calT}\in s_{\calT}(L_{\calT})$ we have
		\[
			\varphi_*^{h'}(l_{\calT},0)
			=[h'h^{-1}]_{\Heis}\varphi_*^h(l_{\calT},0)[h(h')^{-1}]_{\Heis}.
		\]
		Therefore
		\[
		\begin{aligned}
			\varphi_*^h(l_{\calT},0)[h(h')^{-1}]_{\Heis}
				\psi_M^{h'}(\mathbf1_{L_-}\otimes1)
			&=[h(h')^{-1}]_{\Heis}\varphi_*^{h'}(l_{\calT},0)
				\psi_M^{h'}(\mathbf1_{L_-}\otimes1)\\
			&=[h(h')^{-1}]_{\Heis}\psi_M^{h'}(\mathbf1_{L_-}\otimes1).
		\end{aligned}
		\]
		The same conjugation relation holds for $\sigma^p[t_h]_{\Heis}$. Hence this transported vector and $\psi_M^h(\mathbf1_{L_-}\otimes1)$ are fixed by the same section of $L_+$, so they are proportional. Acting by the auxiliary group gives
		\[
			\psi_M^h(w)=C\,[h(h')^{-1}]_{\Heis}\psi_M^{h'}(w),
			\qquad w\in W_\zeta(\Sigma_-)\otimes A,
		\]
		for some $C\in\C\setminus\{0\}$. This is projectively \eqref{monodromy_rep_relation}.

		Finally, fix $h$ and vary the Lagrangians and sections. The subspace of $W_\zeta(\Sigma_+)$ fixed by $\sigma^p[t_h]_{\Heis}$ has the same dimension as $W_\zeta(\Sigma_-)$, by the induced model with $[\nu]$ as one Lagrangian generator. The map just constructed identifies $W_\zeta(\Sigma_-)$ with this subspace: it is injective by irreducibility and the dimensions agree. Thus this subspace is one irreducible copy of the incoming representation for the action through $(i_-)_*$. Any other choices give another intertwiner onto the same subspace, hence the same map up to scalar by Schur's lemma. All these comparisons concern the common auxiliary module map, so the scalar is independent of particle degree.
	\end{proof}

	\begin{remark}
		The first step toward constructing a homological TQFT was taken in \cite{andreev2024abelian}, where an action of cobordisms on Schr\"odinger local systems was provided for a specific category of Lagrangian cobordisms. In that setting, the monodromy representations $W_{\zeta}(L)$ themselves become the state spaces of a TQFT. The action suggested here, when restricted to a subcategory of Lagrangian cobordisms in $3\Cob$, coincides projectively with the action of index $1$ and $2$ cobordisms (up to taking duals with respect to $(-,-)$) and with the twisted projective action of mapping classes denoted by $\calS_{\Heis}$ there.
	\end{remark}
	
	\begin{theorem}\label{thm:Schroedinger_nice}
		The collection of Schr\"odinger local systems $\calW_p:=\{\calW_{p}(\Sigma)\}$, together with the action of elementary cobordisms defined above, is a projectively Morse compatible collection of local systems, \ie all the conditions are satisfied projectively.  
	\end{theorem}
	
	\begin{proof}
		We use notation from Section~\ref{sec:nice_ls} for each condition here:
		\begin{enumerate}
			\item Clearly, $\phi_{\mathrm{Id}}=\mathrm{Id}$. A composition $\phi_{d'}\circ\phi_d$ is an isomorphism of Schr\"odinger representations, as well as $\phi_{d'\circ d}$, therefore they coincide up to a scalar since the Schr\"odinger representation is irreducible.
			\item We choose paths $h$, $h'$, Lagrangians $L_X$, $L_X'$ and their sections $s_X$, $s_X'$ for $X=\{-,+,\calT\}$ as in \eqref{eq:Heisenberg_span} such that they are compatible with $(d_{\calT})_*$. Then:
			\[
				\phi_{d_-}:W_{\zeta}(\Sigma_-)\to W_{\zeta}(\Sigma_-'),\quad 	\phi_{d_+}:W_{\zeta}(\Sigma_+)\to W_{\zeta}(\Sigma_+'),
			\]
			\[
					\phi_{d_-}(\mathbf{1}_{L_-})\sim\mathbf{1}_{L_-'},\quad 	\phi_{d_+}(\mathbf{1}_{L_+})\sim\mathbf{1}_{L_+'}.
			\]
			And on the other hand 
			\[
				\psi_M^h(\mathbf{1}_{L_-}\boxtimes 1)\sim\mathbf{1}_{L_+},\quad \psi_{M'}^{h'}(\mathbf{1}_{L_-'}\boxtimes 1)\sim\mathbf{1}_{L_+'}
			\]
			The relation follows from Lemma~\ref{lem:Schroedinger_handle}. This applies to the full auxiliary actions, including degree zero. The scalar is independent of particle degree: the standard insertion paths move the same $p-1$ strands along the path while all old boundary points remain stationary. 

			The same argument verifies compatibility with belt reflections. A reflection replaces $\sigma^p[t_h]_{\Heis}$ by its inverse, so its fixed subspace is unchanged. The circle character $t\mapsto-1$ is invariant under $t\mapsto t^{-1}$ and composing with $\kappa_r$ multiplies the coefficient map by the constant $(-1)^{(p-1)(p-2)/2}$. Thus the reflection identity specified after \eqref{eq:psi_M_def} holds projectively in all degrees.
			\item By choosing a path $h$ inside $\Sigma_+''$ and a Lagrangian $L_{\calT}$ such that $L_{\calT}=L'\oplus L_{\calT}''$ together with the corresponding sections, we obtain in the standard basis the following diagram for the raw tensor map $\mu^0$, where the input of each handle map has an additional factor $1\in A$:
			\[
				\begin{tikzcd}
					\mathbf{v}_{c_1,\dots,c_{g_1}}\otimes \mathbf{v}_{c'_1,\dots,c'_{g_2}} \arrow[r,maps to, "{\mathrm{Id}\otimes \psi_{M''}^h}"]\arrow[d,maps to,"{\mu^0}"]&	\mathbf{v}_{c_1,\dots,c_{g_1}}\otimes \mathbf{v}_{c'_1,\dots,c'_{g_2},0}\arrow[d,maps to,"\mu^0"]\\
					\mathbf{v}_{c_1,\dots,c_{g_1},c'_1,\dots,c_{g_2}'}\arrow[r,maps to,"{\psi_{M}^h}"]&\mathbf{v}_{c_1,\dots,c_{g_1},c'_1,\dots,c_{g_2}',0}
				\end{tikzcd}	
			\]
			The equality on the displayed vectors extends to the $A$-factor by the circle character $t\mapsto-1$ and to all coefficients by equivariance. The raw tensor identifications and handle maps are the common auxiliary module maps, so this includes zero-particle factors. For incoming genera $g,h$ and bidegree $(n,m)$, replacing $\mu^0$ by $\mu_{\calW_p}$ gives commutativity up to the common scalar $\zeta^{-2gm}/\zeta^{-2g(m+p-1)}=\zeta^{-2g}$ independent of particle degrees.
			
			For left locality \eqref{eq:locality_left}, the raw maps have the same auxiliary matrix as in the right-hand calculation. The additional path comparison shuffles the $p-1$ inserted points past the $m$ points in the unchanged factor. Choosing parallel strands with no additional winding, this shuffle has Heisenberg image $\sigma^{(p-1)m}$. Expressing both routes in the same path coordinates contributes the inverse scalar
			\[
			(-\zeta^{-2})^{-(p-1)m}=\zeta^{-2m}.
			\]
			The normalization of the tensor maps cancels this factor: $\zeta^{-2gm}/\zeta^{-2(g+1)m}\cdot \zeta^{-2m}=1$.
			Thus left locality holds exactly for these representatives. The path comparison is supported in a disk and acts by a scalar, so the calculation applies in every genus and particle degree, including zero.

			\item We choose symplectic arcs in $\Sigma$ as in Fig.~\ref{fig:relation_3} (compare it with Fig.~\ref{fig:handle_cancellation}), and a path $h$ to $*_{\pitchfork}$. The disk inputs are their degree-zero units. To prove \eqref{eq:cancellation} projectively, it is enough to show that
			\[
				\left(\psi_{M_1}^h(1\otimes\eta^A(1)),
				\mathbf \mathfrak{f}^{-1}_{p-1}\,\psi_{\overline M_2}^h(1\otimes\eta^A(1))\right)\ne0,
			\]
			where the second fiber is identified with the conjugate of the first.
			\begin{figure}[h]
				\centering
				\includegraphics[width=0.5\linewidth]{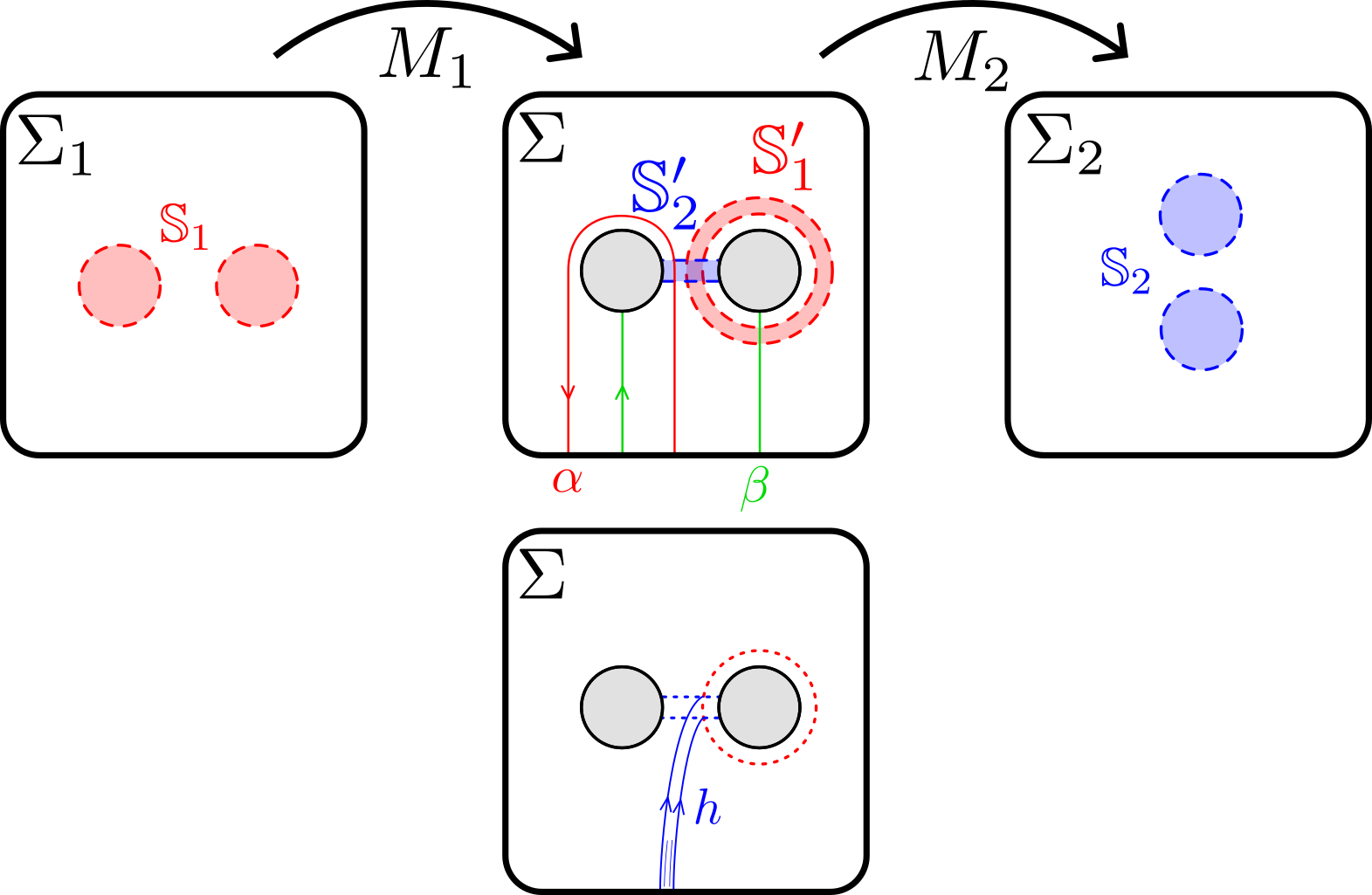}
				\caption{Composition of cancelling cobordisms, $\bbS_1$ and $\bbS_1'$ are attaching and belt tubes of $M_1$, $\bbS_2'$ and $\bbS_2$ are attaching and belt tubes of $M_2$.}
				\label{fig:relation_3}
			\end{figure}

			Let $L,L'$ be the Lagrangians generated by the two cancelling curves, with sections prescribed by the corrected belt lifts. By Lemma~\ref{lem:Schroedinger_handle}, we have
			\[
				\psi_{M_1}^h(1\otimes\eta^A(1))=2p\,\mathbf1_L,\qquad
				\psi_{\overline M_2}^h(1\otimes\eta^A(1))=2p\,\mathbf1_{L'}.
			\]
			The curves intersect once. Hence their Heisenberg lifts have commutator $\sigma^{\pm2}$, acting by a primitive $p$-th root of unity. As in the proof of Theorem~\ref{thm:Stone_von_Neumann}, either lift cyclically permutes the $p$ one-dimensional eigenspaces of the other. These eigenspaces are orthogonal for the invariant Hermitian pairing. A nonzero eigenvector of the second lift has a nonzero component in every eigenspace of the first: if one component vanished, cyclic permutation would force all of them to vanish. In particular,
			\[
				(\mathbf1_L,\mathbf1_{L'})\ne0.
			\]
			Since $\mathfrak{f}_{p-1}^{-1}$ acts by a nonzero scalar, the required contraction is nonzero. The normalization of $\chi$ multiplies this genus-one, degree $p-1$ contraction by $\zeta^{2(p-1)}=\zeta^{-2}$, which is nonzero. This proves projective cancellation.

			\item Let $d_1:\Sigma_1\to \Sigma_1'$ and $d_2:\Sigma_2\to \Sigma_2'$. The proof is similar to (2): choose Lagrangians and sections on the source surfaces and transport them by $d_1$ and $d_2$. Their direct sum is transported by $d_1\natural d_2$. Thus $\phi_{d_1}\otimes\phi_{d_2}$ and $\phi_{d_1\natural d_2}$, compared through $\mu^0$, send the same incoming vector $\mathbf{1}$ to the same outgoing vector $\mathbf{1}$ projectively. Uniqueness for the full auxiliary Heisenberg actions gives one scalar in all degrees. The tensor phase is the same on both sides because genera and particle degrees are preserved, proving \eqref{eq:Monoidal_structure} for $\mu_{\calW_p}$.
		\end{enumerate}
	Together with associativity, the disk units and anti-monoidal conjugation coherence verified above, these five checks prove projective Morse compatibility.
	\end{proof}


	We derive formulas for the action of the standard index 1 and index 2 cobordisms $M_1$ and $M_2$ (see Fig.~\ref{fig:Psi_M_dual}) in the associated projective TQFT $\bbF_{\calW_p}$. Choosing $\psi_{M_1}^h(1\otimes1)=\mathbf v_0$, formulas \eqref{eq:index_1_preaction} and \eqref{W_monodromy1} give
	\begin{equation}\label{eq:bbF_M_1_formula}
		\boxed{\begin{gathered}
				\bbF_{\calW_p}(M_1)(1)=\mathbf{\Gamma}(p-1,0)\otimes ((p-1)(p-2)/2,(p-1)[\beta])\sum\limits_{k=0}^{2p-1}(0,k[\alpha])\cdot \mathbf{v}_{0}=\\
			=\mathbf{\Gamma}(p-1,0)\otimes 2p(-\zeta^2)^{(p-2)(p-1)/2}\mathbf{v}_{-1}.
		\end{gathered}}
	\end{equation}
	
	The computation in Fig.~\ref{fig:overline_M_2_action} gives, projectively, the dual Borel--Moore small cycle
	\[
		\bbF_{\calW_p}(\overline{M}_2)(1)\sim\mathbf{\Gamma}^{\vee}(p-1,0)\otimes 2(-\zeta^2)^{(p-2)(p-1)/2}\sum\limits_{k=0}^{p-1}\mathbf{v}_{k}.
	\]
	To compute its adjoint, first lift this cycle to ordinary homology by the inverse fusion factor $\mathfrak{f}_{p-1}^{-1}$ on its coefficient fiber, as in Section~\ref{sec:small_cycles}. This is a nonzero scalar on the displayed degree $p-1$ summand. Formula \eqref{eq:pairing_formula}, with the intersection normalization of Section~\ref{sec:pairing}, then gives a nonzero scalar multiple of the functional $\mathbf\Gamma(a,b)\otimes\mathbf v_c\mapsto\delta_{a,p-1}\delta_{b,0}$. We choose the following representative of that projective adjoint: 
	\begin{equation}\label{eq:bbF_M_2_formula}
		\boxed{\bbF_{\calW_p}(M_2)(\mathbf{\Gamma}(a,b)\otimes \mathbf{v}_{c})=\delta_{a,p-1}\delta_{b,0}\frac{(-\zeta^2)^{-(p-2)(p-1)/2}}{2}.}
	\end{equation}
	Here $g=1$ and the only supported degree is $p-1$, so the additional factor from $\chi$ is the single scalar $\zeta^{-2}$, absorbed in this representative.
	


	\subsection{Kerler--Lyubashenko TQFT}\label{sec:KL_TQFT}

		
		We first recall a definition of small quantum $\mathfrak{sl}_2$. Let $\mathfrak{u}_{\zeta}\mathfrak{sl}_2$ be the $\C$-algebra with generators $\{E,F^{(1)},K\}$ and relations:
		\begin{equation}
			E^p=(F^{(1)})^p=0,\quad K^p=1,\quad KEK^{-1}=\zeta^2E,\quad KF^{(1)}K^{-1}=\zeta^{-2}F^{(1)}, \quad [E,F^{(1)}]=K-K^{-1}
		\end{equation}
		This algebra admits a structure of factorizable ribbon Hopf algebra (see \cite{lyubashenko1995invariants}). We use the Hopf and ribbon conventions of \cite[Section~4.3]{de2022homological}. In particular,
		\[
		\begin{aligned}
		\Delta(E)&=E\otimes K+1\otimes E,&
		\Delta(F^{(1)})&=K^{-1}\otimes F^{(1)}+F^{(1)}\otimes1,\\
		\Delta(K)&=K\otimes K,&
		\varepsilon(E)&=\varepsilon(F^{(1)})=0,\quad\varepsilon(K)=1.
		\end{aligned}
		\]
		For $0\leq b<p$, write $F^{(b)}=(F^{(1)})^b/[b]_{\zeta}!$, with $F^{(0)}=1$. Introduce
		\[
			T_c=1/p\sum\limits_{l=0}^{p-1}\zeta^{2lc}K^{l}.
		\]
		Then a basis of $\mathfrak{u}_{\zeta}\mathfrak{sl}_2$ is given by:
		\[
			\{E^{a}T_c F^{(b)}|\: a,b,c\in\{0,\dots,p-1\}\}.
		\] 
		\vspace{0.5cm}
	
	    The Kerler--Lyubashenko TQFT is defined on a cobordism category $3\Cob_*^{\Z}$ that is more suitable for the algebraic setting. Define first $3\Cob_*$. An object of this category is a non-negative integer $g$. Morphisms are equivalence classes of cobordisms between $\Sigma_g$ and $\Sigma_{g'}$. Monoidal structure is given by $(g,g')\mapsto g+g'$ on objects which corresponds to the boundary connected sum $\Sigma_{g_1+g_2}\simeq\Sigma_{g_1}\natural\Sigma_{g_2}$. For more details, and for a tangle presentation of this category, see Appendix~\ref{sec:ALG3} and \cite{beliakova2024kerler}. This category was introduced by Kerler in \cite{kerler2001towards} (denoted by $\mathcal{C}ob_0$ there). There is a natural functor 
	    \begin{equation}\label{eq:Cob_functor}
	    	3\Cob_*\to 3\Cob
	    \end{equation}
	     sending object $g$ to the standard surface $\Sigma_g$ and morphisms from $g$ to $g'$ to the corresponding cobordism in $3\Cob$. Let $3\Cob_*^{\Z}$ be the central extension of $3\Cob_*$ by $\Z$ given by the signature defects (the category $\mathcal{C}ob$ in Kerler's paper). 
	
	\begin{theorem}[\cite{kerler2002non,beliakova2024kerler}]
		If $H$ is a factorizable ribbon Hopf algebra, then there exists a braided monoidal functor
		\[
			J_H:3\Cob_*^{\Z}\to H\text{-mod}
		\]
		sending every surface $\Sigma_g$ to $\mathrm{ad}^{\otimes g}$, where $\mathrm{ad}$ is the adjoint representation of $H$. 
	\end{theorem}
	
	We are interested in the case $H=\mathfrak{u}_{\zeta}\mathfrak{sl}_2$. Since signature defects contribute to the Kerler--Lyubashenko TQFT only through scalar coefficients, it descends to a projective TQFT on $3\Cob_*$
	\[
		\mathbb{P}J_{\mathfrak{sl}_2}:3\Cob_*\to \mathbb{P}\mathbf{Vect}_{\C}.
	\]
	The homological TQFT $\bbF_{\calW_p}$ can be pulled back to $3\Cob_*$ via \eqref{eq:Cob_functor}. We use the same notation for this functor:
	\[
		\bbF_{\calW_p}:3\Cob_*\to \bbP\mathrm{Vect}_{\C}
	\]
	and compare it with $\mathbb{P}J_{\mathfrak{sl}_2}$ in the following theorem.


	\begin{maintheorem}\label{thm:KL_isomorphism}
			The projective TQFT $\bbF_{\calW_p}$ arising from the collection of Schr\"odinger local systems $\calW_p$ for an odd $p\geq 3$ is projectively isomorphic to the Kerler--Lyubashenko TQFT $J_{\mathfrak{sl}_2}$ associated with the small quantum $\mathfrak{sl}_2$ at the $p$-th root of unity $\zeta=e^{2\pi i/p}$.
	\end{maintheorem}
	
	 We first recall that a projective homological representation of mapping class groups coinciding with the projective Lyubashenko representation was constructed in \cite{de2022homological}. The action on the small-cycle space $\bbF_{\calW_p}(\Sigma_g)$ is determined by the generating positive Dehn twists
	\[
	\{\tau_{\alpha_i},\tau_{\beta_j},\tau_{\gamma_k}:1\leq i,j\leq g,\ 1\leq k\leq g-1\}
	\]
	along the curves in Fig.~\ref{fig:Dehn_twists}. Their action on twisted cycles is given by
	\[
		\rho_g: \mathrm{Mod}(\Sigma_g)\to \mathrm{PGL}(\bbF_{\calW_p}(\Sigma_g)) 
	\]
	\[
		\rho_g(f)(\tilde{\sigma}\otimes v)=(\tilde f^n)_*(\tilde{\sigma})\otimes \phi(f)(v), \quad n\geq 0,
	\]
	where $\tilde f^n$ is the action on the covering space $\widetilde{\Conf}_n(\Sigma_g)$ induced by $f$, $\tilde{\sigma}$ is a cycle in $C_n(\widetilde\Conf_n(\Sigma_g))$ and $v$ is a vector in the Schr\"odinger representation. The map $\phi:W_{\zeta}(\Sigma_g)\to W_{\zeta}(\Sigma_g)$ is given by the formulas in the standard basis:
	
	\begin{equation}\label{psi_alpha}
		\phi(\tau_{\alpha_i}):\mathbf{v}_{c_1,\dots,c_g}\mapsto \sum\limits_{l=0}^{p-1}\zeta^{-2l(l-1)+4c_il}\mathbf{v}_{c_1,\dots,c_g},
	\end{equation}
	\begin{equation}\label{psi_beta}
		\phi(\tau_{\beta_i}):\mathbf{v}_{c_1,\dots,c_g}\mapsto \sum\limits_{l=0}^{p-1}\zeta^{-2l(l+1)}\mathbf{v}_{c_1,\dots,c_i +l,\dots,c_g},
	\end{equation}
	\begin{equation}\label{psi_gamma}
		\phi(\tau_{\gamma_k}):\mathbf{v}_{c_1,\dots,c_g}\mapsto \sum\limits_{l=0}^{p-1}\zeta^{-2l(l+1)-4lc_k+4lc_{k+1}}\mathbf{v}_{c_1,\dots,c_g},
	\end{equation}

	\begin{figure}[h]
		\centering
		\includegraphics[width=0.5\linewidth]{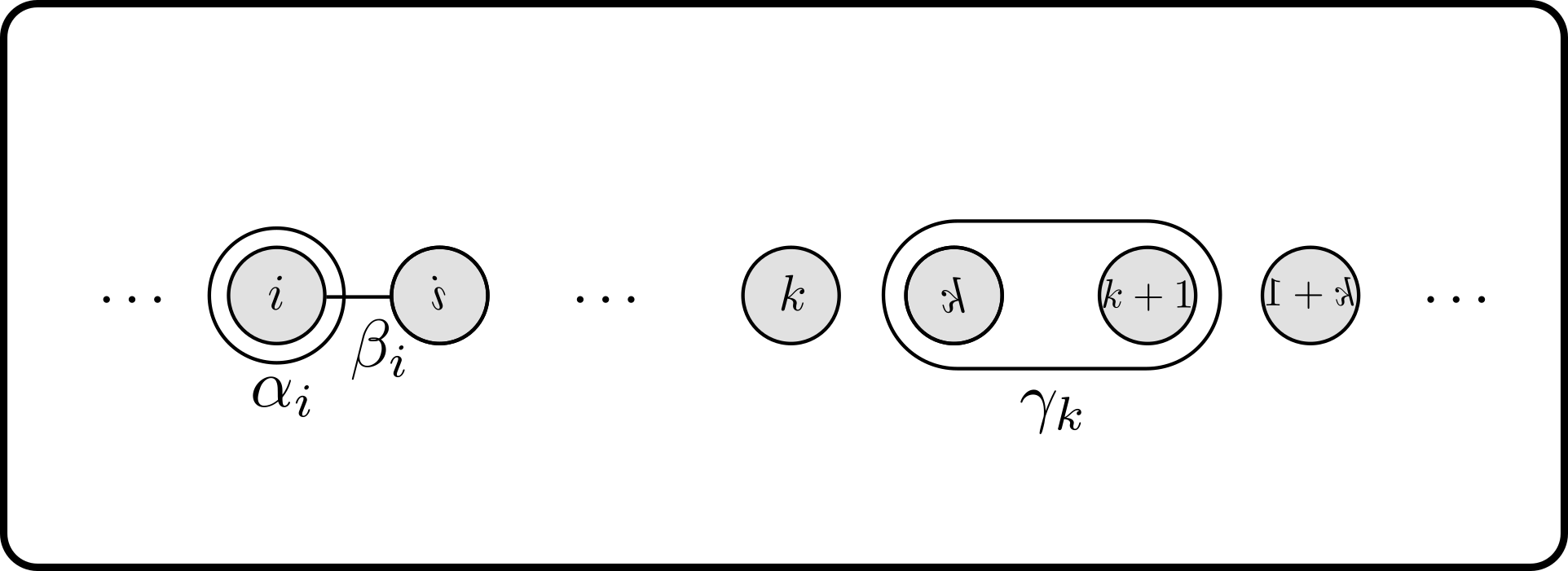}
		\caption{Generating Dehn twists.}
		\label{fig:Dehn_twists}
	\end{figure}

	For the following formula, write $\mathbf\Gamma^{\mathrm{DRM}}(\mathbf a,\mathbf b)$ for the threaded arc basis of \cite[Equation~(8)]{de2022homological}, with its cube-factor ordering from Remark~2.7 there. Move its boundary reference configuration to ours by the ordered boundary path. These tethers need not coincide with those in Fig.~\ref{fig:BM_arcs}: in particular, the beta tether reaches the descending branch on the other side of its handle. Changing between these bases uses a parallel transport between the tethers and the sign of the cube-factor permutation. It preserves each arc-label summand. For $\mathbf b=0$ the tethers agree by a deformation in the boundary collar, with the same orientation.

	\begin{theorem}[\cite{de2022homological}, Theorem 6.1]\label{thm:MJ_theorem}
		With the surface markings and tensor order of \cite{de2022homological}, there is an isomorphism of projective $\mathrm{Mod}(\Sigma_g,\partial\Sigma_{g})$ representations between the homological and quantum representations
		\begin{equation}\label{eq:MJ_isomorphism}
			\mathbf{\Phi}: \mathring{\bbH}(\Sigma_g;W_{\zeta}(\Sigma_g))\to \mathrm{ad}^{\otimes g},
		\end{equation}
		\[
			 \mathbf{\Gamma}^{\mathrm{DRM}}(\mathbf{a},\mathbf{b})\otimes \mathbf{v}_{\mathbf{c}}\mapsto N(\mathbf{a},\mathbf{b},\mathbf{c})E^{\mathbf{p-1}-\overline{\mathbf{b}}}T_{\overline{\mathbf{c}}}F^{\overline{\mathbf{a}}},
		\]
		where $\mathrm{ad}$ is the adjoint representation of $\mathfrak{u}_{\zeta}\mathfrak{sl}_2$ and
		\[
			\mathbf{k}=(k_1,\dots,k_g),\quad \overline{\mathbf{k}}=(k_g,\dots,k_1),\quad \mathbf{p-1}-\mathbf{k}=(p-1-k_1,\dots,p-1-k_g);
		\]

		\[
			N(\mathbf{a},\mathbf{b},\mathbf{c}) := \prod_{1 \le j < k \le g} \zeta^{2(a_j+b_j)(a_k+b_k)} \prod_{j=1}^g N_j(a_j,b_j,c_j),
		\]
		\[
			N_j(a_j,b_j,c_j) := \zeta^{2(a_j+b_j)(j-1)+\frac{a_j(a_j-1)}{2}+2a_jb_j-2(b_j-1)c_j}.
		\]
	\end{theorem}
	We write $\mathbf\Phi_g$ when the genus needs to be specified. In genus zero, the empty-label formula gives $\mathbf\Phi_0=\Id_{\C}$, compatible with the disk unit of Section~\ref{sec:Schroedinger_ls}. To compare with our ordered boundary sums, we will use the genus-one twist comparisons and the genus-two $\gamma$ comparison from the cited theorem, and then extend them by monoidality.
	
	In order to generate the whole category $3\Cob_*$ it is enough to add two more generating cobordisms besides mapping cylinders of generating Dehn twists. Those are integral $\lambda:\mathrm{ad}\to \C$ and cointegral $\Lambda:\C\to \mathrm{ad}$ (see Appendix~\ref{sec:ALG3}).
	\begin{remark}
		The category $3\Cob_*$ admits an algebraic presentation with finitely many generators and relations. In particular, $3\Cob_*$ is equivalent to the braided monoidal category freely generated by a Kerler Hopf algebra, see \cite{beliakova2312algebraic}, Section 3. In order to clarify how this presentation is related to the one we use in this paper, we provide explicit formulas for structure morphisms in the algebraic presentation in terms of generating Dehn twists and integrals/cointegrals. Here $\tau_{\alpha_i}$, $\tau_{\beta_i}$ and $\tau_{\gamma_k}$ denote the mapping cylinders corresponding to the positive Dehn twists in Fig.~\ref{fig:Dehn_twists}. The unindexed twists act on $\Sigma_1$, and tensor products specify their placement.
		\begin{equation*}
			\text{product:}\quad \mu= (\Id \otimes \lambda)\circ \tau_{\gamma_1}^{-1}\circ \tau_{\beta_2}^{-1}:\Sigma_2\to \Sigma_1\natural \Sigma_0=\Sigma_1,
		\end{equation*}
		\begin{equation*}
			\text{coproduct:}\quad (\tau_{\alpha_2}^{-1}\circ\tau_{\beta_2}^{-1}\circ\tau_{\alpha_2})\circ \tau_{\gamma_1}\circ ((\tau_{\beta}\circ \Lambda)\otimes (\tau_{\beta}\circ\tau_{\alpha})):\Sigma_1=\Sigma_0\natural\Sigma_1\to \Sigma_2,
		\end{equation*}
		\begin{equation*}
			\text{unit:}\quad \eta = \tau_{\alpha}\circ\tau_{\beta}\circ\tau_{\alpha}\circ \Lambda:\Sigma_0\to \Sigma_1,\quad\text{counit:}\quad \lambda\circ \tau_{\alpha}^{-1}\circ\tau_{\beta}^{-1}\circ \tau_{\alpha}^{-1}: \Sigma_1\to \Sigma_0,
		\end{equation*}
		\begin{equation*}
			\text{antipode:} \quad S=\tau_{\beta}\circ\tau_{\alpha}^2\circ \tau_{\beta}:\Sigma_1\to \Sigma_1,
		\end{equation*}
		\begin{equation*}
			\text{ribbon:}\quad v_+=\tau_{\alpha}\circ \eta=\tau_{\alpha}^2\circ\tau_{\beta}\circ\tau_{\alpha}\circ \Lambda:\Sigma_0\to \Sigma_1, \quad \text{inverse ribbon:}\quad v_-=\tau_{\alpha}^{-1}\circ \eta=\tau_{\beta}\circ\tau_{\alpha}\circ \Lambda:\Sigma_0\to \Sigma_1.
		\end{equation*}
	\end{remark}
	
	 With an appropriate normalization:
	\[
		\Lambda(1)=\sqrt{p}E^{p-1}F^{(p-1)}T_0,
	\]
	\[
		\lambda(E^{a}F^{(b)}T_c)=\frac{\zeta^{-2c}}{\sqrt{p}}\delta_{p-1,a}\delta_{p-1,b},
	\]
	Since $T_cF^{(a)}=F^{(a)}T_{c-a}$ and $N(p-1,0,-1)=\zeta^{-1}$, these maps can be transported to the basis of twisted cycles via \eqref{eq:MJ_isomorphism}. Only $b=0$ contributes, where the two tether conventions agree. For $b>0$ the integral vanishes on the entire arc-label summand. Thus the formulas also hold in our original basis:
	\begin{equation}\label{eq:cointegral_action}
		\mathbf\Phi_1^{-1}\Lambda(1)=\sqrt{p}\,N(p-1,0,-1)^{-1}\mathbf\Gamma(p-1,0)\otimes\mathbf v_{-1}=\zeta\sqrt{p}\,\mathbf\Gamma(p-1,0)\otimes\mathbf v_{-1}.
	\end{equation}
	\begin{multline}\label{eq:integral_action}
		(\lambda\circ\mathbf\Phi_1)(\mathbf\Gamma(a,b)\otimes\mathbf v_c)
		=N(a,b,c)\lambda(E^{p-1-b}F^{(a)}T_{c-a})\\
		=\frac{N(a,b,c)\zeta^{-2(c-a)}}{\sqrt p}\delta_{a,p-1}\delta_{b,0}
		=\frac{\zeta^{-1}}{\sqrt p}\delta_{a,p-1}\delta_{b,0}.
	\end{multline}
	
	\vspace{0.5cm}
	
	The cross product with coefficient tensor $\mu^0$ concatenates the $\mathbf\Gamma^{\mathrm{DRM}}$ bases. For homogeneous $x\in\mathring{\bbH}_n(\Sigma_g;\calW_p)$ and $y\in\mathring{\bbH}_m(\Sigma_h;\calW_p)$, the formula for $N$ gives
	\begin{equation}\label{eq:Phi_monoidal}
	\mathbf\Phi_{g+h}\mu^{\bbF}_{\calW_p}(x\otimes y)
	=\zeta^{2nm}\bigl(\mathbf\Phi_h(y)\otimes\mathbf\Phi_g(x)\bigr).
	\end{equation}
	Here the ratio of the $N$ factors is $\zeta^{2m(n+g)}$ and the coefficient tensor phase is $\zeta^{-2gm}$. The remaining phase is degree dependent, so it cannot be absorbed into a projective scalar.

	To correct both the phase and the tensor order put
	\[
	Q|_{\mathring{\bbH}_n}=\zeta^{-n(n-1)}\Id,
	\qquad \omega(E)=F^{(1)},\quad\omega(F^{(1)})=E,\quad\omega(K)=K^{-1},
	\qquad \vartheta=S\circ\omega.
	\]
	The map $\omega$ exchanges the raising and lowering generators and inverts $K$. It preserves the algebra relations but reverses the coproduct: $\Delta\omega=(\omega\otimes\omega)\Delta^{\mathrm{op}}$. Thus it is an algebra automorphism, rather than a Hopf algebra automorphism for the fixed coproduct. The map $\vartheta$ is an involutive algebra anti-automorphism and coalgebra automorphism. In particular,
	\[
	\vartheta(E)=-KF^{(1)},\quad \vartheta(F^{(1)})=-EK^{-1},\quad
	\vartheta(K)=K,\quad \vartheta S=S^{-1}\vartheta.
	\]
	The ribbon and integral formulas in \cite[Section~4.3 and Lemma~D.2]{de2022homological} give
	\[
	\vartheta(v)=v,\qquad \lambda\vartheta=\lambda,\qquad \vartheta\Lambda=\Lambda.
	\]
	For the first identity, $\omega$ exchanges the two expressions for $v$ in that lemma and $S(v)=v$. For the other two, $\vartheta$ exchanges the $E$ and $F$ exponents in a PBW monomial. Let $R_g$ reverse the $g$ ordinary vector-space tensor factors, and define
	\begin{equation}\label{eq:KL_monoidal_comparison}
	\mathbf\Psi_g:=\vartheta^{\otimes g}R_g\mathbf\Phi_gQ.
	\end{equation}
	Identity \eqref{eq:Phi_monoidal} and equality $\zeta^{-(n+m)(n+m-1)}\zeta^{2nm}
	=\zeta^{-n(n-1)}\zeta^{-m(m-1)}$ prove
	\begin{equation}\label{eq:KL_tensor_comparison}
	\mathbf\Psi_{g+h}\mu^{\bbF}_{\calW_p}
	=\mathbf\Psi_g\otimes\mathbf\Psi_h,\qquad \mathbf\Psi_0=\Id_{\C}.
	\end{equation}
	The reversal $R_g$ here is a vector-space permutation, it is not being used as a quantum-group intertwiner. The accompanying map $\vartheta$ and the resulting quantum action are addressed in Section~\ref{sec:sl_2_action}.

		\begin{proof}[Proof of Theorem~\ref{thm:KL_isomorphism}]
		
		We first check that the actions of Dehn twists $\tau_{\alpha_i},\tau_{\beta_i},\tau_{\gamma_i}$ coincide for $\bbF_{\calW_p}$ and for De Renzi--Martel action. Since the action of $\tilde{f}^n$ on Borel--Moore cycles naturally coincides with the action of $\bbF_{\calW_p}$ we only need to check it on the level of local systems. 
		
		\textbf{Action of $\tau_{\alpha_i}$}. It is enough to consider genus one surface $\Sigma_1$, so we omit index $i$ everywhere. Take $L_1=\Z[\alpha]$ with the section $s(k[\alpha])=(0,k[\alpha])$. Its finite image $\mathbf L_1$ is indexed by $k\in\Z/p\Z$, including zero. Since $\tau_{\alpha}$ sends $\alpha\mapsto \alpha$ and 
		$\beta\mapsto \alpha \beta$, the image of $\mathbf{1}_{L_1}$ must be the eigenvector of $(0,[\alpha])$ with eigenvalue 1, which is exactly $\mathbf{1}_{L_1}$ again (up to a scalar). We have:
		\[
			(0,[\alpha])\mapsto (0,[\alpha]),\quad(0,[\beta])\mapsto (0,[\alpha])\cdot (0,[\beta]),\quad (0,c[\beta])\mapsto (c(c+1),c[\beta])\cdot (0,c[\alpha]).
		\]
		Therefore, the image of $\mathbf{v}_c$ is:
		\[
			\phi_{\tau_{\alpha}}:(0,c[\beta])\otimes_{\tilde{\mathbf L}_1}1\mapsto(c(c+1),c[\beta])\otimes_{\tilde{\mathbf L}_1} 1=\mathbf{v}_c\zeta^{2c(c+1)}.
		\]
		On the other hand, the sum in \eqref{psi_alpha} can be rewritten (up to coefficient) as:
		\begin{equation}\label{sum}
			\sum\limits_{l=0}^{p-1}\zeta^{-2l(l-1)+4cl}=\sum\limits_{l=0}^{p-1}\zeta^{-2(l-c)(l-c-1)+2c(c+1)}\sim  \zeta^{2c(c+1)}.
		\end{equation}
		Here
		\[
			G=\sum\limits_{l=0}^{p-1}\zeta^{-2l(l-1)}
		\] 
		is non-zero. Indeed, putting $l=m+d$ and summing modulo $p$ gives
		\[
			|G|^2=\sum_{d,m\in\Z/p\Z}\zeta^{-4dm-2d(d-1)}=p,
		\]
		since $4$ is invertible modulo the odd integer $p$, and the inner sum vanishes unless $d=0$.
		
		\textbf{Action of $\tau_{\beta_i}$}. We omit index $i$ since it is enough to consider only the genus 1 case. Again, take $L_1=\Z[\alpha]$ with the same section and finite image. The twist $\tau_{\beta}$ sends $\alpha \mapsto \beta^{-1}\alpha$ and $\beta\mapsto \beta$, therefore 
		\[
			(0,[\alpha])\mapsto (0,-[\beta])\cdot(0,[\alpha]),\quad (0,[\beta])\mapsto (0,[\beta]),\quad (0,k[\alpha])\mapsto (0,-k[\beta])\cdot (-(k-1)k,k[\alpha]).
		\]
		Then the element $\mathbf{1}_{L_1}$ is sent by $\phi_{\tau_{\beta}}$ to the only non-zero vector (up to scalar) invariant under the action of $(0,-[\beta])\cdot (0,[\alpha])$:
		\[
			\sum\limits_{k=0}^{p-1}(0,-k[\beta])\cdot(-(k-1)k,k[\alpha])\otimes_{\tilde{\mathbf L}_1}1=\sum\limits_{k=0}^{p-1}\zeta^{-2k(k-1)}\mathbf{v}_{-k}
		\]
		Then $\mathbf{v}_c$ is sent to:
		\[
			\phi_{\tau_{\beta}}:\mathbf{v}_c\mapsto (0,c[\beta])\cdot \sum\limits_{k=0}^{p-1}\zeta^{-2k(k-1)}\mathbf{v}_{-k}= \sum\limits_{k=0}^{p-1}\zeta^{-2k(k-1)}\mathbf{v}_{c-k}=\sum\limits_{k'=0}^{p-1}\zeta^{-2k'(k'+1)}\mathbf{v}_{c+k'}\sim \phi(\tau_{\beta})(\mathbf{v}_c).
		\]

		\textbf{Action of $\tau_{\gamma_i}$}. It is enough to consider the genus 2 surface, so we omit $i$ in $\gamma_i$. There are only four generators $[\alpha_1]$, $[\alpha_2]$, $[\beta_1]$, $[\beta_2]$ of $H_1(\Sigma_2)$ and the Lagrangian is 
		\[
			L_2=\Z[\alpha_1]\oplus\Z[\alpha_2],\qquad
			s(k_1[\alpha_1]+k_2[\alpha_2])=(0,k_1[\alpha_1]+k_2[\alpha_2]).
		\] 
		Its finite image $\mathbf L_2$ is indexed by $k_1,k_2\in\Z/p\Z$.
		The twist $\tau_{\gamma}$ sends
		\[
		\begin{aligned}
		\alpha_1&\mapsto\alpha_1,&
		\alpha_2&\mapsto\alpha_2\beta_1^{-1}\alpha_1^{-1}\beta_1\alpha_2\beta_1^{-1}\alpha_1\beta_1\alpha_2^{-1},\\
		\beta_1&\mapsto\alpha_1\beta_1\alpha_2^{-1},&
		\beta_2&\mapsto\alpha_2\beta_1^{-1}\alpha_1^{-1}\beta_1\beta_2.
		\end{aligned}
		\]
		Hence we have:
		\[
			(0,[\alpha_1])\mapsto (0,[\alpha_1]), \quad (0,[\alpha_2])\mapsto (0,[\alpha_2])
		\]
		
		\[
			(0,[\beta_1])\mapsto (0,[\beta_1])\cdot (2,[\alpha_1]-[\alpha_2]),\quad (0,[\beta_2])\mapsto (0,[\beta_2])\cdot(0,[\alpha_2]-[\alpha_1])
		\]
		and therefore 
		\[
			(0,c_1[\beta_1])\mapsto (0,c_1[\beta_1])\cdot(c_1(c_1-1)+2c_1,c_1[\alpha_1]-c_1[\alpha_2])
		\]
		\[
			(0,c_2[\beta_2])\mapsto (0,c_2[\beta_2])\cdot(c_2(c_2-1),c_2[\alpha_2]-c_2[\alpha_1])
		\]
		\[
			(0,c_1[\beta_1]+c_2[\beta_2])\mapsto (0,c_1[\beta_1]+c_2[\beta_2])\cdot(-2c_1c_2+c_1(c_1+1)+c_2(c_2-1),(c_1-c_2)[\alpha_1]+(c_2-c_1)[\alpha_2]).
		\]
		The element $\mathbf{1}_{L_2}$ is sent by $\phi_{\tau_{\gamma}}$ to the only (up to scalar) non-zero invariant vector of $[\alpha_1]$ and $[\alpha_2]$ actions, which is again $\mathbf{1}_{L_2}$.
		Then an element $\mathbf{v}_c$ is sent to:
		\begin{multline}
			\phi_{\tau_{\gamma}}:\mathbf{v}_{c_1,c_2}\mapsto(0,c_1[\beta_1]+c_2[\beta_2])\cdot(-2c_1c_2+c_1(c_1+1)+c_2(c_2-1),(c_1-c_2)[\alpha_1]+(c_2-c_1)[\alpha_2])\otimes_{\tilde{\mathbf L}_2}1=\\
			=\zeta^{2(c_1-c_2)(c_1-c_2+1)}\mathbf{v}_{c_1,c_2}
		\end{multline}
		
		On the other hand, similarly to \eqref{sum}, the sum in \eqref{psi_gamma} can be rewritten as:
		\[
			\sum\limits_{l=0}^{p-1}\zeta^{-2l(l+1)-4lc_1+4lc_2}=\sum\limits_{l=0}^{p-1}\zeta^{-2(l+c_1-c_2)(l+1+c_1-c_2)+2(c_1-c_2)(c_1-c_2+1)}\sim \zeta^{2(c_1-c_2)(c_1-c_2+1)}.
		\]
		
		\textbf{Action of $\Lambda$.}
		The action of cointegral in the standard basis has in fact been already computed in \eqref{eq:bbF_M_1_formula} (compare with \eqref{eq:cointegral_action}):
		\begin{equation}\label{eq:index_1_Schroedinger}
			\bbF_{\calW_p}(\Lambda)(1)=2p(-\zeta^{2})^{(p-1)(p-2)/2}\mathbf\Gamma(p-1,0)\otimes\mathbf v_{-1}
			=2\zeta^{-1}(-\zeta^{2})^{(p-1)(p-2)/2}\sqrt p\,\mathbf\Phi_1^{-1}\Lambda(1).
		\end{equation}
		Indeed $\Lambda$ is an elementary index 1 cobordism $\Sigma_0\to \Sigma_1$ with the belt sphere parallel to $\alpha$ (see Appendix~\ref{sec:ALG3} for details), therefore it is exactly $M_1$ in notation of Fig.~\ref{fig:handle_cancellation} or \ref{fig:relation_3}. 
		
		\textbf{Action of $\lambda$.}
		The integral is an elementary index 2 cobordism with attaching sphere parallel to $\beta$ (see Appendix~\ref{sec:ALG3}). It is exactly $M_2$ in notation of Fig.~\ref{fig:handle_cancellation} or \ref{fig:relation_3}. Its action has already been computed in \eqref{eq:bbF_M_2_formula} (compare with \eqref{eq:integral_action}):
		\begin{multline*}
			\bbF_{\calW_p}(\lambda)(\mathbf\Gamma(a,b)\otimes\mathbf v_c)
			=\frac{(-\zeta^2)^{-(p-2)(p-1)/2}}{2}\delta_{a,p-1}\delta_{b,0}\\
			=\frac{\zeta\sqrt p}{2}(-\zeta^2)^{-(p-2)(p-1)/2}(\lambda\circ\mathbf\Phi_1)(\mathbf\Gamma(a,b)\otimes\mathbf v_c).
		\end{multline*}
		\emph{Tensor order and naturality.}
		We now pass from the preceding comparisons to $\mathbf\Psi_g$. Write $w=v^{-1}$. The quantum Dehn twist operators of \cite[Proposition~4.2]{de2022homological} are
		\[
		J_\alpha(x)=wx,\qquad
		J_\beta(x)=\lambda(v_{(2)}x)S(v_{(1)}),\qquad
		J_\gamma(x\otimes y)=xS(w_{(1)})\otimes w_{(2)}y.
		\]
		The identities for $\vartheta$ above give
		\[
		\begin{aligned}
			\vartheta J_\alpha&=J_\alpha\vartheta,\\
			\vartheta J_\beta&=J_\beta\vartheta,\\
			((\vartheta\otimes\vartheta)R_2)J_\gamma
			&=J_\gamma((\vartheta\otimes\vartheta)R_2).
		\end{aligned}
		\]
		The first follows from centrality of $w$ and
		$\vartheta(w)=w$. The second uses
		$\lambda(xy)=\lambda(yS^2(x))$ and the invariance of
		$\Delta(v)$ and $\lambda$ under $S^2$.
		For the third, $S(w)=w$ gives
		\[
		(\Id\otimes S^{-1})\tau\Delta(w)
		=(S\otimes\Id)\Delta(w).
		\]
		
		For every diffeomorphism $d$, the map $\bbF_{\calW_p}(M_d)$ preserves particle degree and therefore commutes with $Q$. The preceding comparisons in genera one and two, together with these identities and tensor compatibility \eqref{eq:KL_tensor_comparison}, thus prove naturality of $\mathbf\Psi_g$ for the Dehn twist generators in Fig.~\ref{fig:Dehn_twists}, and hence for all mapping cylinders.
		
		For the genus-one handle maps, only particle degree $p-1$ contributes, where $Q=\zeta^{-2}\Id$. Since $\vartheta\Lambda=\Lambda$ and $\lambda\vartheta=\lambda$, the computations above give
		\[
		\mathbf\Psi_1\bbF_{\calW_p}(\Lambda)\sim\Lambda,
		\qquad
		\bbF_{\calW_p}(\lambda)\sim\lambda\mathbf\Psi_1.
		\]
		Tensor compatibility and the two locality identities extend these comparisons to boundary connected sums with the identity cobordism on any surface, on either side. Naturality for mapping cylinders then gives the comparison for handles with arbitrary attaching tubes.
		
		Since mapping cylinders and elementary handles generate the cobordism category, $\mathbf\Psi$ is a projective monoidal natural isomorphism. For each cobordism, the projective scalar is independent of particle degree.
	\end{proof}
	

	\subsection{$\mathfrak{u}_{\zeta}\mathfrak{sl}_2$-module structure}\label{sec:sl_2_action}
	
	We first recall the operators $E,F^{(1)},K$ of \cite{de2022homological}. The action compatible with our tensor normalization is obtained from them below. Define $K$ on small cycles as the graded diagonal operator:
	\begin{equation}\label{eq:K_def}
		K:\mathring{\bbH}_n(\Sigma;\calW_p(\Sigma))\to \mathring{\bbH}_n(\Sigma;\calW_p(\Sigma)),
	\end{equation}
	\[
		x\mapsto x \cdot \zeta^{-2(n+g)},
	\]
	where $g$ is the genus of $\Sigma$.
	
	In order to define the action of $F^{(1)}$ consider a strip $(I\times I, \{0\}\times I\sqcup \{1\}\times I)$. Its relative one-dimensional generator is represented by the oriented core arc in Fig.~\ref{fig:kappa}. Insert this class along $\partial_+\Sigma$, after pushing the old surface away from a collar containing the strip. Choose the two embeddings with disjoint images and with their relative boundary arcs in $\partial_-\Sigma$. The union map on configurations is then a proper map of pairs. Precisely speaking the action of $F^{(1)}$ is defined via the composition:
	\begin{equation}\label{eq:F_BM_def}
		\begin{adjustbox}{max width=0.96\linewidth}
		\begin{tikzcd}
			H^{BM}_n(\Conf_n(\Sigma),\Conf_n^-(\Sigma); \calW_p(\Sigma))\arrow[d,"\Id\otimes 1"]\arrow[dddd,rounded corners=3pt , to path={-- ++(-8cm,0)  |-  node [left, near start] {$F^{BM}_-$} (\tikztotarget)}]\\
			H^{BM}_n(\Conf_n(\Sigma),\Conf_n^-(\Sigma); \calW_p(\Sigma))\otimes H^{BM}_1(I\times I,\{0\}\times I\sqcup \{1\}\times I;\C)\arrow[d,"\text{K\"unneth}"]\\
			H^{BM}_{n+1}(\Conf_n(\Sigma)\times I\times I,\Conf_n^-(\Sigma)\times I\times I\cup\Conf_n(\Sigma)\times(\{0\}\times I\sqcup \{1\}\times I);\calW_p(\Sigma)\boxtimes \C)\arrow[d,"\Id_*"]\\
				H^{BM}_{n+1}(\Conf_n(\Sigma)\times I\times I,\Conf_n^-(\Sigma)\times I\times I\cup\Conf_n(\Sigma)\times(\{0\}\times I\sqcup \{1\}\times I);\kappa_-^*\calW_p(\Sigma))\arrow[d,"(\kappa_-)_*"]\\
				H^{BM}_{n+1}(\Conf_{n+1}(\Sigma),\Conf_{n+1}^-(\Sigma);\calW_p(\Sigma))
		\end{tikzcd}
		\end{adjustbox}
	\end{equation}
	where $\Id:\calW_p(\Sigma)\boxtimes\C\to\kappa_-^*\calW_p(\Sigma)$ sends $x\boxtimes1$ to $x$ in the common auxiliary Schr\"odinger module. This is equivariant: the inserted point stays in the boundary collar, disjoint from the old braid generators. With the two embeddings understood, $\kappa_-$ is the union map in Fig.~\ref{fig:kappa}:
	\[
		\kappa_-:\Conf_n(\Sigma)\times I\times I\to \Conf_{n+1}(\Sigma)
	\]
	\[
		\mathbf{x}\times y\mapsto \mathbf{x}\cup y.
	\]
	Define:
	\begin{equation}\label{eq:F_1_def}
		F^{(1)}:\bbH^{BM}_n(\Sigma;\calW_p(\Sigma))\to \bbH^{BM}_{n+1}(\Sigma;\calW_p(\Sigma)),
	\end{equation}
	\[
		x\mapsto F^{BM}_-(x)\cdot \zeta^{2g},
	\]
	where $g$ is the genus of $\Sigma$. The same insertion is defined in ordinary homology, and its cross product, coefficient map, and push-forward commute with the ordinary-to-Borel--Moore map $\iota$. Hence $F^{(1)}$ preserves $\mathrm{Im}(\iota)=\mathring{\bbH}(\Sigma;\calW_p(\Sigma))$. This is the operator of \cite[Definition~2.15]{de2022homological} for $k=1$.

	The action of $E$ can be defined in a similar way, see the right hand side of Fig.~\ref{fig:kappa}. Let 
	\[
		F_+:H_{n-1}(\Conf_{n-1}(\Sigma),\Conf_{n-1}^+(\Sigma);\overline{\calW_p(\Sigma)})\to H_n(\Conf_n(\Sigma),\Conf_n^+(\Sigma);\overline{\calW_p(\Sigma)}),\quad n\geq1,
	\]
	 be the map given by inserting the relative core class on $\partial_-\Sigma$ similarly to \eqref{eq:F_BM_def}, with the conjugate coefficient system. The ordinary/Borel--Moore pairing of Section~\ref{sec:pairing} defines its adjoint degreewise. This is the intrinsic pairing with $\overline{\calW_p(\Sigma)}$, and we have:
	 \begin{equation}\label{eq:E_def}
	 	E:=F_+^{\vee}:\bbH^{BM}_n(\Sigma;\calW_p(\Sigma))\to\bbH^{BM}_{n-1}(\Sigma;\calW_p(\Sigma)),\quad n\geq1,
	 \end{equation}
	with $E=0$ in degree zero. Explicitly,
	\[
		\langle Ex,y\rangle=\langle x,F_+y\rangle,
	\]
	for $x\in\bbH^{BM}_n(\Sigma;\calW_p(\Sigma))$ and
	\[
	 y\in H_{n-1}(\Conf_{n-1}(\Sigma),\Conf_{n-1}^+(\Sigma);\overline{\calW_p(\Sigma)}).
	\]
	Locally, intersecting with the core inserted by $F_+$ takes the boundary where one point reaches $\partial_-\Sigma$, and then deletes that point. The normalization of Section~\ref{sec:pairing} multiplies its adjoint by
	\[
	\frac{(-1)^{n(n-1)/2}}{(-1)^{(n-1)(n-2)/2}}=(-1)^{n-1}.
	\]
	Thus the displayed adjoint is exactly the signed boundary-and-deletion operator of \cite[Definition~2.14]{de2022homological}, with the core orientation in Fig.~\ref{fig:kappa}. This operation only lowers arc labels, so it preserves small cycles by Section~\ref{sec:small_cycles}. These definitions also include the full degree-zero fiber $W_\zeta(\Sigma)$. On the disk state space, $E=F^{(1)}=0$ and $K=1$.
	
		\begin{figure}[h]
		\centering
		\includegraphics[width=0.8\linewidth]{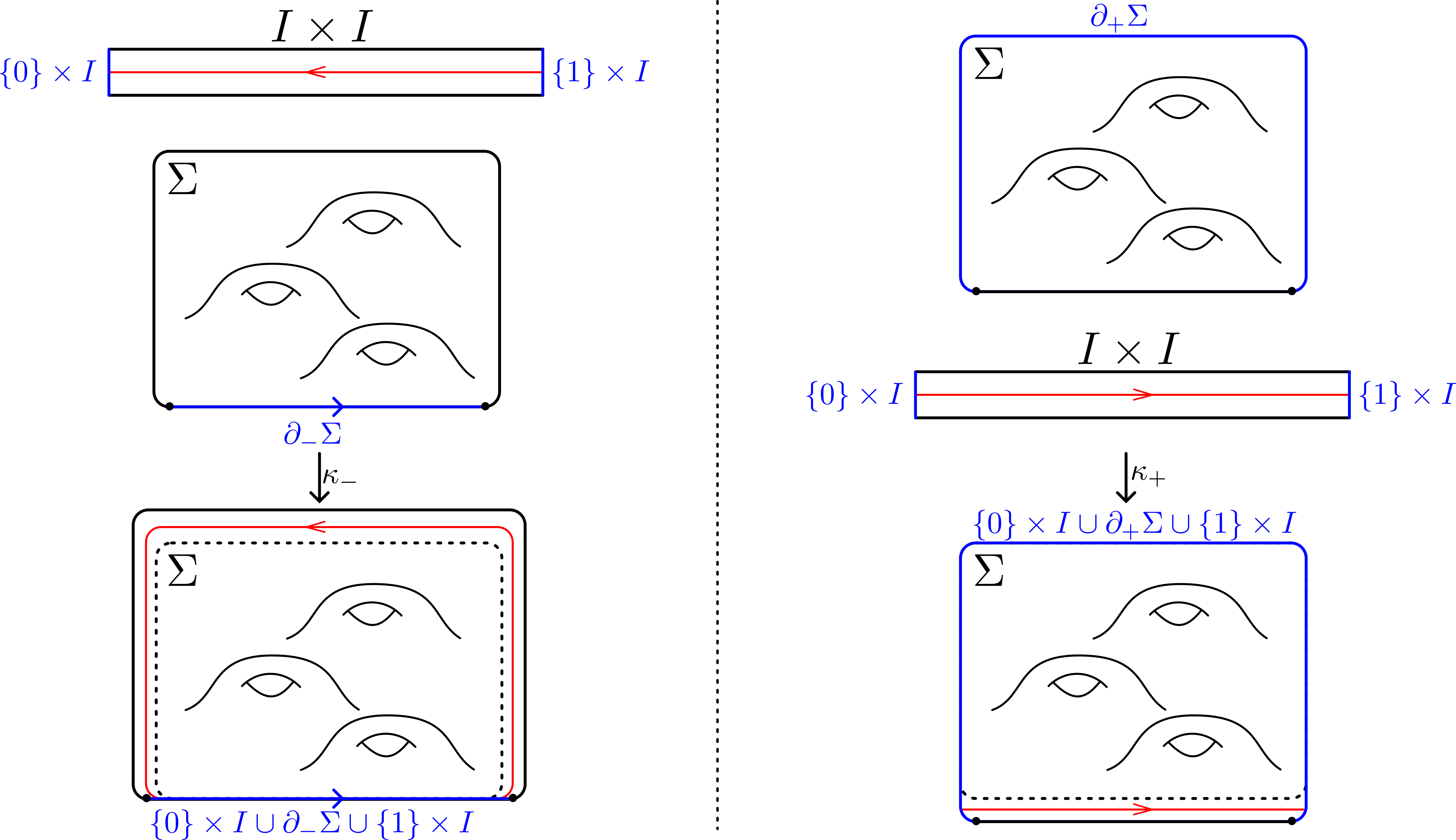}
		\caption{The strips used to define the $\mathfrak{u}_{\zeta}\mathfrak{sl}_2$-action on small cycles by inserting relative core classes along $\partial_+\Sigma$ and $\partial_-\Sigma$. The right-hand insertion uses ordinary homology with conjugate coefficients; its adjoint lowers particle degree.}
		\label{fig:kappa}
	\end{figure}
	
	\begin{lemma}[\cite{de2022homological}, Theorems~2.20, 2.32 and 6.1]
		Linear maps \eqref{eq:K_def}, \eqref{eq:F_1_def} and \eqref{eq:E_def} restricted to small cycles form a well defined action of $\mathfrak{u}_{\zeta}\mathfrak{sl}_2$.
	\end{lemma}

	For the homological TQFT, use instead the action
	\begin{equation}\label{eq:normalized_quantum_action}
	\widehat E:=Q^{-1}F^{(1)}Q,\qquad
	\widehat F^{(1)}:=Q^{-1}EQ,\qquad
	\widehat K:=K^{-1}.
	\end{equation}
	Thus, on particle degree $n$,
	\[
	\widehat E=\zeta^{2n}F^{(1)},\qquad
	\widehat F^{(1)}=\zeta^{-2(n-1)}E,\qquad
	\widehat K=\zeta^{2(n+g)}\Id.
	\]
	These operators satisfy the same small-quantum-group relations, by the automorphism $\omega$ and conjugation by $Q$. They preserve small cycles and give the trivial action on the disk. They are exactly the action transported by $\mathbf\Psi_g$: the identities
	\[
	\vartheta\bigl(\mathrm{ad}_{\omega(u)}x\bigr)=\mathrm{ad}_u\vartheta(x),
	\qquad \Delta\omega=(\omega\otimes\omega)\Delta^{\mathrm{op}}
	\]
	show that $\vartheta^{\otimes g}R_g$ intertwines the $\omega$-transformed action with the usual tensor power of $\mathrm{ad}$. Consequently $\mathbf\Psi_g$ is a module isomorphism for \eqref{eq:normalized_quantum_action}. Together with \eqref{eq:KL_tensor_comparison}, this proves that the homological tensor map intertwines the coproduct fixed above. In particular,
	\[
	\widehat E\mu^{\bbF}_{\calW_p}
	=\mu^{\bbF}_{\calW_p}(\widehat E\otimes\widehat K+1\otimes\widehat E),\qquad
	\widehat F^{(1)}\mu^{\bbF}_{\calW_p}
	=\mu^{\bbF}_{\calW_p}(\widehat K^{-1}\otimes\widehat F^{(1)}+\widehat F^{(1)}\otimes1).
	\]

	The normalized action \eqref{eq:normalized_quantum_action} is intertwined by the action of cobordisms.
	
	\begin{proposition}
	 	For each $M\in \mathrm{Mor}(3\Cob)$, the linear map $\bbF_{\calW_p}(M)$ is a morphism of $\mathfrak{u}_{\zeta}\mathfrak{sl}_2$-modules with respect to \eqref{eq:normalized_quantum_action}.
	\end{proposition}
	
	\begin{proof}
		This is a simple fact following from the construction of homological TQFT. For any elementary cobordism $M:\Sigma_-\to \Sigma_+$ homology class insertions $F_-^{BM}$ and $F_+$ can be defined in the same manner on the trajectory space $\calT_M$. It intertwines both $i_-$ and $i_+$ (or $\varphi$, $\overline{\varphi}$) since all of these maps send $\partial_-\calT_M$ (resp $\partial_+\calT_M$) to $\partial_-\Sigma_-$ (resp. $\partial_+\Sigma_-$) or $\partial_-\Sigma_+$ (resp. $\partial_+\Sigma_+$).  
		
		To check the normalization in \eqref{eq:normalized_quantum_action}, it suffices by locality and the coproduct formulas above to consider the standard handles between the disk and the genus-one surface. The disk operators \(E\) and \(F^{(1)}\) vanish, so the geometric intertwining identities for these handles are vanishing identities and are preserved by the degreewise rescalings in \eqref{eq:normalized_quantum_action}. On the genus-one side the handle maps are supported in degree \(p-1\), where \(\widehat K=\zeta^{2((p-1)+1)}=1\), as on the disk. Mapping cylinders preserve genus and particle degree, hence also the normalization.
	\end{proof}
	
	
	The target therefore improves to the projective category of $\mathfrak{u}_{\zeta}\mathfrak{sl}_2$-modules with the normalized action. The projective Kerler--Lyubashenko functor is braided, with the braiding represented by the mapping cylinder of $d_{\beta}$ from Section~\ref{subs:Cobordism_categories}. Naturality of $\mathbf\Psi$ for this cylinder and \eqref{eq:KL_tensor_comparison} identify the homological braiding with the quantum braiding on the tensor powers of $\mathrm{ad}$. This proves braided compatibility using the same comparison maps.

\section{Discussion}

This paper provides an alternative interpretation of quantum invariants of $3$-manifolds. We conclude with a brief discussion of possible further developments.
\begin{enumerate}
	\item The homological construction presented here should generalize to arbitrary even values of $p$. As often happens with quantum invariants, the even case is slightly more involved. We expect that a projective homological TQFT for even $p$ can be constructed in a way very similar to the odd case. A full TQFT, however, should require additional structure on the cobordism category, for instance a spin structure.
	\item In this paper Schr\"odinger local systems produce only a projective homological TQFT. We believe, however, that by working with the category $3\Cob_*^{\Z}$ it should be possible to define a genuine action yielding a full TQFT isomorphic to the Kerler--Lyubashenko one.
	\item As one sees in Section~\ref{sec:general_construction}, working with finite Heisenberg groups and local systems with a vanishing quantum factorial is needed in order to ``invert'' one of the arrows in the diagram \eqref{eq:homological_index_1_span} after restricting to small cycles: the insertion is independent of the chosen small-cycle lift through the left arrow. One can replace the target category of the TQFT, $\mathrm{Mod}_R$, by a category of spans in $\mathrm{Mod}_R$ (or perhaps correspondences in $\mathrm{Mod}_R$) in order to avoid this issue. This seems very natural, since such categories contain $\mathrm{Mod}_R$ as a subcategory and the duality is less restrictive, \ie $R$-modules may have infinite rank. This approach appears to be much more flexible and has at least two potential advantages: 
	\begin{itemize}
		\item it allows one to work with Verma modules, constructed homologically in \cite{martel2022homological}, which are much larger than the subspaces of small cycles;
		\item it allows one to work with all possible parameters $p$, which can probably be interpreted as a way of working with a generic quantum parameter $q$.
	\end{itemize}
	\item One can expect that this construction should generalize to surfaces with an arbitrary number of labeled punctures, reproducing non-semisimple TQFTs in a more general setting.
	\item The $3$-dimensional Kerler--Lyubashenko functor is part of a $4$-dimensional functor on $2$-handle\-bodies. It is possible that the homological construction extends naturally to the $4$-dimensional setting.
\end{enumerate}

	\section{Appendix}\label{section:appendix}

	\subsection{Frohman--Nicas--Donaldson TQFT on $\mathbf{3}\mathbf{Cob}$}\label{sec:Donaldson}

		We use Juh\'asz's presentation to formulate the Frohman--Nicas--Donaldson TQFT.

	\noindent\textbf{Step 1. State spaces.}
	Let $\Bbbk$ be a field. Let $\Sigma\in 3\Cob$. The state space $D(\Sigma)$ is defined as the homology of the symmetric product $\Sym^n(\Sigma)$ for sufficiently large $n$:
	\[
	D(\Sigma):=\bigoplus_{k=0}^{2g}H_k(\Sym^n(\Sigma);\Bbbk)\simeq \bigoplus_{k=0}^{2g}\bigwedge^k H_1(\Sigma;\Bbbk), \quad n\geq 2g
	\]
	where $g$ is the genus of $\Sigma$.
	
	Note that there is a canonical isomorphism 
	\[
		H_1(\Sigma;\Bbbk)\xrightarrow{\simeq}H_1(\Sigma;\partial_-\Sigma;\Bbbk).
	\]
	Therefore $D(\Sigma)$ can be viewed in a slightly different way, more suitable for comparison with homological TQFTs:
	\begin{equation}\label{eq:Donaldson_decomposition}
		D(\Sigma)\simeq \bigoplus\limits_{k=0}^{2g}\bigwedge^k H_1(\Sigma;\partial_-\Sigma;\Bbbk)\simeq \bigoplus\limits_{n=0}^{2g}H_n(\Sym^n(\Sigma),\Sym^n_-(\Sigma);\Bbbk),
	\end{equation}
	where $\Sym^n_-(\Sigma)$ is the subset of configurations with at least one point on $\partial_-\Sigma$, and $\Sym^0_-(\Sigma)=\varnothing$. Indeed, for $n>0$ collapsing $\partial_-\Sigma$ identifies the relative quotient with
	\[
	\Sym^n(\Sigma/\partial_-\Sigma)/\Sym^{n-1}(\Sigma/\partial_-\Sigma),
	\]
	where the denominator is included by adding the basepoint. The space $\Sigma/\partial_-\Sigma$ is homotopy equivalent to a bouquet of $2g$ circles. The compatible symmetric-product models of \cite[Lemmas~2--3]{kallel2006symmetric} give the $n$-skeleton of a $2g$-torus relative to its $(n-1)$-skeleton. Its relative cellular complex is concentrated in degree $n$, with basis the products of $n$ distinct one-cells. These products give the natural exterior-power identification in \eqref{eq:Donaldson_decomposition}. The $n=0$ term is $\Bbbk$, so the disk has state space $\Bbbk$. The same argument applies to trajectory surfaces, with $2g+2$ circles.

	This in particular gives a way to describe a $\Bbbk$-basis of $D(\Sigma)$; see Fig.~\ref{fig:Donaldson_basis}. For this we choose a symplectic basis $\{[\alpha_i],[\beta_i]\}_{i=1}^{g}$ of $H_1(\Sigma;\Bbbk)$ with respect to the intersection pairing. We draw each $[\alpha_i]$ and $[\beta_i]$ as particular cycles relative to $\partial_-$ identifying $H_1(\Sigma;\Bbbk)\simeq H_1(\Sigma,\partial_-\Sigma;\Bbbk)$. Fix an ordering on this basis, then a basis of $H_n(\Sym^n(\Sigma),\Sym^n_-(\Sigma);\Bbbk)$ is given by classes (compare with a basis description of Heisenberg homology via symplectic arcs in Section~\ref{sec:Heisenberg_homology}):
	\[
		[\alpha_1]^{a_1}\wedge[\beta_1]^{b_1}\wedge\dots \wedge 	[\alpha_g]^{a_g}\wedge[\beta_g]^{b_g}
	\]
	via the isomorphism \eqref{eq:Donaldson_decomposition}, where $a_1,b_1,\dots,a_g,b_g\in\{0,1\}$ and $\sum_i(a_i+b_i)=n$ ($0$ means absence of the corresponding class in the product). In what follows we only write numbers $a_i,b_i$ near the corresponding classes in pictures.
	\begin{figure}[h]
		\centering
		\includegraphics[width=0.4\linewidth]{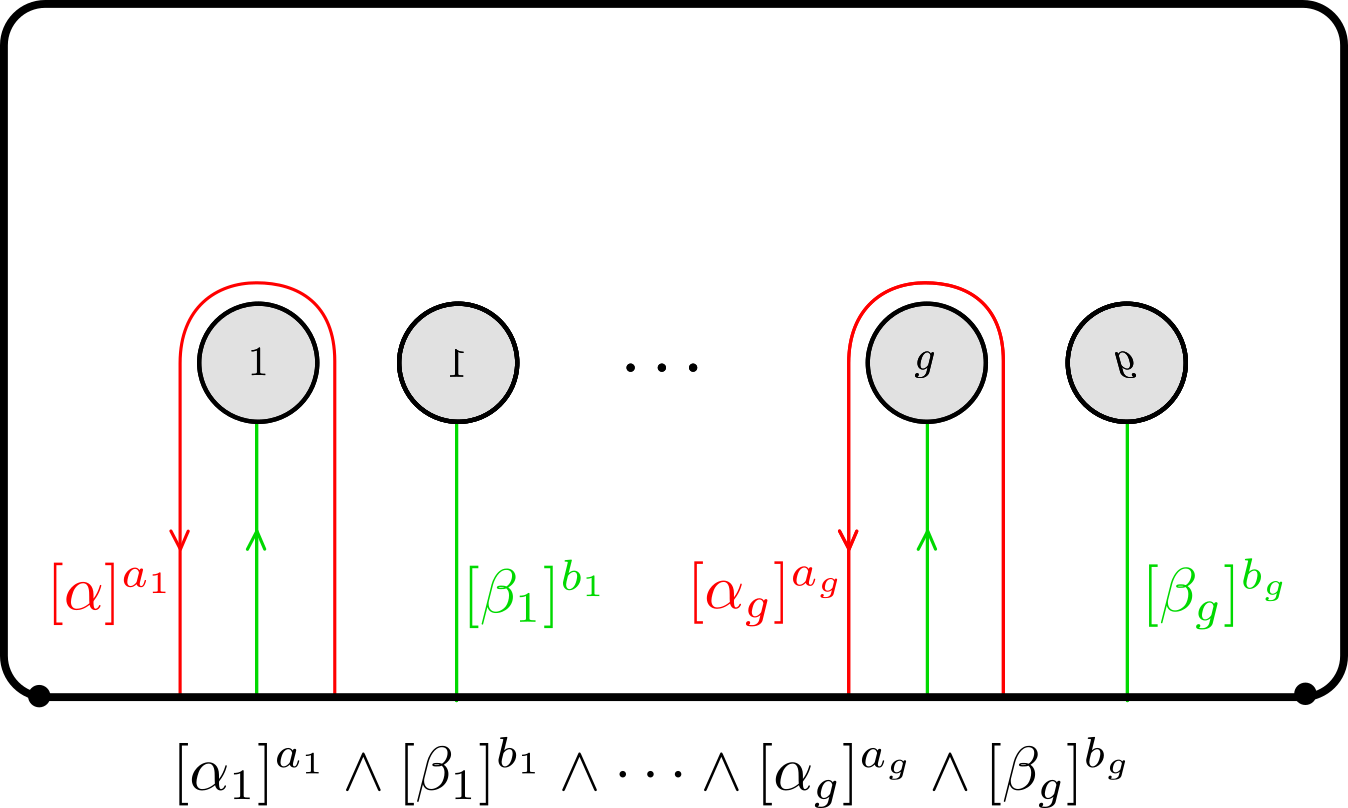}
		\caption{A picture representing a basis element in $H_n(\Sym^n(\Sigma),\Sym^n_-(\Sigma))$, where $n=\sum_i(a_i+b_i)$ and all labels are $0$ or $1$. Each arc corresponds to a homology class in $H_1(\Sigma,\partial_-\Sigma)$ and represents a cycle for that class.}
		\label{fig:Donaldson_basis}
	\end{figure}
	
	For boundary connected sum, the direct-sum decomposition of first homology gives the identification
	\begin{equation}\label{eq:Donaldson_tensorator}
	D(\Sigma_1)\otimes D(\Sigma_2)\longrightarrow D(\Sigma_1\natural\Sigma_2),
	\qquad x\otimes y\longmapsto(-1)^{g_1n_2}x\wedge y,
	\end{equation}
	where $g_1$ is the genus of $\Sigma_1$ and $y$ has exterior degree $n_2$. This sign convention agrees with Section~\ref{sec:monoidality} for $p=2$. Associativity follows from $g_1n_2+(g_1+g_2)n_3=g_2n_3+g_1(n_2+n_3)$, and the disk unit is $1\in\Bbbk$.

	\noindent\textbf{Step 2. Action of elementary cobordisms.}
	
	We formulate the action of elementary cobordisms in terms of their trajectory spaces.

	\begin{enumerate}
		\item  \emph{Action of mapping cylinders}. Let $M_d:\Sigma\to \Sigma'$ be a mapping cylinder given by a diffeomorphism $d:\Sigma\to \Sigma'$. It induces a map between symmetric products $d^{\times n}:\Sym^n(\Sigma)\to \Sym^n(\Sigma')$ and their homology:
		\[
		D(M_d)|_{H_n}:=(d^{\times n})_*:H_n(\Sym^n(\Sigma),\Sym^n_-(\Sigma);\Bbbk)\to H_n(\Sym^n(\Sigma'),\Sym^n_-(\Sigma');\Bbbk). 
		\]
		\item \emph{Action of index 1 cobordisms}. Let $M:\Sigma_-\to \Sigma_+$ be an elementary cobordism of index 1 with belt sphere $\nu:S^1\to \Sigma_+$ and trajectory space $\calT_M$. The trajectory space $\calT_M$ together with the maps $i_-$ and $i_+$ induce the diagram:
		\[
		\begin{tikzcd}
			&\Sym^n(\calT_M)\times S^1\arrow[dl,"i_-^{\times n}\circ \pi"']\arrow[dr,"i_+^{\times n}\times \nu"]&\\
			\Sym^{n}(\Sigma_-)&&\Sym^{n+1}(\Sigma_+),
		\end{tikzcd}
		\]
		where $\pi:\Sym^n(\calT_M)\times S^1\to \Sym^n(\calT_M)$ is the projection and the right arrow takes the union with $\nu(s)$. Define $\phi$ as the composition (compare with \eqref{eq:Psi_M_def}):
			\begin{equation}
			\begin{tikzcd}
				H_n(\Sym^{n}(\calT_M),\Sym^{n}_-(\calT_M);\Bbbk)\arrow[d,"{\Id\otimes [S^1]}"]\arrow[ddd,rounded corners=3pt , to path={-- ++(-4cm,0)  |-  node [left, near start] {$\phi$} (\tikztotarget)}]\\
				H_n(\Sym^n(\calT_M),\Sym^{n}_-(\calT_M);\Bbbk)\otimes H_1(S^1;\Bbbk)\arrow[d,"{\text{K\"unneth}}"]\\
				H_{n+1}(\Sym^n(\calT_M)\times S^1; \Sym^{n}_-(\calT_M)\times S^1;\Bbbk)\arrow[d,"{(i_+^{\times n}\times \nu)_*}"]\\
				H_{n+1}(\Sym^{n+1}(\Sigma_+),\Sym^{n+1}_-(\Sigma_+);\Bbbk)
			\end{tikzcd}
		\end{equation}
		We consider the following diagram:
		\begin{equation}\label{eq:Donaldson_index_1_span}
		\adjustbox{max width=0.96\textwidth}{$
			\begin{tikzcd}
				&H_n(\Sym^n(\calT_M),\Sym^n_-(\calT_M);\Bbbk)\arrow[dl,"{(i_-^{\times n})_*}"']\arrow[dr,"\phi"]&\\
				H_n(\Sym^n(\Sigma_-),\Sym_-^n(\Sigma_-);\Bbbk)\arrow[rr,"\exists !","D(M)"',blue]&&H_{n+1}(\Sym^{n+1}(\Sigma_+),\Sym^{n+1}_-(\Sigma_+);\Bbbk).
			\end{tikzcd}
		$}
		\end{equation}
		We can ``invert'' the left arrow in \eqref{eq:Donaldson_index_1_span} in the following sense: it is surjective and $\Ker (i_-^{\times n})_*\subset\Ker\phi$ (see Fig.~\ref{fig:Donaldson_span}). Indeed, the map $(i_-)_*$ on first homology is surjective with kernel generated by the two puncture classes, whose images under $(i_+)_*$ are multiples of $[\nu]$. The kernel on exterior algebras is the ideal generated by these classes, so multiplication on the right by $[\nu]$ annihilates its image. Hence there exists a unique horizontal map making the diagram commutative. This defines $D(M)$ up to the sign of the orientation of $\nu$. 
		
		\begin{figure}[h]
			\centering
			\includegraphics[width=0.9\linewidth]{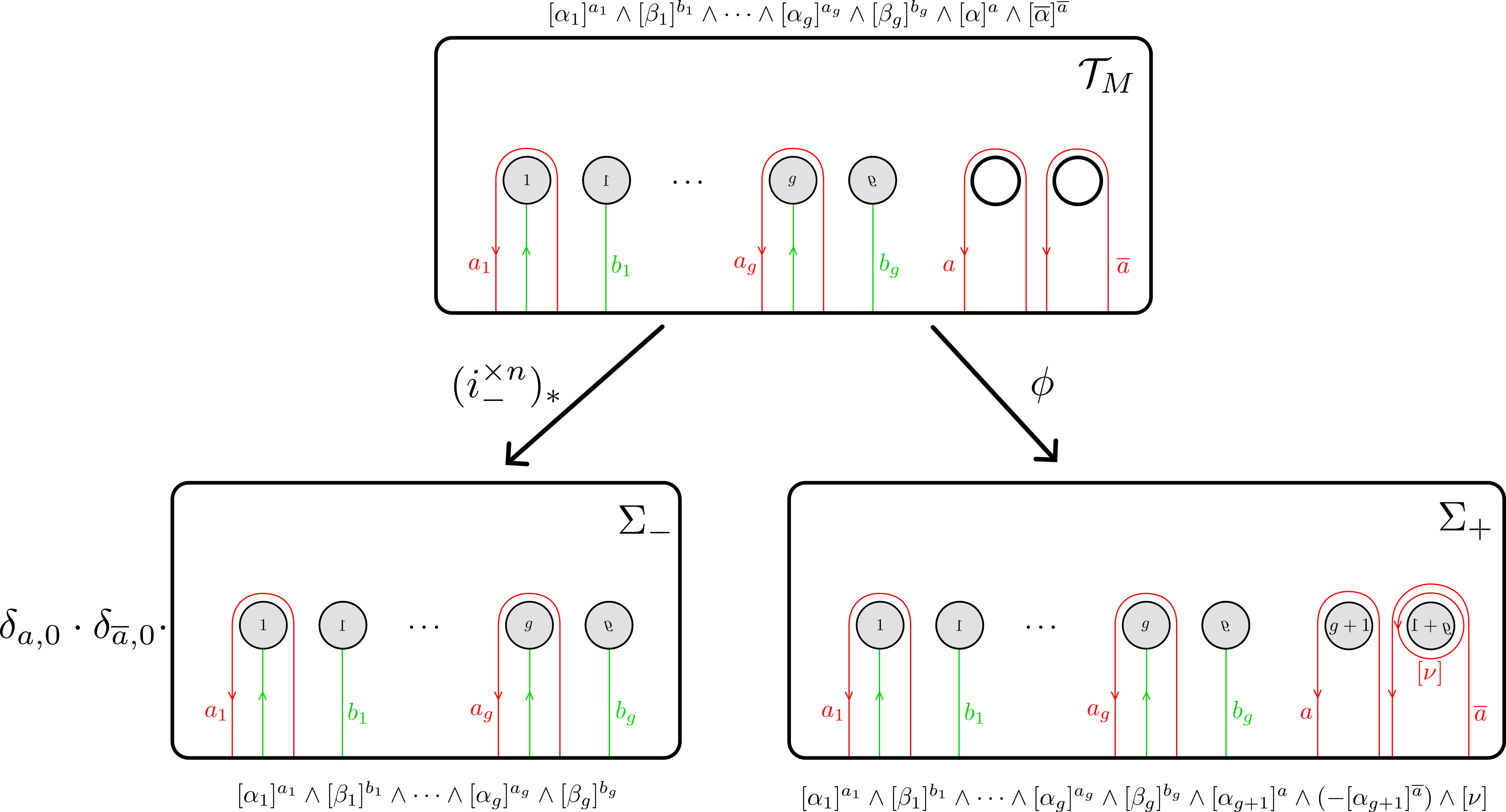}
			\caption{Action of $(i_-^{\times n})_*$ and $\phi$ on a basis element in $H_n(\Sym^n(\calT_M),\Sym^n_-(\calT_M);\Bbbk)$, with all labels in $\{0,1\}$ and their sum equal to $n$. This element is sent to 0 by $(i_-^{\times n})_*$ if and only if $a+\overline{a}>0$. On the other hand, the inserted homology class $[\nu]$ of the meridian (via $\phi$) contributes as multiplication $\wedge[\nu]=\wedge(-[\alpha_{g+1}])$. If $a+\overline{a}>0$ the image under $\phi$ vanishes since the images of the corresponding cycles contribute as $\wedge [\alpha_{g+1}]^{a}\wedge (-[\alpha_{g+1}]^{\overline{a}})$.}
			\label{fig:Donaldson_span}
		\end{figure}
		
		\item \emph{Action of index 2 cobordisms}. Let $M':\Sigma_-'\to \Sigma_+'$ be an index 2 cobordism. We use duality in order to define its action as follows. Exchanging its ends and reversing the orientations of the boundary surfaces, while retaining the orientation of the three-manifold, gives the opposite index 1 cobordism $\overline{M}':\overline{\Sigma}'_+\to \overline{\Sigma}'_-$. Hence $D\left(\overline{M}'\right):D(\overline{\Sigma}_+')\to D(\overline{\Sigma}_-')$ is defined as above. If for each surface $\Sigma\in 3\Cob$ the state space $D(\Sigma)$ is equipped with a perfect bilinear pairing
		\[
			\langle-,-\rangle:D(\Sigma)\otimes D(\overline{\Sigma})\to \Bbbk,
		\]
		then by identifying $\mathrm{Hom}(D(\Sigma),\Bbbk)$ with $D(\overline{\Sigma})$ we obtain that $D\left(\overline{M}'\right)$ induces the dual homomorphism
		\[
		D(M'):=D\left(\overline{M}'\right)^{\vee}:D(\Sigma_-')\to D(\Sigma_+').
		\]
		where the dual is taken with respect to the pairing.
		
		In particular, for each $\Sigma\in 3\Cob$ of genus $g$, there is a perfect pairing:
		\begin{equation}\label{Donaldson_pairing}
			\langle-,-\rangle:\bigwedge^{k} H_1(\Sigma;\Bbbk)\otimes	\bigwedge^{2g-k} H_1(\Sigma;\Bbbk)\xrightarrow{(-1)^{k(k-1)/2}\wedge}\bigwedge^{2g}H_1(\Sigma;\Bbbk)\simeq \Bbbk.
		\end{equation}
		The last isomorphism sends $[\alpha_1]\wedge[\beta_1]\wedge\dots\wedge[\alpha_g]\wedge[\beta_g]$ to $1$. The factor $(-1)^{k(k-1)/2}$ reverses the order of the first exterior factor. In the chosen basis,
			\begin{multline}\label{eq:explicit_duality}
			\left\langle[\alpha_1]^{a_1}\wedge[\beta_1]^{b_1}\wedge\dots\wedge [\alpha_g]^{a_g}\wedge [\beta_g]^{b_g},[\alpha_1]^{a_1'}\wedge[\beta_1]^{b_1'}\wedge\dots\wedge [\alpha_g]^{a_g'}\wedge [\beta_g]^{b_g'}\right\rangle=\\
			=\pm \delta_{a_1,1-a_1'}\delta_{b_1,1-b_1'}\dots\delta_{a_g,1-a_g'}\delta_{b_g,1-b_g'}.
		\end{multline}
		Identifying $H_1(\Sigma;\Bbbk)=H_1(\overline{\Sigma};\Bbbk)$ and pairing noncomplementary degrees by zero gives the required perfect pairing $D(\Sigma)\otimes D(\overline{\Sigma})\to\Bbbk$.
	\end{enumerate}

	\noindent\textbf{Step 3. Checking the relations.}
	By Fig.~\ref{fig:Donaldson_span} the action of an index 1 cobordism in standard basis is given by the formula:
	\begin{equation}\label{eq:Donaldson_index_1}
		[\alpha_1]^{a_1}\wedge [\beta_1]^{b_1}\wedge\dots\wedge [\alpha_g]^{a_g}\wedge [\beta_g]^{b_g}\mapsto \pm[\alpha_1]^{a_1}\wedge [\beta_1]^{b_1}\wedge\dots\wedge [\alpha_g]^{a_g}\wedge [\beta_g]^{b_g}\wedge[\alpha_{g+1}],
	\end{equation}
	where $\pm$ ambiguity appears because flipping orientation of the meridian changes the sign of the result.

	Using duality \eqref{eq:explicit_duality} and formula \eqref{eq:Donaldson_index_1} one can derive a formula for the action of an index 2 cobordism in the standard basis:
	\begin{multline}\label{eq:Donaldson_index_2}
		[\alpha_1]^{a_1}\wedge[\beta_1]^{b_1}\wedge\dots\wedge [\alpha_{g+1}]^{a_{g+1}}\wedge [\beta_{g+1}]^{b_{g+1}}\mapsto\\
		\mapsto\pm [\alpha_1]^{a_1}\wedge[\beta_1]^{b_1}\wedge\dots\wedge [\alpha_{g}]^{a_g}\wedge [\beta_{g}]^{b_g} \cdot \delta_{a_{g+1},0}\delta_{b_{g+1},1}.
	\end{multline}
	
	For the sign in the index 2 formula, let $u$ and $v$ be old exterior monomials of degrees $k$ and $2g-k$. Then
	\[
	\left\langle u\wedge[\beta_{g+1}],v\wedge[\alpha_{g+1}]\right\rangle=-\langle u,v\rangle.
	\]
	Indeed, the ratio of the degree normalizations is $(-1)^k$, while moving the last-handle factors together contributes $(-1)^{k+1}$. All other last-handle monomials pair trivially with $v\wedge[\alpha_{g+1}]$. Thus the sign in \eqref{eq:Donaldson_index_2} is one overall sign, independent of $u$.

	These formulas satisfy the relations of Theorem~\ref{thm:Juhasz_relations}. Mapping cylinders compose by functoriality of exterior powers, and covariance follows from naturality of the filling, insertion, and pairing. For disjoint handles, use a basis adapted to their attaching regions: interchanging the two operations changes only the sign from interchanging their degree-one factors or contractions. For a cancelling pair, the attaching circle meets the belt circle once, and the index 2 operation deletes precisely the factor inserted by the index 1 operation, giving the identity up to sign. Reversing a meridian changes its map by an overall sign, proving the last relation.

	The boundary-sum maps \eqref{eq:Donaldson_tensorator} are compatible with these operations. For an operation in the first factor, moving its inserted or contracted degree-one factor past $y$ contributes $(-1)^{n_2}$, which is compensated by the change of $g_1$ in the tensor sign. For an operation in the second factor, the change of $n_2$ leaves only the overall sign $(-1)^{g_1}$. Together with associativity and the disk unit, this proves monoidality in $\mathrm{Vect}_{\Bbbk}^{\pm}$.

	The mapping-cylinder action and these handle formulas are the usual FND generators; compare \cite[Section~3, equations~(13), (15)--(16)]{kerler2003homology}, interchanging the two last-handle generators when comparing the index 2 maps. 

	\begin{theorem*}
		Modules $D(\Sigma)$ for each $\Sigma\in 3 \Cob$ and morphisms $D(M)$ for every elementary cobordism $M$ in $3\Cob$ give rise to a TQFT $D:3\Cob\to \mathrm{Vect}_{\Bbbk}^{\pm}$ called (3-dimensional) Frohman--Nicas--Donaldson TQFT.
	\end{theorem*}

	\subsection{Diagrammatic calculus for twisted cycles}\label{sec:Diag_calculus}
	 We provide here only the relations needed in our computations. For a full description of diagrammatic calculus for twisted cycles see \cite{de2022homological, martel2022homological}.
	
	\begin{itemize}
		\item \textit{Fusion rule}. (Fig.~\ref{fig:Fusion_rule}) (only for Borel-Moore homology). If there are two parallel arcs with labels $k$ and $l$ and in the same fiber of the Heisenberg cover, then they can be ``fused'' into a single arc with label $k+l$, but twisted by $\sqbinom{k+\ell}{k}_q\cdot q^{k\ell}\in R[\Heis_p(\Sigma)]$;
		\begin{figure}[h]
			\centering
			\includegraphics[width=0.5\linewidth]{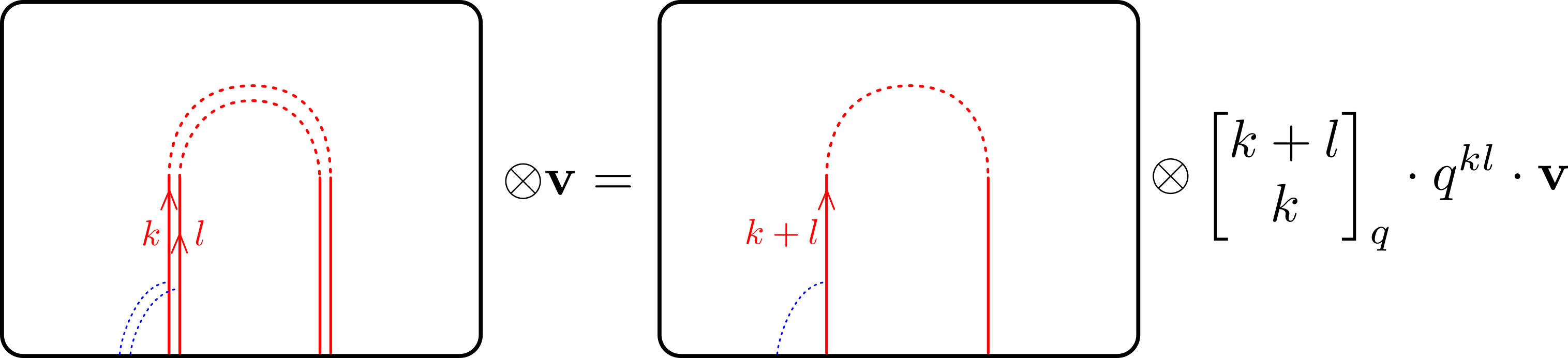}
			\caption{Fusion rule}
			\label{fig:Fusion_rule}
		\end{figure}
		\item \textit{Orientation rule}. (Fig.~\ref{fig:Orientation_rule}) The orientation of an arc with label $k$ can be switched to the opposite by multiplication by $(-1)^k$.
		\begin{figure}[h]
			\centering
			\includegraphics[width=0.5\linewidth]{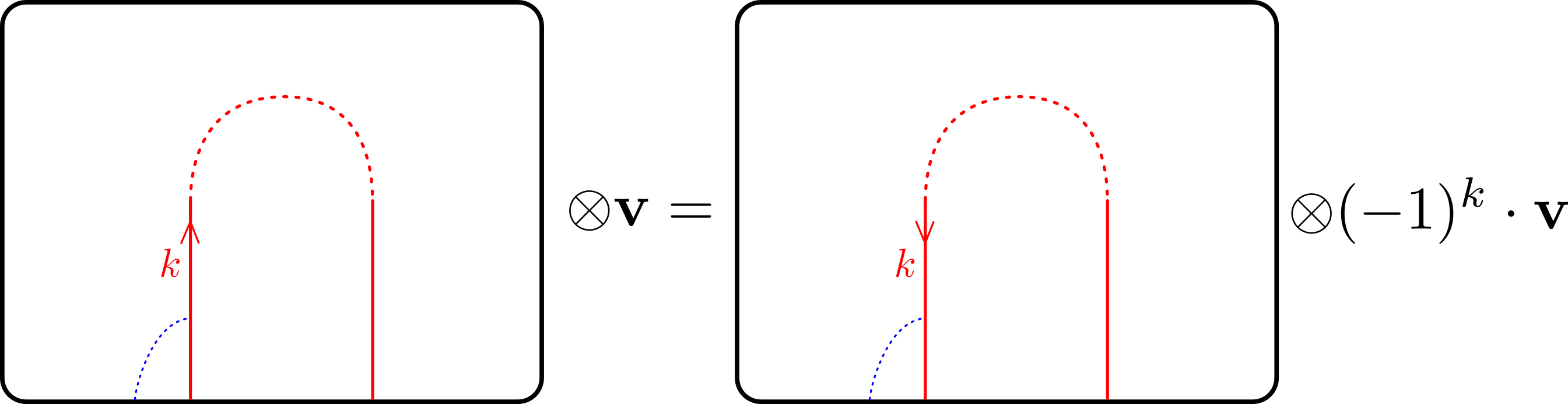}
			\caption{Orientation rule}
			\label{fig:Orientation_rule}
		\end{figure}
		\item \textit{Braid rule}. If the base point of an arc is connected to the base point $*_n$ via two paths $h$ and $h'$ then the corresponding twisted cycles are related by:
		\begin{figure}[h]
			\centering
			\includegraphics[width=0.5\linewidth]{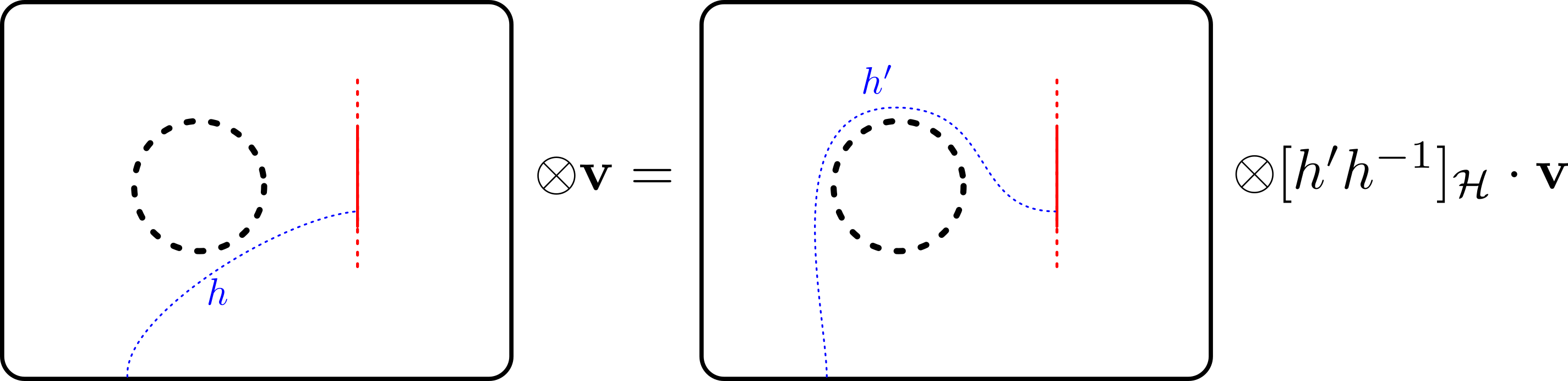}
			\caption{Braid rule}
			\label{fig:Braid_rule}
		\end{figure}
	\end{itemize}
	

	
	\subsection{Connected framed $3$-dimensional cobordisms and top tangles}
	\label{sec:ALG3}
	
	This subsection briefly recalls Kerler's category of connected framed $3$-dimensional cobordisms and its diagrammatic description in terms of top tangles in handlebodies.
	
	We use the standard skeleton of the connected cobordism category, whose objects are fixed connected oriented surfaces $\Sigma_{g}$ of genus $g$ with one boundary component, or equivalently the integers $g\in\mathbb N$. Let
	\[
	L_g\subset H_1(\Sigma_{g};\mathbb R)
	\]
	be the standard Lagrangian subspace generated by the belt circles of the $g$ one-handles.
	
	\textit{The framed cobordism category.}
	The category $3\Cob_*^{\Z}$ then has objects $g\in\mathbb N$. A morphism from $g$ to $g'$ is a pair $(M,n)$, where $M$ is the diffeomorphism class of a connected oriented cobordism from $\Sigma_{g}$ to $\Sigma_{g'}$ and $n\in\mathbb Z$ is an integer, called the \emph{signature defect}. The identity of $g$ is $(\Sigma_{g}\times I,0)$, and the composition of
	\[
	(M,n):g\to g',\qquad (M',n'):g'\to g''
	\]
	is defined by
	\[
	(M',n')\circ (M,n)
	=
	\bigl(M\cup_{\Sigma_{g'}} M',\, n+n'-\mu(M_*(L_g),L_{g'},(M')^*(L_{g''}))\bigr),
	\]
	where $M_*(L_g)$ and $(M')^*(L_{g''})$ are obtained by pushing $L_g$ forward through $M$ and pulling $L_{g''}$ back through $M'$, respectively. Thus all three Lagrangians lie in $H_1(\Sigma_{g'};\mathbb R)$, and $\mu$ denotes their Maslov index. The monoidal product is given on objects by addition,
	\[
	g\otimes g'=g+g',
	\]
	and on morphisms by horizontal gluing, that is, by boundary connected sum along the distinguished boundary components. Forgetting the signature defect gives the connected cobordism category $3\Cob_*$.
	
	In this way, $3\Cob^{\Z}_*$ is a Maslov-corrected $\Z$-extension of $3\Cob_*$: the extra integer keeps track of the framing anomaly that appears when cobordisms are composed.
	
	\textit{Top tangles in handlebodies.}
	Let $H_g^*$ denote the standard $3$-dimensional handlebody of genus $g$ depicted in Fig.~\ref{fig:handlebodies}. A \emph{top $(g,g')$-tangle} is an unoriented framed tangle $T\subset H_g^*$ such that:
	\begin{itemize}
		\item the boundary of $T$ consists of $2g'$ points placed on the top face of $H_g^*$ in the standard order;
		\item each arc component of $T$ joins the two points in one of the prescribed pairs;
		\item closed components are allowed.
	\end{itemize}
	If $\widetilde T\subset T$ denotes the union of the arc components, the composition of a top $(g,g')$-tangle $T$ with a top $(g',g'')$-tangle $T'$ is obtained by removing a tubular neighborhood of $\widetilde T$, gluing the resulting complement to the ambient handlebody of $T'$ along the boundary pattern prescribed by $\widetilde T$, and then rescaling the result back to a standard handlebody. Horizontal juxtaposition defines the monoidal product.
	
\begin{figure}[h]
		\centering
		\includegraphics[width=0.7\linewidth]{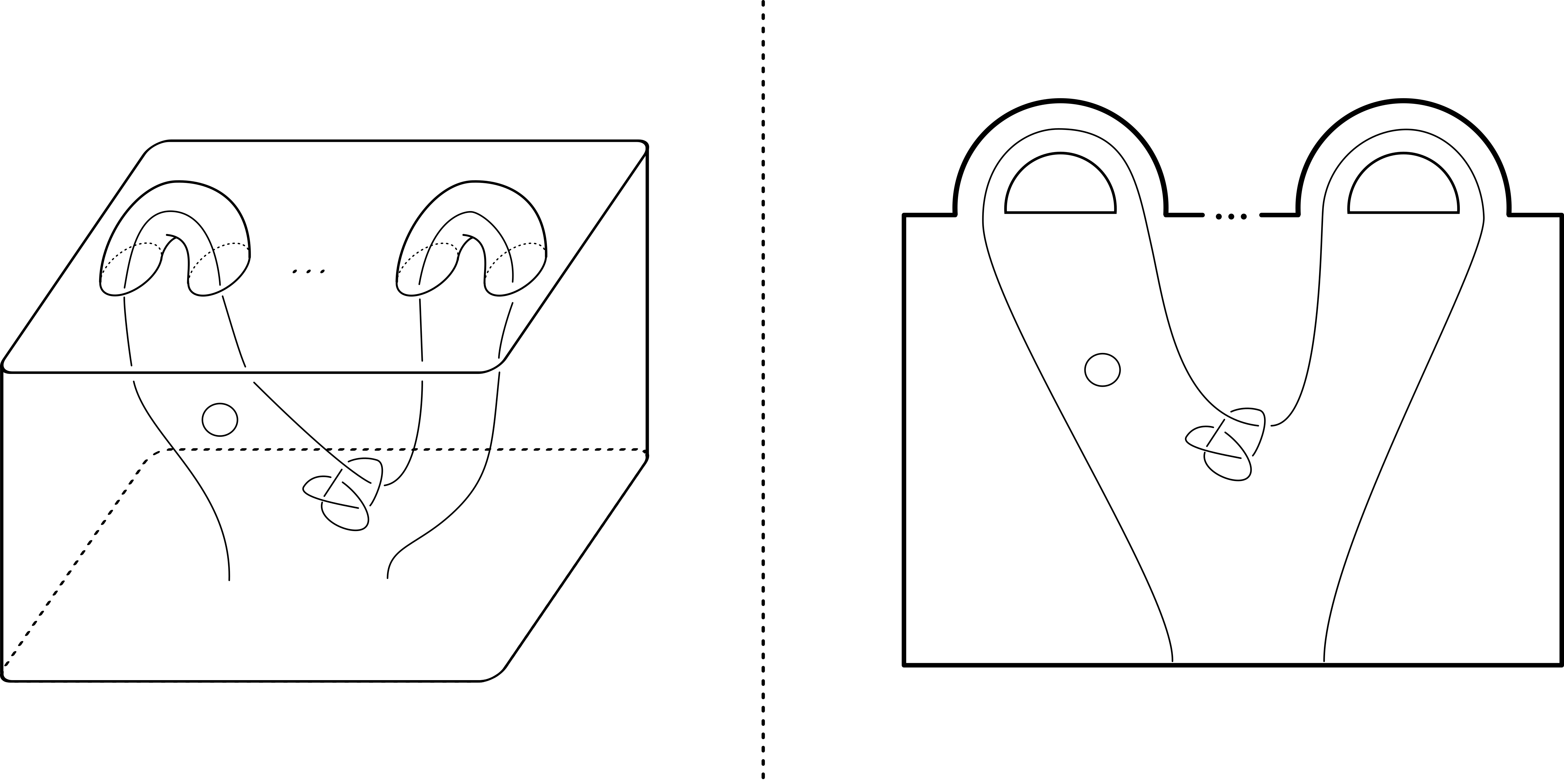}
		\caption{On the left is a tangle representing a cobordism from $g=1$ to $g'=2$. On the right its diagrammatic presentation.}
		\label{fig:handlebodies}
	\end{figure}
	
	Following the bottom-to-top convention, one obtains the category $\mathrm{TTan}^{\Z}$ of \emph{signed top tangles in handlebodies}: its objects are again natural numbers, while morphisms are pairs $(T,n)$ consisting of a top tangle and an integer defect, taken modulo isotopy and the signed Kirby moves $(K\pm 1)$, $(K2)$, and $(K3)$. Equivalently, one quotients by the usual stabilization and handle-slide moves adapted to tangles in a handlebody. Forgetting the integer defect yields the category $\mathrm{TTan}$ of top tangles in handlebodies. (With Habiro's top-to-bottom convention, these become the corresponding categories of bottom tangles.)
	
	\textit{Surgery presentation.}
	Let $T\subset H_g^*$ be a top $(g,g')$-tangle, and write $\widetilde T$ for its arc part. The associated cobordism $C(T)$ is obtained by interpreting the arc components as carving data and the circle components as surgery data: one removes a tubular neighborhood of $\widetilde T$ and performs $2$-surgery along the framed link $T\setminus \widetilde T$. This defines the \emph{signed surgery presentation functor}
	\[
	\chi^{\Z}:\mathrm{TTan}^{\Z}\longrightarrow 3\Cob^{\Z}_*,
	\qquad
	\chi^{\Z}(T,n)=\bigl(C(T),\sigma(T\setminus \widetilde T)+n\bigr),
	\]
	where $\sigma(T\setminus \widetilde T)$ is the signature of the linking matrix of the closed components of $T$. By construction, the arc components encode the boundary parametrization, while the closed components encode the $2$-handle attachments.
	
	\begin{proposition}[\cite{beliakova2024kerler}]
		\label{prop:signed_surgery_equivalence_appendix}
		The functor $\chi^{\Z}:\mathrm{TTan}^{\Z}\to 3\Cob^{\Z}_*$ is an equivalence of categories.
	\end{proposition}
	
	In particular, every connected framed cobordism admits a presentation by a top tangle in a handlebody, and two such presentations determine the same morphism of $3\Cob^{\Z}_*$ if and only if they are related by isotopy and signed Kirby moves. After forgetting signature defects, one obtains an equivalence
	\[
	\chi:\mathrm{TTan}\longrightarrow 3\Cob_*.
	\]
	Thus top tangles provide a convenient diagrammatic model for connected framed $3$-dimensional cobordisms.

	\vspace{0.5cm}

	The integral and cointegral are presented in Habiro's notation in Fig.~\ref{fig:integral_cointegal}.

	\begin{figure}[H]
		\centering
		\includegraphics[width=0.5\linewidth]{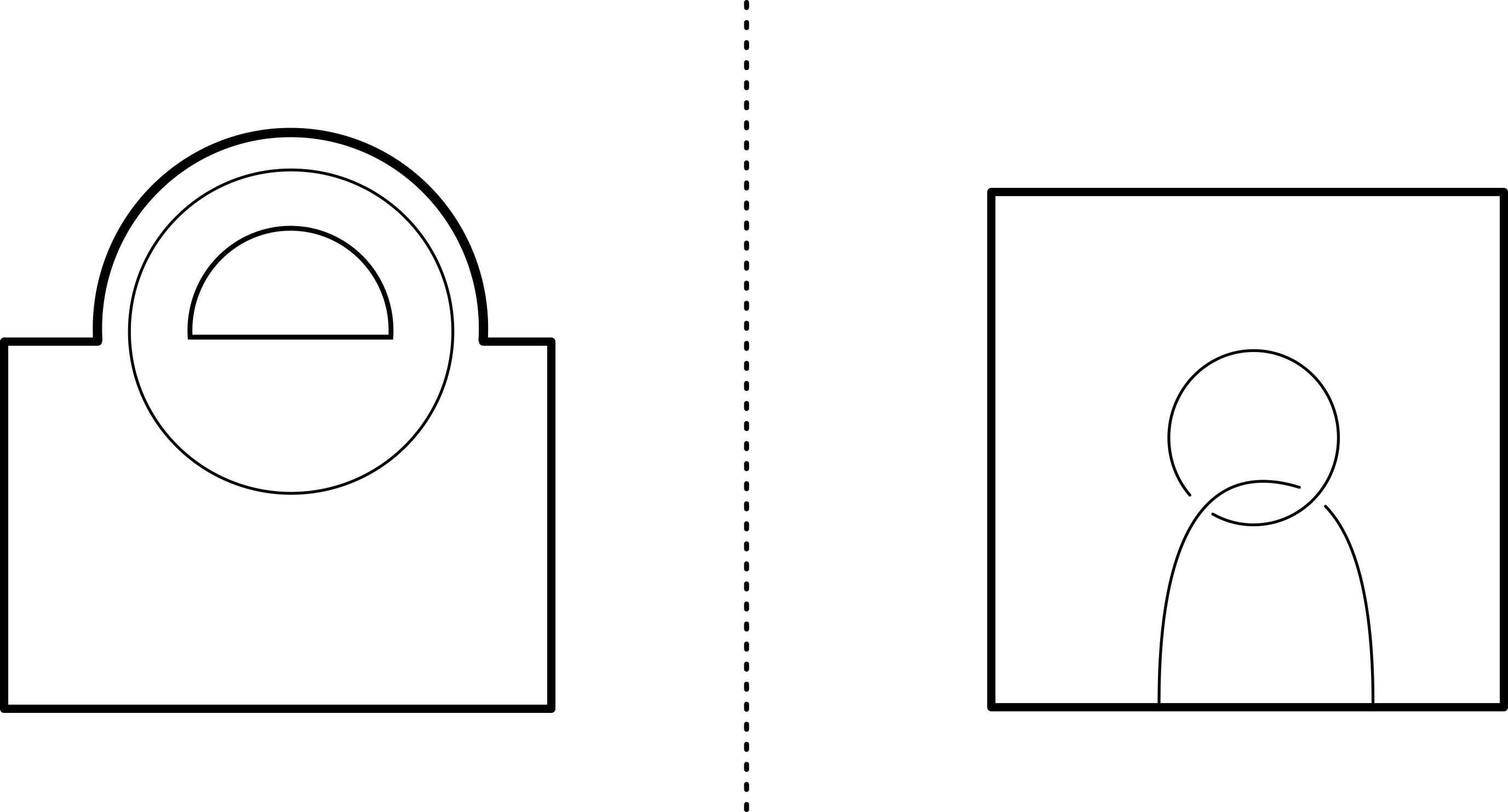}
		\caption{On the left is the tangle presentation of the integral $\lambda$. On the right is the tangle presentation of the cointegral $\Lambda$.}
		\label{fig:integral_cointegal}
	\end{figure}

	\bibliographystyle{alpha}
	\bibliography{biblio}
	
	

\end{document}